%% file: delay_main.tex
\documentclass[12pt]{article}
\usepackage{amsmath,amsfonts,amssymb,amscd}
\usepackage{graphicx}
\usepackage[T2A]{fontenc}
\usepackage{color}
\usepackage{float}

\newtheorem{theo}{Theorem}
\newtheorem{lem}{Lemma}[section]

\newtheorem{cor}{Corollary}[section]

\makeatletter \@addtoreset{equation}{section} \makeatother

\newcommand{\grad}{\mathop{\rm grad}\nolimits}

\newcommand{\mR}{\mathbb{R}}

\newcommand{\eps}{\varepsilon}
\newcommand{\ph}{\varphi}
\newcommand{\thet}{\vartheta}

\newcommand{\ze}{\zeta}
\newcommand{\p}{\partial}

\newcommand{\Imm}{\operatorname{Im}}

\newcommand{\im}{\operatorname{Im}}

\newcommand{\re}{\operatorname{Re}}

\newcommand{\dK}{\mathrm K}

\begin{document}

\title
{On Maximal Delay of Stability Loss  for Dynamical Bifurcations}
\author{Anatoly Neishtadt\thanks{Department of Mathematical Sciences, Loughborough University, Loughborough, LE11 3TU, UK
} }
\date{}
\maketitle

\begin{abstract}
We consider a dynamical bifurcation caused by a slow passage through a static bifurcation point: in a system depending on a parameter, the parameter changes slowly in time and passes through the critical value corresponding to the loss of stability of an equilibrium via a Poincar\'e--Andronov--Hopf bifurcation in the frozen system.
 If the system is analytic, then the loss of stability is inevitably delayed: phase points attracted to the equilibrium in the stability region remain near the equilibrium for a long time after entering the instability region, so that the parameter changes by an amount of order ~1 independently of how slow the variation of the parameter is. Remarkably, there exists a {\it maximal delay}: all phase points attracted to the stable equilibrium before a certain threshold value of the parameter leave a neighbourhood of the unstable equilibrium almost simultaneously near another threshold value of the parameter, known as  {\it a buffer point}. A delay of stability loss beyond the buffer point is impossible unless the initial data have a very special form. We assume that, although the equilibrium is non-degenerate for real values of the parameter, one of its eigenvalues vanishes generically for some complex value of the parameter (a complex analogue of a saddle-node bifurcation), and that this complex singularity is, in a suitable sense, the closest one to the real Poincar\'e--Andronov--Hopf bifurcation point. We show that  the value of maximal delay is determined by this complex singularity: the threshold values defining  the maximal delay are  the intersection points of the Stokes lines associated with this singularity and the real axis. We study these phenomena in the framework of slow--fast dynamical systems.
\end{abstract}

\section{Introduction} 
\label{intro}

There is a vast variety of applied problems leading to study of dynamical systems with phase variables changing on different time scales ({\it slow-fast dynamical systems} with {\it slow and fast variables}). {\it Stability loss delay} is a remarkable feature of dynamical bifurcations in such systems. Existence of maximal delay (known as {\it a buffer point}) is an important phenomenon in stability loss delay that is related to singularities of solutions in complex time in a way which is still to be understood. In this paper we provide an  asymptotic description of the maximal delay in stability loss phenomenon at slow passage through a bifurcation in a generic setting.
\medskip

Finite dimensional slow-fast dynamical systems with continuous time are described by ODEs  
   \begin{eqnarray}
\label{perturbed}
     \frac{dx}{dt}&=&f(x,\kappa,\varepsilon ), \quad   x \in \mR^n,  \\
\frac{d\kappa}{dt}&=&\varepsilon g(x,\kappa,\varepsilon ), \quad  \kappa\in \mR^m.
 \nonumber
\end{eqnarray}
\noindent 

Here  $x$ and $\kappa$ are {\it fast} and {\it slow} variables. Small positive parameter $\eps$ characterises ratio of time scales of these variables. The first equation in (\ref{perturbed}) for 
$\kappa={\rm const}$  and  $\eps=0$ is called {\it the fast system}.
\medskip

Let $x=X(\kappa)$ be a non-degenerate equilibrium of the fast system: 
$f(X(\kappa), \kappa,0)\equiv 0$ and eigenvalues $\lambda_1(\kappa), \lambda_2(\kappa),\ldots, \lambda_n(\kappa)$ of this equilibrium are different from 0. Dynamics of slow variables near this equilibrium is approximately described by {\it the slow system}
 \begin{equation}
 \label {slow}
\frac{d\kappa}{dt} = \varepsilon g(X(\kappa),\kappa,0),\quad  x= X(\kappa).
\end{equation}
Assume that in the domain of slow variables there exists a bifurcation surface that separates regions of asymptotic stability and instability of this equilibrium. We say that there is a dynamical bifurcation with stability loss if drift of slow variables in (\ref{slow}) leads to crossing the bifurcation surface from the stability to instability region. Typically, two complex-conjugate eigenvalues cross imaginary axis (a Poincaré-Andronov-Hopf bifurcation in the fast system). Consider phase points of (\ref {perturbed}) that are attracted to the equilibrium of fast system in the stability region at the distance of order 1 from the bifurcation surface.  We say that there is {\it  a delay of stability loss} if these phase points stay near the equilibrium in the instability region during the time of order $1/\eps$. 

\medskip
The buffer point phenomenon consists of the following. Fix a solution $\kappa=\dK(\tau), x=X(\dK(\tau)), \tau=\eps t$ of the slow system.  Let $\tau_*$ be a bifurcation value of {\it the slow time} $\tau$, i.e. $\dK(\tau)$  crosses the bifurcation surface at $\tau=\tau_*$. The value $\tau_*^+ > \tau_*$  is {\it a buffer point} if  there exists a slow time moment
$\tau_- < \tau_*$  such that all the phase points of (\ref{perturbed})  attracted to a small neighbourhood of equilibrium $X(\dK(\tau))$ before time $\tau_*^-$ will leave  a neighbourhood of the equilibrium near $\tau=\tau_*^+$. Thus, stability delay beyond time 
$\tau_*^+$  is not possible for these phase points. Value $\tau_*^+-\tau_*$   is  {\it a maximal delay}. Phase points of (\ref{perturbed}) attracted to a small neighbourhood of equilibrium $X(\dK(\tau))$  at times  $\tau$   between  $\tau_*^-$ and $\tau_*$ will leave a neighbourhood of equilibrium at times between  $\tau_*$   and $\tau_*^+$.  Stability delay beyond time $\tau_*^+$ occurs for phase points with special initial data. Those are phase points attracted to a small neighbourhood of equilibrium $X(\dK(\tau))$ close to the time $\tau_*^-$. 

\medskip
Theory of stability loss delay was started in Pontryagin’s school by Shishkova in \cite {sh}, where an example of this phenomenon in a non-trivial ($X(\dK(\tau))  \not \equiv {\rm const}$) setting was studied using methods of the dynamics in complex time. Such an approach and its modifications turned out to be a powerful tool in analysis of this phenomenon.  

\medskip
Systematic development of a general theory of the stability loss delay phenomenon was started in  \cite{nei1, nei2.1,nei2.2,ber}. It was shown that if the system is analytic, the stability loss delay necessarily occurs \cite{nei1, nei2.1}. It turns out that the analyticity is a property responsible for the stability loss delay phenomenon. There are examples of systems of smoothness $C^{\infty}$ for which stability loss does not delay: an escape from an equilibrium occurs inside an interval of slow time that shrinks to a bifurcation value as $\eps \to 0$  \cite{nei2.1}. Asymptotic expressions for delay time were obtained in \cite{nei2.2} under certain conditions. Namely, the analytic continuation of slow system solutions along some paths in complex time plane was considered, and it was assumed that the family of these paths is separated from singularities. These complex time singularities are related to existence of a buffer point. From the viewpoint of motion in real time, the buffer point phenomenon is related to an exponentially small mismatch of bounded backward and forward in time solutions \cite{dien}.
\medskip

Stability loss delay accompanies also dynamical bifurcations of periodic trajectories in analytic slow-fast ODEs and fixed points in analytic slow-fast maps \cite{nei1, F91, F92, NST, Su96, Su97, FS}. 
Study of stability loss delay for infinite dimensional dynamical systems is a subject of active research in \cite{Su94, TWK, KV, GKV, AACS, TLFK}. 

Stability loss delay appears in diverse applications, including mechanics, laser physics, chemistry, epidemiology, ecology, neural models, climate modelling, and biochemistry (see references in \cite{NT}).

\medskip
Maximal delay/buffer point was for the first time considered (without using these names) in a particular example in \cite{sh}. In this example there is a complex slow time point $\tau_c$  where an eigenvalue of an equilibrium vanishes.  Points $\tau_*^-$ and $\tau_*^+$ turned out to be points of intersection of Stokes lines passing through $\tau_c$  and the axis of real time. At first glance, the example in \cite{sh} looks as a typical one as it presents a normal form for Poincaré-Andronov-Hopf bifurcation with slowly time dependent parameters. However, in this example position of equilibrium remains an analytic function of complex slow time at  $\tau_c$, while generically  $\tau_c$ should be a branch point of order 2 for this position (a complex counterpart of saddle-node bifurcation). It turns out that the example in \cite{sh} is a degenerate one: two coefficients in expansion near $\tau_c$ vanish while generically they should have non-zero values.  In this paper, we consider the generic case: an eigenvalue vanishes at  $\tau_c$, and the equilibrium $X(\dK(\tau))$  of the fast system has a branch point of order 2 at $\tau_c$. We do this for slow-fast systems of general form (\ref{perturbed}). The result is that, similarly to \cite{sh},  $\tau_*^-$ and $\tau_*^+$ are points of intersection of Stokes lines passing through $\tau_c$   with the axis of real time. Values $\tau_*^-$ and $\tau_*^+$ describe maximal delay phenomenon in the limit as $\eps \to 0$. We  show that accuracy with which $\tau_*^+$ provides the value of maximal delay is  $O(\eps|\ln \eps|)$.

\medskip
Our methodology is based on study of solutions of systems (\ref{perturbed}) in complex time. Let the solution $\kappa=\dK(\tau), x=X(\dK(\tau)), \tau=\eps t$ of the slow system (\ref{slow}) be defined on an interval 
$\tau \in [\tau_0,\tau_1]\subset \mR$.  Let $\lambda_1(  \dK(\tau)), \lambda_2(  \dK(\tau))$ be two complex conjugate eigenvalues of the equilibrium $X(\dK(\tau))$ of the fast system. Assume that real parts of these eigenvalues are negative at  $\tau< \tau_*$, zero at $\tau= \tau_*$, and positive at $\tau> \tau_*$; 
 here  $\tau_*  \in [\tau_0,\tau_1]$ is the bifurcation value of slow time.  Other eigenvalues have negative real parts for $\tau  \in [\tau_0,\tau_1]$. Assuming analyticity of functions $ f,g$ in (\ref{perturbed}) with respect to $x,\kappa$, consider systems (\ref{perturbed}) , (\ref{slow})  for complex values of phase variables and time. In the complex time plane consider a path $\tau=\Gamma(s)$ parametrised by a real variable $s$. Take $s/\eps$  as a new time for the fast system. Then eigenvalues of the equilibrium $X(\dK(\Gamma(s)))$ are $\Lambda_j(s)=\Gamma'(s)   \lambda_j(  \dK(\Gamma(s))), j=1,2,\ldots, n$.   Let us call path $\tau=\Gamma(s)$ {\it iso-expanding}\footnote{This name is suggested by Prof Carles Simo.}, if $\Lambda_1(s)$ is an imaginary number. For motion in real time $t$, initially there is a contraction in all fast variables (for $\eps t<\tau_*$) and then an expansion in two fast variables (for $\eps t>\tau_*$).  For motion along an iso-expanding path there is neither contraction nor expansion in one of fast variables in the linearised system. There is a smooth family of iso-expanding paths that connect points of stable ($\eps t<\tau_*$) and unstable ($\eps t>\tau_*$) parts of the real time axis near $\tau_*$. Study of dynamics when time changes along such a path (see \cite{nei2.2}) provides a proof of existence of stability loss delay and asymptotic formulas for time of delay. Points of the real time axis connected by such a path are (asymptotically) moments of time of attraction to the equilibrium  $X(\dK(\tau))$ and of escape from it. Continuation of iso-expanding paths family away from $\tau_*$ in the domain of analyticity of functions $f,g$  is prevented by singularities of paths. Typical singularities are related to a zero eigenvalue or a double eigenvalue: $\lambda_1(  \dK(\tau_c))=0$,  or
 $\lambda_1(  \dK(\tau_c))=\lambda_2(  \dK(\tau_c))$ at some complex point $\tau_c$. In this paper we consider the case of a zero eigenvalue. Iso-expanding path through $\tau_c$ is responsible for the buffer point phenomenon.

Generically, $\tau_c$ is a branch point of order 2 for $X(\dK(\tau))$. Dynamics near $\tau_c$ is approximately described by a complex Riccati equation whose real counterpart plays a fundamental role in relaxation oscillations theory \cite{mr}. In complex time this equation is considered in \cite{C}. Its role for determining  maximal delay is discussed in \cite{nei95}. We study an effect of this dynamics upon the real time dynamics. This is  achieved via analytic continuation of solutions along iso-expanding paths.  To avoid growing norm solutions we will use the following approach (see \cite{nei2.2}): complex system of $n+m$ differential equations is replaced by a system of $n+m-1$ equations with a complex non-constant delay using the real-analyticity conditions. In study of analytic continuation we use a perturbative approach. Outside $\sim \eps^{2/3}$ - neighbourhood of  $\tau_c$  this is a perturbation relative to solutions of the linearised near the equilibrium equation. 
In this neighbourhood this is a perturbation relative to solutions of the Riccati equation. Choice of $\sim \eps^{2/3}$ - neighbourhood of  $\tau_c$  is determined by the condition that both perturbative processes provide the same accuracy at the boundary of this neighbourhood. On the boundary of this neighbourhood one should  switch from an iso-expanding path determined by linearisation near equilibria of the fast system to an iso-expanding path determined by    the Riccati equation. Inside $\sim \eps^{2/3}$ - neighbourhood of  $\tau_c$  we consider perturbations near two special solutions of the Riccati equation. A special solution of the Riccati equation is a solution that approaches an equilibrium of the fast system constructed for the Riccati equation as time changes along an iso-expanding path through $\tau_c$.  We switch from one special solution to another at  $\tau_c$.  This switch is a cornerstone of the methodology.  At this switch, at  $\tau=\tau_c$,  the phase point of the original system (\ref{perturbed}) turned out to be at a distance $\sim \eps^{1/3}$  from a special solution of the Riccati equation. As a result, for moving along an appropriate iso-expanding path $\Gamma_{+,\eps}$, far from $\tau_c$,  this phase points remains at a distance $\sim \eps^{1/3}$ from the equilibrium of the fast system.  It keeps this distance up to the point $\tau_{+,\eps}$  of intersection of this iso-expanding path with the axis of real time.  At the slow time moment   $\tau_{+,\eps}$ the phase point is in the instability region: $\re \lambda_{1,2} (\dK(\tau_{+,\eps}))>{\rm const}>0$. Thus, during the real slow time interval of length $\sim \eps\ln \eps$ centred at $\tau_{+,\eps}$  the distance of the phase point from the equilibrium of the fast system grows from a value $\sim \eps^{1/3}$ to a value $ \sim 1$. Therefore, 
$\tau_*^+=\tau_{+,\eps} +O(\eps\ln \eps)$. This constitutes the principal result of the paper, Theorem \ref{t1.th} in Section \ref{tosld}.



\section{ Formulation of conditions }
\label{form_conditions}   
 
The natural framework for dynamical bifurcation theory is that of systems with
slow and fast variables (slow-fast systems) (\ref{perturbed}). Here $ (x, \kappa) \in D \subset \mR^{n +m}$, and 
$|\varepsilon| < \varepsilon_1 ={\rm const}$.  The fast system is the first equation  in
 (\ref{perturbed}) with $\kappa={\rm const}$ and $\varepsilon =0$. 
 We suppose that for all $\kappa$
in the projection of  $D$ onto the $\kappa$-space the fast system  has an 
equilibrium position $x=X(\kappa)$ depending continuously on $\kappa$. Let
$\lambda_i(\kappa)$, $i=1,2,...n$, be eigenvalues of this equilibrium. For the {\it slow system} (\ref{slow})
we fix a solution $\kappa=\dK(\tau)$, $x=X(\dK(\tau))$ with the slow time  $\tau = \varepsilon t
 \in [\tau_0,\tau_1]$, $\varepsilon > 0$, and consider behavior of
 the eigenvalues $\lambda_i(\kappa)$, $i=1,2,...n$ along  it. Let $\lambda_1(\dK(\tau))$ be in the left complex
half-plane for $ \tau < \tau_* \in[ \tau_0,\tau_1]$, and cross
imaginary axis for $\tau =\tau_*$ with a nonzero velocity. Let
 $\lambda_1(\dK(\tau_*)) \ne 0$, and $\lambda_2(\dK(\tau_*))
 =\overline{\lambda_1(\dK(\tau_*))}$ (the bar here and below indicates 
 complex conjugation). Let $\lambda_i((\dK(\tau))$, $i=3,\ldots,n$, be in the
left complex half-plane for all  $\tau \in[ \tau_0,\tau_1]$. We assume that the right-hand sides 
in  (\ref{perturbed}) can be continued analytically with respect
 to $x$, $\kappa$ and smoothly in $\varepsilon $ into a complex neighbourhood
 of the point
  $(X(\dK(\tau_*))$, $\dK(\tau_*) )$,  and this neighbourhood does not depend on $\varepsilon$.
 Then 
there is a stability loss delay in  system (\ref{perturbed})
 \cite {nei1, nei2.1}.
 
To formulate  conditions of the main theorem  of this paper we use constructions
 related to  analytic continuation of the slow solution $X, \dK$. 
We may assume that ${\rm Im}\,\lambda_1(\dK(\tau_*))  < 0$. For
$ \tau  \in[ \tau_0,\tau_1]$, let us introduce the complex phase $\Psi (\tau) =
\int\limits_{\tau_*}^{\tau} \lambda_1(\dK(\vartheta)) d\vartheta$. The function $ {\rm Re}\, \Psi$ has a minimum at $\tau = \tau_*$
 since at this point 
${\rm Re}\,\lambda_1(\dK(\tau))  $ changes sign from negative to positive.
Let $\Pi $ be the function that maps
a slow time moment $\tau_- < \tau_*$ into the slow time moment
$\Pi(\tau_-) > \tau_*$ such that $ {\rm Re}\, \Psi(\tau_-) =
 {\rm Re}\, \Psi(\Pi(\tau_-) )$. The function $\Pi$ is well defined on an interval
$[\hat\tau, \tau_*], \quad \hat\tau< \tau_*$.  The slow solution
can be  analytically continued 
into some neighbourhood of the point $\tau_*$ in the plane of the
complex slow time $\tau$. The function $\Psi$ can be  continued into
the same neighbourhood.  Points $\tau_-$
and $\Pi(\tau_-)$ are connected by an arc $\Gamma_{\tau_-}$ of the 
level curve  $ {\rm Re}\, \Psi= {\rm const} $ in the upper half-plane of complex slow time. If $\tau_-$ is
sufficiently close to $\tau_*$, then in the domain $K_{\tau_-}$
bounded by $\Gamma_{\tau_-}$ and its complex conjugate 
${\overline\Gamma}_{\tau_-}$ (Fig.\, \ref{curves_1}) the following conditions are 
satisfied: 1) the slow solution $X,\,\dK$ is analytic; 2) the
right-hand sides of  (\ref{perturbed}) are analytic at the points of
the slow solution; 3) $\lambda_{1,2}(\dK)\ne 0$;
 4) $\lambda_1(\dK) \ne \lambda_2(\dK)$, $\lambda_i(\dK) \ne \lambda_j(\dK),
 \quad i=1,2,\quad j=3,\ldots, n$;  5)  ${\rm Re}\,\lambda_j(\dK)  < 0,
\quad j=3,\ldots , n$, for real $\tau$, and the linearized near the
 equilibrium fast system considered along any curve   
$ {\rm Re}\, \Psi(\tau)= {\rm const} $ has $(n-2)$ eigenvalues with
 negative real parts corresponding to the eigenvalues
$\lambda_3,\ldots ,\lambda_n$;
 6) the tangent lines to the curves $\re\Psi={\rm const}$ are not vertical.
 \begin{figure} [H]
\begin{center}
            \includegraphics[scale=0.6, angle=0.0]{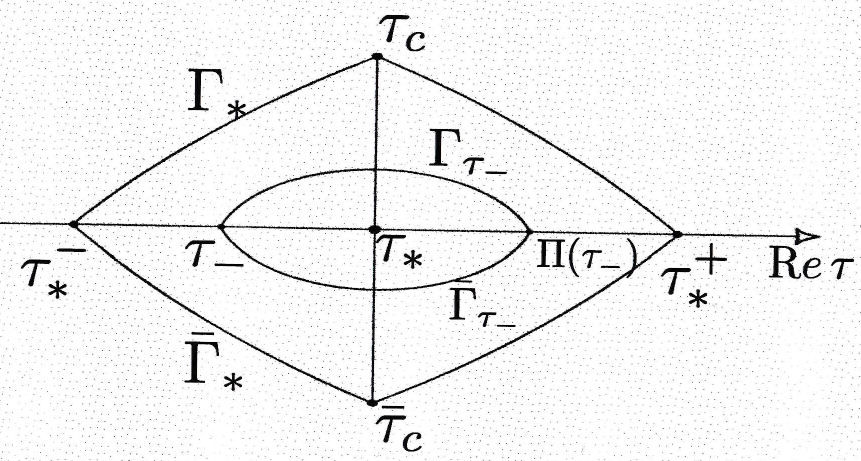}
            \end{center}
           \caption{Curves in the complex slow time plane}
            \label{curves_1}
\end{figure}
Let $\tau_*^-$ be the lower bound
of the values $\tau_-$ for which conditions 1) - 6)  hold for 
$K_{\tau_-}$. Denote
$$
\tau_*^+ = \sup_{\tau_*^- < \tau < \tau_* } \Pi (\tau), \quad 
 K_* =\bigcup_{\tau_*^- < \tau_- < \tau_* } K_{\tau_-}. \quad
$$

{\bf Remark.} Condition 6) in the definition of $\tau_*^-$ can be relaxed. It is sufficient to require that there exists a family  of nonintersecting symmetric with respect to the real axis smooth curves which are  transversal to  curves  $\re \Psi={\rm const}$, and this family  covers  $K_{\tau_-}$.

\medskip 
On the boundary of the region $K_*$ one of the conditions 1) - 6) 
should be violated. We will suppose, that this is condition
3), i.e. $\lambda_1=0$ at some point $\tau_c$ of the
boundary of  $K_*$ in the upper half-plane,
and then $\lambda_2=0$ at the  complex conjugate point $\bar \tau_c$. We will assume also that the 
right hand-sides of the system (\ref{perturbed}) are
 in a general position in a neighbourhood of the point $ X(\dK(\tau_c)), \dK(\tau_c)$. More 
precisely, we suppose that the following assumptions A - E 
and general position conditions F, G are satisfied.
 
 \medskip
A. Values $\tau_*^-, \tau_*^+$ are finite.
 
B. The component $\Gamma_*$ of the $\partial K_*$ in the upper half-plane
is a connected curve. Denote $\bar \Gamma_*$ the complex 
conjugate to  $\Gamma_*$.
 
C. The slow solution is continuous on $\partial K_*$. 
 
D. There is a unique point $\tau_c \in \Gamma_*$ such that
 $\lambda_1(\dK(\tau_c))=0 $. We assume that  $\lambda_j(\dK(\tau_c))\ne 0,\, j=2,\ldots,n. $ 
(Then $\lambda_2(\dK(\bar\tau_c))=0 $, $\lambda_j(\dK(\bar\tau_c))\ne 0,\,
 j=1,3,\ldots,n. $)

E. Conditions 2, 4 - 6 are satisfied in the domain $K_* \cup\partial K_*$.
 
\medskip
Denote 
$$
\kappa_c=\dK(\tau_c),\quad x_c=X(\kappa_c),\quad
 A_c=\partial f(x_c,\kappa_c,0)/\partial x .
$$
 
\noindent
 According to condition D,  
 matrix $A_c$ has one zero eigenvalue.  
Let $z_1$ be the projection of the vector $(x-x_c)$ onto the
 eigenvector of the matrix $A_c$ which
corresponds to this eigenvalue. Let $z_2$ be the projection
of the vector $(x-x_c)$ onto the eigenspace of the matrix $A_c$ which
corresponds to all other eigenvalues. Then 
\begin{equation}
\label{de_z1}
\dot z_1=b \cdot (\kappa-\kappa_c) +az_1^2 +O(\varepsilon +|z_1|^3+
|\kappa-\kappa_c|(|z_1|+|z_2|)+|z_1||z_2|+|z_2|^2+|\kappa-\kappa_c|^2).
\end{equation}
Here and below, the norm $|\ast|$ is defined as the sum of the absolute values of the components of a vector, and ``$\cdot$'' denotes the scalar product.

\medskip

\noindent
We will consider the problem under the following conditions of general position:
 
F. $a\ne 0$.
 
G. $b\cdot g(x_c,\kappa_c,0)\ne 0$.
\medskip

Under conditions F and G, the equilibrium $X(\dK(\tau))$ has a second-order branch  point at $\tau_c$. We formulate this property as the following statement.
\begin{lem} 
\label{FG}
The functions $X(\dK)$ and $\dK$ can be analytically continued into a neighbourhood of the point $\tau_c$ 
as double-valued functions that admit Puiseux expansions in
  $(\tau-\tau_c)^{1/2}$: 
\begin{equation}
\begin{aligned}
&X(\dK(\tau))=x_c+O(|\tau-\tau_c|^{1/2}),\
\dK(\tau)=\kappa_c+g_c(\tau-\tau_c)+O(|\tau-\tau_c|^{3/2}),\\
&Z_1=(a^{-1}b\cdot g_c(\tau-\tau_c))^{1/2}+O(|\tau-\tau_c|).
\end{aligned}
\end{equation}
Here $Z_1$ is the value of the coordinate $z_1$ at the equilibrium $X$, and $g_c=g(x_c,\kappa_c,0)$.
\end{lem}
We omit the standard, albeit rather long, proof of this lemma. The main steps of the proof are to express $z_2$ in terms of $z_1$ and $\kappa$ using the equation $\dot z_2=0$, and to use the Weierstrass preparation theorem to reduce the equation $\dot z_1=0$ to a quadratic equation for $z_1$ with coefficients depending on $\kappa$.

\bigskip
\noindent {\bf Remarks.}
\noindent
1. Problems in \cite{sh, karim, dien} have $a=0$, $b=0$. 
 
\noindent
 2. Maximal delay in the case $a=0$, $b=0$ was studied using geometric desingularisation in \cite{hayes, christ_krist}.
 
\noindent
3. The level curves $ {\rm Re}\, \Psi = {\rm const}$ passing through the point
$\tau_c$  are referred to in the literature as both  Stokes and anti-Stokes  lines. We call them  Stokes  lines, following  the terminology of \cite {fedoryuk}.
 \medskip
 
  In the proof of Theorem \ref{t1.th}  below we will use one  more  condition denoted as  H. This condition simplify the reasonings, but the theorem is valid without this condition.
 
 
 Clearly 
 \begin{equation}
 \label{cond_H1}
 \lambda_1(\dK(\tau))\ne 2\lambda_2(\dK(\tau)), \lambda_2(\dK(\tau))\ne 2\lambda_1(\dK(\tau))   
 \end{equation}
 for $\tau$ close to $\tau_c, \bar \tau_c$ and for real $\tau$. We will assume that
 
 H. Condition (\ref {cond_H1}) is satisfied for $\tau\in K_* \cup\partial K_*$.
 
 

   \medskip
  
  There will be a comment in the proof about an addition to the proof required to avoid using  the condition H.
  
   
   \medskip
   For simplicity in the formulations of intermediate results we will assume the following: $\re\tau_c=\re\tau_*$; along the curves  $\re \Psi(\tau)={\rm const}$ the function  $\im \tau$ attains a unique maximum, and this maximum occurs at $\re\tau=\re \tau_c$. 
      
   \medskip
 The curve $\Gamma_*$ consists of two smooth components that meet  at $\tau_c$ at an angle of $120^{\circ}$.  We will denote them as $\Gamma_{*,1}$ (the left one) and 
  $\Gamma_{*,2}$ (the right one).
  

\section{Time of stability loss delay}
\label{tosld}
 
Let $(x(t), \kappa(t))$ be the phase point of the system (\ref{perturbed})
with initial condition $ \kappa(t_0)=\kappa_0=\dK(\tau_0)$,  $t_0 =\tau_0/\varepsilon$.
 
\begin{theo} 
\label{t1.th}
If $\tau_0 <\tau_*^- -C_1\varepsilon |\ln \varepsilon|$ and the initial point
 $(x(t_0),\kappa(t_0)) $ is in $C_2^{-1}$-neighbourhood of the equilibrium
$(X(\kappa_0),\kappa_0) $, then
 
a) for $\tau_0 + C_1\varepsilon |\ln \varepsilon| \le\varepsilon t \le \tau_*^+ -C_3\varepsilon |\ln \varepsilon|$
the phase point $(x(t),\kappa(t))$ is in a $C_4\varepsilon$-neighbourhood of the equilibrium
$(X(\dK(\varepsilon t)),\dK(\varepsilon t))$,
 
b) there exists $t_d$ such that $|\varepsilon t_d -\tau_*^+| < C_3\varepsilon |\ln \varepsilon|$ 
and $|x(t_d)-X(\kappa(t_d))|>C_5^{-1}$. 
 
\end{theo}
 
\noindent Here and henceforth $C_i, c_i, k_i$ are positive constants. The appearance of $C_i$ in some relation is equivalent to the assertion
that there exists $C_i$  satisfying this relation for small enough $\eps > 0$ (and
similarly for other constants).
 
 \medskip
According to Theorem \ref{t1.th} all the phase points which were
attracted to the slow solution before the moment of  slow time 
$\tau_*^-$ will depart from the slow solution in $O(\varepsilon |\ln \varepsilon|)$-neighbourhood
of the moment of slow time $\tau_*^+$. Hence, $\tau_*^+$ is {\it a buffer point}.

Note that the last statement of Theorem \ref{t1.th} implies that $|x(t_d)-X(\dK(\eps t_d))|>C_6^{-1}$. (Indeed, $\dot \kappa=O(\eps), \dot \dK=O(\eps)$.  Thus, $\kappa(t_d)-\dK(\eps t_d)=O(\eps\ln\eps)$.)
 
The proof is based on analytic continuation of the solution in a certain domain
in the complex slow time  plane. This domain is close to the domain $K_*$. Outside
$O(\varepsilon^{2/3})$-neighbourhoods of the points $\tau_c$, $\bar\tau_c$ in this
 domain, the motion can be described by perturbation theory as a perturbation 
of the slow solution. In $O(\varepsilon^{2/3})$-neighbourhoods of the points $\tau_c$, $\bar\tau_c$ the system can be considered as a perturbation of some 
auxiliary Riccati equation. This equation (for real time) plays an important
role in the theory of relaxation oscillations, and its solution can be
 expressed via Bessel functions. We will consider this equation in
 Section \ref{s_riccati}.
 
 \medskip
 {\bf Remark.} The phase points which were attracted to the slow solution near the moment of the slow time 
$\tau_0\in (\tau_*^-,\tau_*)$ will depart from the slow solution near the moment of the slow time $\Pi(\tau_0)$ \cite{nei1, nei2.2}. 

\medskip
Maximal delay in the situation of Theorem \ref{t1.th} was considered in \cite{nei95}. Numerical demonstration of maximal delay phenomenon in this situation is contained in \cite{ns}.

 \section{Auxiliary constructions}
  \label{aux_constr}
 \subsection{Some curves in the plane of the slow time}
 \label{some_curves}
 Recall that the curve $\Gamma_*$ consists of two smooth components which meet  at $\tau_c$ at an angle of $120^{\circ}$.  They are denoted as  $\Gamma_{*,1}$ (the left one) and $\Gamma_{*,2}$ (the right one). 
  
 Now, consider the slow-fast system obtained by expanding and truncating the original system near the point
  $(X(\kappa_c),   \kappa_c)$:    
\begin{equation}
\label{eq_expanded}
\begin{aligned}
&\dot z_1=b \cdot (\kappa-\kappa_c) +az_1^2 ,\\
&\dot \kappa=\eps g_c,  \ g_c= g(x_c,\kappa_c,0).
\end{aligned}
\end{equation}
(Recall that $z_1$ is the projection of the vector $(x-x_c)$ onto the
 eigenvector of the matrix $A_c$ corresponding  to zero eigenvalue.)
For $\kappa$, we consider the solution $\kappa=\dK_a(\tau)=g_c(\tau-\tau_c)+\kappa_c$.
Take the equilibrium of the fast system ${Z_1}_a(\kappa)= \sqrt{-a^{-1}b \cdot (\kappa-\kappa_c)}$. The eigenvalue of this equilibrium is ${\Lambda_1}_a(\kappa)=2a{Z_1}_a= 2a\sqrt{-a^{-1}b \cdot (\kappa-\kappa_c)}$. We should choose an appropriate  branch of the square root to have an agreement between $X(\dK(\tau))$ and ${Z_1}_a(\dK_a(\tau))$ for 
$\tau$ close to $\tau_c$ in $K_*$.
 This choice is described below.

Introduce the complex phase
\begin{equation}
\label{phase_a_1}
   \psi_{a}(\tau)
 = \int_{\tau_c}^{\tau}{\Lambda_1}_a(\dK_a(\thet))\,d\thet= \int_{\tau_c}^{\tau}2a\sqrt{-a^{-1}b \cdot g_c  (\thet-\tau_c)}\,d\thet=\frac{4}{3}a\sqrt{-a^{-1}b \cdot g_c(\tau-\tau_c)^{3}}.
 \end{equation}
The set $ \{\tau\, :\, {\rm Re }\,\psi_{a}(\tau)=0   \}$ consists of three rays that meet at $\tau_c$ (Fig. \ref{c_rays}).
\begin{figure}[H]
\begin{center}
            \includegraphics[scale=0.4, angle=0.0]{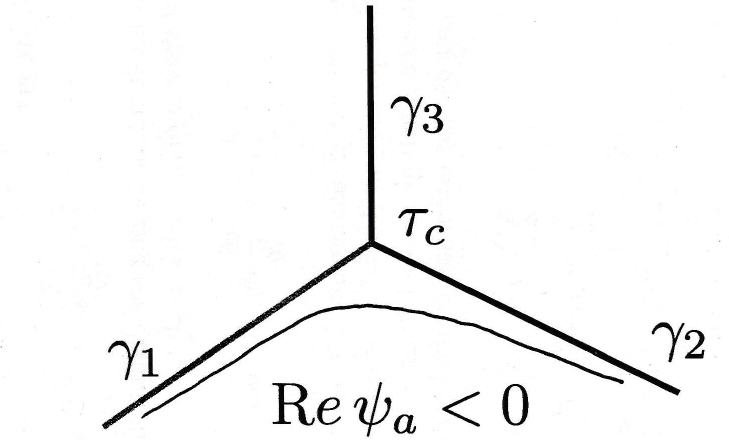}
            \end{center}
           \caption{Complex rays}
            \label{c_rays}
\end{figure}

Denote  by
$ \gamma_{1}$ and $\gamma_{2}$ the rays that are tangent 
\footnote{The order of tangency is $1/2$, see Lemma \ref {lem_Gamma_gamma}.}, respectively,  to
 $\Gamma_{*,1}$ and $\Gamma_{*,2}$ (Fig. \ref{curves_2}). We choose the branch of the square root in the definition of ${Z_{1}}_a(\kappa)$   such that the sign of ${\rm Re }\,\psi_{a}(\tau)$ in the domain between $ \gamma_{1}$ and $\gamma_{2}$ is the same as the sign of ${\rm Re }\,(\Psi(\tau)- \Psi(\tau_c))$ in the domain
 $K_*$.  This choice implies that $\re \psi_{a}(\tau)<0$ in the angle between $\gamma_1$ and $\gamma_2$ because $\re \Psi(\tau)$ has a minimum at $\tau=\tau_*$.
 \begin{figure}[H]
\begin{center}
            \includegraphics[scale=0.4, angle=0.0]{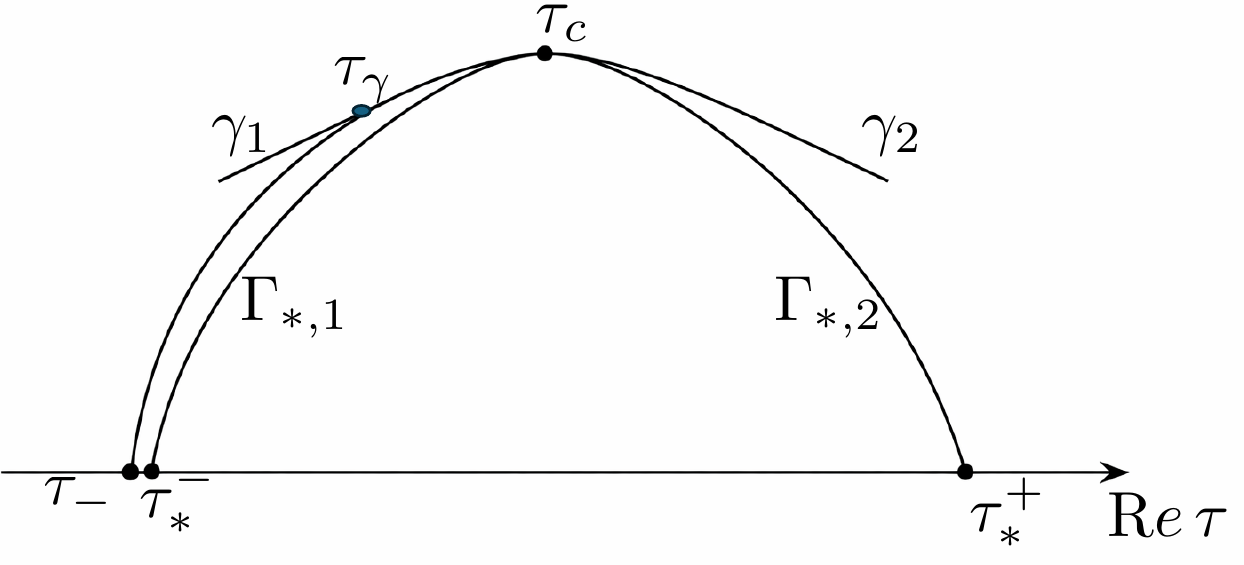}
            \end{center}
           \caption{Curves $\Gamma_{1,*}, \Gamma_{2,*}$ and rays $\gamma_1, \gamma_2$  }
            \label{curves_2}
\end{figure}

 \begin{lem} 
 \label{lem_Gamma_gamma}
 Distance between points of $\Gamma_{*,j}$  and $ \gamma_{j}$ measured along normals to $\Gamma_{*,j}$,\\ $j=1,2$ is $O(|\tau-\tau_c|^{3/2})$, where  $\tau$ is used as a parameter along $\Gamma_{*,j}$. The same estimate is valid if $\tau$ is taken  on $ \gamma_{j}$ and distance is measured along normals to $ \gamma_{j}$.
  \end{lem}
   
 
 Consider also a curve  $\Gamma_{\tau_-}$ with $\tau_-$ close to  $\tau_*^-$ (notation $\Gamma_{\tau_-}$ is introduced in Section \ref{form_conditions}). 
 \begin{lem}  
 \label{lem_Gamma_Gamma}
 Distance between points of  $\Gamma_{\tau_-}$ and  $\Gamma_{*,1}$  measured along normals to $\Gamma_{*,1}$ is 
 of order $|\tau_- -\tau_*^-|/\sqrt{|\tau_{-}-\tau_c|}$ for $\tau$ such that  $|\tau_{-}-\tau_c|>{\rm const} |\tau_{-}-\tau_*^-|^{2/3} $, where $\tau$ is the parameter along $\Gamma_{*,1}$. The same estimate is valid if $\tau$ is taken  on $ \Gamma_{\tau_-}$ and distance is measured along normals to $ \Gamma_{\tau_-}$.

 \end{lem}
 

 
 \medskip
Take any point $\tau_{\gamma}$ on $\gamma_1$ close to $\tau_c$. Its distance from $\Gamma_{*,1}$  is $O(|\tau_{\gamma}-\tau_c|^{3/2})$. Consider the curve  $\re \Psi(\tau)={\rm const}$ passing through $\tau_{\gamma}$ (i.e. $\re \Psi(\tau)=\re \Psi(\tau_{\gamma}))$. This curve crosses the real axis at some point $\tau_{\gamma,-}$ such that $|\tau_{\gamma,-} -\tau_*^-|/\sqrt{|\tau_{\gamma}-\tau_c|}=O(|\tau_{\gamma}-\tau_c|^{3/2})$, i.e. at $|\tau_{\gamma,-} -\tau_*^-|=O(|\tau_{\gamma}-\tau_c|^2)$. These estimates used the assumption \\ $|\tau_{\gamma}-\tau_c|> {\rm const}  |\tau_{\gamma,-} -\tau_*^-|^{2/3}$, i.e. $ |\tau_{\gamma,-} -\tau_*^-|< {\rm const}^{-1}|\tau_{\gamma}-\tau_c|^{3/2}$, which is satisfied. Thus, if  $|\tau_{\gamma}-\tau_c|\sim \eps^{2/3}$, then $|\tau_{\gamma,-} -\tau_*^-|=O(\eps^{4/3})$. 
\medskip

Consider  in the $\tau$-plane the line through the point $\tau_{\gamma}$ parallel to the real axis. Take any point $\tau_p$ on this line with $\re \tau_{\gamma}\le\re\tau_p\le\re \tau_c$ (Fig. \ref {points_1}). Distance of $\tau_p$ from  $\Gamma_{*,1}$ is
$O(|\tau_{\gamma}-\tau_c|)$.
 Denote by $\Gamma_p$ the segment   of the curve $\re \Psi(\tau)={\rm const}$  with $\re\tau\le \re\tau_p$ (assumptions made for simplicity of formulations at the end of Section \ref{form_conditions} imply  that  $\tau_{p}$ is the only intersection of this curve and the line $\im \tau=\im \tau_{\gamma}$  at $\re \tau\le \re \tau_c$). Denote by  $\tau_{p,-}$ the point of intersection of the curve  $\Gamma_p$ and the real axis. 
Then $|\tau_{p,-}-\tau_*^-|=O( |\tau_{\gamma}-\tau_c|\sqrt{|\tau_{\gamma}-\tau_c|})=
O(|\tau_{\gamma}-\tau_c|^{3/2})$.  Thus, if  $|\tau_{\gamma}-\tau_c|\sim \eps^{2/3}$, then $|\tau_{p,-} -\tau_*^-|=O(\eps)$.
\begin{figure}[H]
\begin{center}
            \includegraphics[scale=0.6, angle=0.0]{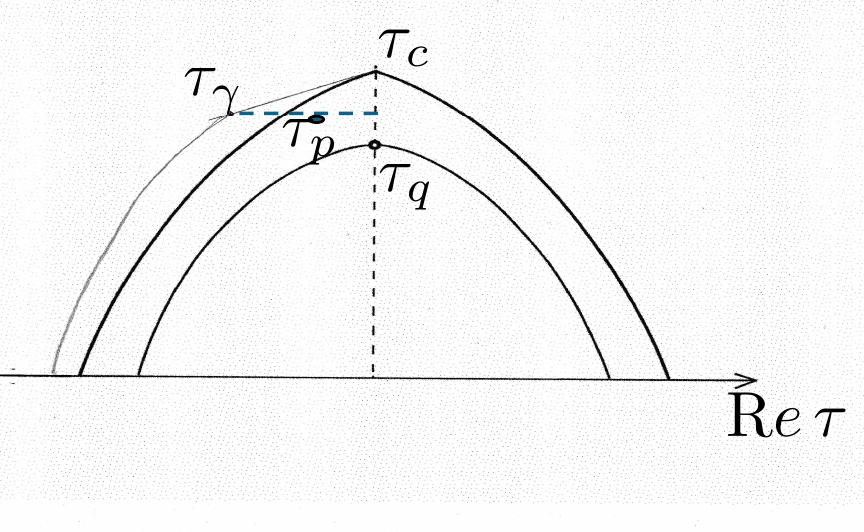}
            \end{center}
            \vskip -1.5cm
           \caption{Points $\tau_{\gamma}, \tau_p, \tau_q$}
            \label{points_1}
\end{figure}
Consider now the vertical line through $\tau_c$. Let $\tau_q$ be a point on this line below $\tau_c$ and close to $\tau_c$ (Fig. \ref {points_1}). Denote by $\Gamma_q$ the curve $\re \Psi(\tau)={\rm const}$ passing through $\tau_{q}$ (i.e. $\re \Psi(\tau)=\re \Psi(\tau_{q}))$. This curve crosses the real axis at some points $\tau_{q,-}$ (the left one) and  $\tau_{q,+}$  (the right one). 
\begin{lem}  
 \label{lem_Gamma_q}
 Distance between points of  $\Gamma_{q}$ and  $\Gamma_{*,1}\cup\Gamma_{*,2} $  measured along normals to $\Gamma_{q}$ is of order $|\tau_q -\tau_c|^{3/2}/\sqrt{|\tau-\tau_c|}$,  where $\tau$ is the parameter along $\Gamma_{q}$. 
  \end{lem}
 In particular, $|\tau_{q,\pm}-\tau_*^{\pm}|$ is of order $|\tau_q -\tau_c|^{3/2}$. If  $|\tau_{q}-\tau_c|\sim \eps^{2/3}$ then $|\tau_{q,\pm}-\tau_*^{\pm}|\sim \eps$.

\input{delay_riccati_1.tex}

 \input{delay_transformation.tex}

  \input{delay_motion.tex}

\input{delay_proofs_of_lemmas_1}

\input{delay_proofs_of_lemmas_2}

  \input{delay_proofs_sect_11}

  \input{delay_proofs_sect_12}

   \input{delay_proofs_sect_13}

    \input{delay_proofs_sect_15}

 {\bf  {Acknowledgements.}} The work was supported by the EPSRC, Grant No. UKRI1724. We are grateful to A. P. Veselov for
useful discussion.

\input{appendix_1}

  \end{document}

%% file: delay_riccati_1.tex
 \subsection{Riccati equation}
 \label{s_riccati}
 Consider equations (\ref{eq_expanded}), and take  $\kappa=\dK_a(\tau)=g_c(\tau-\tau_c)+\kappa_c$:
 \begin{equation}
 \label{e_riccati}
 \frac{ d z_1}{d t}=b \cdot g_c(\tau-\tau_c) +az_1^2 .
 \end{equation}
 This Riccati equation is well known in  the theory of
relaxation oscillations \cite {mr}.
For frozen value of $\tau$ it has equilibria at  $z_1= \sqrt{-a^{-1}b \cdot g_c(\tau-\tau_c)}$
(for $\tau\ne\tau_c$ there are two equilibria corresponding to two values of the square root). The eigenvalues of these equilibria are  $2a\sqrt{-a^{-1}b\cdot  g_c(\tau-\tau_c)}$. The complex phase introduced in  Section  \ref{some_curves} is 
$$
   \psi_{a}(\tau)= \int_{\tau_c}^{\tau}2a\sqrt{-a^{-1} b\cdot g_c  (\thet-\tau_c)}\,d\thet=\frac{4}{3}a\sqrt{-a^{-1}b \cdot g_c(\tau-\tau_c)^{3}}.
 $$
 Let $a=u\exp(iw), \  -a^{-1}b\cdot  g_c = p \exp (i q),\,  \tau-\tau_c = r\exp (i \ph)$ with real $u,w,p,q, r,\ph$; $u>0,  p>0, r\ge0$. (Notation $p,q$ here should not be mixed with $p,q$ in the previous Section.) Then
 $$
  \psi_{a}(\tau)=\frac{4}{3} u p^{1/2}r^{3/2}\exp{(i(3(\ph+2\pi l)+q+2w) /2)},\,  l=1,2,\ldots
 $$ 
 The set $ \{\tau\, :\, {\rm Re }\,\psi_{a}(\tau)=0   \}$ consists of three rays $3\ph+q+2w =3\pi + 2\pi k$, i.e.
$\ph=\ph_k=\pi+(2\pi/3)k -q/3-2w/3, \ k=1,2, 3$, that meet at $\tau_c$. 
These rays where denoted $\gamma_{1,2.3}$ in the previous Section.  We assume that the rays $\gamma_{1}$ and  $\gamma_{2}$ correspond, respectively,  to $k=1$ and  $k=2$, i.e. to $\ph_1= 5\pi/3-q/3-2w/3$ and
$\ph_2= 7\pi/3-q/3-2w/3$\footnote{One could have in mind a model example $\dot z= i(\tau-\tau_c)+z^2$, where $ q=3\pi/2,w=0, \ph_1= 7\pi/6, \ph_2= 11\pi/6$. }. At  $\tau_c$,  these rays are  tangent\footnote{The order of tangency is $1/2$, see Lemma \ref {lem_Gamma_gamma}.}, respectively,  to $\Gamma_{*,1}$ and $\Gamma_{*,2}$. Choice of the branch of the square root  at the definition of the equilibrium and the complex phase is described in the previous Section. This choice is such that $\re \psi_{a}(\tau)<0$ in the $120^{\circ}$-angle   between $\gamma_1$ and $\gamma_2$. This way we have a unique equilibrium $z_1=z_{1,a}(\tau)$ for $\tau$  in this angle.

 Introduce  a new slow time $s$ and a new dependent variable $\ze$ by  formulas\\
 $\tau=\tau_c -s\exp(i\ph_1)$,  $z_1=\exp(i\alpha_1)\ze,\
 \alpha_1=-\pi/6+q/3-w/3$. The equation in the new variables takes the form
 \begin{equation}
 \label{non_normalised_1}
 \eps \frac{  d \ze}{d s}=i(|a|ps +|a|\ze^2 ).
 \end{equation}
 The ray $\gamma_1$ corresponds to $s\le 0$.
 Introduce new variables $\hat \ze, \hat s$ via formulas
 \begin{equation}
\ze=\eps^{1/3}\frac{p^{1/3}}{|a|^{1/3}}\hat \ze,\  s=\eps^{2/3}\frac{1}{p^{1/3}|a|^{2/3} }\hat s.
 \end{equation}  
 The equation for the new variables is
 \begin{equation}
 \label {normalised_1}
 \frac{  d \hat\ze}{d \hat s}=i(\hat s +\hat{\ze}^2 ).
 \end{equation}
 This equation by replacing $\hat \ze$ with $i\hat \ze$ and $\hat s$ with $-\hat s$ is reduced to equation (9.2) in \cite{mr}, whose general solution is given by formula (9.6) in \cite{mr}. Returning to variables  $\hat \ze, \hat s$  we get that the general solution to
equation (\ref{normalised_1}) is
\begin{equation}
\label{solution_1}
\hat\ze=-i\sqrt{-\hat s}\,\frac{J_{-2/3}(v)-RJ_{2/3}(v)}{RJ_{-1/3}(v)+J_{1/3}(v)},
\quad v=\frac{2}{3}(-\hat s)^{3/2} .
\end{equation}
\noindent Here $J_{\nu}(\cdot)$ is the Bessel function of the order $\nu$,
and $R$ is a constant which should be defined by initial conditions. 

%

We will use asymptotic expansions of Bessel functions for large and small values of $v$ \cite{be}:
\begin{equation}
\begin{aligned}
\label{big_v_J}
J_{\nu}(v)&=(\pi v/2)^{-1/2}\cos(v-\nu\pi/2-\pi/4)[1+O(|v|^{-2})]\\
&-(\pi v/2)^{-1/2}  \sin(v-\nu\pi/2-\pi/4)[ \frac{1}{2v}(\nu^2-\frac{1}{4}) +O(|v|^{-3})]  , 
     \  -\pi<\arg v<\pi
\end{aligned}
 \end{equation}
 and
 \begin{equation}
 \label{small_v_J}
 J_{\nu}(v)=\frac{v^{\nu}}{2^{\nu}}\frac{1}{\Gamma(\nu+1)}[1+O(|v|^2)].
\end{equation}
\begin{lem}
\label{expansion_minus}
Asymptotic expansion (\ref{big_v_J}) implies that
\begin{equation}
\label{special_1}
\hat\ze=-\sqrt{-\hat s}[1+O(|v|^{-1})]\ {\rm for} \  \im s=0,\ s\le 0, \ R=R_{-}=\exp((2\pi i)/3).
\end{equation}
\end{lem}
\begin{lem}
\label{ze_small_v}
Asymptotic expansion (\ref{small_v_J}) implies that
\begin{equation}
\label{as1_small}
\hat\ze=\frac{1}{R}\frac{-2\pi i}{\Gamma^2(1/3)3^{1/6}}  [1+O(|v|^{2/3})]  .
\end{equation}
\end{lem}

\medskip
Now, introduce in equations (\ref{eq_expanded}) a new
  slow time $\sigma $ and a new dependent variable $\chi$ by formulas
 $\tau=\tau_c -\sigma \exp(i\ph_2), \  z_1=\exp(i\alpha_2)\chi$, where
 $\ph_2= 7\pi/3-q/3-2w/3, \   \alpha_2=-5\pi/6+q/3-w/3$. The equation in the new variables takes the same form
 (\ref {non_normalised_1}), but now for variables $\sigma, \chi$:
 \begin{equation}
 \label{non_normalised_2}
 \eps \frac{  d \chi}{d \sigma}=i(|a|p\sigma +|a|\chi^2 ).
 \end{equation}
 The ray $\gamma_2$ corresponds to $\sigma \le 0$.
 Introduce new variables $\hat \chi, \hat \sigma$ via formulas
 \begin{equation}
 \label{normalisation_1}
\chi=\eps^{1/3}\frac{p^{1/3}}{|a|^{1/3}}\hat \chi,\  \sigma=\eps^{2/3}\frac{1}{p^{1/3}|a|^{2/3} }\hat \sigma.
 \end{equation} 
 The equation in the new variables takes the same form
 (\ref {normalised_1}), but now for variables $\hat\sigma, \hat \chi$: 
 \begin{equation}
 \label {normalised_2}
 \frac{  d \hat\chi}{d \hat \sigma}=i(\hat \sigma +\hat{\chi}^2 ).
 \end{equation}
  The general solution to
this equation is again(\ref{solution_1}):
\begin{equation}
\label{solution_2}
\hat\chi=-i\sqrt{-\hat \sigma}\,\frac{J_{-2/3}(v)-RJ_{2/3}(v)}{RJ_{-1/3}(v)+J_{1/3}(v)},
\quad v=\frac{2}{3}(-\hat \sigma)^{3/2} .
\end{equation}


\begin{lem}
\label{expansion_plus}
Asymptotic expansion (\ref{big_v_J}) implies that
\begin{equation}
\label{special_2}
\hat\chi=\sqrt{-\hat \sigma}[1+O(|v|^{-1})]\ {\rm for} \  \im \sigma=0,\ \sigma\le 0, \ R=R_{+}=\exp(-(2\pi i)/3).
\end {equation}
\end{lem}
\medskip

For $\tau\in\gamma_1$  (respectively, $\tau\in\gamma_2$) the equilibrium $z_{1,a}(\tau)$ is  given by   $\hat \ze =-\sqrt{-\hat s}$ (respectively, $\hat \chi =\sqrt{-\hat \sigma}$). We will use these formulas for  $\tau$ in the $120^{\circ}$ angle  between  $\gamma_1$ and $\gamma_2$, including its boundaries  $\gamma_1$ and $\gamma_2$. The branches of the square roots in these formulas are chosen so that $\sqrt{-\hat s}$, for real $\hat s \le 0$, and  $\sqrt{-\hat \sigma}$,  for real $\hat \sigma \le 0$,  are the positive arithmetic square roots.
 For the phase  (\ref{phase_a_1}), we obtain $\psi_a= \frac {4}{3}i(-\hat s)^{3/2}$ 
  (respectively, $\psi_a= -\frac {4}{3}i(-\hat \sigma)^{3/2}$). Here we use the fact that
    $\psi_a$  is invariant under transformations of the time and the dependent variable. Note that $\re \psi_a <0$ inside the angle between $\gamma_1$ and $\gamma_2$, as required.  To verify this, consider a point close to $\gamma _1$:
  $ \hat s = -|\hat s_0| -i \mu, \ 0<\mu\ll 1$. Then $\re\psi_a= \re (\frac {4}{3}i(|\hat s_0| +i \mu)^{3/2})=- 2  |\hat s_0|^{1/2}\mu+O(\mu^2)<0$. Finally, note that  $\re\psi_a$ vanishes only on the rays $\gamma_j$. 
  
  Formulas (\ref{solution_1}), (\ref{solution_2}) will be used for  $\tau$ in the $120^{\circ}$ angle  between  $\gamma_1$ and $\gamma_2$, including the boundary  rays, $\gamma_1$ and $\gamma_2$. The branches of the square roots in these formulas are chosen as described above; namely,
$\sqrt{-\hat s}$ for real $\hat s \le 0$  (respectively,  $\sqrt{-\hat \sigma}$ for real $\hat \sigma \le 0$) denotes the positive arithmetic square root.

\medskip
Consider any solution to  equation (\ref{e_riccati}). On $\gamma_1$ it has the form (\ref{solution_1}) with $R=R_l$. On
$\gamma_2$ the same solution has the form (\ref{solution_2})
 with $R=R_r$. The following relation between $R_l$ and $R_r$
follows from the previous transformation formulas  and asymptotic 
formulas for Bessel functions at the point $v=0$ (\ref{small_v_J}).
\begin{lem}
\label{lem_relation}
\begin{equation}
\label{r_relation}
R_r=R_le^{-2\pi i/3}  .
\end{equation}  
\end{lem}
\noindent The system (\ref{e_riccati}) has two special solutions which
 have, respectively,  asymptotic behaviour  (\ref{special_1}) on $\gamma_1$ and
 (\ref{special_2}) on $\gamma_2$. These are different solutions because the relation 
 (\ref{r_relation}) is not satisfied for them. Splitting of these two solutions is the
 reason for existence of maximal delay time. 
 The first of these special solutions, which is described by formula
 (\ref{solution_1}) with $R=R_l=R_{-}=\exp((2\pi i)/3)$ in variables $\ze, s$, is described by formula (\ref{solution_2}) 
 with $R=R_r=R_-e^{-2\pi i/3} =1$ in variables $\chi, \sigma$.  This solution  has poles on  $\gamma_2$
because the function $J_{-1/3}(v)+J_{1/3}(v)$ has zeroes on
the positive real semi-axis \cite{be}, Sec. 7.9. Inside the   angle between $\gamma_1$ and $\gamma_2$ we have ${\rm Im\, }v<0$ (here $v=\frac{2}{3}(-\hat \sigma)^{3/2}$). 
\begin{lem}
\label{r_expansion}
For ${\rm Im\, }v<0$ the function   (\ref{solution_2}) with $R\ne e^{2\pi i/3}$    has the expansion
\begin{equation}
\label {expansion}
\hat\chi=\sqrt{-\hat\sigma}\Bigl [1-2e^{-2iv}e^{\pi i/6}
\frac{R-e^{-2\pi i/3}}{R- e^{2\pi i/3}}
+O\Bigl (e^{-4|{\rm Im\, }v|}+\frac{1}{|v|}\Bigr)\Bigr ].  
\end{equation}
\end{lem}
 \begin{lem}
 \label{lem_dr}
 For the function $\hat\ze$ (\ref{solution_1}) we have
 \begin{equation}
\frac{\partial \hat\ze}{\partial R}=-i \frac{3\sqrt{3}}{2\pi \hat s}\frac{1}{(RJ_{-1/3}(v)+J_{1/3}(v))^2},
\quad v=\frac{2}{3}(-\hat s)^{3/2} .
\end {equation}
For the derivative ${\partial \hat\chi}/{\partial R}$  such relation is valid with the replacement of $\hat s$ with $\hat \sigma$.
  \end{lem}
  
  In the domain bounded by rays $\gamma_1$ and $\gamma_2$, consider the solution of equation (\ref{e_riccati}) that in variables $\hat \chi,\hat \sigma$ is described by the formula (\ref{solution_2}) with $R=1$. This solution has poles on $\gamma_2$ (i.e. for real $\hat \sigma<0$)  and does not have other singularities. This is because its poles are related to  zeros of the Airy function $\rm {Ai(z)}$:
   $$
  {\rm Ai}(z)=\sqrt\frac{-z}{9}\left[J_{-1/3}(\frac{2}{3}(-z)^{\frac{3}{2}})+J_{1/3}(\frac{2}{3}(-z)^{\frac{3}{2}})\right],
  $$
  and  all  zeros of  $\rm {Ai(z)}$ are  located on the negative real axis of the complex plane.  In the variables $\hat\ze, \hat s$ this solution is described by  formula (\ref{solution_1}) with $R=\exp((2\pi i)/3)$. This implies \footnote{Indeed, if function $\hat \ze$ is regular at some $v$, but the denominator in  (\ref{solution_1}) vanishes at this $v$, then the numerator in (\ref{solution_1}) should also vanish at this $v$. This would imply that  $J_{1/3}(v)J_{2/3}(v)+J_{-1/3}(v)J_{-2/3}(v)$ is equal to 0.
However, according to  \cite{be}, Sec. 7.11, formula (33), this expression is equal to ${\sqrt{3}}/({\pi v})$.}
   that, in the considered domain, the function $\exp((2\pi i)/3)J_{-1/3}(v)+J_{1/3}(v), v=\frac{2}{3}(-\hat s)^{3/2}$  has zeros only on the ray $\gamma_2$ (which has equation 
   $\hat s=\exp((2\pi i)/3)\hat \sigma$ with real $\hat \sigma<0$).

%% file: delay_transformation.tex
 \subsection{Additional notation}
 \label{add_notation}
 
 Denote
 $$
 L_*=\{x, \kappa \ : \  x=X(\kappa), \kappa=\dK(\tau), \tau\in K_*\}.
 $$
 Take $c_{l,1}^{-1}< \Imm \tau_c$ and denote
\begin{eqnarray*}
K_*^+&=&\{\tau \,: \,  \tau\in K_*,  \,  \Imm\tau \ge  -c_{l,1}^{-1}   \}, \ K_*^-=\{\tau \,: \,  \tau\in K_*,  \,  \Imm\tau \le  c_{l,1}^{-1}   \},\\
  L_*^+&=&\{x, \kappa \ : \  x=X(\kappa), \kappa=\dK(\tau), \tau\in K_*^+\},\,  L_*^- = \{x, \kappa \ : \  x=X(\kappa), \kappa=\dK(\tau), \tau\in K_*^-\}.
\end {eqnarray*}
Denote by $W$ the $c_{l,2}^{-1}$-neighbourhood of $L_*$, where $c_{l,2}^{-1}$ is such that the right hand sides of the system (\ref{perturbed}) are analytic for $(x,\kappa)\in W $.  
Denote by $W^+$ and  $W^-$ the $c_{l,2}^{-1}$-neighbourhoods of $L_*^+$ and  $L_*^-$, respectively.
 Denote by $V, V^+, V^-$ the projections of $W, W^+,W^-$ onto the $\kappa$-space.

 Denote 
 \begin{equation}
 \begin{aligned}
K^{+}_{*,\delta}&=\{\tau \ : \ \tau\in K_*^+,   \   \im\tau<\im  \tau_c- \delta\}, \
 K^{-}_{*,\delta}=\{\tau \ : \ \tau\in K_*^-,   \     \im\tau>\im \bar{\tau}_c+ \delta  \}  ,\\
  K_{*,\delta}&=K^{+}_{*,\delta}\cup K^{-}_{*,\delta}.   
 \end{aligned}
 \end{equation}
 \vskip -0.3cm
  ({\rm Fig.\ \ref{domain_K}})
 \begin{figure}[H]
\begin{center}
            \includegraphics[scale=0.4, angle=0.0]{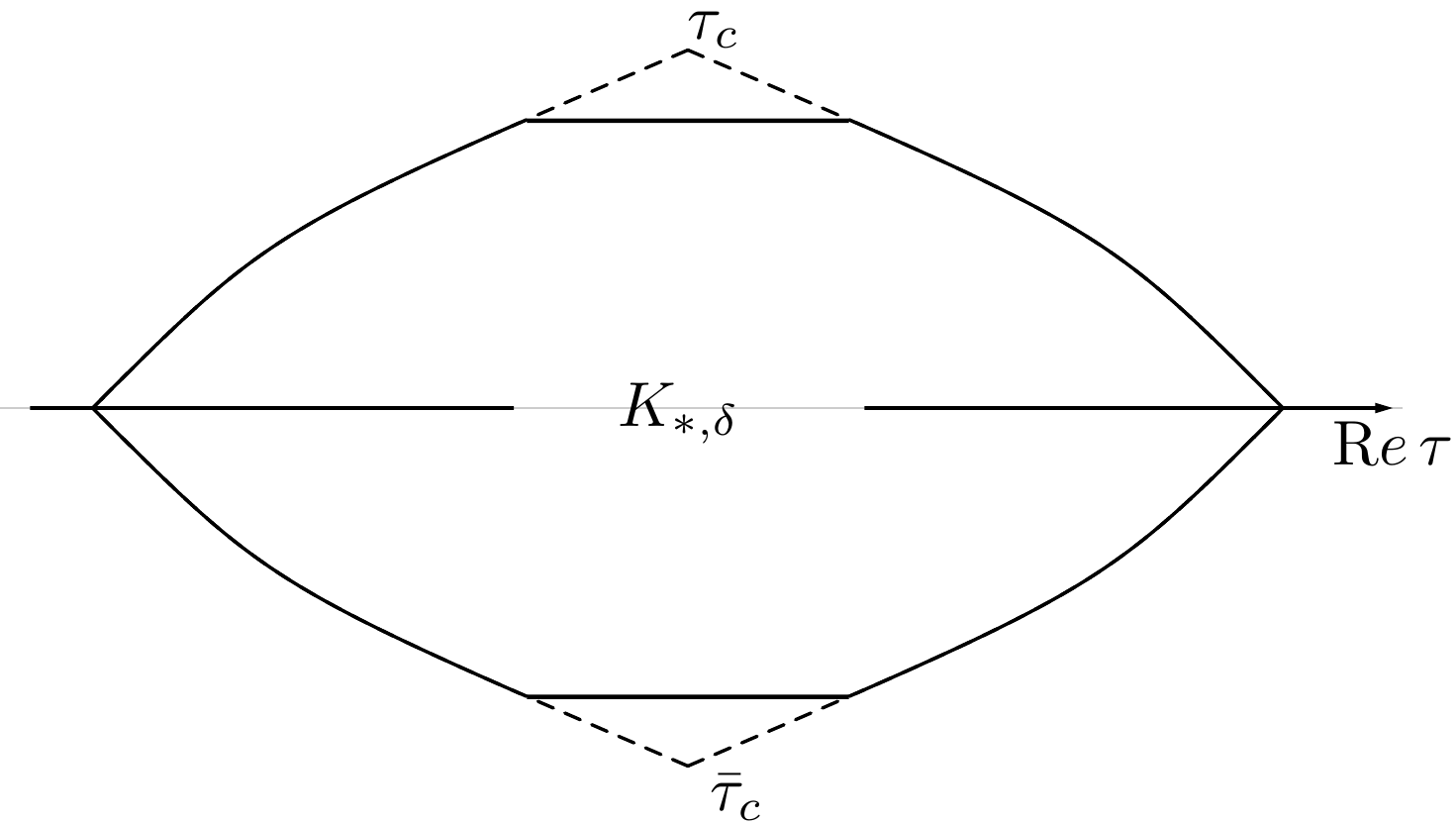}
            \end{center}
           \caption{Slow time domain $ K_{*,\delta}$}
            \label{domain_K}
\end{figure}


  Here $\delta=C_{\delta}\eps^{2/3}$, where $C_{\delta}$ is a sufficiently large  constant whose value will be defined later, at the end of Section \ref{s_transformations}. 
 
  Denote
 \begin{equation}
  B_{*,{\delta}}^{\pm}=\{\kappa \ : \   \kappa=\dK(\tau), \tau\in K^{\pm}_{*,\delta}\}, \ B_{*,{\delta}}=B_{*,{\delta}}^{+}\cup B_{*,{\delta}}^{-} \, .
\end{equation}
\begin{lem}
\label{lem_implicit}
If $C_{\delta}>c_{l,3}$, then in a $c_{l,4}^{-1}\eps^{2/3}$-neighbourhood of $ B_{*,{\delta}}$ (in the complex $\kappa$-space)  the equilibrium $X(\kappa)$  of the fast system is well defined and is an analytic  function of $\kappa$. 
\end{lem}

The proof is a direct application of the Implicit Function Theorem with quantitative estimates. For $\kappa_*\in B_{*,{\delta}}, \, x_*=X(\kappa_*)$ we have 
$|{\det}\left(\partial f(x_*,\kappa_*,0)/\partial x\right)|>k_1^{-1}|\kappa_* -\kappa_c|^{1/2}$. Hence, for $\kappa$ satisfying
$|\kappa-\kappa_*|<k_2^{-1}\eps^{2/3}$, there exists an equilibrium $X(\kappa)$ of the fast system such  that
$|X(\kappa)-x_*| <k_3|\kappa_* -\kappa_c|^{-1/2}|\kappa-\kappa_*|<k_4\eps^{1/3}$, i.e. $(X(\kappa), \kappa)\in W$. We omit the detailed proof.

\medskip
We denote the  $c_{l,4}^{-1}\eps^{2/3}$-neighbourhoods of $ B_{*,{\delta}}^{\pm}$ by   $ V_{{\delta}}^{\pm}$,
and denote $ V_{{\delta}}=V_{{\delta}}^{+}\cup V_{{\delta}}^{-}$. 

Thus, the slow system (\ref{slow}) is well defined for $\kappa\in V_{{\delta}}$.
\begin{lem}
The solution $\dK$ of the slow system (\ref{slow}) can be continued to a\\  $c_{l,5}^{-1}\eps^{2/3}$-neighbourhoods of
the domains $K^{\pm}_{*,\delta}$\, .
\end{lem}
The proof is a direct application of the Existence and Uniqueness Theorem with estimates for analytic ODEs. We omit the  proof.

\medskip
Denote by $\hat K^{\pm}_{*,\delta}$ the parts of  $c_{l,5}^{-1}\eps^{2/3}$-neighbourhoods of
the domains $K^{\pm}_{*,\delta}$ for which\\ $\im  \bar \tau_c+ \delta    < \im\tau<\im  \tau_c- \delta$. Denote $\hat K_{*,\delta} = \hat K^{+}_{*,\delta}\cup \hat K^{-}_{*,\delta}$.

\medskip

Denote 
\begin{equation}
\begin{aligned}
 W_{\delta}=\{x,\kappa \ : \  (x,\kappa)\in  W, \ \kappa \in  V_{{\delta}}  \}, \quad  W_{\delta}^{\pm}=\{x,\kappa \ : \  (x,\kappa)\in  W, \ \kappa \in  V_{{\delta}}^{\pm} \}.
\end{aligned}
\end{equation}
Denote $\tilde d_+=\tilde d_+(\kappa)=  b\cdot(\kappa-\kappa_c), \  d_+=d_+(\kappa)=|\tilde d_+|, \  \tilde d_-=\tilde d_-(\kappa)=b\cdot(\kappa-\bar \kappa_c),  \\ d_-=d_-(\kappa)=|\tilde d_-|  $. 
\begin{lem}
\label{lem_for_domain_V} 
If the constant $c_{l,4}$ in Lemma \ref{lem_implicit} is chosen  sufficiently large, then  \\  $c_{l,6}^{-1}d_{+}<|\kappa-\kappa_c| < c_{l,6}d_{+}$   in $V_{\delta}^{+}$, and  $c_{l,6}^{-1}d_{-}<|\kappa-\bar \kappa_c| < c_{l,6}d_{-}$   in $V_{\delta}^{-}$.
 \end{lem}
 
 In what follows we take such  $c_{l,4}$. We choose $c_{l,5}$ such that the solution  $\dK$ of the slow system with $\tau$ from the  $c_{l,5}^{-1}\eps^{2/3}$-neighbourhood of
the domain $K_{*,\delta}$ does not leave $ V_{{\delta}}$.
 
 \section{Transformations of variables}
 \label{s_transformations}
 For $(x,\kappa) \in W_{\delta}$,  introduce  $\xi=x-X(\kappa)$ as a new variable. In the statements and proofs of lemmas in this Section we  use, in addition to the notation $O(\cdot)$, the notation $O^*(\cdot)$, which  indicates that the function inside the  $O$-symbol is a homogeneous polynomial in $\xi$  of a specified degree. For example, $O^*(d_+^{-1/2})$ and  $O^*(|\xi|)$, are used, respectively,  for terms that are independent of $\xi$ and terms that are linear in $\xi$, etc.  

 
 \begin{lem}
\label{lem_transform_0} 
(a) For $C_{\delta}>c_{t,1}$, if the constant $c_{l,4}$ in Lemma \ref{lem_implicit} is chosen sufficiently large, then, in the variables $\xi,\kappa$, the system (\ref{perturbed}) in the domain $W^+_{\delta}$ takes the form
 \begin{eqnarray}
\label{before_transform_0}
&    \dot\xi
   = \left({\cal {A}}(\kappa)+ \eps O^*(d_+^{-1/2})\right)\xi+O(|\xi|^2)+\eps O^*(d_+^{-1/2}), \\
     &\dot\kappa
   = \eps G(\kappa)+\eps O(|\xi|)+O^*(\eps^2), & \\
\nonumber
&    {\cal {A}}
   = \partial f(X(\kappa),\kappa,0)/\partial x,  \quad
     G
   = g(X(\kappa),\kappa,0).  &
\end{eqnarray}
Differentiation with respect to $\kappa$ of  $O$-terms  that do not contain explicitly $d_+$  multiplies  estimates  of these terms by $O(d_{+}^{-1/2})$.    Differentiation with respect to $\kappa$ of  $O$-terms  that  contain explicitly $d_+$  multiplies estimates of these terms by $O(d_{+}^{-1})$.
Derivative of
$G$  with respect to $\kappa$ is $O(d_{+}^{-1/2})$.
 
 
 (b) More accurate component-wise estimates are valid. For 
 $(x,\kappa) \in W^+_{\delta} $, decompose $\xi$  into projections onto the invariant subspace of the  matrix $A_c$  corresponding to the eigenvalue 0 (the first component of $\xi$) and onto the invariant subspace corresponding to all other eigenvalues. Then $ O^*(d_+^{-1/2})$ in the last term of the first equation in (\ref{before_transform_0}) can be replaced by $O^*(1)$  for all components of $\dot \xi$ except the first.

 In the domain $ W^-_{\delta}$, analogous  estimates hold, with  $d_+$ replaced by $d_-$.

\end{lem}
Here and below, the estimates of the $O$-terms are uniform in $C_{\delta}$ unless stated otherwise. In what follows, we assume that the constant $c_{l,4}$ is chosen as required by Lemma \ref{lem_transform_0}.

  \begin{lem}
\label{lem_transform_00}
For $C_{\delta}>c_{t,2}$, one can make in system (\ref{before_transform_0}) a real-analytic transformation of variables
$\xi =C(\kappa)\tilde \xi$ such that the matrix $\tilde{\cal {A}}=C^{-1}{\cal {A}}C$ has the block-diagonal form with blocks $\tilde{\cal {A}}_0$ of size $2\times2$ and $\tilde{\cal {A}}_1$ of size $(n-2)\times (n-2)$.  The first block has the form 
$$
\tilde{\cal {A}}_0= \frac12
      \begin{pmatrix}
            \lambda_1+\lambda_2 & i(\lambda_1-\lambda_2 ) \\
          i(\lambda_2-\lambda_1)& \lambda_1+\lambda_2
      \end{pmatrix} .
$$
In the new variables, system (\ref{perturbed})  in  domain $ W^+_{\delta}$  takes the following form (we omit   the ``tildes'' over the new variables)
\begin{equation}
\begin{aligned}
\label{after_0_transform}
  & \dot\xi
  = \left(\tilde{\cal {A}}(\kappa)+\eps O^*(d_+^{-1/2}) \right) \xi+O(|\xi|^2) +\eps O^*(d_{+}^{-1/2}), \quad\\
    &\dot\kappa
  =  \eps G(\kappa)+\eps O(|\xi|)+O^*(\eps^2).
\end{aligned}
 \end{equation}
  
  Differentiation with respect to $\kappa$ of  $O$-terms  that do not contain explicitly $d_+$  multiplies estimates of these terms by $O(d_{+}^{-1/2})$.    Differentiation with respect to $\kappa$ of  $O$-terms  that  contain explicitly $d_+$  multiplies estimates of these terms by $O(d_{+}^{-1})$.

  In the domain $W^-_{\delta}$, similar estimates hold with $d_+$ replaced by $d_-$ and with $z$ and $w$ interchanged.
  


\end{lem}
  
 
\begin{lem}
\label{lem_transform_1}
For $C_{\delta}>c_{t, 3}$,  in $W_{\delta}$   there exists a real-analytic transformation of variables 
$\xi \mapsto\hat\xi$
with the following properties. 
In  $W^{+}_{\delta}$ this transformation 
differs by $O(\eps d_{+}^{-1})$ from a linear (affine) transformation of $\xi$;
system (\ref{perturbed}) in the new variables takes the form (we omit the ``hats'' over the new variables)
\begin{eqnarray}
\label{after_transform}
   \dot\xi &=& A(\kappa,\eps)\xi+O(|\xi|^2)+\eps^3O^*(d_{+}^{-3}|\xi|)+\eps^3O^*(d_{+}^{-7/2}), \\
   \dot\kappa
  &=&  \eps  F(\kappa,\eps)+\eps O(|\xi|)+O^*(\eps^2),\\
  F&=& G(\kappa)+O^*(\eps d_{+}^{-1}).
  \label{FandG}
   \end{eqnarray}
  
  The matrix $A$  is block-diagonal with blocks $A_0$ of size
$2\times 2$ and  $A_1$ of size $(n-2)\times(n-2)$. These blocks are $O^*(\eps d_+^{-1})$-close to $\tilde{\cal {A}}_0$ and $\tilde{\cal {A}}_1$, respectively.
  The first block has the form
$$
 A_0=\frac12
      \begin{pmatrix}
            \Lambda_1+\Lambda_2 & i(\Lambda_1-\Lambda_2 ) \\
          i(\Lambda_2-\Lambda_1)& \Lambda_1+\Lambda_2
      \end{pmatrix} .
$$

  If one uses variables $z=\xi_1+i\xi_2, w= \xi_1-i\xi_2$ instead of $\xi_1, \xi_2$, 
 then  for $(x,\kappa) \in W^+_{\delta}$ the estimate  $O^*(\eps d_{+}^{-1})$ for shift of variables can be replaced by  $O^*(\eps d_{+}^{-1/2})$ for $w$ and remaining components of $\xi$,
  the term $ O^*(d_+^{-7/2})$ can be replaced by $ O^*(d_+^{-3})$  for $\dot w$ and remaining components of $\dot \xi$.
  
  Differentiation with respect to $\kappa$ of  $O$-terms  that do not contain explicitly $d_+$  multiplies estimates of  these terms by $O(d_{+}^{-1/2})$.    Differentiation with respect to $\kappa$ of  $O$-terms  that  contain explicitly $d_+$  multiplies estimates of  these terms by $O(d_{+}^{-1})$.
   
   In the domain $W^-_{\delta}$, similar estimates hold with $d_+$ replaced by $d_-$ and with $z$ and $w$ interchanged.

 \end{lem}
 
  \begin{lem}
\label{lem_transform_2}
For  system (\ref{after_transform}), for  $C_{\delta}>c_{t, 4}$ and $|\hat \xi| <c_{t, 5}\eps^{1/3}$, in $ W_{\delta}$   one can make a real-analytic transformation  of variables $(\xi, \kappa)=(\xi_1,\xi_2, \eta, \kappa)\mapsto  (\hat \xi, \kappa)=(\hat\xi_1,\hat\xi_2,\hat \eta,  \kappa)$ 
  that eliminates some quadratic, cubic, and fourth order terms in equations. 
   In the domain $W^+_{\delta}$   this transformation in variables $z=\xi_1+i\xi_2, w= \xi_1-i\xi_2, \ \eta,\kappa$ meets estimates
   \begin{equation*}
   \begin{aligned}
  &\hat z=z+ O(|z|^2 d_+^{-1/2}+ |\xi|^2) ,\
    \hat w=w+O(|zw| d_+^{-1/2}+
  |\xi|^2 ),  \ 
  \hat \eta =\eta +O( |\xi|^2) 
  \end{aligned}
  \end{equation*}
    (``hats'' for new variables). The transformed system is such that the equation for $\dot z$ contains  terms $O(\eps^3d_+^{-3}|\xi|)+ O^*(\eps^{3}d_+^{-7/2})$, and the remaining part of $\dot z$  contains the monomial $z^2$  with a coefficient $O^*(\eps d_+^{-3/2})$,  other quadratic monomials containing $z$ and monomial $w^2$ with coefficients   $O^*(\eps  d_+^{-1/2})$, and all other  quadratic monomials with coefficients   $O^*(1)$ (we omit the ``hats'' over the new variables). Higher order in $\xi$ terms  in equation for $\dot z$, that are not bounded above by estimates for quadratic terms,   are estimated as
  $     O(|\xi|_*^3d_+^{-1/2})+O(|z|^5d_+^{-3/2})$,
  where ``star'' indicates that in computation  of $|\xi|_*^3$ the term  $|z|^3$ is not included (i.e.\footnote{Recall that $|\cdot|$  is the sum of absolute values of coordinates of a vector.}
    $ |\xi|_*^3=|\xi|^3-|z|^3$ assuming that $\xi$ is represented in coordinates $z, w,\eta$).   
        The equation for $\dot w $ contains  terms $O(\eps^3d_+^{-3}|\xi|)+ O^*(\eps^{3}d_+^{-3})$, and the remaining part of $\dot w$ contains monomial $wz$ with  coefficient $O^*(\eps d_+^{-3/2})$,  monomials $z^2, w^2,  w\eta$  with  coefficients $O^*(\eps d_+^{-1/2})$, other  quadratic monomials with coefficients   $O^*(1)$. 
  The equation for $\dot \eta$ contains  terms $O(\eps^3d_+^{-3}|\xi|)+ O^*(\eps^{3}d_+^{-3})$, and the remaining part of $\dot \eta$  contains  monomials $z^2, w^2$  with  coefficients $O^*(\eps d_+^{-1/2})$, other  quadratic monomials with coefficients   $O^*(1)$.
  Higher order in $\xi$ terms in equations for $\eta, w$ are estimated, respectively,   as $O(|\xi|_*^3d_+^{-1/2})  +O(|z|^4)$ and  $O(|\xi|_*^3d_+^{-1/2})  +O(|z|^3)$. 
   
   In the domain $W^-_{\delta}$, similar estimates hold with $d_+$ replaced by $d_-$ and with $z$ and $w$ interchanged.  
   The  transformation of variables meets estimates
   \begin{equation*}
   \begin{aligned}
  &\hat z=z+ O(|zw| d_-^{-1/2}+ |\xi|^2) ,\
    \hat w=w+O(|w|^2 d_-^{-1/2} +|\xi|^2 ),  \ 
  \hat \eta =\eta +O( |\xi|^2).
  \end{aligned}
  \end{equation*}
In this domain, the equation for $\dot z$ contains terms $O(\eps^3d_-^{-3}|\xi|)+ O^*(\eps^3d_-^{-3})$, and in the remaining part of $\dot z$   quadratic monomials are estimated as $O^*(|\eta|(|\eta|+|w|))+  O^*(\eps |zw|d_-^{-3/2})+O^*(\eps |\xi|^2d_-^{-1/2})$,   higher order terms, that are not bounded above by estimates for  quadratic terms, are estimated as 
$$ O(|\xi|_{**}^3d_{-}^{-1/2})+ O(|w|^3).$$ 
  Here  $ |\xi|_{**}^3=|\xi|^3-|w|^3$.
 \end{lem}
 
  \begin{lem}
\label{lem_transform_kappa}
For $C_{\delta}>c_{t,6}$, 
 in $W_{\delta}$   there exists a real-analytic transformation of variables $\xi,\kappa\mapsto\xi, \hat\kappa$ with the following properties. 
In  $W^{+}_{\delta}$ this transformation 
satisfies estimates
$$
\hat \kappa =\kappa +  \eps O(|z|d_-^{-1/2}+|\xi|_*).
$$
Equation for $\hat \kappa$ has the form (the ``hats" are omitted)
\begin{equation}
\begin{aligned}
\label{eq_kappa_transformed}
 & \dot {\kappa}=\eps G(\kappa) +O^*(\eps^2 d_+^{-1})+\eps^2 O(|z|^2d_+^{-2})+ \eps^2 O(|\xi|_*^2d_+^{-3/2}) \\
 &
 +\eps O\left(|z|^4+|z|^5d_+^{-2}+ |\xi|_*^3d_+^{-1}\right)
 +\eps^3 O(|z|d_+^{-3} + |\xi|_*d_+^{-3/2})
 +\eps^4O(d_+^{-7/2}|\xi|).
    \end{aligned}
\end{equation}

 In the domain $W^-_{\delta}$, similar estimates hold with $d_+$ replaced by $d_-$ and with $z$ and $w$ interchanged.
   
\end{lem}

Let $\xi(t), \kappa(t)$ be a solution of the system obtained from (\ref{after_transform}) after transformation from Lemma \ref{lem_transform_kappa}. Consider the real initial conditions $\xi(t_0), \kappa(t_0)$.

\begin{lem}
\label{lem_real}
Suppose that, in the complex $t$-plane, the solutions $\xi(t)$ and $\kappa(t)$ can be analytically continued to a neighbourhood $U$ of a real interval. Then they can also be analytically continued to the reflected set $\overline{U}$, and
\[
\xi(\bar t)= \overline{\xi(t)}, \qquad \kappa(\bar t)= \overline{\kappa(t)},
\]
for all $t \in U \cup \overline{U}$.
\end{lem}

The proof is evident.
\smallskip

Introduce $z=\xi_1+i\xi_2, w= \xi_1-i\xi_2$ and rewrite the system in the variables $z,w,\eta=(\xi_3,\ldots, \xi_n),  \kappa$. Lemma \ref {lem_real} implies that $w(t)= \overline {z(\bar t)}$ for real-analytic solutions. This equation will be used instead of the differential equation for $w$. In  the domain $W^{+}_{\delta}$  we have
\begin{equation}
\begin{aligned}
\label{d_equation}
     &\dot z=  \Lambda_1(\kappa)z+  \eps O(|z|^2 d_+^{-3/2})+
     \eps O(|\xi|_*^2 d_+^{-1/2})  
     +O(|\eta|(|\eta|+|w|))\\
            &+ O(|\xi|_*^3d_+^{-1/2})+O(|z|^5d_+^{-3/2}) +\eps^3O(d_{+}^{-3}|\xi|)+\eps^3O^*(d_{+}^{-7/2}),\\
  &\dot \eta=B(\kappa)\eta +O(|\eta|^2) + O(|\eta |(|z|+|w|)) +O(|zw|) \\
   & \skip 0.7cm+     \eps O((|z|^2  +|w|^2)d_+^{-1/2})  
     +O(|\xi|_*^3d_+^{-1/2})+ O(|z|^4)+\eps^3O(d_{+}^{-3}|\xi|)+\eps^3O^*(d_{+}^{-3}),\\
      & \dot {\kappa}=\eps F(\kappa) +\eps^2 O(|z|^2d_+^{-2})+\eps^2 O(|\xi|_*^2d_+^{-3/2}) \\
 &
 +\eps O\left(|z|^4+|z|^5d_+^{-2}+ |\xi|_*^3d_+^{-1}\right)
 +\eps^3 O(|z|d_+^{-3} + |\xi|_*d_+^{-3/2})+\eps^4O(d_+^{-7/2}|\xi|),\\
     & \hskip 1cm F= G(\kappa)+\eps O^*(d_{+}^{-1}).
\end{aligned}
\end{equation}

Here the ``star'' in $|\xi|_*^2, |\xi|_*^3$ indicates that in the computation of that quantities the term  $|z|^2$ (respectively, $|z|^3$)  is not included; that is,
    $$ |\xi|_*^2=|\xi|^2-|z|^2, |\xi|_*^3=|\xi|^3-|z|^3$$
    where $\xi$ is represented in coordinates $z, w,\eta$.
    
    In the domain $W^-_{\delta}$, similar estimates hold with $d_+$ replaced by $d_-$, $\Lambda_1$ replaced by
$\Lambda_2$,   
 and with $z$ and $w$ interchanged.

Equation for  $z$ in $W^{-}_{\delta}$ is
\begin{equation}
\begin{aligned}\dot z&=  \Lambda_1(\kappa)z+O(|\eta|(|\eta|+|w|))+  
\eps O(|z||w| d_-^{-3/2})   + \eps O(|\xi|^2d_{-}^{-1/2})
     +O(|\xi|_{**}^3d_{-}^{-1/2})+ O(|w|^3) \\
     &+\eps^3O(d_{-}^{-3}|\xi|)
     +\eps^3O^*(d_{-}^{-3}).
\label{d_equation1}
\end{aligned}
\end{equation}
Here  $ |\xi|_{**}^3=|\xi|^3-|w|^3$.

The argument $\eps$ is omitted from the right hand sides of these equations for brevity.   Instead of the system of $n+m$ ordinary differential equations we obtain a system of $n+m-1$ delay differential equations with an imaginary non-constant value of the delay.

\medskip
{\it Remark.} If condition  H  is  dropped, then for $\kappa$ near 
$\kappa_c$ and $\bar\kappa_c$ we would have the same equations (\ref{d_equation}),
(\ref{d_equation1}). However,  for $\kappa$ far from  
$\kappa_c$ and $\bar\kappa_c$, the term $O(|w|^2)$
  should be added to the right hand side of the equation for $\dot z$.

\medskip
From now on, we fix $C_{\delta}> c_{t,6}$.  

\section{Preliminary lemmas about continuation of solutions}
\label{s_continuation}
\subsection{Curves determined by improved slow equation}
\label{improved_curves}
Consider ``improved'' slow equation (see (\ref{d_equation}))
\begin{equation}
\label{improved_slow}
\dot\kappa= \eps F(\kappa), \ F= g(X(\kappa), \kappa,0)+O^*(\eps d_{\pm}^{-1}) , \ \kappa \in V_{\delta}^{\pm}.
\end{equation}
  Consider solution $\dK_{\eps}(\tau)$ of  equation (\ref {improved_slow}) with some initial condition $\dK_{\eps}(\tau_*^-)$ of the form $\dK_{\eps}(\tau_*^-)= \dK(\tau_*^-)+O(\eps)$. 
Without loss of generality, we assume that $\dK_{\eps}(\tau)$ is well defined for $\tau\in  K_{*,\delta} $ (otherwise we would introduce a smaller time domain  where both  $\dK$ and $\dK_{\eps}$ are well defined).  
Denote $\hat {d}_{+}(\tau)=| b\cdot( \dK(\tau)-\kappa_c)|, \ \hat {d}_{-}=|b\cdot(\dK(\tau)-\bar \kappa_c)|, \\  \hat {d}_{+,\eps}(\tau)=| b\cdot( \dK_{\eps}(\tau)-\kappa_c)|,\  \hat {d}_{-,\eps}=|b\cdot(\dK_{\eps}(\tau)-\bar \kappa_c)| $.
\begin{lem}
\label{K_and_K_eps}
In the domain $  K_{*,\delta} $ we have
\begin{equation*}
|\dK(\tau)-\dK_{\eps}(\tau)|=O(\eps(1+|\ln \hat {d}_{\pm}(\tau|)),\quad   
 0.5\hat {d}_{\pm}(\tau)  <\hat {d}_{\pm,\eps}(\tau)< 2\hat {d}_{\pm}(\tau).
\end{equation*}
One should take $\hat {d}_+$ for  $\im \tau> -c_{l,1}^{-1}$, and  $\hat {d}_-$ for  $\im \tau< c_{l,1}^{-1}$.
\end{lem}

\medskip

Introduce the ``phase''
\begin{equation}
\label{phase_eps}
   \Psi_{\eps}(\tau)
 = \int_{\tau_*}^{\tau}\Lambda_1(\dK_{\eps}(\vartheta))\,d\vartheta
\end{equation}
and consider in the upper half-plane of the complex variable $\tau$ arcs of level curves \\ $\re \Psi_{\eps}={\rm const}$ with endpoints on the real axis. 

\medskip
Take any point $\tau_{\gamma}$ on $\gamma_1$ close to $\tau_c$, but such that $\im \tau_{\gamma}<\im\tau_c-{\delta}$. According to Lemma \ref{lem_Gamma_gamma},  the distance of $\tau_{\gamma}$  from $\Gamma_{*,1}$  is $O(|\tau_{\gamma}-\tau_c|^{3/2})$. Consider curves $\re \Psi(\tau)={\rm const}$ and  $\re \Psi_{\eps}(\tau)={\rm const}$ passing through $\tau_{\gamma}$ (i.e. $\re \Psi(\tau)=\re \Psi(\tau_{\gamma})$ and $\re \Psi_{\eps}(\tau)=\re \Psi_{\eps}(\tau_{\gamma}))$. The first of these curves crosses the real axis at  a point $\tau_{-}$ such that    $|\tau_- -\tau_*^-|=O(|\tau_{\gamma}-\tau_c|^2)$,
see Sect. \ref{some_curves}. 
\begin{lem}
\label{second_line}
The curve $\re \Psi_{\eps}(\tau)=\re \Psi_{\eps}(\tau_{\gamma})$ crosses the real axis at  a point $\tau_{ *,\eps,-}=\tau_{ *,\eps,-}(\tau_{\gamma})= \tau_-+ O(\eps\ln \eps)$.
\end{lem}

Thus,   $\tau_{ *,\eps,-}= \tau_-+ O(\eps\ln \eps) = \tau_- -  \tau_*^- + \tau_*^-  +O(\eps\ln \eps ) =\tau_*^-  +O(|\tau_{\gamma}-\tau_c|^2)
+O(\eps\ln \eps )$.
If  $|\tau_{\gamma}-\tau_c|\sim \eps^{2/3}$, then
$\tau_{ *,\eps,-}=\tau_*^ - +O(\eps^{4/3}) +O(\eps\ln \eps)=O(\eps\ln \eps)$.

The required value  $\tau_{\gamma}$ will be defined later. We will have
$ \tau_{ *,\eps,-}(\tau_{\gamma})-\tau_*^-=O(\eps|\ln \eps|)$. The estimates below  are uniform with respect to $ \tau_{ *,\eps,-}(\tau_{\gamma})$ from the considered interval till the part of the paper where we choose $\tau_{\gamma}$ (Section \ref{domain_3}). We will use a notation $\Gamma_{*,1, \eps}$ for the curve defined in Lemma \ref{second_line} omitting an indication of dependence on $\tau_{\gamma}$.

Consider,  in the $\tau$-plane,  the line passing through the point $\tau_{\gamma}$ and parallel to the real axis. Take any point $\tau_p$ on this line with $\re \tau_{\gamma}\le\re\tau_p\le\re \tau_c$. Consider curves $\re \Psi(\tau)={\rm const}$ and  $\re \Psi_{\eps}(\tau)={\rm const}$ passing through $\tau_{p}$ (i.e. $\re \Psi(\tau)=\re \Psi(\tau_{p})$ and $\re \Psi_{\eps}(\tau)=\re \Psi_{\eps}(\tau_{p}))$. The first of these curves is  $\Gamma_p$ from Section  \ref{some_curves}.  It crosses the real axis at a point $\tau_{p,-}=\tau_*^- +O(|\tau_{\gamma}-\tau_c|^{3/2})=\tau_-+O(|\tau_{\gamma}-\tau_c|^{3/2})$. Denote the second of these curves as $\Gamma_{p, \eps}$. Denote the point of intersection of this curve with the real axis by $\tau_{p,\eps, -}$.

\begin{lem}
\label{third_line}
The curve $\re \Psi_{\eps}(\tau)=\re \Psi_{\eps}(\tau_{p})$ crosses the real axis at  a point $\tau_{p,\eps, -}= \tau_-+ O(\eps\ln \eps)$.
\end{lem}

Consider now a vertical line through $\tau_c$. Take  a point $\tau_q$  on this line below $\tau_c$, close to $\tau_c$,  but such that $\im \tau_{q}<\im\tau_c-{\delta}$. Consider curves $\re \Psi(\tau)={\rm const}$ and  $\re \Psi_{\eps}(\tau)={\rm const}$ passing through $\tau_{q}$ (i.e. $\re \Psi(\tau)=\re \Psi(\tau_{q})$ and $\re \Psi_{\eps}(\tau)=\re \Psi_{\eps}(\tau_{q}))$. The first of these curves is  $\Gamma_q$ from Section  \ref{some_curves}.  It crosses the real axis at points $\tau_{q,\pm}=\tau_*^{\pm}+O(|\tau_q -\tau_c|^{3/2})$. Denote the second of these curves by
 $\Gamma_{q, \eps}$. Denote points of intersection of this curve with the real axis by $\tau_{q,\eps,\pm}$.

\begin{lem}
\label{fourth_line}
The curve $\re \Psi_{\eps}(\tau)=\re \Psi_{\eps}(\tau_{q})$ crosses the real axis at  points $\tau_{q,\eps,\pm}=\tau_{q,\pm}+O(\eps\ln \eps)$.
\end{lem}
Thus, $\tau_{q,\eps,\pm}=\tau_*^{\pm}+O(|\tau_q -\tau_c|^{3/2})+O(\eps\ln \eps)$. In particular, if
$|\tau_q -\tau_c|\sim \eps^{2/3}$, then   $\tau_{q,\eps,\pm}=\tau_*^{\pm}+O(\eps\ln \eps)$.

 \medskip
 We can introduce now domains $D_{\gamma}$,  $D_p$ and $D_q$  in $\tau$-plane to which solutions will be continued. 

Domain  $D_{\gamma}$ is bounded by the curve  $\Gamma_{*,1, \eps}$, by the complex conjugate to it curve 
 $\bar\Gamma_{*,1, \eps}$, and by the lines $\im \tau=\im \tau_{\gamma}$,  $\im \tau=-\im \tau_{\gamma}$, 
 $\re \tau=\re \tau_{\gamma}$.
 
 Domain  $D_{p}$ is bounded by the curve  $\Gamma_{p, \eps}$, by the complex conjugate to it curve $\bar\Gamma_{p, \eps}$, and by the lines $\im \tau=\im \tau_{\gamma}$,  $\im \tau=-\im \tau_{\gamma}, 
 \re \tau=\re \tau_p$.
 
  Domain  $D_{up}$ is the union of domains $D_{p}$ for all the considered values $\tau_p$.
 
 Domain  $D_{q}$ is bounded by the curve  $\Gamma_{q, \eps}$ and by the complex conjugate to it curve 
 $\bar\Gamma_{q, \eps}$.
 
 \subsection{Continuation into domains $D_{\gamma}, D_{up} $.}
 \label{cont_D_gamma}
 
 Denote $t_3=\tau_{ *,\eps,-}(\tau_{\gamma})/\eps$. 
  Let $\hat d_{+}(\tau_{\gamma})=C_{\gamma,*}\eps^{2/3}$, where $C_{\gamma,*}$ is a constant. Then $\im \tau_{\gamma}= \im \tau_c -(C_{\gamma}+o(1))\eps^{2/3}$ with a constant $C_{\gamma}$ determined by the constant $C_{\gamma,*}$. We take $C_{\gamma,*}$ such that
  $C_{\gamma}>C_{\delta}$.  Consider solution $z(t), \eta(t), \kappa(t)$ of  system (\ref{d_equation}) with initial conditions of the form
 \begin{equation}
 \label{init_gamma}
  |\kappa(t_3)-\dK_{\eps}(\eps t_3)|=O(\eps^6 |\ln \eps|) ,
 \  |z(t_3)|= O(\eps^3),\ |\eta(t_3)|= O(\eps^3) .
 \end{equation}
 \begin{lem}
 If $C_{\gamma}>c_{e,1}$, then 
the solution $z(t), \eta(t), \kappa(t)$ can be   analytically  continued into the domains $D_{\gamma}$ and $D_{up}$ with the following estimates.
\label{lem_cont_D_gamma}

If $\im \tau> -c_{l,1}^{-1}$, then
\begin{equation}
|\kappa(t)-\dK_{\eps}(\eps t)|<  c_{e,2}\eps^4 \hat{ d}_{+}^{-9/2}  ,\ 
 | z(t)|<c_{e,3}\eps^2 \hat{ d}_{+}^{-5/2}, \  | \eta(t)|<c_{e,4}\eps^{3} \hat{ d}_{+}^{-3}.  \end{equation}
 
 If\quad  $-2 c_{l,1}^{-1}<  \im \tau < 2 c_{l,1}^{-1}$, then
  \begin{equation}
|\kappa(t)-\dK_{\eps}(\eps t)|=O(\eps^4 ) ,\ 
 | z(t)|=O(\eps^3), \  | \eta(t)|=O(\eps^{3}).  \end{equation}

 If $\im \tau< c_{l,1}^{-1}$, then
\begin{equation}
|\kappa(t)-\dK_{\eps}(\eps t)|<  c_{e,2}\eps^4 \hat{ d}_{-}^{-9/2}  ,\ 
 | z(t)|<c_{e,4}\eps^2 \hat{ d}_{-}^{-3}, \  | \eta(t)|<c_{e,4}\eps^{3} \hat{ d}_{-}^{-3}. 
 \end{equation}
\end{lem}
Constants in this Lemma depend on constants in estimates ``$O(\cdot)$'' for initial conditions in (\ref{init_gamma}) and  initial condition for solution of  the ``improved'' slow equation in  Section \ref{improved_curves}.  

\subsection{Continuation into domain $D_{q}$.}
 \label{cont_D_q}

Denote by $D_{q,l}$ and  $D_{q,r}$ the left (with $\re \tau\le \re \tau_c)$ and right  (with $\re \tau\ge \re \tau_c)$ parts of the domain  $D_{q}$.

  Let $\im \tau_q= \im \tau_c -C_{q}\eps^{2/3}$, where $C_{q}$ is a constant, $C_{q}>C_{\delta}$.    
  Then \\$\hat d_+(\tau_q)=(C_{q,1}+o(1))\eps^{2/3}$ with a constant $C_{q,1}$ determined by the constant $C_{q}$.
  Denote $t_4=\tau_{q,\eps,-}/\eps$.  Consider solution $z(t), \eta(t), \kappa(t)$ of  system (\ref{d_equation}) with initial conditions of the form
 \begin{equation}
  |\kappa(t_4)-\dK_{\eps}(\eps t_4)|=O(\eps^6 |\ln \eps|) ,
 \  |z(t_4)|= O(\eps^3),\ |\eta(t_4)|= O(\eps^3) .
 \end{equation}
 
 \subsubsection{Continuation into domain $D_{q,l}$.}
\label{cont_D_q_l}

 \begin{lem}
 \label{lem_cont_D_q_l}

 If $C_{q}>c_{e,5}$, then 
the solution $z(t), \eta(t), \kappa(t)$ can be  analytically continued  into the domain $D_{q,l}$ 
with the following estimates.

\medskip

If $\im \tau>- c_{l,1}^{-1}$, then
\begin{equation}
|\kappa(t)-\dK_{\eps}(\eps t)|<  c_{e,6}\eps^4 \hat{ d}_{+}^{-9/2}  ,\ 
 | z(t)|<c_{e,7}\eps^2 \hat{ d}_{+}^{-5/2}, \  | \eta(t)|<c_{e,8}\eps^{3} \hat{ d}_{+}^{-3}.  \end{equation}
  
  If\quad  $-2 c_{l,1}^{-1}<  \im \tau < 2 c_{l,1}^{-1}$, then
  \begin{equation}
|\kappa(t)-\dK_{\eps}(\eps t)|=O(\eps^4 ) ,\ 
 | z(t)|=O(\eps^3), \  | \eta(t)|=O(\eps^{3}).  \end{equation}

 If  $\im \tau< c_{l,1}^{-1}$, then
\begin{equation}
|\kappa(t)-\dK_{\eps}(\eps t)|<  c_{e,6}\eps^4 \hat{ d}_{-}^{-9/2},\ 
 |z(t)|<c_{e,8}\eps^3 {\hat d}_{-}^{-3}, 
   | \eta(t)|<c_{e,8}\eps^3 \hat{ d}_{-}^{-3}.  \end{equation}
   \end{lem}
   We need also some improved estimate for motion vertically down from $\tau_q$. 
   \begin{lem}
   \label{l_improved_I}
   On the vertical line $\re \tau=\re \tau_c$, if \\
   $ \hat{ d}_{+}(\tau_q)+c_{e,9,1}\eps^{2/3}C_q^{-1/2}(\ln C_q)<\hat{ d}_{+}<c_{e,9,2}^{-1} $, then
 \begin{equation}
 | z(t)|<c_{e,9,3}\eps^3 \hat{ d}_{+}^{-4}.
\end{equation}
   \end{lem}
We denote by $A_q$ the segment of the line $\re \tau=\re \tau_c$ with\\
$ \hat{ d}_{+}(\tau_q) \le \hat{ d}_{+}\le \hat{ d}_{+}(\tau_q)+c_{e,9,1}\eps^{2/3}C_q^{-1/2}(\ln C_q)$
(Fig. \ref{dom_and_c_1}).

  \begin{figure}[H]
\begin{center}
            \includegraphics[scale=0.4, angle=0.0]{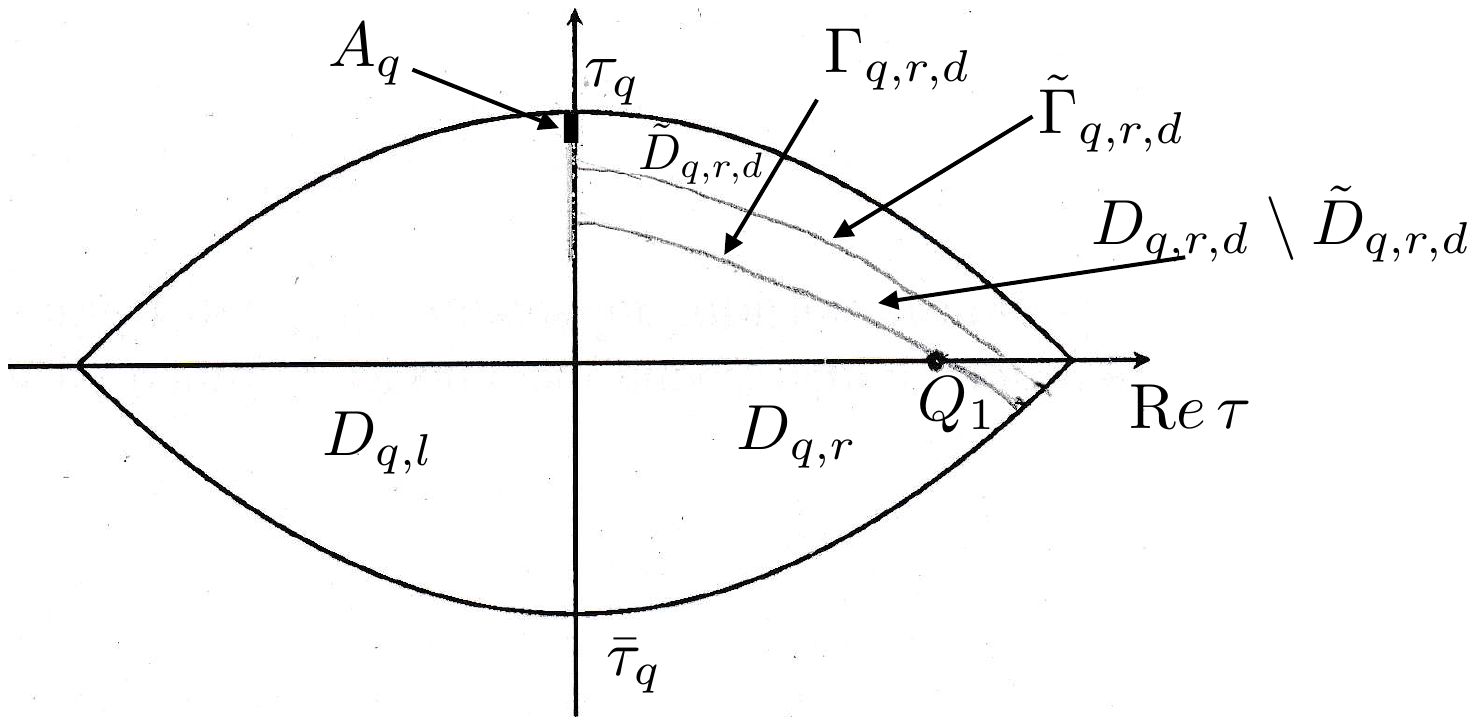}
            \end{center}
            \vskip-2cm
           \caption{Domains and curves for Lemmas \ref{lem_cont_D_q_l}, \ref{l_improved_I}, \ref{lem_cont_D_q_r}, and \ref{lem_improved_kappa_old_1}}
            \label{dom_and_c_1}
\end{figure}

\subsubsection{Continuation into domain $D_{q,r}$.}
\label{cont_D_q_r}

   \begin{lem}
    \label{lem_cont_D_q_r}
   If $C_{q}>c_{e,10}>c_{e,5}$, then 
the solution $z(t), \eta(t), \kappa(t)$ can be analytically continued   into the domain $D_{q,r}$ 
with the following estimates.

\medskip
Denote by $D_{q,r,d}$ (respectively,  by $D_{q,r,d}'$) the part of the domain $D_{q,r}$ covered by vertical segments  of lengths less than or equal to  (respectively, equal to)  $c_{e, 12} \eps \hat d_u^{-1/2}|\ln( c_{e,11}^{-1}\eps \hat d_u^{-3/2}C_q^{15/16})|$  drawn downward from all points     $\tau_u\in \Gamma_{q,\eps}$, where $\hat d_u$ is the value $\hat{ d}_{+}$ at the point $\tau_u$.  Denote by $ \Gamma_{q,r,d}$ (respectively, by  $ \Gamma_{q,r,d}'$) the lower boundary of $ D_{q,r,d}$ (respectively, of $D_{q,r,d}'$).  
Denote by  $Q_1$  the point of intersection of the curve $\Gamma_{q,r,d}$  and the real axis in the $\tau$-plane (Fig.  \ref{dom_and_c_1}). 

\medskip

Denote by $\tilde D_{q,r,d}$ (respectively, by $\tilde D_{q,r,d}'$) the part of the domain $D_{q,r,d}$ covered by vertical segments  of lengths less than  or equal to (respectively, equal to)  $c_{e,12,1} \eps \hat d_u^{-1/2}|\ln( c_{e,11,1}^{-1}\eps \hat d_u^{-3/2}C_q^{15/16})|$ drawn downward  from all points   $\tau_u\in \Gamma_{q,\eps}$.  Denote by $\tilde \Gamma_{q,r,d}$ (respectively, by  $ \tilde \Gamma_{q,r,d}'$) the lower boundary of $\tilde D_{q,r,d}$ (respectively, of $\tilde D_{q,r,d}'$). 

Denote by $\overline D_{q,r,d}, \overline {D'}_{q,r,d}, \overline{\tilde D}_{q,r,d}, \overline{{\tilde D}'}_{q,r,d} $ the domains complex conjugate   to the domains $D_{q,r,d}, D'_{q,r,d}, \tilde D_{q,r,d}, \tilde D'_{q,r,d} $.  Denote by $\overline \Gamma_{q,r,d}, \overline { \Gamma'}_{q,r,d}, \overline{\tilde  \Gamma}_{q,r,d}, \overline{{\tilde  \Gamma}'}_{q,r,d} $ the curves complex conjugate   to the curves  $ \Gamma_{q,r,d},  \Gamma'_{q,r,d}, \tilde  \Gamma_{q,r,d}, \tilde  \Gamma'_{q,r,d} $. 


 

 \bigskip
 If $\tau \in  D_{q,r,d}$,  then 
   
  $|z(t)|< c_{e,13} \eps^{1/3}C_q^{-5/2}, \  |\eta(t)|< c_{e,14} \eps C_q^{-3}, \  
     |\kappa(t)-\dK_{\eps}(\eps t)|<  c_{e,15}\eps  C_q^{-9/2}$; 
     
      \bigskip

 additionally,   if $\tau\in  D_{q,r,d}\setminus  \tilde D_{q,r,d}$,  
  then 
   
 $|z(t)|< c_{e,16} (\eps^{11/6}C_q^{-6}\hat {d}_{+}^{-1/2}+\eps^3 \hat {d}_{+}^{-4}), \  |\eta(t)|< c_{e,17} \eps^{3}\hat{ d}_{+ }^{-3}$;
  
  \bigskip
  additionally,   if $\tau\in  D_{q,r,d}'\setminus  \tilde D_{q,r,d}$,  
  then
   
    $ |\kappa(t)-\dK_{\eps}(\eps t)|<  c_{e,18}\eps^4 \hat {d}_{+}^{-9/2}$;  
     
       \bigskip

additionally,   if $\tau\in  \Gamma'_{q,r,d}$, then $|z(t)|< c_{e,19} \eps^3 \hat{ d}_{+ }^{-4}$.

\bigskip

 If $\tau\in  \bar D_{q,r,d}$,  then  
    
      $|z(t)|< c_{e,20} \eps C_q^{-3}, \    |\eta(t)|< c_{e,20}  \eps C_q^{-3}, \ 
     |\kappa(t)-\dK_{\eps}(\eps t)|<  c_{e,21}\eps C_q^{-9/2}$; 
     
      \bigskip

 additionally,    if $\tau\in  \bar D_{q,r,d}\setminus \overline {\tilde D}_{q,r,d}$,  then      
      $|z(t)|< c_{e,22} \eps^3 \hat{ d}_{- }^{-3},\    |\eta(t)|< c_{e,22} \eps^{3}\hat{ d}_{- }^{-3}, \\ 
\hskip 1cm     |\kappa(t)-\dK_{\eps}(\eps t)|<  c_{e,23}\eps^4 \hat{ d}_{-}^{-9/2} $.
     
      \bigskip

     If $\tau\in  (D_{q,r} \setminus(D_{q,r,d}\cup \bar D_{q,r,d}))\cap \{\im\tau>- c_{l,1}^{-1}\}$,  then  
     
      $|z(t)|< c_{e,24} \eps^{3}\hat{ d}_{+ }^{-4}, \
     |\eta(t)|< c_{e,24} \eps^{3}\hat{ d}_{+ }^{-3}, 
      |\kappa(t)-\dK_{\eps}(\eps t)|<  c_{e,25}\eps^4 \hat{ d}_{+}^{-9/2} $.
      
       \bigskip

        If $\tau\in  (D_{q,r} \setminus(D_{q,r,d}\cup \bar D_{q,r,d}))\cap \{\im\tau< c_{l,1}^{-1}\}$,  then  
     
      $|z(t)|< c_{e,24} \eps^{3}\hat{ d}_{- }^{-3}, \
     |\eta(t)|< c_{e,24} \eps^{3}\hat{ d}_{- }^{-3}, \ 
      |\kappa(t)-\dK_{\eps}(\eps t)|<  c_{e,25}\eps^4 \hat{ d}_{-}^{-9/2} $.

   \end{lem}
   
   Estimate for $\kappa(t)-\dK_{\eps}(\eps t)$ in the domain $\tilde D_{q,r,d}'$ can be improved. 
  \begin{lem}
   \label{lem_improved_kappa_old_1}
   If $\tau\in \tilde D_{q,r,d}'$, 
   then
   \begin{equation}
   \label{eq_improved_kappa_old_1}
    \begin{aligned}
    &|\kappa(t)-\dK_{\eps}(\eps t)|<c_{e,26}\eps^4 \hat{ d}_{+}^{-9/2} +c_{e,27}(\eps \hat d_+^{-1/2}|\ln(\eps \hat d_+^{-3/2}C_{q}^{15/16})|)\\
 &\cdot \left( \eps^{5/3}C_q^{-5}d_+^{-2}
  +\eps^{4/3}C_q^{-10}
 +\eps^{7/3}C_q^{-5/2}d_+^{-3} \right).
    \end{aligned}
       \end{equation}
  \end{lem}

Constants in  Lemmas \ref {lem_cont_D_q_l}, \ref {l_improved_I}, \ref {lem_cont_D_q_r}, \ref{lem_improved_kappa_old_1} depend on constants in estimates ``$O(\cdot)$'' for initial conditions (including initial condition for solution of  the ``improved'' slow equation in  Section \ref{improved_curves}).  

\medskip
Lemmas on the continuation of solutions into the domains $D_{\gamma}, D_{up}, D_{q,l}$, and $D_{q,r}$ will be  used in the proof of Theorem \ref{t1.th} as follows. We will show that the considered solution can be continued into $D_{\gamma}$  and $D_{up}$ with $\hat d_{+}(\tau_{\gamma})=C_{\gamma,*}\eps^{2/3}$ for any sufficiently large  constant 
$C_{\gamma,*}$.  We  will also show that  this solution can be continued  into $D_{q,l}$ with $\im \tau_q =\im \tau_c-C_q\eps^{2/3}$, provided that $C_q$ is sufficiently large. Continuation into $D_{\gamma}$ and  $D_{up}$ will be used in continuation into a larger domain $D_c$ with  vertices  at $\tau_c$ and $\bar \tau_c$. Continuation into $D_{q,l}$  will be used to extend the solution into $D_{q,r}$, provided $C_q$ is sufficiently large.

%% file: delay_motion.tex
\section{Proof of  Theorem 1.}
\label{t1.th.p}

\subsection{Initial part of  motion}

Let $x(t),\kappa(t)$ be the solution of system (\ref{perturbed}) with
 $ \kappa(t_0)=\kappa_0=\dK(\tau_0)$,  $t_0 =\tau_0/\varepsilon$. According to \cite{vasil'eva_butuzov}, p.55,  if the initial   point is in a $C_{2}^{-1}$-neighbourhood of the equilibrium of the fast system, then for some $t_1=t_0+O(|\ln \eps|)$ the solution comes to an $O(\eps)$-neighbourhood of the equilibrium, and  $\kappa(t_1)=\dK(\eps t_1)+O(\eps)$.  
After this, up to a slow time  $\tau_* -c_{m,1}^{-1}$ the phase point  moves in an $O(\eps)$-neighbourhood of the equilibrium, and  $\kappa(t)=\dK(\eps t)+O(\eps)$. Moreover, for description of this motion one can use the same variables as in Lemmas \ref{lem_transform_1}, \ref{lem_transform_2}, and \ref{lem_transform_kappa}. In these variables, starting from a slow  time $\eps t_1+c_{m,2}\eps|\ln \eps|$, we have
$\xi(t)=O(\eps^3),  \kappa( t)=\dK(\eps t)+O(\eps)$.  Thus, we can take  $\eps t_2= \tau_*^- +O(\eps |\ln \eps|)$
such that  at $t=t_2$ the solution in these new variables  is in  an  $O(\eps^3)$-neighbourhood of the equilibrium $z=0, \eta=0$, and $\kappa( t_2)=\dK(\eps t_2)+O(\eps)$ (we assume that the constant $C_{1}$ in the statement of Theorem \ref{t1.th} is chosen  such that for $\tau_0<\tau_*^- -C_{1}\eps|\ln\eps|$ we have  $\eps t_2<  \tau_*^-$).

\subsection{First part of the principal  part of  motion}
\label{first_principal_part}
  We consider  solution $z(t), \eta(t), \kappa(t)$ of  system (\ref{d_equation}). Take as the initial condition for the ``improved''  slow equation (\ref{improved_slow}) at $t= \tau_*^-/\eps$ the value 
  $\kappa(\tau_*^-/\eps)$:   $\dK_{\eps}(\tau_*^-)=\kappa(\tau_*^-/\eps)$.  Then the assumption  $\dK_{\eps}(\tau_*^-)= \dK(\tau_*^-)+O(\eps)$ from Section \ref{improved_curves} is satisfied.
  
Take  $\tau_{\gamma}$ (see Section \ref{improved_curves}) such that     $\hat d_{+}(\tau_{\gamma})=C_{\gamma,*}\eps^{2/3}$, where $C_{\gamma,*}$ is a sufficiently large positive constant to be specified later, at the end of Section 
\ref{first_principal_part}.  Then  $\im \tau_{\gamma}= \im \tau_c -(C_{\gamma}+o(1))\eps^{2/3}$, where $C_{\gamma}$ is a constant determined by the constant $C_{\gamma,*}$.  Denote
  $t_3=\tau_{-,*,\eps}(\tau_{\gamma})/\eps$. 
  Take value $\tau_q$ (see Section \ref{improved_curves}) such that  $\im \tau_q= \im \tau_c -C_{q}\eps^{2/3}$, with some not yet defined constant $C_{q}$.  Denote $t_4=\tau_{q,\eps,-}/\eps$.
  Lemmas \ref{lem_Gamma_Gamma}, \ref{lem_Gamma_q}, \ref{second_line}, \ref{fourth_line} imply that
   $t_{3,4}=\tau_*^- +O(\eps|\ln\eps|)$.
   We will choose the constant $C_{1}$ in formulation of
  Theorem \ref {t1.th} such that $\eps t_{3,4}>\tau_*^- -C_{1}\eps|\ln \eps|$.
  
    \medskip
  We assume that $C_{\gamma}>C_{\delta}$, $C_{q}>C_{\delta}$.
  
  \begin{lem}
  \label{lem_init_cond_3} 
  \begin{equation}
  \label{init_cond_3}
 |\kappa(t_{3,4})-\dK_{\eps}(\eps t_{3,4})|< c_{m,3}\eps^6 |\ln \eps| ,
 \  |z(t_{3,4})|<  c_{m,4}\eps^3,\ |\eta(t_{3,4})|<  c_{m,5}\eps^3  .
  \end{equation}
  \end{lem}
  (Recall that   $\dK_{\eps}(\tau_*^-)=\kappa(\tau_*^-/\eps)$ and $\eps t_{3,4}=\tau_*^-+O(\eps|\ln \eps|)$).
  \medskip
  
  According to Lemmas \ref{lem_cont_D_gamma}, \ref{lem_cont_D_q_l},  and \ref{lem_cont_D_q_r}, the considered solution can be analytically continued   into domains $D_{\gamma}, D_{up}$, and $D_{q}$ for sufficiently large $C_{\gamma,*}$ and $ C_{q}$ with estimates given by these lemmas.     These estimates are stated in terms of the variables obtained  by transformations in Lemmas \ref{lem_transform_1},  \ref{lem_transform_2}, and \ref{lem_transform_kappa}. For the original variables $z,w, \eta, \kappa$, the same lemmas  imply  the following. 

In  $D_{\gamma},  D_{up}$ and $D_{q,l}$  for $\im \tau> -c_{l,1}^{-1}$ we have
\begin{equation}
\label{after_lemma_+}
|\kappa(t)-\dK_{\eps}(\eps t)|< c_{m,6}\eps^3 \hat {d}_{+}^{-3}  ,\ 
 | z(t)|<c_{m,7}\eps\hat { d}_{+}^{-1}, \  |\eta(t)|<c_{m,8}\eps \hat{ d}_{+}^{-1/2}.  
\end{equation}
In  $D_{\gamma},  D_{up}$ and $D_{q,l}$  for $\im \tau< c_{l,1}^{-1}$ we have
\begin{equation}
|\kappa(t)-\dK_{\eps}(\eps t)|< c_{m,6}\eps^3 \hat {d}_{-}^{-3}  ,\ 
 | z(t)|<c_{m,9}\eps\hat { d}_{-}^{-1/2}, \  |\eta(t)|<c_{m,8}\eps \hat{ d}_{-}^{-1/2}.  
\end{equation}

\medskip

Values $z(t), \eta (t)$ are obtained from the components of the vector\\
$ C^{-1}(\kappa(t))(x(t)- X(\kappa(t)))$. Denote by $\hat z(t), \hat \eta (t)$  the corresponding quantities obtained in the same way from the components of the vector $ C^{-1}(\dK(\eps t))(x(t)- X(\dK(\eps t)))$. 
\begin{lem}
\label{l_XC}
For $\im \tau> -c_{l,1}^{-1}$, we have
\begin{equation}
\begin{aligned}
&|\hat z(t)|<c_{m,10}\eps\hat { d}_{+}^{-1}, 
|\hat \eta(t)|<c_{m,11}\eps\hat { d}_{+}^{-1/2} .
\end{aligned}
\end{equation}
For $\im \tau< c_{l,1}^{-1}$, we have
\begin{equation}
\begin{aligned}
&|\hat z(t)|<c_{m,12}\eps\hat { d}_{-}^{-1/2}, \
|\hat \eta(t)|<c_{m,11}\eps\hat { d}_{-}^{-1/2}.
\end{aligned}
\end{equation}
\end{lem}

\medskip
Now we should obtain estimate that would later specify  the value $\tau_{\gamma}$, which also specifies the curve $\Gamma_{*,1, \eps}$.  We have chosen  the point $\tau_{\gamma}$ on the ray $\gamma_1$ such that $\im \tau_{\gamma}< \im \tau_c-C_{\delta}\eps^{2/3}$ and $\hat {d}_{+}(\tau_{\gamma})=C_{\gamma,*}\eps^{2/3}$.
Recall that the curve $\Gamma_{*,1, \eps}$ and the ray $\gamma_1$  intersect at 
$\tau=\tau_{\gamma}$. Thus,  estimates (\ref{after_lemma_+}) are satisfied at this value of $\tau$. We should specify the value  $C_{\gamma,*}$. For $t_{\gamma}=\tau_{\gamma}
/\eps$, we have $| \hat z(t_{\gamma})|<c_{m,10}\eps\hat { d}_{+}^{-1} =\frac{c_{m,10}}{C_{\gamma,*}}\eps^{1/3}, \  |\hat \eta(t_{\gamma})|<c_{m,11}\eps \hat{ d}_{+}^{-1/2}=\frac{c_{m,11}}{C_{\gamma,*}^{1/2}}\eps^{2/3}$. Values $ \hat z(t_{\gamma}),\hat  \eta(t_{\gamma})$ are the projections of the deviations of the phase point from $X(\dK(\tau_{\gamma}))$ onto the eigenspaces of the equilibrium $X(\dK(\tau_{\gamma}))$ of the fast system. These estimates guarantee that the exact solution is close to the equilibrium of the fast system at $\kappa=\dK(\tau_{\gamma})$,  if $\eps$ is sufficiently small. We should choose $C_{\gamma,*}$  sufficiently large to  guarantee that the exact solution is close  to the special solution with  $R=R_ -=\exp(2\pi i/3)$ of Riccati equation
 (\ref{e_riccati}). For this special solution, at $\tau=\tau_{\gamma}$, we have (see (\ref{special_1}))
 $$
 |\hat \ze -(-\sqrt{-\hat  s})|<c_{m,13}|\sqrt{-\hat s}\,  v^{-1}|, 
 $$
 where $\hat \ze, \hat  s, v$ are the variables introduced in Section \ref{s_riccati}. In the original variables this is reduced to
 $$
 |z_1-\sqrt{-a^{-1}b g_c(\tau-\tau_c)}|<\frac{c_{m,14}}{C_{\gamma,*}} \eps^{1/3}.
 $$
 (Recall that  $z_1$ is the first coordinate in the eigen-coordinate system for $\kappa=\kappa_c$.) 
 
 \medskip
 Thus, at $\tau=\tau_{\gamma}$:
 
 - the distance of the exact solution from the equilibrium $X((\kappa(\tau_{\gamma}/\eps))$ of the fast system is 
 $O(\eps\hat { d}_{+}^{-1})
 =O(\eps^{1/3}/C_{\gamma,*})$ in $z$, and $O( \eps\hat { d}_{+}^{-1/2})
  =O(\eps^{2/3}/C_{\gamma,*}^{1/2})$ in $\eta, w$;

 - the value  $\dK(\tau_{\gamma})$ differs from the value of $\kappa$  in  the simplified (expanded) slow system by
  $$O(|\tau_{\gamma} -\tau_c|^{3/2})=O((C_{\gamma,*}\eps^{2/3})^{3/2})=
 O(C_{\gamma,*}^{3/2}\eps);$$
 
  -the value $\kappa(t)$ differs from the value of $\kappa$  in  the simplified (expanded) slow system by 
  the sum of differences of $\kappa$ and $\dK_{\eps}$, of  $\dK_{\eps}$ and  $\dK$, and of  $\dK$ and the value of $\kappa$  in the simplified (expanded) slow system, i. e.  by 
    $$O(c_{m,6}\eps^3 \hat {d}_{+}^{-3}+\eps(1+|\ln \hat {d}_{+}(\tau|) +(C_{\gamma,*}^{3/2}\eps))=
  O(\eps|\ln\eps|+C_{\gamma,*}^{3/2}\eps)
  ;$$
 
- the distance between the equilibrium $X(\dK(\tau_{\gamma}))$ and the equilibrium of the simplified (expanded) fast system is $O(C_{\gamma,*}\eps^{2/3})$ in $z_1$   and  $O(C_{\gamma,*}^{3/2}\eps)$ in $\eta$.
(The second  of these values is obtained as a difference between $\dK(\tau_{\gamma})$  and the value of $\kappa$  in  the simplified (expanded) slow system, which is  $O(|\tau_{\gamma} -\tau_c|^{3/2})= O(C_{\gamma,*}^{3/2}\eps)$. The first of these values is obtained as the product of this difference and $\hat d_+^{-1/2}$, which is $O(|\tau_{\gamma} -\tau_c|^{-1/2}|\tau_{\gamma} -\tau_c|^{3/2} )=O(|\tau_{\gamma} -\tau_c|)=C_{\gamma,*}\eps^{2/3}$. These estimates also take into account a difference  of eigen-axes of equilibria for $\kappa=\dK(\tau_{\gamma})$ and $\kappa=\kappa_c$.)

\medskip
Thus, at $\tau= \tau_{\gamma}$, in the variable $z_1$, the distance between the exact solution and the special solution can be estimated as

 the distance between the exact solution and the equilibrium of the fast system at  $\kappa=\kappa(t_{\gamma})$, 
 $O(c_{m,10}\eps\hat { d}_{+}^{-1}),$\\
 $+$\\
 the distance between the equilibrium of the fast system at $\kappa=\kappa(t_{\gamma})$ and the equilibrium of the fast system at $\kappa= \dK_{\eps}(\tau_{\gamma})$, 
 $O(\hat { d}_{+}^{-1/2}c_{m,6}\eps^3 \hat {d}_{+}^{-3}),$\\
 $+$\\
  the distance between the equilibrium of the fast system at $\kappa=\dK_{\eps}(\tau_{\gamma})$ and the equilibrium of the fast system at $\kappa=\dK(\tau_{\gamma})$,
  $O(\hat { d}_{+}^{-1/2}\eps(1+|\ln \hat {d}_{+}(\tau|)),$\\
 $+$\\
 the distance between the equilibrium of the fast system at $\kappa=\dK(\tau_{\gamma})$ and the equilibrium of the expanded system at $\tau=\tau_{\gamma}$ ,
 $O((|\tau_{\gamma} -\tau_c|^{-1/2}     (|\tau_{\gamma} -\tau_c|^{3/2})=O(|\tau_{\gamma} -\tau_c|),$\\
 $+$\\
  the distance between the equilibrium of the expanded system at $\tau=\tau_{\gamma}$  and the special solution at
 $\tau=\tau_{\gamma}$, 
 $O(\frac{c_{m,14}}{C_{\gamma,*}} \eps^{1/3})$.\\
 
  Thus, we get  the following estimate for this value 
\begin{equation*}
\begin{aligned}
&O(c_{m,10}\eps\hat { d}_{+}^{-1})+O(\hat { d}_{+}^{-1/2}c_{m,6}\eps^3 \hat {d}_{+}^{-3})+O(\hat { d}_{+}^{-1/2}\eps(1+|\ln \hat {d}_{+}(\tau|))+O(|\tau_{\gamma} -\tau_c|)+O(\frac{c_{m,14}}{C_{\gamma,*}} \eps^{1/3})\\
&=O(\eps^{1/3}/C_{\gamma,*}+\eps^{2/3}/C_{\gamma,*}^{7/2}+\eps^{2/3}|\ln\eps|/C_{\gamma,*}^{1/2}+C_{\gamma,*}\eps^{2/3})+O(\eps^{1/3}/C_{\gamma,*})\\
&=O(\eps^{1/3}/C_{\gamma,*}  +C_{\gamma,*}\eps^{2/3}+\eps^{2/3}|\ln\eps|/C_{\gamma,*}^{1/2} ).
\end{aligned}
\end{equation*}
\medskip
Denote by $z_{1,\eps}(t)$ the first component of the solution $x(t)$ in the eigen-coordinate system constructed for the point $\kappa_c$. Denote by $z_{1,sp}(t)$ the value $z_1$ along the special solution. We have shown that
\begin{equation*}
\begin{aligned}
z_{1,\eps}(t_\gamma)-z_{1,sp}(t_\gamma)=O(\eps^{1/3}/C_{\gamma,*} +C_{\gamma,*}\eps^{2/3}+\eps^{2/3}|\ln\eps|/C_{\gamma,*}^{1/2}).
\end{aligned}
\end{equation*}
Introduce $R_{\gamma}$ as the value of $R$ in solution (\ref{solution_1}) that would give  $z_1=z_{1,\eps}(t_\gamma)$ at $t=t_\gamma$.  The value of $R$ that gives $z_{1,sp}(t)$  is 
$R_- =\exp(2\pi i/3)$.
\begin{lem}
\label{delta_r_1}
\begin{equation*}
\begin{aligned}
R_{\gamma}-R_-=O(|z_{1,\eps}(t_\gamma)-z_{1,sp}(t_\gamma)|/(C_{\gamma,*}^{1/2}\eps^{1/3} )).
\end{aligned}
\end{equation*}
\end {lem}
Thus,
\begin{equation*}
R_{\gamma}-R_-=O(1/C_{\gamma,*}^{3/2} +C_{\gamma,*}^{1/2}\eps^{1/3}+\eps^{1/3}|\ln\eps|/C_{\gamma,*}).
\end{equation*}
\medskip

 Denote
  $$\tau_{\eps}(t)=\tau_c +b \cdot(\kappa(t)-\kappa_c)/ (b \cdot g_c).$$
 Here
$g_c=g(x_c, \kappa_c,0)$.
Denote by $s_{\eps}$ the value obtained from $\tau_{\eps}(t)$ by  the formulas expressing $s$ via $\tau$ (see Section \ref{s_riccati}). Denote $ \zeta_{\eps}(t)=e^{-i\alpha_1}z_{1,\eps}(t)$, and  introduce corresponding values $\hat s_{\eps}(t), \hat \zeta_{\eps}(t)$,  as in  Section \ref{s_riccati}. 
 Introduce $v_{\eps}(t)=\frac{2}{3}(-\hat s_{\eps}(t))^{3/2}$.  We have
\begin{equation}
\begin{aligned}
\label{for_tau_eps}
\tau_{\eps}(t_{\gamma})&-\tau_{\gamma}=\tau_c+b \cdot(\kappa(t_{\gamma})-\kappa_c)/(b \cdot g_c)-\tau_{\gamma}\\
&=\tau_c+b \cdot(\dK(\tau_{\gamma})-\kappa_c)/(b\cdot g_c)-\tau_{\gamma}+O(\kappa(t_{\gamma})-\dK(\tau_{\gamma}))\\
&=\tau_c+b \cdot(\kappa_c+g_c(\tau_{\gamma}-\tau_c)+O(\tau_{\gamma}-\tau_c)^{3/2}-\kappa_c)/(b\cdot g_c)-\tau_{\gamma} +O(\eps|\ln\eps|+C_{\gamma,*}^{-3}\eps)\\
&=O((\tau_{\gamma}-\tau_c)^{3/2})+ O(\eps|\ln\eps|+C_{\gamma,*}^{-3}\eps)
= O(\eps|\ln\eps|+C_{\gamma,*}^{3/2}\eps).
\end{aligned}
\end{equation}

  Introduce $R_{\eps}(t)$  such that
  \begin{equation}
  \label{solution_eps}
\hat\ze_{\eps}(t)=-i\sqrt{-\hat s_{\eps}(t)}\,
\frac{J_{-2/3}(v_{\eps}(t))-R_{\eps}(t)J_{2/3}(v_{\eps}(t))}{R_{\eps}(t)J_{-1/3}(v_{\eps}(t))+J_{1/3}(v_{\eps}(t))}.
\end{equation}
This value is well defined at least at $t=t_{\gamma}$.
\begin{lem}
\label{delta_r_2}
$$
R_{\eps}(t_{\gamma})-R_\gamma=  O(|\tau_{\eps}(t_{\gamma})-\tau_{\gamma}|/(C_{\gamma,*}\eps^{2/3})).
$$
\end{lem}
Thus,
$$
R_{\eps}(t_{\gamma})-R_\gamma=O(\eps^{1/3}|\ln\eps|/C_{\gamma,*}+\eps^{1/3}C_{\gamma,*}^{1/2}).
$$
Therefore,
\begin{equation*}
\begin{aligned}
&|R_{\eps}(t_{\gamma})-R_-|\le |R_{\eps}(t_{\gamma})-R_{\gamma}|+|R_{\gamma}-R_-|\\
&=O(\eps^{1/3}|\ln\eps|/C_{\gamma,*}+\eps^{1/3}C_{*}^{1/2})
+O(1/C_{\gamma,*}^{3/2}+\eps^{1/3}C_{\gamma,*}^{1/2}.
+\eps^{1/3}|\ln\eps|/C_{\gamma,*})
\\
&=O(1/C_{\gamma,*}^{3/2}+\eps^{1/3}(C_{\gamma,*}^{1/2}+C_{\gamma,*}^{-1}|\ln\eps|)).
\end{aligned}
\end{equation*}

For $C_{\gamma,*}>c_{m,15}$,  we have
\begin{equation}
\label{delta_R_1}
|R_{\eps}(t_{\gamma})-R_{-}|< \frac {c_{m,16}}{C_{\gamma,*}^{3/2}}<\frac{1}{200}. 
\end{equation}

\medskip
Here we  fix $C_{\gamma,*}, C_q$.  For this purpose, we will need one more constant, $c_{m,19}$, which we will specify later, in Section  \ref{domain_3}. This constant is determined by the properties of the function   $\hat \chi$ in (\ref {solution_2}).
Recall that  $C_{\gamma}$ is determined by  $C_{\gamma,*}$.  We take $C_q=C_{\gamma}$ and choose $C_{\gamma,*}$ sufficiently large so that
 $$C_{\gamma,*}>\max \{c_{m,15}, c_{m,19} \}  \ \mbox {and} \ C_{\gamma}>\max \{c_{e,1}, c_{e,10}\}  .$$
 
Since  $C_{\gamma,*}$ is now a fixed constant, it is not necessary to indicate dependence on it.  We will nevertheless indicate this dependence where it is instructive. 

\subsection{Second part of the principal  part of  motion: continuation of solution up to  singularity}
\label{second_principal_part}

Denote by $D_{\triangle}$ the triangle in the complex slow time plane bounded by the lines $\gamma_1, \im \eps t =\im \tau_{\gamma}, \re \eps t =\re \tau_c$ (Fig.  \ref{d_triangle}). Let  $\bar D_{\triangle}$ be  the triangle complex conjugate to $D_{\triangle}$. Denote \\$D_c= D_{up}\cup D_{\triangle} \cup \bar D_{\triangle}$ (notation $D_{up}$ is introduced at the end of Section \ref{improved_curves}).
 Denote by $z_{sm},w_{sm},\eta_{sm},\xi_{sm}$ values analogous to $z,w,\eta,\xi$ in Lemma \ref{lem_transform_00}, but constructed for 
 $x-X(\kappa_c)$ and matrix $A_c={\cal A}(\kappa_c)$. Denote by $z_{ms},w_{ms},\eta_{ms}, \xi_{ms}$ values analogous to $z,w,\eta,\xi$ in Lemma \ref{lem_transform_00},  but constructed for 
 $x-X(\bar \kappa_c)$ and matrix ${\cal A}(\bar\kappa_c)$.
 \begin{figure}[H]
\begin{center}
            \includegraphics[scale=0.4, angle=0.0]{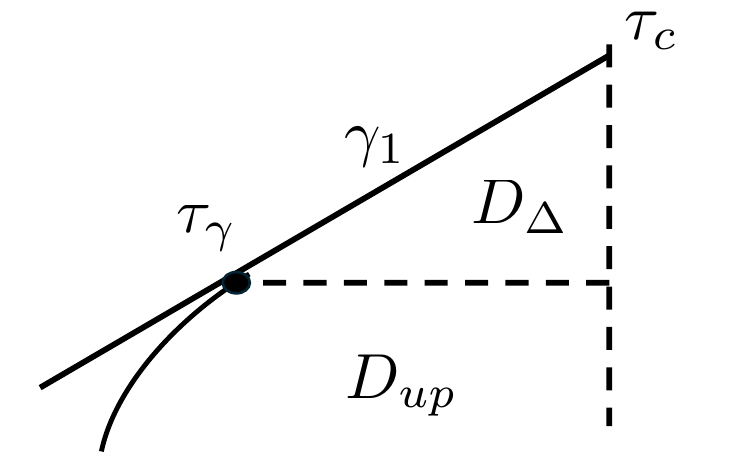}
            \end{center}
           \caption{Triangle $D_{\triangle}$ }
            \label{d_triangle}
\end{figure}

\begin{lem}
\label{lem_to_c}
Solution $x(t), \kappa(t)$ can be analytically continued   into the domain $D_c$ with the following estimates.

\medskip
If $\tau \in D_{\triangle}$, then
\begin{equation}
\begin{aligned}
&|\kappa(t)-\kappa_c|=O(\eps^{2/3} ) ,\ 
 | z_{sm}(t)|=O(\eps^{1/3}), \ | w_{sm}(t)|=O(\eps^{2/3}), \  | \eta_{sm}(t)|=O(\eps^{2/3}), \\
 &\kappa(t)=\overline{\kappa(\bar t)},  \ z_{sm}(t)= \overline {w_{ms}(\bar t)},\ w_{sm}(t)= \overline {z_{ms}(\bar t)}, \ \eta_{sm}(t)= \overline {\eta_{ms}(\bar t)},\\
 &  |\kappa(\tau/\eps)-\dK(\tau)|=O(\eps|\ln\eps|),\ |\kappa(\tau_c/\eps)-\kappa_c|=O(\eps|\ln\eps|).
 \end{aligned}
  \end{equation}
 
 In this domain values
$ R_{\eps}(t ), \hat s_{\eps}(t), \hat\ze_{\eps}(t )$ are well defined and
 \begin{equation}
\begin{aligned}
\label{delta_R_2}
&|R_{\eps}(\tau_{c}/\eps )-R_{\eps}(\tau_{\gamma}/\eps)|=O(\eps^{1/3} (|\ln\eps|+C_{\gamma,*}^{3/2})), 
\end{aligned}
\end{equation}
\begin{equation}
\begin{aligned}
&|\hat s_{\eps}(\tau_{c}/\eps )|=O(\eps^{1/3} (|\ln\eps|+C_{\gamma,*}^{3/2})),\\
&|\hat\ze_{\eps}(\tau_{c}/\eps )- \frac{1}{R_{\eps}(\tau_{\gamma}/\eps)}\frac{-2\pi i}{\Gamma^2(1/3)3^{1/6}}    |=O(\eps^{1/3} (|\ln\eps|+C_{\gamma,*}^{3/2})).
\end{aligned}
\end{equation}

The estimates for $\tau\in D_{up}$
  are given by Lemma \ref{lem_cont_D_gamma}.

\medskip
    If $\tau \in \bar D_{\triangle}$, then
\begin{equation}
\begin{aligned}
&|\kappa(t)-\bar \kappa_c|=O(\eps^{2/3})  ,\ 
 | z_{ms}(t)|=O(\eps^{2/3}), \ | w_{ms}(t)|=O(\eps^{1/3}),\ | \eta_{ms}(t)|=O(\eps^{2/3}).
 \end{aligned} 
 \end{equation}
\end{lem}

\medskip
Note that $\kappa$ in the present lemma denotes the original variable
 $\kappa$, whereas in Lemma \ref{lem_cont_D_gamma} it denotes the variable obtained by the transformation described in Lemma \ref{lem_transform_kappa}.

\begin{cor}
Estimates (\ref{delta_R_1}) and  (\ref{delta_R_2}) imply that
$$|R_{\eps}(\tau_{c}/\eps )-R_{-}|<
\frac {c_{m,17}}{C_{\gamma,*}^{3/2}}<\frac{1}{100}, \quad c_{m,17}=2c_{m,16}.$$
\end{cor}


\subsection{A domain to which the solution will be continued }
\label{domain_3}

Recall that $R_{-}=e^{2\pi i/3}$. Then $ R_{+}=e^{-2\pi i/3}R_{-}=1$.
Denote 
$
R^+_{\eps}(\tau_{c}/\eps )=e^{-2\pi i/3} R_{\eps}(\tau_{c}/\eps ).
$
Then
$$
|R^+_{\eps}(\tau_{c}/\eps )-1|<\frac {c_{m,17}}{C_{\gamma,*}^{3/2}}< \frac{1}{100}.
$$
We will use functions $\hat\sigma_{\eps}(t)=e^{-2\pi i/3}\hat s_{\eps}(t)$ and  $\hat \chi_{\eps}(t)=e^{2\pi i/3}\hat\ze_{\eps}(t)$.  Then 
\begin{equation}
\label{solution_eps_chi}
\hat\chi_{\eps}(t)=-i\sqrt{-\hat \sigma_{\eps}(t)}\,
\frac{J_{-2/3}(v_{\eps}(t))-R^+_{\eps}(t)J_{2/3}(v_{\eps}(t))}{R^+_{\eps}(t)J_{-1/3}(v_{\eps}(t))+J_{1/3}(v_{\eps}(t))}, 
\ v_{\eps}(t)=\frac{2}{3}(-\hat \sigma_{\eps}(t))^{3/2}.
\end{equation}
Indeed, we know that the substitution $s=e^{2\pi i/3}\sigma$ together with the  replacement of $R$ with $ e^{-2\pi i/3}R$ in (\ref{solution_1}) results  in multiplication of the right hand side of  (\ref{solution_1}) by $e^{2\pi i/3}$ (cf.  (\ref{solution_1}) ,  (\ref{solution_2}), and
 (\ref{r_relation})). This implies (\ref{solution_eps_chi}). The left- and right-hand sides of (\ref {solution_eps_chi}) are well defined at least for $\re \tau= \re \tau_c,\ \im \tau\le \im \tau_c$.


\medskip

Take $c_{m,18}, c_{m,19}$ such that, for  $ C_{\gamma,*}>c_{m,19}$,
 the function $\hat \chi$ (\ref {solution_2}) with $R$ satisfying  \\
 $|R-1|\le \frac {c_{m,17}}{C_{\gamma,*}^{3/2}}< 1/100$ has no poles on $\gamma_2$ for $ \re \tau_{c}  \le  \re\tau  \le \re \tau_c+c_{m,18}^{-1}\eps^{2/3}$ or  below  $\gamma_2$.

\medskip


 Denote by $D_{1, 1 }$ the curvilinear trapezoid bounded by the  lines   $\re \tau =\re \tau_{c},\\ \re \tau = \re \tau_c+c_{m,18}^{-1}\eps^{2/3},\ \gamma_2$ , and by the upper boundary of $D_{q,r}$. Let $\bar D_{1, 1}$ be its complex conjugate curvilinear trapezoid.
 Denote by $D_{1, 2 }$ the part of the domain $D_{q,r}$ with $\re \tau_{c} \le\re \tau\le \re \tau_c+c_{m,18}^{-1}\eps^{2/3}$.
 
 \medskip

\medskip

Let us describe the domain in the slow complex time plane  to which we should continue the solution (Fig. \ref{d_2}).
 \begin{figure}[H]
\begin{center}
            \includegraphics[scale=0.4, angle=0.0]{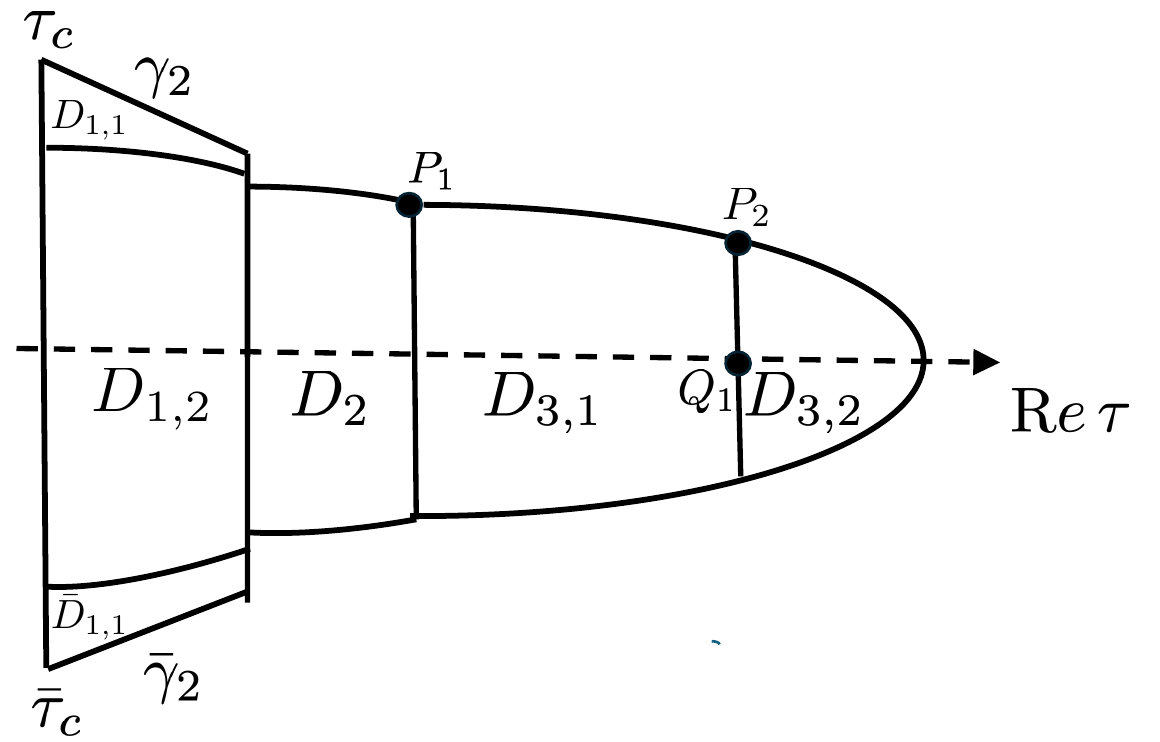}
            \end{center}
           \caption{Domains $D_1=(D_{1,1}\cup D_{1,2}\cup \bar D_{1,1}),\  D_2$ and $D_3=(D_{3,1}\cup D_{3,2}) $ }
            \label{d_2}
\end{figure}
\noindent
This domain is the union of three domains, $D_1, D_2$, and $D_3$. The domain $D_1$ is a trapezoid bounded by the lines   $\re \tau =\re \tau_{c},\  \re \tau = \re \tau_c+c_{m,18}^{-1}\eps^{2/3},\ \gamma_2,$ and  $\bar{\gamma_2}$. 
The domain $D_2$ is bounded by  the curve  $\re \psi_{a}=-C_{a,0}\eps$, the complex conjugate to it curve, the line $ \re \tau = \re \tau_c+c_{m,18}^{-1}\eps^{2/3}$,  and the line  $ \re \tau = \re \tau_c+C_{a,1}\eps^{2/3}$. Here $C_{a,0}$ and  $C_{a,1}$ are constants that will be specified later. The complex phase   $\psi_a$ is introduced in (\ref {phase_a_1}). Denote by $P_1$ the upper right corner of the domain $D_2$. The domain $D_3$ is bounded  by the line $ \re \tau = \re \tau_c+C_{a,1}\eps^{2/3}$,    the segment of the curve $\re \Psi_{\eps}={\rm const}$ from $P_1$ till the real axis (denote by $\tau_{+, C_{a,0}, \eps}$ the corresponding point of the real axis), and the complex conjugate to it curve. 
Denote by  $Q_1$  the point of intersection of the curve $\Gamma_{q,r,d}$  and the real axis in the $\tau$-plane. 
Note  that for $d_{+}\sim 1$ the vertical  distance between the upper boundary of the domain $D_{q,r}$ (curve 
$\Gamma_{q,\eps}$) and the considered level curve $\re \Psi_{\eps}={\rm const}$ is of order $\eps$, while the vertical
 width of the domain $\tilde D_{q,r,d}$ is of order $\eps |\ln \eps|$. Thus, for  $d_{+}\sim 1$, the considered level curve  $\re \Psi_{\eps}={\rm const}$ is inside the domain  $\tilde D_{q,r,d}$ and passes above $Q_1$. We represent the domain $D_3$ as the union of two domains,  $D_3=D_{3,1}\cup D_{3,2}$.
The border between $D_{3,1}$ and  $D_{3,2}$  is the vertical line $\re \tau =\re Q_1$. Note that  $\tau_{+, C_{a,0}, \eps}-\re Q_1=O(\eps|\ln\eps|)$.

 
 \medskip
Define the  constant  $c_{m,22}$ as follows. Consider the expansion  (\ref{expansion}):
\begin{equation}
\label{expansion_copy}
\hat\chi=\sqrt{-\hat\sigma}\Bigl [1-2e^{-2iv}e^{\pi i/6}
\frac{R-e^{-2\pi i/3}}{R- e^{2\pi i/3}}
+O\Bigl (e^{-4|{\rm Im\, }v|}+\frac{1}{|v|}\Bigr)\Bigr ]. 
\end{equation}
Here $\hat\chi,\ \hat\sigma$  can be replaced with $\chi=\eps^{1/3}\hat\chi, \  \sigma=\eps^{2/3}\hat\sigma$. We have
$v=\frac{2}{3}(-\hat\sigma)^{3/2}=\frac{2}{3}(-\sigma)^{3/2}/\eps$.

For $|R-1|<1/100$, we have  
$$
|\frac{R-e^{-2\pi i/3}}{R- e^{2\pi i/3}}|>c_{m,20}^{-1}.
$$
We choose $c_{m,21}, c_{m,22}$  such that on the curve   $\eps^{-1}\re \psi_{a}=\re(-2iv)=2\im v= -C_{a,0}$, for $C_{a,0}>c_{m,21}$, we have in (\ref{expansion_copy})
\begin{equation}
\label{exp_esimate_1}
|2e^{-2iv}e^{\pi i/6}
\frac{R-e^{-2\pi i/3}}{R- e^{2\pi i/3}}
+O\Bigl (e^{-4|{\rm Im\, }v|} \Bigr )|
> c_{m,22}^{-1}e^{-C_{a,0}}.
\end{equation}

 We take  $P_1$   on the curve   $\eps^{-1}\re \psi_{a}=-C_{a,0}$  in such a way that  for $|R-1|<1/100$ at $P_1$ we have   in (\ref{expansion_copy})
\begin{equation}
\label{exp_esimate_2}
|\Bigl [-2e^{-2iv}e^{\pi i/6}
\frac{R-e^{-2\pi i/3}}{R- e^{2\pi i/3}}
+O\Bigl (e^{-4|{\rm Im\, }v|}+\frac{1}{|v|}\Bigr)\Bigr ]|> 0.5 c_{m,22}^{-1} e^{-C_{a,0}}. 
\end{equation}
This is guaranteed, provided that $|v| \ge c_{m,23}e^{C_{a,0}}$ at $P_1$.
At the point $P_1$ we have
\begin{equation}
c_{m,24}^{-1}e^{-C_{a,0}}<|\frac{\chi-\sqrt{-\sigma}}
{\sqrt{-\sigma}}|<c_{m,24}e^{-C_{a,0}}. 
\end{equation}

Now we can determine $C_{a,0}$. Constant  $C_{a,0}>c_{m,21}$ is such that the curve\\ $\eps^{-1}\re \psi_{a}=-C_0$ with $ \re \tau_c \le\re \tau \le \re \tau_c+c_{m,18}^{-1}\eps^{2/3}$ belongs to $D_q$, and is at a distance larger than 1 (in $\tau/\eps^{2/3}$) from the boundary of $D_q$.
This guarantees that, for the choice of $C_{a,1}$ indicated below, the arc of the curve $\eps^{-1}\re \psi_{a}=-C_{a,0}$ ending at the point $P_{1}$ belongs to $D_{q}$ and remains at a positive distance from its boundary.
 The choice of $C_{a,1}$ ensuring that (\ref {exp_esimate_2}) holds is  $C_{a,1}>c_{m,25}e^{2C_{a,0}/3}$. Later we will impose an additional condition on   $C_{a,1}$
  (see Lemmas \ref {lem_after_c_2_1},  \ref {lem_after_c_3}).
 
\subsection{Third  part of the principal  part of  motion}
\label{third_principal_part}

\begin{lem}
\label{lem_after_c_1}
The solution $x(t), \kappa(t)$ can be  analytically continued  into $D_1$  with the following estimates.

\medskip
If $\tau \in D_{1,1}$, then
\begin{equation}
\begin{aligned}
\label{e_after_c_1_1}
&|\kappa(t)-\kappa_c|=O(\eps^{2/3} ) ,\ 
 | z_{sm}(t)|=O(\eps^{1/3}), \ | w_{sm}(t)|=O(\eps^{2/3}), \  | \eta_{sm}(t)|=O(\eps^{2/3}), \\
 &\kappa(t)=\overline{\kappa(\bar t)},  \ z_{sm}(t)= \overline {w_{ms}(\bar t)},\ w_{sm}(t)= \overline {z_{ms}(\bar t)}, \ \eta_{sm}(t)= \overline {\eta_{ms}(\bar t)}.
 \end{aligned}
  \end{equation}
 The estimates for $\tau\in D_{1,2}$
  are given by Lemma \ref{lem_cont_D_q_r}.

\medskip
    If $\tau \in \bar D_{1,1}$, then
\begin{equation}
\begin{aligned}
\label{e_after_c_1_2}
&|\kappa(t)-\bar \kappa_c|=O(\eps^{2/3})  ,\ 
 | z_{ms}(t)|=O(\eps^{2/3}), \ | w_{ms}(t)|=O(\eps^{1/3}),\ | \eta_{ms}(t)|=O(\eps^{2/3}).
 \end{aligned} 
 \end{equation}
 
 On the part of the curve $\eps^{-1}\re \psi_{a}=-C_{a,0}$ with  $   \re \tau_{c}  \le  \re\tau  \le \re \tau_c+c_{m,18}^{-1}\eps^{2/3}$, we have
 \begin{equation}
 \begin{aligned}
 \label{est_on_psi_a}
 &\hat \sigma_{\eps}(\tau/\eps)=\hat \sigma
+O(\eps^{1/3}(|\ln \eps|+C_{a,0})),\  R_{\eps}^+(\tau/\eps)=R_{\eps}^+ +O(\eps^{1/3}C_{a,0}^{1/3}e^{C_{a,0}} ),\\
 &\hat \chi_ {\eps}(\tau/\eps)=-i\sqrt{-\hat \sigma}\,
\frac{J_{-2/3}(v)-R^+_{\eps}(\tau_c/\eps)J_{2/3}(v)}{R^+_{\eps}(\tau_c/\eps)J_{-1/3}(v)+J_{1/3}(v))}
+O(\eps^{1/3}(|\ln \eps|C_{a,0}^{-1/3}+C_{a,0}^{2/3})),\\
& \hat \chi_ {\eps}(\tau/\eps)=\sqrt{-\hat\sigma}\Bigl [1-2e^{-2iv}e^{\pi i/6}
\frac{R^+_{\eps}(\tau_c/\eps)-e^{-2\pi i/3}}{R^+_{\eps}(\tau_c/\eps)- e^{2\pi i/3}}
+O\Bigl (e^{-4|{\rm Im\, }v|}+\frac{1}{|v|}\Bigr)\Bigr ]\\
&+O(\eps^{1/3}(|\ln \eps|C_{a,0}^{-1/3}+C_{a,0}^{2/3})),\\
&c_{m,26}^{-1}e^{-C_{a,0}}-c_{m,27}\frac{1}{|v|}<|\frac{\hat \chi_ {\eps}(\tau/\eps)-\sqrt{-\hat\sigma}}{\sqrt{-\hat\sigma}}|
+O(\eps^{1/3}(C_{a,0}^{-2/3}|\ln \eps|+C_{a,0}^{1/3}))< c_{m,26}e^{-C_{a,0}}+c_{m,27}\frac{1}{|v|},\\
&c_{m,26}^{-1}e^{-C_{a,0}}-c_{m,27}\frac{1}{|v|}<|\frac{z(t)}
{\sqrt{\sigma}}|+O(\eps^{1/3}(C_{a,0}^{-2/3}|\ln \eps|+C_{a,0}^{1/3}))<c_{m,26}e^{-C_{a,0}} +c_{m,27}\frac{1}{|v|},\\
&  |\kappa(t)-\dK_{\eps}(\eps t)|=O(\eps). 
\end{aligned}
\end{equation}
\end{lem}
Note that $z, \kappa$ in the present lemma and the next one denote the original variables  $z, \kappa$, whereas in Lemma \ref{lem_cont_D_q_r} they denote the transformed variables introduced in
 Lemmas \ref{lem_transform_2} and  \ref{lem_transform_kappa}. 
\medskip

{\bf Remark.}  Since the constant $C_{a,0}$ was fixed before Lemma \ref{lem_after_c_1}, the form of the dependence of the other constants on $C_{a,0}$ is not important. It is given here merely for definiteness.

\medskip

\subsection{Fourth  part of the principal  part of  motion}
\label{fourth_principal_part}
 
 \begin{lem}
 \label{lem_after_c_2}
For $C_{a,1}>c_{m,25}e^{2C_{a,0}/3}$, 
 the solution $z(t), \eta(t), \kappa(t)$ can be  analytically continued   into the domain $D_2$ with the estimates given in Lemma  \ref{lem_cont_D_q_r}. At the point $P_1$ we have 
\begin{equation}
\label{est_after c_2}
 c_{m,28}^{-1}\eps^{1/3}\sqrt{|\hat \sigma |}e^{-C_{a,0}}<|z|<c_{m,28}\eps^{1/3} \sqrt{|\hat \sigma |}e^{-C_{a,0}}.
 \end{equation}

 
   
\end{lem}

Let  $\hat z$ be the variable introduced  in Lemma  \ref{lem_transform_2}. 
\begin{lem}
 \label{lem_after_c_2_1}
For $C_{a,1}>c_{m,29}e^{2C_{a,0}/3}$, at the point $P_1$, we have 
\begin{equation}
\label{est_after c_2_1}
 c_{m,30}^{-1}\eps^{1/3}\sqrt{|\hat \sigma |}e^{-C_{a,0}}<|\hat z|<c_{m,30}\eps^{1/3} \sqrt{|\hat \sigma |}e^{-C_{a,0}}.
 \end{equation}
 \end{lem}
 From now until Lemma \ref {lem_to_original} we omit the ``hat'' over $z$.

\subsection{Fifths part of the principal  part of  motion}
\label{fifths_principal_part}

Denote by $\Gamma^{C_{a,0},\eps}$ the arc of the curve  $\re \Psi_{\eps}={\rm const}$ passing through the point $P_1$. 
\begin{lem}
\label{lem_last_arc}
The arc  $\Gamma^{C_{a,0},\eps}$ crosses the real axis $\im \tau =0$ at a point\\ $\tau_{+,C_{a,0}, \eps}=\tau_{*}^{+}+ O(\eps|\ln \eps|)$. 
\end{lem}
Denote by  $\hat \sigma_1$ and  $z_1$  the values of $\hat\sigma$ and $z$ at the point $P_1$. We have \\
$c_{m,31}^{-1}C_{a,1}<\hat \sigma_1<c_{m,31}C_{a,1}$,\quad $c_{m,30}^{-1}\eps^{1/3}\sqrt{|\hat \sigma_1|}e^{-C_{a,0}}<|z_1|<c_{m,30}\eps^{1/3} \sqrt{|\hat \sigma_1|}e^{-C_{a,0}}$.
 \begin{lem}
 \label{lem_after_c_3}
 For  $C_{a,1}> {\rm max} \{ c_{m,29}e^{2C_{a,0}/3}, c_{m,32},  c_{m,33}e^{C_{a,0}/3}\}$, 
  the solution $z(t), \eta(t), \kappa(t)$ can be   analytically continued  into the domain $D_{3,1}$  with estimates given in Lemmas  \ref{lem_cont_D_q_r} and \ref{lem_improved_kappa_old_1}. 
On the curve $\Gamma^{C_{a,0},\eps}$ we have
 $$0.5|z_1|<|z|<2|z_1|, \ \eta=O(\eps), \ w=O(\eps).$$


\end{lem}

\begin{lem}
\label{lem_to_D32}

 The solution $z(t), \eta(t), \kappa(t)$ can be  analytically  continued  into the domain $D_{3,2}$, and at  $\tau=\tau_{+,C_{a,0,} \eps}$ we have $|z(t)|=|w(t)|>c_{m,34}^{-1}\eps^{1/3}, \eta(t)=O(\eps), |\kappa(t)-\dK_{\eps}(\eps t)|=O(\eps)$.
\end{lem}

The variables $z,w, \eta, \kappa$ here are those obtained after transformations  in Lemma \ref{lem_transform_2}. Now we should return to the original variables.
\begin{lem}
\label{lem_to_original}
At $\tau=\tau_{+,C_{a,0}, \eps}$ in the the original variables 
 $z,w, \eta, \kappa$  we have  $|z(t)|=|w(t)|>c_{m,35}^{-1}\eps^{1/3}, \eta(t)=O(\eps^{2/3}), |\kappa(t)-\dK(\eps t)|=O(\eps)$.
\end{lem}

\begin{lem}
\label{lem_escape}
There exists  a real  $\tau_d=\tau^{+}_* + O(\eps|\ln \eps|)$  such that  $|x(t_d)- X(\kappa(t_d))|> c_{m,36}^{-1}$. 
\end{lem}

\medskip
This completes the proof of Theorem \ref{t1.th}.


 

%% file: delay_proofs_of_lemmas_1.tex
\section{Proofs of lemmas about curves}
 \label{proofs_curves}
 
 {\bf Proof of Lemma \ref {lem_Gamma_gamma}.}
 
For definiteness,  consider the case $j=1$. The proof for $j=2$ is analogous.  Let $Oxy$ be the orthogonal coordinate  frame in the complex slow time plane such that the point $O$ is at $\tau_c$,  and the axis $Ox$ is directed along    $ \gamma_{1}$ (Fig. \ref {l_gamma_gamma}).
\begin{figure}[H]
\begin{center}
            \includegraphics[scale=0.6, angle=0.0]{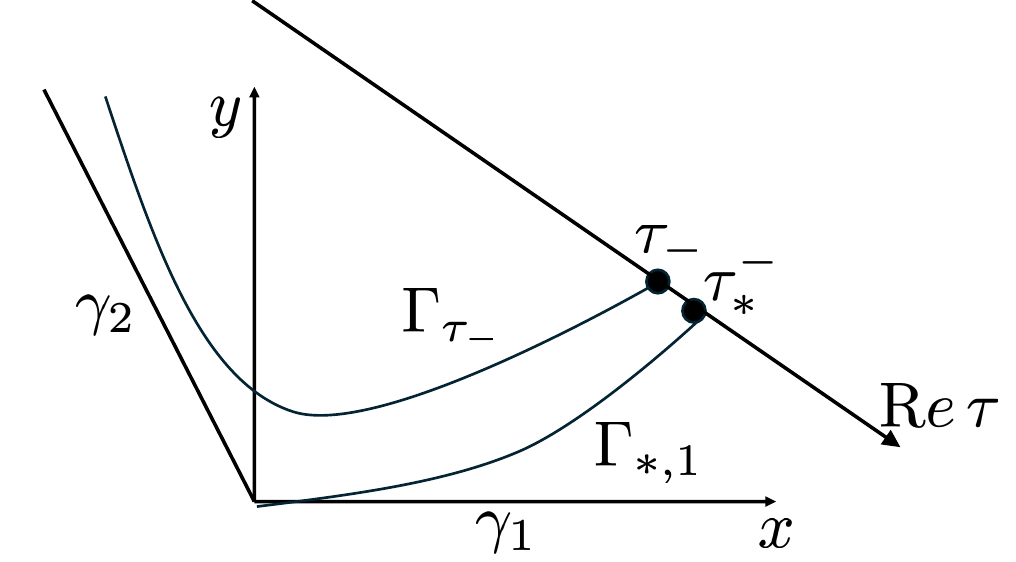}
            \end{center}
           \caption{For lemmas \ref{lem_Gamma_gamma}, \ref {lem_Gamma_Gamma}}
            \label{l_gamma_gamma}
\end{figure}
Let $\rho, \theta$ be the polar coordinates: $x=\rho \cos\theta,  y=\rho \sin\theta$. Then $\Psi(\tau)-\Psi(\tau_c)=i\alpha (x+iy)^{3/2}
+O(\rho^2)=i\alpha \rho^{3/2}(\cos (3\theta/2)+i\sin (3\theta/2)+O(\rho^2)  $.  Here $\alpha$ is a positive constant. The equation for 
$\Gamma_{*,1}$ is $\re(\Psi(\tau)-\Psi(\tau_c))=0$, which implies $-\alpha \rho^{3/2}\sin (3\theta/2)+O(\rho^2)=0$. Thus,
 $\sin (3\theta/2)+O(\rho^{1/2})=0$,
 and $\theta =O(\rho^{1/2})$. The distance between points of $ \gamma_{1}$ and  $\Gamma_{*,1}$  
 measured along normals to $ \gamma_{1} $ 
 is $y=\rho \sin\theta=O(\rho^{3/2})= O(|\tau-\tau_c|^{3/2})$.  The same estimate for the distance remains valid if this distance is measured along normals to $\Gamma_{*,1}$ and  $\tau$ is taken as the parameter along $\Gamma_{*,1}$.  This is because the angle between the tangent line to $\Gamma_{*,1}$ and the $Ox$-axis is $O(\rho^{1/2})$.

 \hskip 12cm $\square$

 \medskip
 {\bf Proof of Lemma \ref{lem_Gamma_Gamma}.}
 
One should give  a proof for $\tau $ in a small neighbourhood of $\tau_c$, because outside any fixed neighbourhood of $\tau_c$ the estimate in this lemma is evident.  
  
 Let $x,y,\rho,\theta$  be variables introduced in the proof of Lemma  \ref{lem_Gamma_gamma} (Fig.  \ref {l_gamma_gamma}).
 Denote\\ $\Phi (x,y)=\re(\Psi(\tau)-\Psi(\tau_c))$.  Equations of curves 
  $\Gamma_{*,1}$ and $\Gamma_{\tau_-}$ are, respectively,  $\Phi (x,y)=0$ and  $\Phi (x,y)=a$ with $  |a| \sim |\tau_--\tau_*^-|$. We have $\Phi (x,y)=-\alpha \rho^{3/2}\sin (3\theta/2)+O(\rho^2)$ (see proof of Lemma \ref {lem_Gamma_gamma}). Here $\alpha$ is a positive constant. If $|\theta|<k_1^{-1}$, then for the distance $\Delta$ between points of  $\Gamma_{\tau_-}$ and  $\Gamma_{*,1}$ we have
  $$
  \Delta\sim \frac{|a|}{|\grad \Phi |}\sim \frac{|a|}{\rho^{1/2} }\sim \frac{ |\tau_--\tau_*^-|}{|\tau-\tau_c |^{1/2}}\,.
  $$
  Here $\tau$ is taken as a parameter along $\Gamma_{*,1}$, distance is measured along normals to  $\Gamma_{*,1}$, and the gradient is calculated at the point $\tau$.
  
  The condition $|\theta|<k_1^{-1}$ is satisfied provided $\Delta/\rho<k_2^{-1}$. This is because the distance between points of $\Gamma_{*,1}$ and $Ox$-axis is $O(\rho^{3/2}) \ll \rho$ (see Lemma \ref {lem_Gamma_gamma}). The condition $\Delta/\rho<k_2^{-1}$  is satisfied provided
 $ |\tau_{-}-\tau_*^-|/|\tau-\tau_c |^{3/2}<k_3^{-1}$, i.e. $|\tau-\tau_c|> k_3^{2/3}  |\tau_{-} -\tau_*^-| ^{2/3}$.
 
 The same estimate for the distance remains valid if this distance is measured along normals to $\Gamma_{\tau_-}$ and  $\tau$ is taken as the parameter along $\Gamma_{\tau_-}$.

\hskip 12cm $\square$

\medskip
{\bf Proof of Lemma \ref {lem_Gamma_q}.}

For $|\tau-\tau_c| \sim |\tau_q-\tau_c|$ the statement of the Lemma is evident, because in this case the distance between points of  $\Gamma_{q}$ and  $\Gamma_{*,1}\cup\Gamma_{*,2} $  is of order  $ |\tau_q-\tau_c|  $. Thus, one should only consider the case 
 $|\tau-\tau_c| >>  |\tau_q-\tau_c|$. For this case, the proof is completely analogous to the proof of Lemma \ref{lem_Gamma_Gamma} with $ |a| \sim |\tau_q-\tau_c|^{3/2}$.

 \hskip 12cm $\square$

\section{Proofs of lemmas about Riccati equation.}
\label{s_proofs_riccati}

{\bf Proof of lemma \ref{expansion_minus}.}
\label{p_expansion_minus}

Substitute into (\ref{solution_1}) the asymptotic formulas for Bessel functions with large $|v|$. For real $v$, we obtain
\begin{equation*}
\begin{aligned}
\hat\ze&=-i\sqrt{-\hat s}\,\frac{\cos(v+\pi/12)-R\cos(v-7\pi/12)}
{R\cos(v-\pi/12)+\cos(v-5\pi/12)}(1+O(|v|^{-1}))\\
&=-i\sqrt{-\hat s}\,\frac{\cos(v+\pi/12)-R\sin(v-\pi/12)}
{R\cos(v-\pi/12)+\sin(v+\pi/12)}(1+O(|v|^{-1})).
\end{aligned}
\end{equation*}
We would like to find $R$ such that
\begin{equation}
\label{eq_for_R}
\frac{\cos(v+\pi/12)-R\sin(v-\pi/12)}
{R\cos(v-\pi/12)+\sin(v+\pi/12)}=-i.
\end{equation} 
This relation gives
$$
\cos(v+\pi/12)-R\sin(v-\pi/12)= -iR\cos(v-\pi/12)-i\sin(v+\pi/12).
$$
Thus,
\begin{equation*}
\begin{aligned}
R&=\frac{\cos(v+\pi/12)+i\sin(v+\pi/12)}{-i[\cos(v-\pi/12)+i\sin(v-\pi/12)]}
=e^{i\pi/2}\frac{e^{i\pi/12}}{e^{-i\pi/12}}=e^{i\pi/2}e^{i\pi/6}=e^{i2\pi/3}.
\end{aligned}
\end{equation*}

\hskip 12cm $\square$
\medskip

{\bf Proof of lemma \ref{ze_small_v}.}

We have 
\begin{equation*}
\begin{aligned}
\hat\ze&=-i\sqrt{-\hat s}\,\frac{J_{-2/3}(v)+O(|v|^{2/3})}{RJ_{-1/3}(v)+O(|v|^{1/3})}
=-i\sqrt{-\hat s}\,
\frac{(v/2)^{-2/3}/\Gamma(1/3)+O(|v|^{2/3})}  {R(v/2)^{-1/3}/\Gamma(2/3)+O(|v|^{1/3})}\\
&=-i\sqrt{-\hat s}\,\frac{1}{R}\left(\frac{(v/2)^{-1/3}\Gamma(2/3)}  {\Gamma(1/3)}+O(|v|^{1/3}) \right)
=-i\sqrt{-\hat s}\,\frac{1}{R}\left(\frac{(-s)^{3/2}/3)^{-1/3}\Gamma(2/3)}  {\Gamma(1/3)}+O(|v|^{1/3}) \right)\\
&=-i\,\frac{1}{R}\left(\frac{3^{1/3}\Gamma(2/3)}  {\Gamma(1/3)}+O(|v|^{2/3}) \right).
\end{aligned}
\end{equation*}
Using the reflection formula for the Gamma function
$$
\Gamma(\nu)\Gamma(1-\nu)=\frac{\pi}{\sin{\pi \nu}}
$$
with $\nu=1/3$, we obtain
\begin{equation*}
\begin{aligned}
\hat\ze&=\,\frac{1}{R}\frac{-2\pi i}  {\Gamma^2(1/3) 3^{1/6}}(1+O(|v|^{2/3})).
\end{aligned}
\end{equation*}

\hskip 12cm $\square$
\medskip

{\bf Proof of lemma \ref{expansion_plus}.}
\label{p_expansion_plus}

This proof is analogous to the proof of Lemma \ref{expansion_minus}, but one  should replace $i$ with $-i$ in the right hand side of equality (\ref{eq_for_R}). Then we obtain $R=e^{-i2\pi/3}$.

\hskip 12cm $\square$
\medskip

{\bf Proof of lemma \ref{lem_relation}.}
\label{p_relation}
The values $\ze$ and $\chi$ at $\tau =\tau_c$ for the same solution of  equation \ref {e_riccati} are related by
$e^{-i\pi/6}\ze= e^{-i5\pi/6}\chi$, which implies $\chi= e^{i2\pi/3}\ze$. On the other hand, Lemma \ref{ze_small_v} implies that, for $\tau =\tau_c$,
$$
\ze =\frac{k}{R_l}, \quad \chi=\frac{k}{R_r}
$$
with the same constant $k$. Hence, $R_r=e^{-i2\pi/3}R_l$.
 
\hskip 12cm $\square$
\medskip

{\bf Proof of lemma \ref{r_expansion}.}

Substituting into (\ref{solution_2}) the asymptotic formulas for Bessel functions with large $|v|$, we obtain
\begin{equation*}
\begin{aligned}
\hat\chi&=-i\sqrt{-\hat \sigma}\,\frac{\cos(v+\pi/12)-R\cos(v-7\pi/12)}
{R\cos(v-\pi/12)+\cos(v-5\pi/12)}(1+O(|v|^{-1}))\\
&=-i\sqrt{-\hat \sigma}\,\frac{e^{i(v+\pi/12)}+e^{-i(v+\pi/12)}-R(e^{i(v-7\pi/12)}+e^{-i(v-7\pi/12)})      }
{R(e^{i(v-\pi/12)}+e^{-i(v-\pi/12)})+e^{i(v-5\pi/12)}+e^{-i(v-5\pi/12)}}(1+O(|v|^{-1}))\\
&=-i\sqrt{-\hat \sigma}\,\frac{e^{i\pi/12}-Re^{-i7\pi/12}+e^{-2iv-i\pi/12} -Re^{-2iv+i7\pi/12}      }
{Re^{-i\pi/12}+e^{-i5\pi/12} +Re^{-2iv+\pi/12}  +e^{-2iv+i5\pi/12} }(1+O(|v|^{-1})\\
&=i\sqrt{-\hat \sigma}\,\frac{Re^{-i7\pi/12}-e^{i\pi/12}-e^{-2iv-i\pi/12} +Re^{-2iv+i7\pi/12}      }
{Re^{-i\pi/12}+e^{-i5\pi/12} +Re^{-2iv+\pi/12}  +e^{-2iv+i5\pi/12} }(1+O(|v|^{-1}))\\
&=i\sqrt{-\hat \sigma}\,\frac{e^{-i7\pi/12} [  R-e^{i8\pi/12}-e^{-2iv+i6\pi/12} +Re^{-2iv+i14\pi/12}]      }
{e^{-i\pi/12}[R+e^{-i4\pi/12} +Re^{-2iv+2\pi/12}  +e^{-2iv+6i\pi/12} ]}(1+O(|v|^{-1}))\\
&=i\sqrt{-\hat \sigma}\, e^{-i\pi/2} \frac{   R-e^{i2\pi/3}+e^{-2iv}[-e^{i\pi/2} -Re^{i\pi/6}]      }
{R-e^{i2\pi/3} +e^{-2iv}[Re^{i\pi/6}  +e^{i\pi/2} ]}(1+O(|v|^{-1}))\\
&=\sqrt{-\hat \sigma}\left (1+ e^{-2iv}  \frac{2 e^{-i\pi/2} -2Re^{i\pi/6}}      
{R-e^{i2\pi/3}}
+O(e^{-4|\im v|}\right)(1+O(|v|^{-1}))\\
&=\sqrt{-\hat \sigma}\left (1- 2e^{-2iv}  e^{i\pi/6} \frac{R-e^{-i2\pi/3}}      
{R-e^{i2\pi/3}}
+O(e^{-4|\im v|}+|v|^{-1})\right).
\end{aligned}
\end {equation*}

\hskip 12cm $\square$

\medskip
{\bf Proof of lemma \ref{lem_dr}.}

We have
\begin{equation}
\frac{\partial \hat\ze}{\partial R}=i\sqrt{-\hat s}\ \frac{J_{2/3}(v)J_{1/3}(v)+J_{-2/3}(v)J_{-1/3}(v)}{(RJ_{-1/3}(v)+J_{1/3}(v))^2},
\quad v=\frac{2}{3}(-\hat s)^{3/2} .
\end{equation} 
According to \cite{be}, Sec. 7.11, formula (33),
$$
J_{\nu}(v)J_{-\nu+1}(v)+J_{-\nu}(v)J_{\nu-1}(v)=\frac{2\sin(\nu\pi)}{\pi v}.
$$
For $\nu=1/3$, we obtain
$$
J_{1/3}(v)J_{2/3}(v)+J_{-1/3}(v)J_{-2/3}(v)=\frac{\sqrt{3}}{\pi v}=\frac{3\sqrt{3}}{2\pi (-\hat s)^{3/2}}
$$
Thus,
$$
\frac{\partial \hat\ze}{\partial R}=-i \frac{3\sqrt{3}}{2\pi \hat s}\frac{1}{(RJ_{-1/3}(v)+J_{1/3}(v))^2}
$$

\hskip 12cm $\square$
\medskip

\section{Proofs of lemmas about transformation of variables}
 \label{proofs_transformations}
 
 {\bf Proof of Lemma \ref{lem_for_domain_V}.}
 
 For definiteness, we give the proof for the domain $V^+_{\delta}$. Recall notation:
 
  \begin{equation*}
 \begin{aligned}
 &K^{+}_{*,\delta}=\{\tau \ : \ \tau\in K_*^{+},   \  \im\tau<\im  \tau_c- \delta\}, \ \delta=C_{\delta}\eps^{2/3},\\
   &B_{*,{\delta}}^{+}=\{\kappa \ : \   \kappa=\dK(\tau), \tau\in K^{+}_{*,\delta}\}, \\
   &d_+=d_+(\kappa)=|b\cdot(\kappa-\kappa_c)|,
  \end{aligned}
 \end{equation*}
 $ V_{{\delta}}^{+}$ is the $c_{l,4}^{-1}\eps^{2/3}$-neighbourhood of $ B_{*,{\delta}}^{+}$.  The equilibrium $X(\kappa)$  of the fast system is well defined and is an analytic  function of $\kappa$ in $ V_{{\delta}}^{+}$. 
\medskip
 
 We should prove that for
    sufficiently large $c_{l,4}$  we have   $c_{l,6}^{-1}d_{+}<|\kappa-\kappa_c| < c_{l,6}d_{+}$   in $V_{\delta}^{+}$.
 \medskip
 
 Take any $\kappa_v\in V_{\delta}^{+}$. Let $\kappa_b$ be the point of $B_{*,{\delta}}^{+}$ closest to $\kappa_v$. Denote\\ $d_{+,v}=d_+(\kappa_v), d_{+,b}=d_+(\kappa_b)$.
 Then
  $$|\kappa_v-\kappa_b|\le c_{l,4}^{-1}\eps^{2/3}, \quad 
 |d_{+,v}-d_{+,b}|\le k_1c_{l,4}^{-1}\eps^{2/3}.$$
 
 Let us first derive the required estimates for $\kappa_b$. The domain $B_{*,{\delta}}^{+}$ is parametrised by values of the slow time $\tau$ from $K^{+}_{*,\delta}$. For
 $\kappa_b$, 
 we have $\kappa_b=\dK(\tau), \tau\in K^{+}_{*,\delta}$. Thus, for such $\kappa_b$, we have
 \begin{equation*}
 \begin{aligned}
 \kappa_b-\kappa_c&=
   \dK(\tau)-\dK(\tau_c)= \frac{d \dK(\tau_i)}{d\tau}(\tau-\tau_c)=g(X(\dK(\tau_i)), \tau_i,0)(\tau-\tau_c) \\
  & =g(x_c, \tau_c,0)(\tau-\tau_c) + O(|\tau-\tau_c|^{3/2}),\\
   b\cdot (\kappa_b-\kappa_c)&=
  b\cdot( \dK(\tau)-\dK(\tau_c))= b\cdot \frac{d \dK(\tau_i)}{d\tau}(\tau-\tau_c)=b\cdot g(X(\dK(\tau_i)), \tau_i,0)(\tau-\tau_c)\\
 & =b\cdot g(x_c,\kappa_c,0)(\tau-\tau_c) +O(|\tau-\tau_c|^{3/2}),
   \end{aligned}
 \end{equation*}
 where $\tau_i$ is a point on the segment joining $\tau$ and $\tau_c$.  We have
 \begin{equation*}
 d_{+,b} =|b\cdot (\kappa_b-\kappa_c) |=|b\cdot g(x_c,\kappa_c,0)||\tau-\tau_c| +O(|\tau-\tau_c|^{3/2}).
\end{equation*}
According to assumption G in Section \ref{form_conditions},  $ |b\cdot g(x_c,\kappa_c)|>k_2^{-1}$.  Then
\begin{equation*}
k_3^{-1}|\tau-\tau_c|< |\kappa_b-\kappa_c|<k_3|\tau-\tau_c|, \  k_4^{-1}d_{+,b} <|\tau-\tau_c| <  k_4d_{+,b}, \ k_5^{-1}d_{+,b}  <|\kappa_b-\kappa_c| <  k_5d_{+,b}.
 \end{equation*}
This, in particular, implies that  
$$
    |\kappa_b-\kappa_c|> k_3^{-1}|\tau-\tau_c| >k_3^{-1}\delta=k_3^{-1}C_{\delta}\eps^{2/3}, \quad    d_{+,b}> k_4^{-1}|\tau-\tau_c| >k_4^{-1}\delta=k_4^{-1}C_{\delta}\eps^{2/3}.
  $$
  Then, for sufficiently large  constant $ c_{l,4}$, we have
  $$
  |\kappa_v-\kappa_b|< 0.5|\kappa_b-\kappa_c|,\quad    |d_{+,v}-d_{+,b}|< 0.5d_{+,b}
  $$
  which implies
  $$
0.25 k_5^{-1}d_{+,v}  <0.5|\kappa_b-\kappa_c| <|\kappa_v-\kappa_c| < 1.5|\kappa_b-\kappa_c|< 3 k_5d_{+,v} .
  $$
  Thus, 
    $$c_{l,6}^{-1}d_{+}(\kappa)<|\kappa-\kappa_c| < c_{l,6}d_{+}(\kappa)$$
   in $V_{\delta}^{+}$ for a sufficiently large constant $ c_{l,6}$.

\hskip 12cm $\square$

  \medskip
 {\bf Proof of Lemma \ref{lem_transform_0}.}
 
 For definiteness, we give the proof for the domain $W^+_{\delta}$.  Substitute $x=\xi+X(\kappa)$ into (\ref{perturbed}). We obtain
 \begin{eqnarray}
\label{sfe_1}
     \frac{d\xi}{dt}&=&f(\xi+X(\kappa),\kappa,0) -\eps\frac{\p X(\kappa)}{\p \kappa}g(\xi+X(\kappa),\kappa,0) 
    +O(\eps)
      ,\\
\frac{d\kappa}{dt}&=&\varepsilon g(\xi+X(\kappa),\kappa,0)
+O(\eps^2).
 \nonumber
\end{eqnarray}

Consider first equation  (\ref{perturbed}) for $\eps=0$ in variables $z_1, z_2$ used in (\ref{de_z1}). We obtain
\begin{equation}
\begin{aligned}
&\dot z_1=b \cdot (\kappa-\kappa_c) +az_1^2 +O(
z_1|^3+
|\kappa-\kappa_c|(|z_1|+|z_2|)+|z_1||z_2|+|z_2|^2+|\kappa-\kappa_c|^2),\\
&\dot z_2={\cal A}_{c,2}\, z_2+b_1 \cdot (\kappa-\kappa_c) + O(
|z_1|^2+
|\kappa-\kappa_c|(|z_1|+|z_2|)+|z_1||z_2|+|z_2|^2+|\kappa-\kappa_c|^2),
\end{aligned}
\end{equation}
where ${\cal A}_{c,2}$ is a non-degenerate matrix, $ b_1={\rm const}$. To find the equilibrium, we first solve equation $\dot z_2=0$ for $z_2$. This determines $z_2$ as an analytic function of $z_1, \kappa$, whose expansion with respect to $z_1,  \kappa-\kappa_c$ starts with terms proportional to $z_1^2$ and  to $\kappa-\kappa_c$.
We then substitute this function into the equation
 $\dot z_1=0$ and solve it for $z_1$. Finally, we substitute this value of 
 $z_1$  into the previously obtained solution of the equation $\dot z_2=0$. Denote $Z_1(\kappa)$ and $Z_2(\kappa)$ equilibrium values of $z_1$ and $z_2$.
 
To obtain the estimates, first consider $(x, \kappa)\in L_*^+$ (the notation $L_*^+$ is introduced in Section \ref{add_notation}). The set $L_*^+$ is parametrised by values of $\tau\in K_*^+$.  Lemmas \ref {FG} and \ref{lem_for_domain_V} imply that
$|\partial Z_1(\kappa)/\partial \kappa|=O((d_+(\kappa))^{-1/2})$.
For $\kappa \in V_{\delta}$, let $\kappa_b$ be the point of $L_*^+$ closest to $\kappa$. For sufficiently large $C_{\delta}, c_{l,4}$, we have
$|\partial Z_1(\kappa)/\partial \kappa|<2|\partial Z_1(\kappa_b)/\partial \kappa|=O((d_+(\kappa_b))^{-1/2})=O((d_+(\kappa))^{-1/2})$.
The expression of $Z_2$ in terms of  $Z_1$ and $\kappa$ implies that $|\partial Z_2(\kappa)/\partial \kappa|=O(1)$.

\hskip 12cm $\square$


\bigskip
{\bf Proof of Lemma \ref{lem_transform_00}.}

For definiteness, we give the proof for the domain $W^+_{\delta}$. Condition 4) is Section \ref{form_conditions} implies that $\lambda_1(\kappa) \ne \lambda_2(\kappa)$, $\lambda_i(\kappa) \ne \lambda_j(\kappa),
 \quad i=1,2,\quad j=3,\ldots, n$ for $\kappa \in  V^+_{\delta}$. This allows to choose a real-analytic basis which is a union of  bases in eigenspaces  of the matrix ${\cal A}$, corresponding to eigenvalues  $\lambda_1, \lambda_2$  and $\lambda_3, \lambda_4,\ldots, \lambda_n$. This basis can first be constructed locally; its definition can then be extended globally by propagating it along solutions of the slow system.
Details of this construction are given in \cite{nei_diss}, pp. 282 - 285.  The matrix $C$ provides the reduction of the matrix ${\cal A}$ to this basis. The matrix $C$ is non-degenerate and depends analytically on the position of the equilibrium  $X(\kappa)$, for which we have
 $\partial X(\kappa)/\partial \kappa=O^*(d_+^{-1/2})$.
Substitute $\xi=C(\kappa)\tilde\xi$ into (\ref{perturbed}). Differentiation of $C(\kappa)$ over time leads  to  terms  $\eps O^*(d_+^{-1/2})\xi$ and   $\eps O(d_+^{-1/2}|\xi|^2)$ in equations. 
  The second of these terms is absorbed into the term   $O(|\xi|^2)$   in (\ref{after_0_transform}).   
 
 \hskip 12cm $\square$
 
 \medskip
 {\bf Proof of Lemma \ref{lem_transform_1}.}

For definiteness, we give the proof for the domain $W^+_{\delta}$. Let $\xi=(z, w,\eta)$.

(a)  We make a shift of the coordinate origin $\xi\mapsto \tilde \xi $ so that the new origin coincides with  the equilibrium point of the  equation  for new fast variables with  frozen  $\kappa$. We compute this   shift  in three steps. First, we solve the equations $\dot w=0, \dot \eta=0$ for $w, \eta$ and substitute the obtained solution into the equation $\dot z=0$. Second, we solve the resulting equation $\dot z=0$ for   $z$. Third, we substitute this value of $z$ into the already obtained solution of the equations $\dot w=0, \dot \eta=0$.  
This procedure yields an estimate for the shift of $z$ of order 
$O^*(\eps d_+^{-1})$,  and for the other fast variables of order $O^*(\eps)+O^*(\eps^2 d_+^{-3/2})+O^*(\eps^2d_+^{-2})=O^*(\eps d_+^{-1/2})$. 
The  term $O^*(\eps^2 d_+^{-3/2})$ appears  due to the term $O^*(\eps d_+^{-1/2}\xi) $ in (\ref{after_0_transform}), while the term    $O^*(\eps^2d_+^{-2})$ appears due to  $\sim z^2$ terms in the equations $\dot w=0, \dot \eta=0$. Note that  an additional linear term $O^*(\eps d_+^{-1})\tilde \xi $  appears in the equations for new fast variables. This is due to the shift in  terms $\sim z^2$.  There are also new nonlinear terms 
  $O(\eps d_+^{-1}|\tilde \xi|^2) $. These terms are absorbed in the term  $O(|\tilde \xi|^2)$. 
 
 The considered shift in $\xi$ leads to change of form of the equation for $\dot \kappa$. An additional $\tilde \xi$-independent term $O^*(\eps^2 d_+^{-1})$ appears due to the shift in the  term $\sim z$. Also,  an additional linear in $\tilde \xi$ term $O^*( \eps^2 d_+^{-1}|\tilde \xi|)$ and nonlinear in $\tilde \xi$ term $O( \eps^2 d_+^{-1}|\tilde \xi|^2)$ appear there.  These terms are absorbed in the terms $O^*(\eps|\tilde \xi|)$ and  $O(\eps|\tilde \xi|^2)$.

 Additional new terms appear  in the right hand sides of the differential equations for the new variables due to the time dependence of $\kappa$.  The terms that do not vanish at $\tilde \xi=0 $ are $O^*(\eps^2d_+^{-2})$  for $\dot{\tilde z}$ and  $O^*(\eps^2 d_+^{-3/2})$ for other fast variables. Linear terms in $\dot \kappa$ produce a linear term  $O^*(\eps^2d_+^{-2})\tilde \xi =O^*(\eps d_+^{-1})\tilde \xi $ in   $\dot{\tilde \xi}$. Nonlinear terms in $\dot \kappa$ produce a nonlinear  term  $O(\eps^2d_+^{-2}|\tilde \xi|^2)$ in
$\dot{\tilde \xi}$. This term is absorbed in the term   $O(\eps|\tilde \xi|^2)$.

  Repeat the same step in the new variables.
     Now  the shift is $O^*(\eps^2 d_+^{-5/2})$ for the  new variable $z$ and 
     $O^*(\eps^2 d_+^{-2})$ for other fast variables. In the equation for the new variables the term that does not vanish at $\tilde \xi=0 $ is $O^*(\eps^3 d_+^{-7/2})$  for $\dot{\tilde z}$ and $O^*(\eps^3 d_+^{-3})$ for other fast variables.

 (b) The matrix of the linearised near $\xi=0$ system differs from a block-diagonal form by  $O^*(\eps d_+^{-1})$. We perform a linear transformation of variables that reduces this matrix to  block-diagonal form for fixed $\kappa$. The matrix of this transformation  differs from the unit matrix by  $O^*(\eps d_+^{-1})$. As a result, we obtain a system in which the matrix of the linearised near $\xi=0$ system differs from a block-diagonal form by  $O^*(\eps^2 d_+^{-2})$. One further similar step reduces the deviation from block-diagonality   to  $O^*(\eps^3 d_+^{-3})$. 
 Thus, the equation for $\xi$ has the form of  (\ref{after_transform}) with the block diagonal matrix $A$.
 
 \hskip 12cm $\square$
 
\medskip
{\bf Proof of Lemma \ref{lem_transform_2}.}

We use the standard procedure from normal form theory for the elimination of non-resonant terms (see, e.g., \cite{arnold_geometric}, Ch. 5). We start with the proof for the domain $W^+_{\delta}$.

First, we perform a standard real-analytic transformation  of variables
 $\eta, z,w$ which, for frozen $\kappa$, eliminates from the right-hand side of the equation for
  $\dot \eta$ 
   the  monomials   $z^2, w^2,  z^3, w^3$;
from the equation for  $\dot z$ the  monomial  $w^2$;
and from the equation for $\dot w$ the  monomial  $z^2$.
 For this transformation we have the estimate
  $\hat \xi=\xi+O(|\xi|^2)$ (``hat'' for new variables). New quadratic  and higher order terms appear in the transformed equations due to dependence of the transformation on $\kappa$. These terms are proportional to $\eps$. The estimate for these terms is $ \eps O(|\xi|^2d_+^{-1/2})$.

 Then, we perform a standard real-analytic transformation of  variables  $z, w$  
  which,  for frozen $\kappa$, eliminates  from 
   the equation for $\dot z$ all  quadratic monomials containing $z$, and from equation for $\dot w$ all quadratic monomials containing $w$.
  This transformation satisfies the estimates
\begin{equation*}
   \begin{aligned}
   \hat z=z+ O(|z|^2 d_+^{-1/2}+ |\xi|^2) ,\
    \hat w=w+ O(|zw|d_+^{-1/2}+
  |\xi|^2)  \ 
  \end{aligned}
  \end{equation*}
    (``hats'' for new variables).
   New quadratic and higher order terms appear in the transformed equations due to the dependence of the transformation on $\kappa$. These terms are proportional to $\eps$. The estimates for these terms are $ \eps O(|z|^2 d_+^{-3/2})+
     \eps O(|\xi|_*^2d_+^{-1/2}) $ in the equation for $\dot z$, and  $ \eps O(|zw| d_+^{-3/2})+
     \eps O(|\xi|^2 d_+^{-1/2}) $ in the equation for $\dot w$
      (we omit “hats” over
the new variables).

The transformation under consideration leads to the appearance of new third- and higher-order terms (not proportional to  $\eps$) due to the substitution of the formulas for the new variables into the right-hand sides of the equations.
   In equation for  $\dot z$, the terms  not proportional to $z^3$ are estimated by  $O(|\xi|_*^3 d_+^{-1/2}) $.  The terms proportional to $z^3$ are estimated by $z^3O^* (d_+^{-1/2})+z^4O^* (d_+^{-1})$ and  $z^4O^* (d_+^{-1/2})+z^5O(d_+^{-1}) $ (these terms originate from the terms $\sim z^2$ and $\sim z^3 $ in the original equation, respectively).  The term $z^3O^* (d_+^{-1/2}) $ can be eliminated by a transformation of variables. The new large term created by this  is $z^5O (d_+^{-3/2})$.  The terms $z^4O^* (d_+^{-1/2})$ and $z^4O^* (d_+^{-1}) $ we also eliminate by a transformation of variables.  The new large term created by this is $z^7O (d_+^{-5/2})$. This term  is majorated by   $z^5O (d_+^{-3/2})$. Other  terms generated by these transformations are smaller than those already present in the equation. To preserve real analyticity, we perform a similar transformation in the equation for
   $\dot w$ replacing $z$ with $w$.
 In the equation for $\dot \eta$  the higher order terms are $O(|\xi|_*^3d_+^{-1/2})+ O(|z|^4)$.
  In the equation for $\dot w$  the higher order terms are $O(|\xi|_*^3d_+^{-1/2})+ O(|z|^3)$.
  
  Combining the above estimates,  we obtain the  formulas for $\dot z, \dot \eta$ in (\ref{d_equation}).
  
  \medskip
  In the domain $W^-_{\delta}$, the nearly-resonant quadratic monomial in the expression for $\dot z$ is $zw$.  The transformation of $z$ is estimated by 
  \begin{equation*}
   \begin{aligned}
   \hat z=z+ O(|zw| d_-^{-1/2}+ |\xi|^2).
  \end{aligned}
  \end{equation*}
 The  dependence of this transformation on $\kappa$ gives terms $\eps O^*(|zw| d_-^{-3/2}+ |\xi|^2d_-^{-1/2})$ (``hats'' are omitted). This transformation  creates some new cubic terms. The cubic and higher order  terms  are  bounded above by $O(|\xi|_{**}^3d_-^{-1/2})$+ $O(|w|^3)$.
  Hence, we obtain (\ref{d_equation1}).
  
  \hskip 12cm $\square$
  
 \medskip
{\bf Proof of Lemma \ref{lem_transform_kappa}.}

For definiteness, we give the proof for the domain $W^+_{\delta}$.

We start by working on the terms linear in $\xi$  in the  equation for $\dot \kappa$. To this end, we make the transformation of variables 
 \begin{equation}
 \label{kappa&M}
  \hat \kappa=\kappa+\eps M(\kappa, \eps)\xi,
 \end{equation}
 where $M$ is a matrix to be determined. The new variable  $\hat \kappa$ satisfies the equation 
 \begin{equation*}
 \begin{aligned}
 \dot {\hat \kappa}&=\eps G(\kappa) +O^*(\eps^2 d_+^{-1}) +\eps E(\kappa)\xi  +\eps O(|\xi|^2) \\
 &+\eps M(\kappa, \eps) A(\kappa, \eps)\xi 
  +\eps M \alpha+\eps \dot M(\kappa, \eps)\xi  ,
\end{aligned}
\end{equation*}
where $A(\kappa, \eps)$ is the matrix of the linearised  near $\xi=0$  system, $E(\kappa)=O(1$), $\eps O(|\xi|^2)$ is the sum of nonlinear terms in $\dot \kappa$, and $\alpha$ denotes all terms in $\dot \xi$ except $A\xi$. 
 
Solving  the equation $M(\kappa, \eps)A(\kappa, \eps) +E(\kappa)=0$ for $M$,  we obtain $M=-A^{-1}E$. In the variables $z,w,\eta$, the matrix $A$ has a block diagonal structure. Thus,  $A^{-1}$ has the  block diagonal structure, where the first $1\times 1$ block is $O(d_+^{-1/2})$, while entries in all other blocks are $O(1)$. Note that derivatives in $\kappa$ of these blocks are $O(d_+^{-3/2})$ and $O(d_+^{-1/2})$, respectively. Using expressions for $\dot z, \dot \eta$ in (\ref{d_equation}), and information about  $\dot w$ in Lemma \ref{lem_transform_2}, we obtain
 \begin{equation*}
 \begin{aligned}
& \dot {\hat \kappa}=\eps G(\kappa) +O^*(\eps^2 d_+^{-1})+\eps O^*(|\xi|^2) + \eps O^*(|\xi|^3) +\eps O(|\xi|^4)\\
&+\eps d_+^{-1/2}O\left( \eps|z|^2d_+^{-3/2}
     +\eps (|\xi|_*^2 d_+^{-1/2})  +O(|\eta|(|\eta|+|w|))
       + (|\xi|_*^3d_+^{-1/2})+(|z|^5d_+^{-3/2}) +\eps^3(d_{+}^{-3}|\xi|)   \right)\\
      &+O^*(\eps d_+^{-1/2}\eps^3(d_{+}^{-7/2}))\\
&+\eps O\left(   |\eta|^2+|z\eta | 
   +     \eps ((|z|^2  +|w|^2+|w\eta|)d_+^{-1/2}) +(\eps |wz|d_+^{-3/2})
     +(|\xi|^3_*d_+^{-1/2})+ (|z|^3)+ \eps^3(d_{+}^{-3}|\xi|)\right) \\
     &+O^*(\eps \eps^3(d_{+}^{-3}))\\
   &+\eps O\left(  |\eta|^2+ (|\eta |(|z|+|w|) +O(|zw|) 
   +     \eps ((|z|^2  +|w|^2)d_+^{-1/2})  
     +(|\xi|_*^3d_+^{-1/2})+ (|z|^4)+\eps^3(d_{+}^{-3}|\xi|)\right) \\
       &+O^*(\eps \eps^3(d_{+}^{-3}))\\
     &+\eps^2O(|z|d_+^{-3/2}+(|w|+|\eta|)d_+^{-1/2} ).   
 \end{aligned}
\end{equation*}
 The substitution of $\hat \kappa$ from (\ref{kappa&M}) into $G(\kappa)$ creates additional  linear in $\xi$ terms \\
$\eps^2 O(d_+^{-1}|z| + d_+^{-1/2}|\xi|_*)$.
Comparing the magnitudes of the different terms, we obtain the following form of this equation (``hat'' over $\kappa$ is omitted):
\begin{equation*}
 \begin{aligned}
 & \dot {\kappa}=\eps G(\kappa) +O^*(\eps^2 d_+^{-1})+\eps O\left (|z|^2+d_+^{-1/2}|\xi|_*^2\right) +\eps d _+^{-1}O(|\xi|_*^3)
 +\eps O^*(|z|^3)\\
 &+\eps^2 O(|z|^2d_+^{-2})
 +\eps O\left(|z|^4+|z|^5d_+^{-2}\right)
 + \eps^2O\left(|z|d_+^{-3/2}+d_+^{-1/2}(|w|+|\eta|) \right)
 +\eps^4O(d_+^{-7/2}|\xi|).
   \end{aligned}
\end{equation*}

 Repeat similar step for the new linear in $\xi$ terms in the equation. We obtain the equation in which linear in $\xi$ term is
  $\eps^3 O(d_+^{-3}|z|+ d_+^{-3/2}(|w|+|\eta|))$. The estimates for the other terms do not change. This is because the new additional terms arising from $\dot \xi$ are the same as those from  the previous step, but multiplied by $O(\eps d_+^{-3/2})=O(1)$.  
 
 \medskip 
  In a similar way, we can eliminate the quadratic  in $\xi$  terms  $ \eps O^*\left (|z|^2+d_+^{-1/2}|\xi|_*^2\right)$ and  the cubic in $z$ term  $\eps O^*(|z|^3)$. The estimate for the new quadratic and cubic   terms   is $\eps^2 O\left (d_+^{-3/2}|z|^2+d_+^{-3/2}|\xi|_*^2 \right)$. The estimates for the other terms do not change. This is because the new additional terms arising from $\dot \xi$ are the same  as those on the previous step, but multiplied by $O( d_+^{-1/2}|\xi|)=O(1)$. 
  
 Comparing the magnitudes of the different terms, we obtain (\ref{eq_kappa_transformed}).
  
\hskip 12cm $\square$

%% file: delay_proofs_of_lemmas_2.tex
\section{Proofs of lemmas about continuation of solutions.}
\label{proofs_continuation}
{\bf Proof of lemma \ref{K_and_K_eps}.}

The differential equations for $\dK_{\eps}$ and $\dK$ in the domain $V_{\delta}^{\pm}$ are
\begin{equation}
\label{F_and_G}
\dot \dK_{\eps}=\eps F(\dK_{\eps}), \quad \dot \dK=\eps G(\dK), \quad F(\kappa)=G(\kappa)+\eps O (d_{\pm}^{-1}).
\end{equation}
For the initial condition of $\dK_{\eps}$, we have
$$
\dK_{\eps}(\tau_*^-)=  \dK(\tau_*^-)+ O(\eps).
$$
  In $V_{\delta}^{\pm}$, we have $|\partial G/\partial \kappa|= O(d_{\pm}^{-1/2})$, $\hat {d}_{\pm}(\tau)>k_1^{-1}\eps^{2/3}$. The estimates for solutions of (\ref {F_and_G}) give
$$
|\dK(\tau)-\dK_{\eps}(\tau)|=O(\eps(1+|\ln \hat d_{\pm}|)), \quad   
 0.5\hat {d}_{\pm}(\tau)  <\hat {d}_{\pm,\eps}(\tau)< 2\hat {d}_{\pm}(\tau)
$$
in $V_{\delta}$. 
\medskip

\hskip 12cm $\square$

\medskip

{\bf Proof of lemmas \ref{second_line}, \ref{third_line}, \ref{fourth_line}.}

We provide the proof of Lemma \ref{second_line}. Proofs  of other lemmas are completely analogous.

We have
\begin{equation*}
\re    \int_{\tau_{\gamma}}^{\tau_{ *,\eps,-}}\Lambda_1(\dK_{\eps}(\vartheta))\,d\vartheta=0, \quad \re    \int_{\tau_{\gamma}}^{\tau_{-}}\lambda_1(\dK(\vartheta))\,d\vartheta=0
\end{equation*}
and
\begin{equation*}
\Lambda_1(\dK_{\eps}(\tau))=\lambda_1(\dK_{\eps}(\tau))+\eps O(\hat {d}_{+,\eps}^{-1})=\lambda_1(\dK(\tau))
+\eps|\ln \hat d_{+}| O(\hat d_{+}^{-1/2})+ \eps O(\hat {d}_{+}^{-1}).
\end{equation*}
This implies
\begin{equation*}
\re    \int_{\tau_{ *,\eps,-}}^{\tau_{-}}\lambda_1(\dK(\vartheta))\,d\vartheta=\eps \re \int_{\tau_{\gamma}}^{\tau_{ *,\eps,-}}
\left(|\ln \hat d_{+}|  O(\hat d_{+}^{-1/2})+ O(\hat {d}_{+}^{-1})\right)d\vartheta=O(\eps\ln\eps).
\end{equation*}
This implies $\tau_{ *,\eps,-}=\tau_{-}+O(\eps\ln\eps)$.

\medskip

\hskip 12cm $\square$

\medskip
{\bf Proof of lemmas \ref{lem_cont_D_gamma}, \ref{lem_cont_D_q_l}.}

We will give a detailed proof of Lemma \ref{lem_cont_D_q_l}. At the end of this proof we will make a comment concerning modifications required for the proof of Lemma \ref{lem_cont_D_gamma}.

\medskip

Denote by $O_1, O_2,  O_3, O_4$ last $O$-terms in equations  (\ref{d_equation}),  (\ref{d_equation1}). We have
\begin{equation}
|O_1({d}_{+}^{-7/2})|< c_{r,1} {d}_{+}^{-7/2}, \ |O_2({d}_{+}^{-3})|<c_{r,2} {d}_{+}^{-3},\ 
 |O_3({d}_{+}^{-1})|< c_{r,3}{d}_{+}^{-1}, \ |O_4({d}_{-}^{-3})|<c_{r,4} {d}_{-}^{-3}.
\end{equation}

Rewrite system (\ref{d_equation}) in the domain  $W^{+}_{\delta}$ in the form
 \begin{equation}
\begin{aligned}
\label{with_beta_+}
     &\dot z=  \Lambda_1(\dK_{\eps})z+ ( \Lambda_1(\kappa)-\Lambda_1(\dK_{\eps}))z + \beta_1+\eps^3O_1(d_{+}^{-7/2}),\\
  &\dot \eta=B(\dK_{\eps})\eta + (B(\kappa)-B(\dK_{\eps}))\eta +\beta_2+\eps^3O_2(d_{+}^{-3}),\\
   & \dot {\kappa}=\eps F(\kappa) +\eps^2 O(|z|^2d_+^{-2})+\eps^2 O(|\xi|_*^2d_+^{-3/2}) \\
 &
 +\eps O\left(|z|^4+|z|^5d_+^{-2}+ |\xi|_*^3d_+^{-1}\right)
 +\eps^3 O(|z|d_+^{-3} + |\xi|_*d_+^{-3/2})+\eps^4O(d_+^{-7/2}|\xi|),\\
      &\hskip 1 cm F= g(X(\kappa), \kappa,0)+\eps O_3( d_{+}^{-1}).
\end{aligned}
\end{equation}
In the domain $W^{-}_{\delta}$, we have similar equations with $z$ and $w$ interchanged and with $d_+$ replaced by $d_-$, and $\Lambda_1$ replaced by
$\Lambda_2$.

Rewrite equation for $z$  in the domain  $W^{-}_{\delta}$ in the form
\begin{equation}
\begin{aligned}
\label{with_beta_-}
\dot z=  \Lambda_1(\dK_{\eps})z+  ( \Lambda_1(\kappa)-\Lambda_1(\dK_{\eps}))z + \beta_4
     +\eps^3O_4(d_{-}^{-3}).
     \end{aligned}
     \end{equation}

We have
 \begin{equation}
  |\kappa(t_4)-\dK_{\eps}(\eps t_4)|=O(\eps^6 |\ln \eps|) ,
 \  |z(t_4)|= O(\eps^3),\ |\eta(t_4)|=O(\eps^3)  .
  \end{equation}
  
  Denote by $S(T)$ the part of the domain $D_{q,l}$ (the sector) in the plane of the complex slow time where  
  $t_4 \le \re t \le T$. We consider  the solution $z(t), \eta(t), \kappa(t)$ of  system (\ref{d_equation}), (\ref{d_equation1}).
  By Cauchy's theorem \cite{golub}, there exists $T>t_4$ such that  the solution $z(t), \eta(t), \kappa(t)$ can be  analytically continued  into the sector $S(T)$.
  
  \medskip
  Take $T$ such that for $\eps t\in S(T)$  we have:
  \begin{equation}
  \begin{aligned}
  \label{induct_D_l}
 &|\beta_1|< 0.5 \eps^{3}c_{r,1} {d}_{+}^{-7/2}, \ |\beta_2| < 0.5  \eps^{3}c_{r,2} {d}_{+}^{-3}, \  |\beta_4| < 0.5 \eps^{3} c_{r,4} {d}_{-}^{-3},\\
 &  0.5 \hat d_{\pm}(\eps t)   \le d_{\pm}(\kappa(t)) \le 2 \hat d_{\pm}(\eps t),\\
 &  |\kappa(t)- \dK_{\eps}(t)|< \mu_1 \eps^{4}\hat d_{\pm}^{-9/2},\\
 & |\xi(t)|< c_{t,5}\eps^{1/3}
  \end{aligned}
  \end{equation} 
  (``+'' and ``-'' correspond respectively for $\im \tau\ge 0$ and $\im \tau\le 0$). Here,  $\mu_1$ is a positive constant whose value does not depend on choice of the constant $C_q$, provided that  $C_q$ is sufficiently large. The value $\mu_1$ is determined after the statement  of Lemma \ref {est_of_shorten}. The constant $c_{t,5}$ is introduced in the statement of Lemma \ref{lem_transform_2}.

  Note that  inequalities (\ref {induct_D_l}) are certainly satisfied for sufficiently small  $T-t_4$. 
  
  \medskip
  
  Assumptions  (\ref{induct_D_l}) imply that
  \begin{equation*}
  \begin{aligned}
&|\Lambda_1(\kappa(t))-\Lambda_1(\dK_{\eps}(\eps t))||z|=\mu_1 \eps^4O(\hat d_{\pm}^{-5})\eps^{1/3}<0.1c_{r,1}\eps^3\hat d_{\pm}^{-7/2},\\
&|B(\kappa(t))-B(\dK_{\eps}(\eps t))||\eta|=\mu_1 \eps^4O(\hat d_{\pm}^{-5})\eps^{1/3}<0.1c_{r,2}\eps^3\hat d_{\pm}^{-3}.
\end{aligned}
\end{equation*}
(Here we use that $|\xi| <c_{t,5}\eps^{1/3}$ and that $C_q$ is sufficiently large.) 

Thus, equations (\ref{with_beta_+}) have the form
 \begin{equation}
\begin{aligned}
\label{a_shorten_+}
     &\dot z=  \Lambda_1(\dK_{\eps})z+\eps^3\tilde O_1(d_{+}^{-7/2}),\\
  &\dot \eta=B(\dK_{\eps})\eta + \eps^3 \tilde O_2( d_+^{-3})
  \end{aligned}
\end{equation}
with $|\tilde O_1(d_{+}^{-7/2})|< 2c_{r,1}d_{+}^{-7/2},\ |\tilde O_2(d_{+}^{-3})|< 2c_{r,2}d_{+}^{-3}$.

Equation (\ref{with_beta_-}) has  the form
\begin{equation}
\begin{aligned}
\label{a_shorthen_-}
\dot z=  \Lambda_1(\dK_{\eps})z+\eps^3\tilde O_4(d_{-}^{-3})
     \end{aligned}
\end{equation}
with $|\tilde O_4(d_{-}^{3})|< 2c_{r,4}d_{-}^{-3}$.

\begin{lem}
\label{est_of_shorten}
In $S(T)$,  for  $\im \tau\ge -c_{l,1}^{-1}$ we have
\begin{equation}
\begin{aligned}
\label{ind_1+}
& | z(t)|<c_{r,5}\eps^2 {d}_{+}^{-5/2}, \  |\eta(t)|<c_{r,6}\eps^3 { d}_{+}^{-3},\\
 &|\kappa(t)-\dK_{\eps}(\eps t)|<  c_{r,7}\eps^4 \hat{ d}_{+}^{-9/2};
 \end{aligned}
   \end{equation}
 for $\im \tau\le  c_{l,1}^{-1}$ we have
\begin{equation}
\begin{aligned}
\label{ind_1-}
 &|z(t)|<c_{r,6}\eps^3 {d}_{-}^{-3}, 
 \  | \eta(t)|<c_{r,6}\eps^3 {d}_{-}^{-3},\\
 & |\kappa(t)-\dK_{\eps}(\eps t)|<  c_{r,7}\eps^4 \hat{ d}_{-}^{-9/2}
 \end{aligned} 
\end{equation}
The constants $c_{r,5}, c_{r,6}, c_{r,7}$ do not depend on the value $T$ or on  the choice of the constant $\mu_1$.
\end{lem}
We take $\mu_1= 2c_{r,7}$.
\begin{lem}
\label{margin_D_l}
If the constant $c_{e,5}$ is chosen sufficiently large, then for any $\eps T< \re \tau_c$ the assumptions   (\ref{induct_D_l}) are satisfied with a margin.
\end{lem}
Thus, one can take $\eps T= \re \tau_c$. Then, the estimates in Lemma \ref{est_of_shorten} imply the estimates in Lemma \ref{lem_cont_D_q_l}.

\medskip

The proof of Lemma \ref{lem_cont_D_gamma} is completely analogous. Just Lemma \ref{est_of_shorten} should be replaced with an analogous lemma for the domain $D_{p}$. Namely, in the statement  of Lemma \ref{est_of_shorten}, one should replace $D_{q}, C_q$ by $D_{\gamma}, C_{\gamma}$, or $D_{p}, C_{p}$. 
 It should be noted that the constants in the
estimates can be chosen uniformly with respect to the choice of $\tau_p$.
\medskip

\hskip 12 cm $\square$.

 \medskip
{\bf Proof of lemma \ref{l_improved_I}.}

Consider downward motion  along the  vertical line $\re \tau = \re\tau_c$ from the point $\tau=\tau _q$.
 For this motion, the equation for $z$ in (\ref {d_equation}) takes the form
 \begin{equation}
 \label{eq_z_vert_l}
  \frac{d z}{ds}
 =-i\Lambda_1(\dK_{\eps}(\eps t)) z +  \eps^3 O (\hat {d}_{+}^{-7/2}), \
    s= -\im t, 
 \end {equation}
 where $-i\Lambda_1(\dK_{\eps}(\eps t))$ is real and negative, bounded above by  $-c_{a,1}^{-1}\hat d_{+}^{1/2}$.
 For $|z|$,  we obtain a differential inequality
 \begin{equation}
 \label{neq_z_vert_l}
  \frac{d |z|}{ds}
 <- c_{a,1}^{-1}\hat d_{+}^{1/2}|z| + c_{a,2} \eps^3 \hat {d}_{+}^{-7/2}. \
 \end {equation}
 Consider a linear differential equation for a real variable $u$: 
 \begin{equation}
 \label{u_equation}
 \frac{d u}{ds}=-c_{a,1}^{-1}\hat d_{+}^{1/2}u+c_{a,2}\eps^3\hat {d}_{+}^{-7/2}.
 \end{equation}
 Denote by $s_0$ the value of $s$ corresponding to $\tau=\tau_q$.  Note that
 $|z(\tau_q/\eps)|<c_{e,7}\eps^2\hat {d}_{+}^{-5/2}(\tau_q)\\< c_{a,3}\eps^{1/3} C_q^{-5/2}$.  

Denote by $u(s)$ the solution of (\ref{u_equation}) with the  initial condition $u(s_0)=c_{a,3}\eps^{1/3} C_q^{-5/2}$. According to Lemma \ref {lem_compar_ODE_1},
 for $s_0 \le s\le s_0+ c_{a,4}^{-1}/\eps$,  we have
 $ |z(t)|< u(s)$.
 \medskip
 Denote  $\hat d_q=\hat d_+(\tau_q)$. Choose $s_1$ such that the change 
 in $\hat d_+$  on the interval  $[s_0, s_1]$ does not exceed $\hat d_q/2$. We may take 
  $s_1-s_0= c_{a,5}^{-1}\eps^{-1}\hat d_q$. Consider the linear differential equation with constant coefficients
 \begin{equation}
 \label{v_equation}
 \frac{d v}{ds}=-\nu v+\alpha,\
  \nu=c_{a,1}^{-1}2^{-1/2}\hat d_q^{1/2},\  \alpha=c_{a,2}(3/2)^{7/2}\eps^3\hat {d}_{q}^{-7/2}.
 \end{equation}
 Denote by $v(s)$ the solution of this equation with the initial condition $v(s_0)=u(s_0)$. For $s_0<s \le s_1$ we have $u(s)<v(s)$.
 
According to Lemma  \ref{l_aux_r_1}, for 
\begin{equation}
 \begin{aligned}
 &s\ge s_0 +\frac{1}{\nu}\left| \ln \left(\frac{\nu v(s_0)}{\alpha}\right)\right |
 = s_0 +\frac{1}{c_{a,1}^{-1}2^{-1/2}\hat d_q^{1/2}}
 \left |\ln \left( \frac {c_{a,1}^{-1}2^{-1/2} \hat d_q^{1/2}c_{e,7}\eps^2\hat {d}_{q}^{-5/2}}{c_{a,2}(3/2)^{7/2}\eps^3\hat {d}_{q}^{-7/2}}\right)\right|\\
 &=s_0 +c_{a,8}\hat d_q^{-1/2}|\ln(c_{a,7}^{-1}\eps \hat {d}_{q}^{-3/2})|
 \end{aligned}
 \end{equation}
 we have 
 $$
 v(s)<2\frac{\alpha}{\nu}=2\frac{c_{a,2}(3/2)^{7/2}\eps^3\hat {d}_{q}^{-7/2}}{c_{a,1}^{-1}2^{-1/2}\hat d_q^{1/2}}=c_{a,6}\eps^3\hat {d}_{q}^{-4}.
 $$
  We have
 $$
  c_{a,8}\hat d_q^{-1/2}|\ln(c_{a,7}^{-1}\eps \hat {d}_{q}^{-3/2})|<c_{a,9}\eps^{-1/3}C_q^{-1/2}(\ln C_q).
 $$
 Denote  $s_2=s_0+c_{a,9}\eps^{-1/3}C_q^{-1/2}(\ln C_q)$. For $s_0\le s\le s_2$ we have $0.5 \hat d_{q}<\hat d_{+}< 1.5\hat d_{q}$, if $C_q$ is sufficiently large.
 Thus,  $|z(t)|\le u(s)\le v(s)< c_{a,6}\eps^3\hat {d}_{q}^{-4}<c_{a,10}\eps^3\hat {d}_{+}^{-4}$
  for  $s=s_2$.
 
  \medskip
Consider  for $s\ge s_2$ solution $u(s)$ of the equation (\ref {u_equation}) with initial condition \\
$u(s)=c_{a,10}\eps^3\hat {d}_{+}^{-4}$ at  $s=s_2$.  
 On some time interval 
 $[s_2,s_*)$,  we have $|u(s)|<c_{a,11}\eps^3 \hat d_{+}^{-4}, \ c_{a,11}= \max (2c_{a,10}, 3c_{a,2}c_{a,1})  $.  Let $s_*$  be the first value of $s$ for which this inequality is not satisfied.   At $s=s_*$,  we have
$$
 \frac{du}{ds}=-c_{a,1}^{-1}\hat d_{+}^{1/2} c_{a,11} \eps^3 \hat d_{+}^{-4}+c_{a,2}\eps^3 \hat d_{+}^{-7/2}< -c_{a,1}^{-1}\hat d_{+}^{1/2} 2c_{a,1}c_{a,2} \eps^3 \hat d_{+}^{-4}+c_{a,2}\eps^3 \hat d_{+}^{-7/2}=-c_{a,2}\eps^3 \hat d_{+}^{-7/2}.
$$
At this value of $s$,
$$
\left |\frac{d}{ds}\left(c_{a,11}\eps^3 \hat d_{+}^{-4} \right)\right |<c_{a,12}\eps^4 \hat d_{+}^{-5}< c_{a,2}\eps^3 \hat d_{+}^{-7/2},
$$
if $c_{e,5}$ is sufficiently large. Thus,
the value
 $u(s)-c_{a,11}\eps^3 \hat d_{+}^{-4}$ decays at $s=s_*$ and, therefore, is positive just before $s_*$. This contradicts the definition of $s_*$. Therefore, the inequality $|u(s)|<c_{a,11}\eps^3 \hat d_{+}^{-4}$ is satisfied at least  while $\im \tau>-c_{a,4}^{-1}$. Then,  $ |z(t)|< u(s)< c_{a,11}\eps^3 \hat d_{+}^{-4}$  while $\im \tau>-c_{a,4}^{-1}$. 
Thus, we obtain the result of the Lemma with $c_{e,9,1}=c_{a,9},  c_{e,9,3}=c_{a,11}$ and with the constant $c_{e,9,2}$ determined by the constant $c_{a,4}$.

 \hskip 12cm $\square$

\medskip

\medskip
{\bf Proof of lemma \ref{lem_cont_D_q_r}.}

\medskip

Denote by $S(T)$ the part of the domain $D_{q,r}$  in the complex time plane where  
  $\re t_c \le \re t \le T$. We consider  the solution $z(t), \eta(t), \kappa(t)$ of  system (\ref{d_equation}).
  By Cauchy's theorem \cite{golub}, there exists $T>\re t_c$ such that  the solution $z(t), \eta(t), \kappa(t)$ can be  analytically continued  into  $S(T)$.
  
According to Lemma \ref{lem_cont_D_q_l}, on the line $\re t=\re t_c$ we have:

If $\im \tau> -c_{l,1}^{-1}$, then
\begin{equation}
|\kappa(t)-\dK_{\eps}(\eps t)|<  c_{e,6}\eps^4 \hat{ d}_{+}^{-9/2}  ,\ 
 | z(t)|<c_{e,7}\eps^2 \hat{ d}_{+}^{-5/2}, \  | \eta(t)|<c_{e,8}\eps^{3} \hat{ d}_{+}^{-3}.  \end{equation}
 
  If\quad  $-2 c_{l,1}^{-1}<  \im \tau < 2 c_{l,1}^{-1}$, then
  \begin{equation}
|\kappa(t)-\dK_{\eps}(\eps t)|=O(\eps^4 ) ,\ 
 | z(t)|=O(\eps^3), \  | \eta(t)|=O(\eps^{3}).  \end{equation}

 If  $\im \tau< c_{l,1}^{-1}$, then
\begin{equation}
|\kappa(t)-\dK_{\eps}(\eps t)|<  c_{e,6}\eps^4 \hat{ d}_{-}^{-9/2},\ 
 |z(t)|<c_{e,8}\eps^3 {\hat d}_{-}^{-3}, 
   | \eta(t)|<c_{e,8}\eps^3 \hat{ d}_{-}^{-3}.  \end{equation}
  
  According to Lemma \ref{l_improved_I}, on the line $\re t=\re t_c$ for \\
  $ \hat{ d}_{+}(\tau_q)+c_{e,9,1}\eps^{2/3}C_q^{-1/2}(\ln C_q)<\hat{ d}_{+}<c_{e,9,2}^{-1} $
    we have
 \begin{equation}
 | z(t)|<c_{e,9,3}\eps^3 \hat{ d}_{+}^{-4}
 \end{equation}
 
 \medskip
   
   We begin in the same way as in the proof of Lemma \ref{lem_cont_D_q_l}.
   
   \medskip

Denote by $O_1, O_2,  O_3, O_4$ last $O$-terms in equations  (\ref{d_equation}),   (\ref{d_equation1}). We have
\begin{equation}
|O_1({d}_{+}^{-7/2})|< c_{r,1} {d}_{+}^{-7/2}, \ |O_2({d}_{+}^{-3})|<c_{r,2} {d}_{+}^{-3},\ 
 |O_3({d}_{+}^{-1})|< c_{r,3}{d}_{+}^{-1}, \ |O_4({d}_{-}^{-3})|<c_{r,4} {d}_{-}^{-3}.
\end{equation}


 The statement of Lemma \ref{lem_cont_D_q_r} contains 
  the constants $c_{e,11}, c_{e,12}, c_{e,11,1}, c_{e,12,1}$. Their definitions are lengthy and are therefore given below in the proof of Lemma \ref {lem_in_Dqr}. The values of these constants do not depend on the value of $C_q$ provided that $C_q$ is sufficiently large. These constants will be used in estimates near the curve $\Gamma_{q, \eps}$. 
For each point $\tau_u\in \Gamma_{q, \eps}$, denote  $\tilde d_u=b\cdot(\dK(\tau_u)-\kappa_c), \hat d_u=|\tilde d_u|$.  Recall the notation from  Lemma \ref{lem_cont_D_q_r}:  $D_{q,r,d}$  (respectively,  $D_{q,r,d}'$) is the part of $D_{q,r}$ covered by the vertical segments of length less than or equal to (respectively, equal to) $c_{e,12} \eps \hat d_u^{-1/2}|\ln( c_{e,11}^{-1}\eps \hat d_u^{-3/2}C_q^{15/16})|$ drawn downward from all points
$\tau_u$; $\bar D_{q,r,d}$ and  $\bar  D_{q,r,d}' $   are the domains complex conjugate to   $ D_{q,r,d}$ and   $D_{q,r,d}' $, respectively.

Recall the notation:  $\Gamma_{q,r,d}$  (respectively,   $ \Gamma_{q,r,d}'$)  is the lower boundary of the domain $D_{q,r,d}$ (respectively, of the domain $D_{q,r,d}'$);   $Q_1$ is the point of intersection of the curve $\Gamma_{q,r,d}$ with the real axis in $\tau$-plane.

Recall the notation: $\tilde D_{q,r,d}$ (respectively,   $ \tilde D_{q,r,d}'$) is the part of $D_{q,r,d}$ covered by the vertical segments  of length less than or equal to (respectively, equal to) \\ $c_{e,12,1} \eps \hat d_u^{-1/2}|\ln( c_{e,11,1}^{-1}\eps \hat d_u^{-3/2}C_q^{15/16})|$ drawn downward from all points $\tau_u\in \Gamma_{q,\eps}$;   $\tilde \Gamma_{q,r,d}$ (respectively,   $ \tilde \Gamma_{q,r,d}'$) is the lower boundary of the domain $\tilde D_{q,r,d}$ (respectively, of the domain $\tilde D_{q,r,d}'$).




   \medskip
   
Rewrite system (\ref{d_equation}) in the domain  $W^{+}_{\delta}$ in the form
 \begin{equation}
\begin{aligned}
\label{with_beta_+r}
     &\dot z=  \Lambda_1(\dK_{\eps})z+ ( \Lambda_1(\kappa)-\Lambda_1(\dK_{\eps}))z + \beta_1+\eps^3O_1(d_{+}^{-7/2}),\\
  &\dot \eta=B(\dK_{\eps})\eta + (B(\kappa)-B(\dK_{\eps}))\eta +\beta_2+\eps^3O_2(d_{+}^{-3}),\\
   & \dot {\kappa}=\eps F(\kappa) +\eps^2 O(|z|^2d_+^{-2})+\eps^2 O(|\xi|_*^2d_+^{-3/2}) \\
 &
 +\eps O\left(|z|^4+|z|^5d_+^{-2}+ |\xi|_*^3 d_+^{-1}\right)
 +\eps^3 O(|z|d_+^{-3} + |\xi|_*d_+^{-3/2})+\eps^4O(d_+^{-7/2}|\xi|),\\
  &\hskip 1cm F= g(X(\kappa), \kappa,0)+\eps O_3( d_{+}^{-1}).
\end{aligned}
\end{equation}
In the domain $W^{-}_{\delta}$, we have analogous  equations with  $z$ and $w$ interchanged,  $d_+$ replaced by $d_-$, and $\Lambda_1$ replaced by
$\Lambda_2$.

Rewrite equation for $z$  in the domain  $W^{-}_{\delta}$ in the form
\begin{equation}
\begin{aligned}
\label{with_beta_-r}
\dot z=  \Lambda_1(\dK_{\eps})z+  ( \Lambda_1(\kappa)-\Lambda_1(\dK_{\eps}))z + \beta_4
     +\eps^3O_4(d_{-}^{-3}).
     \end{aligned}
\end{equation}

 \medskip
 Define $\hat d_{+,q}=\hat d_+(\tau_q)$,  $\hat d_{-,q}=\hat d_-(\bar\tau_q)$.
 Thus, $\hat d_{+,q}=\hat d_{-,q}$

 \medskip
  Take $T$ such that for $\eps t\in S(T) $  we have: 
  \begin{equation}
  \label{induct_xi_1/3}
   |\xi(t)|< c_{t,5}\eps^{1/3},
\end{equation}
   if $\tau \in  D_{q,r,d}$, then
   \begin{equation}
   \label{induct_D_r_d}
     |\eta(t)|< \eps^{17/6}\hat d_{+,q}^{-3}, \
   \ |w(t)|<\eps^{17/6}\hat  d_{+,q}^{-3}, \ |\kappa(t)- \dK_{\eps}(t)|< \mu_1 \eps^{4}\hat d_{+,q}^{-9/2}, \ 0.5 \hat d_{+}(\eps t)   \le d_{+}(\kappa(t)) \le 2 \hat d_{+}(\eps t),
   \end{equation}

  additionally,  if $\tau \in D_{q,r,d}\setminus \tilde D_{q,r,d}$, 
   then
 \begin{equation}
   \label{induct_tilde_D_r_d}
     |\eta(t)|< \eps^{17/6}\hat d_{+}^{-3}, |z(t)|<\mu_2(\eps^{11/6}C_q^{-6}\hat {d}_{+}^{-1/2}+\eps^3 \hat {d}_{+}^{-4}),
   \ |w(t)|<\eps^{17/6}\hat  d_{+}^{-3}, \      \end{equation}
   
   \medskip
    additionally,   if $\tau\in  D_{q,r,d}'\setminus  \tilde D_{q,r,d}$,  then 
    \begin{equation}
   \label{induct_mu_1_r}
     |\kappa(t)- \dK_{\eps}(t)|< \mu_1 \eps^{4}\hat d_{+}^{-9/2},
\end {equation}
    
  \vskip 0.3 cm
    additionally,  if $\tau \in \Gamma_{q,r,d}'$, then 
     \begin{equation}
   \label{induct_Gamma_r_d}
      |\eta(t)|< \eps^{17/6}\hat d_{+}^{-3}, \
   \ |w(t)|<\eps^{17/6}\hat  d_{+}^{-3}, |\kappa(t)- \dK_{\eps}(t)|< \mu_1 \eps^{4}\hat d_{+}^{-9/2},
   \end{equation}

   if  $\tau \in  (D_{q,r}\setminus ( D_{q,r,d}\cup\bar D_{q,r,d})\cap \{\im \tau\ge -c_{l,1}^{-1} \}$, then
   
   \begin{equation}
  \begin{aligned}
  \label{induct_D_r}
 &|\beta_1|< 0.5\eps^3 c_{r,1} {d}_{+}^{-7/2}, \ |\beta_2| < 0.5 \eps^3 c_{r,2} {d}_{+}^{-3}, 
  \end{aligned}
  \end{equation}
 and
 \begin{equation}
  \begin{aligned}
\label {induct_D_r_kappa}
 &|\kappa(t)- \dK_{\eps}(t)|< \mu_1 \eps^{4}\hat d_{+}^{-9/2},\\
 &  0.5 \hat d_{+}(\eps t)   \le d_{+}(\kappa(t)) \le 2 \hat d_{+}(\eps t). 
 \end{aligned}
  \end{equation}
  
  \bigskip
    if  $\tau \in  (D_{q,r}\setminus ( D_{q,r,d}\cup\bar D_{q,r,d})\cap \{\im \tau\le c_{l,1}^{-1} \}$, then
  \begin{equation}
  \begin{aligned}
   \label{induct_D_r_4}
    |\beta_4| < 0.5 \eps^3 c_{r,4}   {d}_{-}^{-3},
   \end{aligned}
  \end{equation}
  Here $\mu_1, \mu_2$ are  positive constants whose values do not depend on the choice of the constant $C_q$, provided that  $C_q$ is sufficiently large. The values $\mu_1, \mu_2$ are determined after the statement of Lemma \ref {est_of_shorten_r}.

 
  Note that  inequalities (\ref{induct_xi_1/3}) - (\ref{induct_D_r_4})  are certainly satisfied for sufficiently small  $T-\re(t_c)>0$. Also note that $\kappa(t)=\overline{\kappa(\bar t)}, \eta(t)=\overline{\eta(\bar t)}, \dK_{\eps}(t)=\overline{\dK_{\eps}(\bar t)}$.
  
  \medskip

  Assumption   (\ref{induct_D_r_kappa}) implies  that
  \begin{equation*}
  \begin{aligned}
  &|\Lambda_1(\kappa(t))-\Lambda_1(\dK_{\eps}(\eps t))||z|=\mu_1 \eps^4O(\hat d_{\pm}^{-5})\eps^{1/3}<0.1c_{r,1}\eps^3\hat d_{\pm}^{-7/2},\\
&|B(\kappa(t))-B(\dK_{\eps}(\eps t))||\eta|=\mu_1 \eps^4O(\hat d_{\pm}^{-5})\eps^{1/3}<0.1c_{r,2}\eps^3\hat d_{\pm}^{-3}\\
&\mbox{ (``+'' for $\im \tau\ge -c_{l,1}^{-1}$,
 ``-'' for $\im \tau\le c_{l,1}^{-1}$).}
\end{aligned}
\end{equation*}
(Here we used the estimate $|\xi| <c_{t,3}\eps^{1/3}$.)

Thus, equations (\ref{with_beta_+r}) have the form
 \begin{equation}
\begin{aligned}
\label{a_shorten_+r}
     &\dot z=  \Lambda_1(\dK_{\eps})z+\eps^3\tilde O_1(d_{+}^{-7/2}),\\
  &\dot \eta=B(\dK_{\eps})\eta + \eps^3 \tilde O_2( d_+^{-3})
  \end{aligned}
\end{equation}
with $|\tilde O_1(d_{+}^{-7/2})|< 2c_{r,1}d_{+}^{-7/2},\ |\tilde O_2(d_{+}^{-3})|< 2c_{r,2}d_{+}^{-3}$.

Equation (\ref{with_beta_-r}) has  the form
\begin{equation}
\begin{aligned}
\label{a_shorten_-r}
\dot z=  \Lambda_1(\dK_{\eps})z+\eps^3\tilde O_4(d_{-}^{-3})
     \end{aligned}
\end{equation}
with $|\tilde O_4(d_{-}^{3})|< 2c_{r,4}d_{-}^{-3}$.

\medskip

\begin{lem}
\label{lem_0.5_2}
Consider in the domain $D_{q,r}$ a segment of vertical line drawn downward from a point on $\Gamma_{q, \eps}$.    Let $\hat d_u$ denote the value of   $\hat d_{+}$ at the upper endpoint of this segment. For sufficiently large $C_q$, if the length of this segment is less than or equal to $c_{r,5}^{-1}\hat d_u$, then on this segment 
$$
0.5\hat d_u <\hat d_{+}< 2\hat d_u.
$$
\end{lem}
\begin{lem}
\label{lem_about_width}
For sufficiently large $C_q$, the vertical width of the domain $D_{q,r,d}$  satisfies the conditions of Lemma \ref{lem_0.5_2}, namely
        $$c_{e,12} \eps \hat d_u^{-1/2}|\ln( c_{e,11}^{-1}\eps \hat d_u^{-3/2}C_q^{15/16})|< c_{r,5}^{-1}\hat d_u .$$
\end{lem}

\begin{lem}
\label{lem_closeness}
By choosing $C_q$ sufficiently large,   the tangent directions of the curve   $\tilde\Gamma_{q,r,d}' $  and of the curve $\re \Psi_{\eps}={\rm const}$  passing through the same point can be made arbitrary close to each other. The same holds for the curve $\Gamma_{q,r,d}' $.
\end{lem}

\begin{lem}
\label{est_of_shorten_r}

For  $\tau \in S(T)$, we have the following estimates.

 \bigskip

 If $\tau\in  D_{q,r,d}$,    then

 $|z(t)|< c_{r,6}  \eps^2 \hat d_{+,q}^{-5/2},\ |\eta (t)|<   c_{r,7} \eps^3 \hat d_{+,q}^{-3},\ |w(t)|<c_{r,7} \eps^3 \hat d_{+,q}^{-3}, 
 |\kappa(t)-\dK_{\eps}(\eps t)|<  c_{r,8}\eps^4 \hat{ d}_{+,q}^{-9/2};$ 
 

 \medskip
additionally, if $\tau \in D_{q,r,d}\setminus \tilde D_{q,r,d}$,  
 then

 $|z(t)|< c_{r,9} (\eps^{11/6}C_q^{-6}\hat {d}_{+}^{-1/2}+\eps^3 \hat {d}_{+}^{-4}),\
  |\eta(t)|<  c_{r,10}\eps^{3}\hat d_{+}^{-3}, \
  |w(t)|<c_{r,10}  \eps^{3}\hat  d_{+}^{-3}$;
  
  \medskip
  additionally, if $\tau \in D_{q,r,d}'\setminus \tilde D_{q,r,d}$,  
   then
   $  |\kappa(t)-\dK_{\eps}(\eps t)|<  c_{r,11}\eps^4 \hat{ d}_{+}^{-9/2}$; 
 
\bigskip

additionally, if   $\tau \in  \Gamma_{q,r,d}'$,  then  
 
 $|z(t)|< c_{r,12} \eps^{3}\hat d_{+}^{-4}$.

\bigskip
 
  If $\tau \in  \bar D_{q,r,d}$, then
   
   $|z(t)|< c_{r,7}  \eps^3 \hat d_{-,q}^{-3},\ |\eta (t)|<   c_{r,7} \eps^3 \hat d_{-,q}^{-3},\ |w(t)|<c_{r,6} \eps^2 \hat d_{-,q}^{-5/2}, \\
 |\kappa(t)-\dK_{\eps}(\eps t)|<  c_{r,8}\eps^4 \hat{ d}_{-,q}^{-9/2}$;
    
    \bigskip
    additionally,    if $\tau\in  \bar D_{q,r,d}\setminus \overline {\tilde D}_{q,r,d}$,  then      
      $|z(t)|< c_{r,13} \eps^3 \hat{ d}_{- }^{-3},\    |\eta(t)|< c_{r,13} \eps^{3}\hat{ d}_{- }^{-3}, \\ 
     |\kappa(t)-\dK_{\eps}(\eps t)|<  c_{r,14}\eps^4 \hat{ d}_{-}^{-9/2} $.

    \bigskip

 If $\tau \in (D_{q,r} \setminus  ( D_{q,r,d}\cup \bar  D_{q,r,d})) \cap  \{\im\tau >-c_{l,1}^{-1}\}$, then 

 $|z(t)|< c_{r,15} \eps^{3}\hat{ d}_{+}^{-4}, \ | \eta(t)|<c_{r,16} \eps^{3} \hat{ d}_{+}^{-3}, \  
 |\kappa(t)-\dK_{\eps}(\eps t)|<  c_{r,17}\eps^4 \hat{ d}_{+}^{-9/2} . $   
 
 \bigskip
 If $\tau \in (D_{q,r} \setminus  ( D_{q,r,d}\cup \bar  D_{q,r,d})) \cap  \{\im\tau >-c_{l,1}^{-1}\}$, then 

 $|z(t)|< c_{r,16} \eps^{3}\hat{ d}_{-}^{-3}, \ | \eta(t)|<c_{r,16} \eps^{3} \hat{ d}_{-}^{-3}, \  |\kappa(t)-\dK_{\eps}(\eps t)|<  c_{r,17}\eps^4 \hat{ d}_{-}^{-9/2} . $  
     
     \bigskip 
 The constants $ c_{r,6}, c_{r,7}, \ldots, c_{r,17}$ do not depend on the value $T$ or on the choice of the constants $\mu_1, \mu_2 $.
\end{lem}
We take $\mu_1=2\max\{c_{r,8}, c_{r,11}, c_{r,14}, c_{r,17} \}, \mu_2=2 c_{r,9} $.

\begin{lem}
\label{margin_D_r}
If the constant $c_{e,10}$ is chosen sufficiently large, then for any $\eps T< \tau_{q,\eps,+}$ the assumptions (\ref{induct_xi_1/3}) -  (\ref{induct_D_r_4}) are satisfied with a margin.
\end{lem}
Thus, one can take $\eps T= \tau_{q,\eps,+}$. Then, the estimates in Lemma \ref{est_of_shorten_r} imply the estimates in Lemma \ref{lem_cont_D_q_r}. Note that $\eps^2d_{+,q}^{-5/2}=O(\eps^{1/3}C_q^{-5/2}), \ \eps^3 d_{+,q}^{-3}=O(\eps C_q^{-3}), \ 
\eps^4 d_{+,q}^{-9/2}=O(\eps C_q^{-9/2})$.

\medskip
This completes the proof of Lemma \ref{lem_cont_D_q_r}.

\medskip
\hskip 12cm $\square$

   {\bf Proof of Lemma \ref{lem_improved_kappa_old_1}.}
   
   \medskip
   According to (\ref{d_equation}),  in the domain $ D_{q,r,d}$ we have
   \begin{equation}
   \begin{aligned}
   \label{e_dot_kappa_old_1}
   & \dot {\kappa}=\eps F(\kappa) +\eps^2 O(|z|^2d_+^{-2})+\eps^2 O(|\xi|_*^2d_+^{-3/2}) \\
 &
 +\eps O\left(|z|^4+|z|^5d_+^{-2}+ |\xi|_*^3d_+^{-1}\right)
 +\eps^3 O(|z|d_+^{-3} + |\xi|_*d_+^{-3/2})+\eps^4O(d_+^{-7/2}|\xi|).\\
 %
 \end{aligned}
   \end{equation} 
   On  $\Gamma_{q,r,d}'$, we have
   \begin{equation}
   \label{e_kappa_1}
    |\kappa(t)-\dK_{\eps}(\eps t)|<  c_{r,15}\eps^4 \hat{ d}_{+}^{-9/2}. 
   \end{equation}
   We estimate $|\kappa(t)-\dK_{\eps}(\eps t)|$ in $D_{q,r,d}$ by moving vertically upward from $\Gamma_{q,r,d}'$. The vertical width of  $D_{q,r,d}$ is $c_{e, 12} \eps \hat d_u^{-1/2}|\ln( c_{e,11}^{-1}\eps \hat d_u^{-3/2}C_{q}^{15/16})|$.  We know that 
 $\partial F(\kappa)/\partial \kappa= O( \hat{ d}_{+}^{-1/2})$.
  According to Lemmas \ref{lem_0.5_2} and \ref{lem_about_width}, on the vertical line, $d_u$ can be replaced by  $\hat { d}_{+}$ in $O(\cdot)$-estimates.  
  Then, for $\tau\in D_{q,r,d}$, equations (\ref{e_dot_kappa_old_1}) and (\ref{e_kappa_1})  imply 
    \begin{equation*}
    \begin{aligned}
    &|\kappa(t)-\dK_{\eps}(\eps t)|=O(\eps^4 \hat{ d}_{+}^{-9/2})\\
    &+(c_{e, 12} \eps \hat d_+^{-1/2}|\ln( c_{e,11}^{-1}\eps \hat d_+^{-3/2}C_{q}^{15/16})|)
    O\left( \eps (\eps^{1/3}C_q^{-5/2})^2d_+^{-2})+\eps (\eps^{1/3}C_q^{-5/2}) (\eps C_q^{-3} )d_+^{-3/2}\right. \\
  &+((\eps^{1/3}C_q^{-5/2})^4+(\eps^{1/3}C_q^{-5/2})^5d_+^{-2}+ (\eps^{1/3}C_q^{-5/2})^2 (\eps C_q^{-3}) d_+^{-1}
 +\eps^2 (\eps^{1/3}C_q^{-5/2})d_+^{-3} \\
 &+ \eps^2  (\eps C_q^{-3} )d_+^{-3/2})
 \left.+\eps^3(\eps^{1/3}C_q^{-5/2})d_+^{-7/2}           \right)      \\     
 &=O(\eps^4 \hat{ d}_{+}^{-9/2}) +(\eps \hat d_+^{-1/2}|\ln(\eps \hat d_+^{-3/2}C_{q}^{15/16})|)\\
 &\cdot O\left( \eps^{5/3}C_q^{-5}d_+^{-2}+\eps^{7/3}C_q^{-11/2} d_+^{-3/2}
  +\eps^{4/3}C_q^{-10}+\eps^{5/3}C_q^{-25/2}d_+^{-2}+ \eps^{5/3}C_q^{-8} d_+^{-1}\right. \\
 &+\eps^{7/3}C_q^{-5/2}d_+^{-3} 
 + \eps^3  C_q^{-3} d_+^{-3/2}
 \left.+\eps^{10/3}C_q^{-5/2}d_+^{-7/2}           \right)      \\  
 &=O(\eps^4 \hat{ d}_{+}^{-9/2}) +(\eps \hat d_+^{-1/2}|\ln(\eps \hat d_+^{-3/2}C_{q}^{15/16})|)\\
 &\cdot O\left( \eps^{5/3}C_q^{-5}d_+^{-2}
  +\eps^{4/3}C_q^{-10}
 +\eps^{7/3}C_q^{-5/2}d_+^{-3} \right).
    \end{aligned}
    \end{equation*}
    Thus,  \begin{equation*}
    \begin{aligned}
    &|\kappa(t)-\dK_{\eps}(\eps t)|<c_{e,26}\eps^4 \hat{ d}_{+}^{-9/2} +c_{e,27}(\eps \hat d_+^{-1/2}|\ln(\eps \hat d_+^{-3/2}C_{q}^{15/16})|)\\
 &\cdot \left( \eps^{5/3}C_q^{-5}d_+^{-2}
  +\eps^{4/3}C_q^{-10}
 +\eps^{7/3}C_q^{-5/2}d_+^{-3} \right).
    \end{aligned}
       \end{equation*}
   
        \hskip 12cm $\square$


\medskip



 \medskip

\section{Proofs of lemmas about motion.}
\label{s_proofs_motion}

{\bf Proof of Lemma \ref{lem_init_cond_3}.}

Recall that   $\dK_{\eps}(\tau_*^-)= \kappa(\tau_*^-/\eps)$, $z(\tau_*^-/\eps)=O(\eps^3), \eta(\tau_*^-/\eps)=O(\eps^3)$,  and $\eps t_3=\tau_{ *,\eps,-}(\tau_{\gamma})=\tau_*^ - +O(\eps|\ln \eps|)$. For  $t$ between   $ \tau_*^-/\eps$ and  $t_3$, we have $z(t)=O(\eps^3), \eta(t)=O(\eps^3)$. Thus, during this time interval, according to (\ref{d_equation}),
$$
\dot{\kappa}=\eps F(\kappa)+O(\eps^6),
$$
while
$$
 \dot \dK_{\eps}=\eps F(\dK_{\eps}).
$$
Thus, during the considered time interval we have $\kappa(t)= \dK_{\eps}(\eps t)+O(\eps^6|\ln \eps|$).

\medskip

\hskip 12cm $\square$

\medskip

{\bf Proof of Lemma \ref{l_XC}.}

\medskip
For $\im \tau> -c_{l,1}^{-1}$, we have
\begin{equation}
\begin{aligned}
&C^{-1}(\dK(\eps t))(x(t)- X(\dK(\eps t)))=C^{-1}(\dK(\eps t))C(\kappa(t))C^{-1}(\kappa(t))(x(t)- X(\dK(\eps t)))\\
&=C^{-1}(\dK(\eps t))C(\kappa(t))\left[C^{-1}( \kappa(t))
\left(x(t)- X(\kappa(t))     \right)+ C^{-1} (\kappa(t))\left( X(\kappa(t))- X(\dK(\eps t))\right)\right]
\end{aligned}
\end{equation}
The estimate (\ref{after_lemma_+}) and Lemma \ref{K_and_K_eps}  imply that 
$$
C^{-1}(\dK(\eps t))C(\kappa(t))=I +\hat d_{+}^{-1/2}\left (O(\eps^3 \hat d_{+}^{-3})+O(\eps(1+|\ln \hat d_{+})|)\right), 
$$
where $I$ is the unit matrix. We use here the fact  that the derivative of $C$ with respect to $\kappa$ is $O(\hat d_{+}^{-1/2})$.

We have $C^{-1}( \kappa(t)) \left(x(t)- X(\kappa(t))\right)=\tilde \xi(t)$, where $\tilde \xi(t)$ is the vector introduced in Lemma \ref {lem_transform_00}. For components of this vector, we have 
$$
\tilde z(t)=O(\eps \hat d_{+}^{-1}), \  \tilde w(t)=O(\eps \hat d_{+}^{-1/2}), \  \tilde \eta (t)=O(\eps \hat d_{+}^{-1/2}).
$$
For components of the vector $C^{-1}( \kappa(t))( X(\kappa(t))- X(\dK(\eps t)))= \xi'(t)$, we have
\begin{equation*}
\begin{aligned}
 &z'(t)= \hat d_{+}^{-1/2}\left(O(\eps^3 \hat d_{+}^{-3})+O(\eps(1+|\ln \hat d_{+})|)\right)=O(\eps \hat d_{+}^{-1}),\\
 &|w'(t)|+|\eta'(t)|=\left(O(\eps^3 \hat d_{+}^{-3})+O(\eps(1+|\ln \hat d_{+})|)\right)=O(\eps \hat d_{+}^{-1/2}).
 \end{aligned}
\end{equation*}
We use here that, according to (\ref{de_z1}), the  derivative of the $z_1$ (respectively, $z_2$) component of the equilibrium of the fast system with respect to $\kappa$ is $O(\hat d_{+}^{-1/2})$ (respectively, $O(1)$). 
 
Combining above estimates, we obtain the result of the lemma  for  $\im \tau> -c_{l,1}^{-1}$.
 Similar estimates give the result of the lemma  for  $\im \tau< c_{l,1}^{-1}$.
\medskip

\hskip 12cm $\square$

{\bf Proof of Lemma \ref{delta_r_1}.}

According to Lemma \ref{lem_dr},
\begin{equation}
\frac{\partial \hat\ze}{\partial R}=-i \frac{3\sqrt{3}}{2\pi \hat s}\frac{1}{(RJ_{-1/3}(v)+J_{1/3}(v))^2},
\quad v=\frac{2}{3}(-\hat s)^{3/2} .
\end {equation}
Thus, according to asymptotic formula (\ref{big_v_J}), for real $s$,
$$
|\frac{\partial \hat\ze}{\partial R}|_{R=e^{2\pi i/3}}>c_{a,1}\frac{v}{|\hat s|}\sim \frac{|\hat s|^{3/2}}{|\hat s|}> c_{a,2}|\hat s|^{1/2}.
$$
Therefore,
$$
R_{\gamma}-R_-= O(|z_{1,\eps}(t_\gamma)-z_{1,sp}(t_\gamma)|/ |\tau_{\gamma}-\tau_{c}|^{1/2})  =  O(|z_{1,\eps}(t_\gamma)-z_{1,sp}(t_\gamma)|/(C_{\gamma,*}^{1/2}\eps^{1/3} )).
$$

\medskip

\hskip 12cm $\square$

\medskip

{\bf Proof of Lemma \ref{delta_r_2}.}

According to Lemma \ref{lem_dr} and the asymptotic formula  (\ref{big_v_J}), for real $s$,
$$
|\frac{\partial \hat\ze}{\partial R}|_{R=e^{2\pi i/3}}> c_{a,1} |\hat s|^{1/2}.
$$
On the other hand, formula  (\ref{big_v_J}) implies that
$$
|\frac{\partial \hat\ze}{\partial \hat s}|_{R=e^{2\pi i/3}} =O(|\hat s|^{-1/2}), \quad 
|\frac{\partial \hat\ze}{\partial  \tau}|_{R=e^{2\pi i/3}} =O(\eps^{-2/3}|\hat s|^{-1/2}).
$$
Thus 
\begin{equation*}
\begin{aligned}
&R_{\eps}(t_{\gamma})-R_{\gamma}=
O(|\tau_{\eps}(t_{\gamma})-\tau_{\gamma}|/(\eps^{2/3}|\hat s|))
=O(|\tau_{\eps}(t_{\gamma})-\tau_{\gamma}|
/(|\tau_{\gamma}-\tau_c|)\\
&=O(|\tau_{\eps}(t_{\gamma})-\tau_{\gamma}|/(C_{\gamma,*}\eps^{2/3})).
\end{aligned}
\end{equation*}
\hskip 12cm $\square$

\medskip

{\bf Proof of Lemma \ref{lem_to_c}.}

Consider the equations for the variables $ z_{sm},\eta_{sm}$, and  $\kappa$ for $t\in D_{\triangle}$. These equations have the form 
\begin{equation}
\begin{aligned}
\label{in_triangle}
&\dot z_{sm}=b \cdot (\kappa-\kappa_c) +az_{sm}^2 +O(\varepsilon +|z_{sm}|^3+|\kappa-\kappa_c|^2+
|\kappa-\kappa_c| |\xi_{sm}|
+|\xi_{sm}|_*^2),\\
&\dot \eta_{sm}=B\eta_{sm}+O(\eps+|\kappa-\kappa_c|
+|\xi_{sm}|^2),\\ 
&\dot \kappa=\eps (g_c+O(\eps+|\kappa-\kappa_c|+|\xi_{sm}|)).
\end{aligned}
\end{equation}

Initial conditions for this system are taken at $t=t_{\gamma}$ in accordance with values of $z,\eta, w, \kappa$ at 
$t_{\gamma}$ and the following lemma. 
\begin{lem}
\label{symmetry_triangle_1}
On the boundary, where $\im \tau=\im  \tau_{\gamma}$, we have 
 \begin{equation}
 \begin{aligned}
 \label{est1_triangle}
&z_{sm}=(1+O(\eps^{1/3}))z+O(\eps^{1/3})w+O(\eps^{1/3})\eta+O(\eps^{1/3}), \\ 
&w_{sm}=(1+O(\eps^{1/3}))w+O(\eps^{1/3})z+O(\eps^{1/3})\eta+O(\eps^{2/3}).
\end{aligned}
\end{equation}
Similarly, on the boundary, where $\im \tau=-\im  \tau_{\gamma}$, we have
 \begin{equation}
 \begin{aligned}
  \label{est2_triangle}
&z_{ms}=(1+O(\eps^{1/3}))z+O(\eps^{1/3})w+O(\eps^{1/3})\eta +O(\eps^{2/3}), \\ 
&w_{ms}=(1+O(\eps^{1/3}))w+O(\eps^{1/3})z+O(\eps^{1/3})\eta+O(\eps^{1/3}). 
\end{aligned}
\end{equation}
\end{lem}
Thus, we have
\begin{equation*}
z_{sm}(t_{\gamma})<c_{t,1}\eps^{1/3},\  \eta_{sm}(t_{\gamma})<c_{t,2}\eps^{2/3},  \ w_{sm}(t_{\gamma})<c_{t,3}\eps^{2/3},
|\kappa(t_{\gamma})-\kappa_c|< c_{t,4}\eps^{2/3}.
\end{equation*}
Discarding in (\ref{in_triangle}) $O(\cdot)$ terms, we obtain the system
\begin{equation}
\label{shorten_triangle}
\dot {\hat z}_{sm}=b \cdot (\hat\kappa-\kappa_c) +a\hat z_{sm}^2,  \
\dot {\hat \eta}_{sm}=B\hat \eta_{sm},\  \dot{\hat\kappa}=\eps g_c.
\end{equation}
Solutions of these equations for  ${\hat z}_{sm},  {\hat \eta}_{sm} ,\hat \kappa$  with the same initial conditions as for  (\ref{in_triangle}) are well defined in $D_{\triangle}$ and satisfy there the estimates
\begin{equation}
|\hat z_{sm}|< c_{t,5}\eps^{1/3}, \ {\hat \eta}_{sm}<c_{t,6}\eps^{2/3}, \  |\hat\kappa-\kappa_c|< c_{t,7}\eps^{2/3}.
\end{equation}
This is evident for $\hat\kappa$.  For $\hat z_{sm}$, this follows from the explicit formulas for the solution in Section \ref{s_riccati}. 
The solution ${\hat \eta}_{sm}$ can be estimated by considering a path that first follows the lower boundary of  $ D_{\triangle}$ 
and then the curve
$\re \psi_a(\tau)={\rm const}$, where the function $\psi_a$ was introduced in Section   \ref{s_riccati}. Along each part of this path, the solution satisfies a linear system of ODEs with constant coefficients whose eigenvalues have negative real parts. Therefore, the solution is well defined on the whole considered part of the curve $\re \psi_a(\tau)={\rm const}$ in  $ D_{\triangle}$, and satisfies the required estimate. Every  point of $ D_{\triangle}$ can be reached along such a path.

An analogous construction can be performed in the domain
$\bar D_{\triangle}$. We denote corresponding variables 
$z_{ms}, \eta_{ms}, w_{ms}, \xi_{ms}, \hat z_{ms}, \hat \eta_{ms, t_u},  \hat \xi_{ms,t_u}, \hat \kappa$. The estimates are 
\begin{equation}
|\hat w_{ms}|< c_{t,5}\eps^{1/3},\ |{\hat \eta}_{ms,t_u}|<c_{t,6}\eps^{2/3},\ |\hat\kappa-\bar \kappa_c|< c_{t,7}\eps^{2/3}.
\end{equation}

Denote by $S(T)$ the part of $D_c$ where $\re t\le T$. The solution of  system (\ref{d_equation}) can be continued at least into $S(T_1)$ such that in this domain, in $D_{\triangle}$, we have
\begin{equation}
\label{sm_induction}
| z_{sm}|< 2c_{t,5}\eps^{1/3},\ |{ \eta}_{sm}|<\mu \eps^{2/3},\ |\kappa-\kappa_c|< 2c_{t,7}\eps^{2/3},\ |w_{sm}|<\eps^{1/2},
\end{equation}
and in $\bar D_{\triangle}$ we have
 \begin{equation}
 \label{ms_induction}
|w_{ms}|< 2c_{t,5}\eps^{1/3},\ |{ \eta}_{ms}|<\mu \eps^{2/3},\ |\kappa-\bar \kappa_c|< 2c_{t,7}\eps^{2/3},\ |z_{ms}| <\eps^{1/2},
\end{equation}
where $\mu$ is a positive constant to be chosen later in the proof. 

\medskip
Using estimates (\ref{sm_induction}), (\ref{ms_induction}), and considering equations (\ref {in_triangle}) as a perturbation of equations (\ref{shorten_triangle}),   we can estimate $z_{sm}, \eta_{sm}$ and  $w_{ms}, \eta_{ms}$ for $\eps t \in D_{\triangle},  \re t\le T_1$:
\begin{equation}
\begin{aligned}
&| z_{sm}(t)|< 1.5c_{t,5}\eps^{1/3}, |{ \eta}_{sm}(t)|<c_{t,8}\eps^{2/3}, |\kappa(t)-\kappa_c|< 1.5c_{t,7}\eps^{2/3},\\
&|w_{ms}(\bar t)|< 1.5c_{t,5}\eps^{1/3}, |{ \eta}_{ms}(\bar  t)|<c_{t,8}\eps^{2/3}, |\kappa(\bar  t)-\bar  \kappa_c|< 1.5c_{t,7}\eps^{2/3}.
\end{aligned}
\end{equation}
Constant $c_{t,8}$ is determined independently of $\mu$. We then set  $\mu =2c_{t,8}$.
\begin{lem}
\label{symmetry_triangle_2}
For $\re t \le T_1$, we have
 \begin{equation}
 \label {e_symmetry_triangle}
 \kappa(\bar t)=\overline{\kappa(t)},\ z_{sm}(\bar t)=\overline{ w_{ms}(t)}, \ w_{sm}(\bar t)=\overline{w_{ms}(t)}, \ \eta_{sm}(\bar t)=\overline{ \eta_{ms}(t)}.
\end {equation}
\end{lem}

\medskip
On the lower boundary of $D_{up}$, which is also the upper boundary of $\bar D_{\triangle}$, we have 
$|z(t)|=O(\eps^{2/3}),|w(t)|=O(\eps^{1/3}), \eta(t)=O(\eps^{2/3})$ and, therefore, according to (\ref{est2_triangle}),  $|z_{ms}(t)|=O(\eps^{2/3}),|w_{ms}(t))|=O(\eps^{1/3}), \eta_{ms}(t)=O(\eps^{2/3})$. 
Considering motion vertically downward from this boundary, we obtain  $|z_{ms}(t)|=O(\eps^{2/3})$ for  $\eps t \in \bar D_{\triangle},  \re t\le T_1$. According to (\ref {e_symmetry_triangle}),  $|w_{sm} (t)| = |z_{ms}(\bar t )|$. Thus, 
$|w_{sm} (t)|=O(\eps^{2/3})$ for  $\eps t \in D_{\triangle},  \re t\le T_1$. Therefore, conditions (\ref{sm_induction}), (\ref{ms_induction}) are satisfied at $\re t=T_1$ with a margin. Therefore, one can take $T_1=\re \tau_c$. Additionally to these estimates, we obtain $|w_{sm} (t)|=O(\eps^{2/3})$ for  $t \in D_{\triangle}$, and $|z_{ms} (t)|=O(\eps^{2/3})$ for  $\eps t \in \bar D_{\triangle}$.

\medskip
 For $\eps t \in  D_{\triangle}$, we have
\begin{equation*}
\dot \kappa=\eps(g_c+O(\eps+|\kappa-\kappa_c|+|\xi_{sm}|)), 
\end{equation*}
which implies
\begin{equation*}
\dot \kappa=\eps(g_c+O(\eps^{1/3})). 
\end{equation*}
For $\dK$ in the considered domain we have
\begin{equation*}
\dot \dK=\eps(g_c+O(\eps^{1/3})). 
\end{equation*}
We also have $\kappa(\tau_{\gamma}/\eps)= \dK(\tau_{\gamma})+ O(\eps \ln\eps)$. Thus, we have
\begin{equation*}
\kappa(t)= \dK(\tau)+ O(\eps |\ln\eps|)
\end{equation*}
for $\eps t \in  D_{\triangle}$. This also implies $\kappa(\tau_c/\eps)=\kappa_c+ O(\eps |\ln\eps|)$.

Denote by $\hat s_{(\gamma)}$ the value of $\hat s$ at $\tau=\tau_{\gamma}$. Then 
$ \hat s_{\eps}(\tau_{\gamma}/\eps)=\hat s_{(\gamma)}+O(\eps^{1/3} |\ln\eps|+\eps^{1/3}C_{\gamma,*}^{3/2})$.  An estimate similar to that in (\ref{for_tau_eps}), but for the slow time spent in $D_{\Delta}$,   then implies  that  
$\hat s_{\eps}(\tau_c/\eps)=O(\eps^{1/3}| \ln\eps|+ \eps^{1/3}C_{\gamma,*}^{3/2})$.

\medskip
Differentiate the value $\hat \ze_{\eps}(t)$ (\ref {solution_eps}) with respect to time $\hat s$. On the one hand, because  $\hat \ze_{\eps}(t)$ is an exact solution of the equation, whose expansion is equation (\ref{eq_expanded}), using the already obtained estimates in $D_{\triangle}$, we obtain
\begin{equation}
\label{der_ze_1}
 \frac {d{\hat \ze}_{\eps}(t)}{d\hat s}=i( \eps^{-2/3} \nu b\cdot(\kappa(t)-\kappa_c)/(b\cdot g_c) +  ({\hat \ze}_{\eps}(t))^2 )+O(\eps^{1/3})=i(\hat s_{\eps} +({\hat \ze}_{\eps}(t))^2 )+O(\eps^{1/3}).
 \end{equation}
 Here $\nu$ is a complex number such that $\hat s=\eps^{-2/3}\mu (\tau-\tau_c)$.
 On the other hand, using Lemma \ref{lem_dr}, we obtain
\begin{equation}
\begin{aligned}
\label{der_ze_2}
 &\frac {d{\hat \ze}_{\eps}(t)}{d\hat s}
 =\frac{\partial {\hat \ze}_{\eps}(t)}{\partial R_{\eps}}\frac{d R_{\eps}}{d\hat  s}
 +\frac{\partial {\hat \ze}_{\eps}(t)}{\partial \hat s_{\eps}}\frac{d\hat  s_{\eps}}{d\hat s}\\
& =-i \frac{3\sqrt{3}}{2\pi \hat s_{\eps}}\frac{1}{(R_{\eps}J_{-1/3}(v_{\eps}(t))+J_{1/3}(v_{\eps}(t)))^2}\frac{d R_{\eps}}{d\hat  s} +i( \hat s_{\eps} + ({\hat \ze}_{\eps}(t))^2)\frac{d\hat  s_{\eps}}{d\hat s}.
\end{aligned}
 \end{equation}
 Also
 \begin{equation}
\frac{d\hat  s_{\eps}}{d\hat s} =\frac{d \tau_{\eps}}{d \tau}=
\frac{d (g_c(\kappa(t)-\kappa_c)/ (b\cdot g_c))}{d\tau}=1+O(\eps^{1/3}).
 \end{equation}
 Comparing (\ref{der_ze_1}) and (\ref{der_ze_2}), we obtain
  \begin{equation}
  \frac{3\sqrt{3}}{2\pi \hat s_{\eps}}\frac{1}{(R_{\eps}J_{-1/3}(v_{\eps}(t))+J_{1/3}(v_{\eps}(t)))^2}\frac{d R_{\eps}}{d\hat  s}= O(\eps^{1/3}).
  \end{equation}
  Thus,
   \begin{equation}
   \label{dR_e}
   \frac{d R_{\eps}}{d\hat  s}= O(\eps^{1/3} \hat s_{\eps}   (R_{\eps}J_{-1/3}(v_{\eps}(t))+J_{1/3}(v_{\eps}(t)))^2).  
   \end{equation}
   According to (\ref{small_v_J}), for small values of $|\hat s_{\eps}|$, the right hand side of (\ref{dR_e}) is\\
   $O(\eps^{1/3} |\hat s_{\eps}|  |\hat s_{\eps}| ^{-2/3\cdot{3/2}}  =O(\eps^{1/3})$. According to  (\ref{big_v_J}), for
   large values of $|\hat s_{\eps}|$, the right hand side of (\ref{dR_e}) is  $O(\eps^{1/3} |\hat s_{\eps}| |\hat s_{\eps}|^{-3/2})=O(\eps^{1/3} |\hat s_{\eps}|^{-1/2})  $. This implies
   $$
   |R_{\eps}(\tau_{c}/\eps)-R_{\eps}(\tau_{\gamma}/\eps)|=O(\eps^{1/3}).
   $$
   Together with the estimate
   $$
   \hat s_{\eps}(\tau_c/\eps)=O(\eps^{1/3} (|\ln\eps|+C_{\gamma,*}^{3/2}))
$$
and the asymptotic formula (\ref{as1_small}), this implies
$$
|\hat\ze_{\eps}(\tau_{c}/\eps )- \frac{1}{R_{\eps}(\tau_{\gamma}/\eps)}\frac{-2\pi i}{\Gamma^2(1/3)3^{1/6}}    |=O(\eps^{1/3} (|\ln\eps|+C_{\gamma,*}^{3/2})).
$$
\hskip 12cm $\square$

\medskip
{\bf Proof of Lemma \ref{lem_after_c_1}.}



Estimates  (\ref  {e_after_c_1_1}) and (\ref  {e_after_c_1_2}) can be obtained in the same way as the estimates in Lemma \ref{lem_to_c}. We omit the details.

Similarly to the proof of Lemma  \ref{lem_to_c},  we can show that 
$$\hat \sigma_{\eps}(\tau/\eps)=\hat \sigma+O(\eps^{1/3}(|\ln \eps|+|\hat \sigma|^{3/2})), \ \frac{d\hat 
\sigma _{\eps}}{d\hat \sigma} =1+O(\eps^{1/3} |\hat \chi_{\eps}|).
$$

Similarly  to the proof of Lemma  \ref{lem_to_c}, we have two different expressions for   ${d{\hat \chi}_{\eps}(t)}/{d\hat \sigma}$:
\begin{equation}
\label{der_chi_1}
\frac {d{\hat \chi}_{\eps}(t)}{d\hat \sigma}
=i(\hat \sigma_{\eps} +({\hat \chi}_{\eps}(t))^2 )
+O(\eps^{1/3}
(|{\hat \chi}_{\eps} |^3)
\end{equation}
and
\begin{equation}
\begin{aligned}
\label{der_chi_2}
 &\frac {d{\hat \chi}_{\eps}(t)}{d\hat \sigma}
& =-i \frac{3\sqrt{3}}{2\pi \hat \sigma_{\eps}}\frac{1}{(R_{\eps}^+J_{-1/3}(v_{\eps}(t))+J_{1/3}(v_{\eps}(t)))^2}\frac{d R_{\eps}^+}{d\hat  \sigma} +i( \hat \sigma_{\eps} + ({\hat \chi}_{\eps}(t))^2)\frac{d\hat  \sigma_{\eps}}{d\hat \sigma}.
\end{aligned}
 \end{equation}
 Asymptotic expansions   (\ref{big_v_J}) and  (\ref{expansion}) imply that
 \begin{equation*}
 \begin{aligned}
 & \hat \sigma_{\eps} + ({\hat \chi}_{\eps}(t))^2
 =O(\hat \sigma_{\eps}(e^{-2|\im v_{\eps }|}+|v_{\eps}|^{-1})),\\
 &\hat \sigma_{\eps}(R_{\eps}^+J_{-1/3}(v_{\eps}(t))+J_{1/3}(v_{\eps}(t)))^2=O(|\hat \sigma_{\eps}|^{-1/2}e^{2|\im v_{\eps }|})
 \end{aligned}
 \end{equation*}
 Equating  expressions (\ref{der_chi_1})
 and (\ref{der_chi_2}) we obtain
      $$
  \frac{d R_{\eps}^+}{d\hat  \sigma}= O\left(\eps^{1/3} 
 |\hat \sigma| e^{2\im v}\right).  
   $$
   Integrating this inequality shows that,  for $\tau$  on the curve $\eps^{-1}\psi_a=-C_{a,0}$, we have
   $$
    R_{\eps}^+(\tau/\eps)=R_{\eps}^+(\tau_c/\eps)+O(\eps^{1/3}C_{a,0}^{1/3}e^{C_{a,0}} ). 
   $$
  This implies
\begin {equation}
\begin{aligned}
&\hat \chi_ {\eps}(\tau/\eps)=-i\sqrt{-\hat \sigma}\,
\frac{J_{-2/3}(v)-R^+_{\eps}(\tau_c/\eps)J_{2/3}(v)}{R^+_{\eps}(\tau_c/\eps)J_{-1/3}(v)+J_{1/3}(v))}+O(\eps^{1/3}(|\ln \eps|C_{a,0}^{-1/3}+C_{a,0}^{2/3})).
\end{aligned}
\end{equation}
From the expansion (\ref{expansion}), we obtain
\begin {equation}
\hat \chi_ {\eps}(\tau/\eps)=\sqrt{-\hat\sigma}\Bigl [1-2e^{-2iv}e^{\pi i/6}
\frac{R^+_{\eps}(\tau_c/\eps)-e^{-2\pi i/3}}{R^+_{\eps}(\tau_c/\eps)- e^{2\pi i/3}}
+O\Bigl (e^{-4|{\rm Im\, }v|}+\frac{1}{|v|}\Bigr)\Bigr ]+O(\eps^{1/3}(|\ln \eps|C_{a,0}^{-1/3}+C_{a,0}^{2/3})).
\end{equation}
This implies, in view of (\ref {exp_esimate_1}),
\begin{equation}
 c_{t,1}^{-1}e^{-2C_{a,0}}-c_{t,2}\frac{1}{|v|}<|\frac{\hat \chi_ {\eps}(\tau/\eps)-\sqrt{-\hat\sigma}}{\sqrt{-\hat\sigma}}|
 +O(\eps^{1/3}(C_{a,0}^{-2/3}|\ln \eps|+C_{a,0}^{1/3}))< c_{t,1}e^{-2C_{a,0}}+c_{t,2}\frac{1}{|v|}.
\end{equation}
The equilibrium of the fast system on the considered curve in variables $\hat\chi, \hat\sigma$ is $\hat\chi_e$ such that
$|\hat\chi_e-\sqrt{-\hat\sigma}|=O(\eps^{1/3}C_{a,0}^{2/3})$.
In the original variables, we obtain
\begin{equation}
 \begin{aligned}
&c_{m,26}^{-1}e^{-2C_{a,0}}-c_{m,27}\frac{1}{|v|}<|\frac{z(t)}
{\sqrt{\sigma}}|+O(\eps^{1/3}(C_{a,0}^{-2/3}|\ln \eps|+C_{a,0}^{1/3}))<c_{m,26}e^{-2C_{a,0}} +c_{m,27}\frac{1}{|v|}.
\end{aligned}
\end{equation}
The estimate for  $|\kappa(t)-\dK_{\eps}(\eps t)|$ follows from the estimate for  $|\kappa-\hat \kappa|$  in Lemma \ref {lem_transform_kappa} and from the estimate for  $|\hat \kappa(t)-\dK_{\eps}(\eps t)|$  in Lemma \ref {lem_cont_D_q_r}.

\hskip 12cm $\square$

\medskip
{\bf Proof of Lemma \ref{lem_after_c_2}.}

\medskip
We know that the considered solution can be continued into the domain $D_{q,r}$ with estimates given
by Lemma \ref{lem_cont_D_q_r}. As $D_2\subset D_{q,r}$, it remains to prove the inequalities (\ref {est_after c_2}). This proof
follows the lines of the proof in Lemma \ref{lem_after_c_1}. We have 
$$\hat \sigma_{\eps}(\tau/\eps)=\hat \sigma+O(\eps^{1/3}(|\ln \eps|+|\hat \sigma|^{3/2})), \ \frac{d\hat 
\sigma _{\eps}}{d\hat \sigma} =1+O(\eps^{1/3} |\hat \chi_{\eps}|).
$$
Similarly to the proof in Lemma \ref{lem_after_c_2}, for the value $R_{\eps}^+(\tau/\eps)$ at the point $P_1$ we get
$$
    R_{\eps}^+(\tau/\eps)=R_{\eps}^+(\tau_c/\eps)+O(\eps^{1/3}C_{a,1}^{2}e^{C_{a,0}} ). 
   $$
 \medskip
Then at the point $P_1$ we have 
\begin{equation}
\begin{aligned}
&\hat \chi_ {\eps}(\tau/\eps)=\sqrt{-\hat\sigma}\Bigl [1-2e^{-2iv}e^{\pi i/6}
\frac{R^+_{\eps}(\tau_c/\eps)-e^{-2\pi i/3}}{R^+_{\eps}(\tau_c/\eps)- e^{2\pi i/3}}
+O\Bigl (e^{-4|{\rm Im\, }v|}+\frac{1}{|v|}\Bigr)\Bigr ]\\
&+O(\eps^{1/3}(C_{a,1}^{-1/2}|\ln \eps|+C_{a,1}^{5/2})),
\end{aligned}
\end{equation}
where $O\Bigl (e^{-4|{\rm Im\, }v|}+{1}/{|v|}\Bigr)$ is the same term as in (\ref {exp_esimate_2}).
We omit the detailed derivation of the expressions for the coefficients in this formula in terms of $C_{a,1}$, since only the existence of such coefficients, rather than their explicit form, is important.

  Taking into account that
$|R^+_{\eps}(\tau_c/\eps)-1|< 1/100$, and the definition of the point $P_1$ (see (\ref{exp_esimate_2})), we obtain that at this point 
\begin{equation}
0.5 c_{m,22}^{-1}e^{-C_{a,0}}<|\frac{\hat \chi_ {\eps}(\tau/\eps)-\sqrt{-\hat\sigma}}{\sqrt{-\hat\sigma}}|+O(\eps^{1/3}
(C_{a,1}^{-1}|\ln \eps|+C_{a,1}^{2}))
<c_{t,1}e^{-C_{a,0}}.
\end{equation} 
 
The equilibrium of the fast system at the point $P_1$  in variables $\hat\chi, \hat\sigma$ is $\hat\chi_e$ such that
$|\hat\chi_e-\sqrt{-\hat\sigma}|=O(\eps^{1/3}C_{a,1})$.
In the original variables, we obtain
\begin{equation}
 \begin{aligned}
  c_{m,28}^{-1}\eps^{1/3}\sqrt{|\hat \sigma |}e^{-C_{a,0}}<|z|<c_{m,28}\eps^{1/3} \sqrt{|\hat \sigma |}e^{-C_{a,0}}.
\end{aligned}
\end{equation}
\hskip 12cm $\square$
\medskip

\medskip
{\bf Proof of Lemma \ref{lem_after_c_2_1}.}
\medskip

According to Lemmas \ref{lem_transform_1}, \ref{lem_transform_2},
$$
\hat z=z+O(\eps d_+^{-1}) +O(|z|^2d_+^{-1/2}+|\xi|^2).
$$ 
At the point $P_1$ we have $d_+>c_{t,1}C_{a,1}\eps^{2/3}, \ c_{t,2}^{-1}C_{a,1} < |\hat \sigma |<c_{t,2}C_{a,1}, \eta=O(\eps^{2/3}), \\w=O(\eps^{2/3})$.
Thus, at point $P_1$ we have
\begin{equation}
| \hat z-z|<c_{t,3}\eps^{1/3}( C_{a,1}^{-1} +C_{a,1}^{1/2}e^{-2C_{a,0}}+  O(\eps^{1/3})).
\end{equation}

We take $C_{a,1}>c_{m,29}e^{2C_{a,0}/3}$ with  $c_{m,29}>c_{m,25}$  such that  we have\\ $| \hat z-z|<0.5c_{m,28}^{-1}\eps^{1/3} \sqrt{|\hat \sigma |}e^{-C_{a,0}}$. Then
$$
  c_{m,30}^{-1}\eps^{1/3}\sqrt{|\hat \sigma |}e^{-C_{a,0}}<|\hat z|<c_{m,30}\eps^{1/3} \sqrt{|\hat \sigma |}e^{-C_{a,0}}. 
  $$
\hskip 12cm $\square$

{\bf Proof of Lemma \ref{lem_last_arc}.}

\medskip

Denote by $\Gamma^{C_{a,0}}$ the arc of the curve  $\re \Psi ={\rm const}$ passing through the point $P_1$. Denote by 
$\tau_{+,C_{a,0}}$ the point  where this arc crosses  the real axis $\im \tau =0$. The distances of  $P_1$ to
$\tau_c$ and to the curve  $\Gamma_{*,2}$ are of order $\eps^{2/3}$. Thus, we have
$|\tau_{+,C_{a,0}}-\tau_*^+|/\sqrt{\eps^{2/3}}=O(\eps^{2/3})$ (cf. Lemma \ref{lem_Gamma_Gamma}). Thus
$|\tau_{+,C_{a,0}}-\tau_*^+|=O(\eps)$. 

Estimates in Lemma \ref {K_and_K_eps} and  the estimates for $\Lambda_1$ in Lemma \ref{lem_transform_1} imply that 
distance between the points where the curves  $\Gamma^{C_{a,0}}$ and $\Gamma^{C_{a,0}, \eps}$  cross the real axis is 
$O(\eps|\ln\eps|)$. Thus, $|\tau_{+,C_{a,0}, \eps}-\tau_{*}^{+}|= O(\eps|\ln \eps|)$. 
  
\hskip 12cm $\square$

\medskip
        {\bf Proof of Lemma \ref {lem_after_c_3}.}
   
   \medskip
        
        The possibility of continuing  the solution into the domain $D_{3,1}$ with the estimates given by Lemmas  \ref{lem_cont_D_q_r}, 
        \ref{lem_improved_kappa_old_1} follows directly from these lemmas. 
        
        The arc $\Gamma^{C_{a,o}, \eps}$  corresponding to  $\re\tau_c+ C_{a,1}\eps^{2/3}     \le \re \tau \le \re Q_1$ may consist of parts lying in the domains $D_{q,r}\setminus (D_{q,r,d}\cup \bar D_{q,r,d})$ and $D_{q,r,d}\setminus  \bar D_{q,r,d}$  (cf. Lemma \ref{crossing}).
       The estimates $\eta=O(\eps), w=O(\eps)$ on this arc in these domains follow from  Lemma  \ref{lem_cont_D_q_r}. 
        
        \medskip
        According to (\ref{d_equation}), on the   $\Gamma^{C_{a,o}, \eps}$ we have
        \begin{equation*}
        \begin{aligned}
           &\dot z=  \Lambda_1(\kappa)z+  \eps O(|z|^2 d_+^{-3/2})+
     \eps O(|\xi|_*^2 d_+^{-1/2})  +O(|\eta|(|\eta|+|w|))\\
       &+ O(|\xi|_*^3d_+^{-1/2})+O(|z|^5d_+^{-3/2}) +\eps^3O(d_{+}^{-7/2}).
       \end{aligned}
        \end{equation*}
        This can be rewritten as
         \begin{equation*}
        \begin{aligned}
           &\dot z=  \Lambda_1(\kappa)z+  \eps O(|z|^2 \hat d_+^{-3/2})+
     \eps^2 O(|z| \hat d_+^{-1/2})  \\
       &+ O(\eps |z|^2 \hat d_+^{-1/2})+O(|z|^5 d_+^{-3/2}) +\eps^3O(\hat d_{+}^{-7/2}) + O(\eps^2).
       \end{aligned}
        \end{equation*}

        Lemma \ref{lem_improved_kappa_old_1} implies that on  $\Gamma^{C_{a,o}, \eps}$, for $\re \tau \le \re Q_1$, we have
        \begin{equation*}
        \begin{aligned}
       |\Lambda_1(\kappa(t))-\Lambda_1(\dK_{\eps}(\eps t)|= O(
          \eps^4 \hat{ d}_{+}^{-5}) +(\eps \hat d_+^{-1}|\ln(\eps \hat d_+^{-3/2})|)
  \left(\eps^{4/3}+ \eps^{5/3}d_+^{-2}
   +\eps^{7/3}d_+^{-3} \right).
    \end{aligned}
       \end{equation*}
    (We do not indicate the dependence on $C_q$ here, since the  value of   $C_q$ is already fixed.)   
       
          
        Thus, we have
        \begin{equation*}
        \begin{aligned}
         &\dot z=  (\Lambda_1(\dK_{\eps}(\eps t)) +  \alpha)z+\beta,\\
        & \alpha=O(\eps^4 \hat{ d}_{+}^{-5}+(\eps \hat d_+^{-1}|\ln(\eps \hat d_+^{-3/2})|)
  \left(\eps^{4/3}+ \eps^{5/3}d_+^{-2}
   +\eps^{7/3}d_+^{-3} \right)    \\
      &+\eps |z| \hat d_+^{-3/2}+ \eps^2 \hat d_+^{-1/2}+\eps |z| \hat d_+^{-1/2}+|z|^4 d_+^{-3/2}),
       \\
        &\beta=\eps^3O(\hat d_{+}^{-7/2}) + O(\eps^2).
         \end{aligned}
        \end{equation*}
        As $|z|=O(\eps^{1/3})$, we obtain
        \begin{equation*}
        \begin{aligned}
        & \alpha=O(\eps^4 \hat{ d}_{+}^{-5}
       + (\eps \hat d_+^{-1}|\ln(\eps \hat d_+^{-3/2})|)
  \left(\eps^{4/3}+ \eps^{5/3}d_+^{-2}
   +\eps^{7/3}d_+^{-3} \right)  \\
  & +\eps^{4/3} \hat d_+^{-3/2}+ \eps^2 \hat d_+^{-1/2}+\eps^{4/3} \hat d_+^{-1/2}+\eps^{4/3} d_+^{-3/2}).
        \end{aligned}
        \end{equation*}
        The integral of $|\alpha|$ over  the part of $\Gamma^{C_{a,0}, \eps}$ with $\re \tau_c+ C_{a,1}\eps^{2/3} \le \re \tau\le \re Q_1$ does not exceed $c_{t,1}C_{a,1}^{-1/2}$.
        \footnote{The main contribution $\sim 1$
         comes from the term $\eps^{4/3} \hat d_+^{-3/2}$. 
         }
         The integral of $|\beta|$  over the same part of $\Gamma^{C_{a,0}, \eps}$ does not exceed
         $c_{t,2}\eps^{1/3} C_{a,1}^{-5/2}$.
         
         \medskip
         We have $c_{m,31}^{-1} C_{a,1}<\hat \sigma_1< c_{m,31} C_{a,1} $, \  $c_{m,30}^{-1}\eps^{1/3}\sqrt{|\hat \sigma_1|}e^{-C_{a,0}}<|z_1|<c_{m,30}\eps^{1/3} \sqrt{|\hat \sigma_1|}e^{-C_{a,0}}$.
         
         Already obtained estimates in this lemma  imply that the change of $z$ along the considered part of  $\Gamma^{C_{a,o}, \eps}$ does not exceed
         $$c_{t,3}C_{a,1}^{-1/2}|z_1| +c_{t,4}\eps^{1/3} C_{a,1}^{-5/2}.$$ We would like  to choose  $C_{a,1}$ such that this change is  smaller than  $0.5|z_1|$. For this,  it suffices to require 
    $c_{t,3}C_{a,1}^{-1/2}<0.25, \ c_{t,4}\eps^{1/3} C_{a,1}^{-5/2}<0.25 c_{m,31}^{-1}\eps^{1/3} \sqrt{|\hat \sigma_1 |}e^{-C_{a,0}}
 $, which is satisfied if $C_{a,1}>c_{m,32}, \ C_{a,1}>c_{m,33} e^{C_{a,0}/3}$. The value $C_{a,1}$ should also satisfy the condition $C_{a,1}>c_{m,29}e^{2C_{a,0}/3}$ introduced by Lemma \ref {lem_after_c_2_1}.

         \hskip 12cm $\square$
         
         \medskip
        {\bf Proof of Lemma \ref {lem_to_D32}.}
   
   \medskip
   The possibility of continuing the solution into the domain
   $D_{3,2}$ with the estimates given by Lemma  \ref{lem_cont_D_q_r} follows directly from 
        Lemma  \ref{lem_cont_D_q_r}.  In particular, in this domain we have 
        $$
        z(t)=O(\eps^{1/3}), \ w(t)=O(\eps^{1/3}), \ \eta(t)= O(\eps),\     |\kappa(t)-\dK_{\eps}(\eps t)|= O(\eps).
        $$
        According to (\ref{d_equation}), on the curve  $\Gamma^{C_{a,o}, \eps}$ in  $D_{3,2}$ we have
        \begin{equation*}
        \begin{aligned}
        &\dot z=  \Lambda_1(\kappa)z+  \eps O(|z|^2 d_+^{-3/2})+
     \eps O(|\xi|_*^2 d_+^{-1/2})  +O(|\eta|(|\eta|+|w|))\\
       &+ O(|\xi|_*^3d_+^{-1/2})+O(|z|^5d_+^{-3/2}) +\eps^3O(d_{+}^{-3}|\xi|)+\eps^3O^*(d_{+}^{-7/2}).  
       \end{aligned}
        \end{equation*}
        This can be rewritten as
        \begin{equation*}
        \begin{aligned}
          &\dot z=  \Lambda_1(\kappa)z+  \eps O(|z|^2)+O(\eps^{4/3}|z|)+O(\eps^{4/3}) \\
          &+O(\eps^{1/3}|z|^2)+O(|z|^5) +O(\eps^3|z|)+\eps^{10/3}+O(\eps^3).
        \end{aligned}
        \end{equation*}
        This can  further be rewritten as
         \begin{equation}
        \begin{aligned}
        \label{e_Q2_to_R}
           &\dot z=  (\Lambda_1(\dK_{\eps}(\eps t)) +\alpha )z+\beta,\\
           &\alpha=O(\eps+\eps^{4/3}+\eps^{4/3}+\eps^{2/3}+\eps^{4/3}+\eps^3)=O(\eps^{2/3}),\\
            &\beta=O(\eps^{4/3}+\eps^{10/3}+\eps^{3})=O(\eps^{4/3}).
       \end{aligned}
        \end{equation}
        Denote by $P_2$ the point of   the curve  $\Gamma^{C_{a,o}, \eps}$ where $\re P_2=\re Q_1$. Let $z_2$ be the value of $z(t)$ at this point. According to Lemma \ref {lem_after_c_3},  $c_{t,1}^{-1}\eps^{1/3}<|z_2|<c_{t,1}\eps^{1/3}$. The change in $t$ along the curve  $\Gamma^{C_{a,o}, \eps}$  from the point $P_2$ till the real axis is $O(|\ln\eps|)$. Then  (\ref{e_Q2_to_R}) implies that the change of
 $|z(t)|$ along  the curve  $\Gamma^{C_{a,o}, \eps}$  from the point $P_2$ till the real axis is $O(\eps|\ln\eps|)$. Thus, at the point  $\tau_{+,C_{a,0}, \eps}$ of intersection of  $\Gamma^{C_{a,o}, \eps}$ with the real axis, we have $|z(t)|=|w(t)|>c_{m,34}^{-1}\eps^{1/3}$ with  $c_{m,34} =2 c_{t,1}$.
 
 \hskip 12cm $\square$
 
  \medskip
        {\bf Proof of Lemma \ref {lem_to_original}.}
           
   \medskip
   Estimates for the transformations in Lemmas \ref{lem_transform_00}, \ref{lem_transform_1}, \ref {lem_transform_2}, and \ref {lem_transform_kappa}  show that the difference between the original and transformed variables at  $\tau= \tau_{+,C_{a,0}, \eps}$  is $O(\eps^{2/3})$ for $z,w,
   \eta$ and  $O(\eps^{4/3})$ for $\kappa$. Thus, in the original variables at  $\tau= \tau_{+,C_{a,0}, \eps}$, we have
   $$
  c_{m,35}^{-1}\eps^{1/3} < |z(t)|=|w(t)|<  c_{m,35}\eps^{1/3},\ \eta(t)=O(\eps^{2/3}), \   |\kappa(t)-\dK_{\eps}(\eps t)|= O(\eps).   
   $$
   According to Lemma \ref{K_and_K_eps},  $ |\dK(\eps t)-\dK_{\eps}(\eps t)|= O(\eps)$. Thus,  $ |\kappa(t)-\dK(\eps t)|= O(\eps)$.
   
    \hskip 12cm $\square$
    
    \medskip
        {\bf Proof of Lemma \ref {lem_escape}.}
           
   \medskip
   This is a standard situation of rapid departure from a non-degenerate linearly unstable equilibrium of the fast system, when the phase point is initially not too close to this equilibrium. We include the proof for completeness of the exposition.

    \medskip
   Denote $t_5=\tau_{+,C_{a,0},\eps}/\eps$. Note that $\tau_{+,C_{a,0},\eps}=\tau_*^+ + O(\eps|\ln \eps|)$.
   
   For real $t \ge t_5$, we have 
   \begin{equation}
        \begin{aligned}
        \label{e_for_departure_1}
        &\dot z =\lambda_1(\dK(\eps t))z +O(|z^2| +|z\eta| +|\eta|^2+|\kappa-\dK(\eps t)||z|+ \eps),\\
        &\dot \eta=B(\dK(\eps t))\eta +O(|z^2| +|z\eta| +|\eta|^2+|\kappa-\dK(\eps t)||\eta|+\eps),\\
        &\dot \kappa =\eps G(\kappa)+\eps O(|z|+|\eta|+\eps),\\
        &\dot \dK(\eps t)=\eps G(\dK(\eps t)).
   \end{aligned}
   \end{equation}
   At $t =t_5$, we have 
   $$
   c_{t,1}^{-1}\eps^{1/3} < |z(t_5)|<  c_{t,1}\eps^{1/3},\ \eta(t_5)=O(\eps^{2/3}), \   |\kappa(t_5)-\dK(\eps t_5)|= O(\eps).
   $$
   We also have
   $$
   \re\lambda_{1,2}(\dK(\eps t_5))>c_{t,2}^{-1},\  \re\lambda_{j}(\dK(\eps t_5))<-c_{t,3}^{-1}, \ j=3,\ldots , n.
   $$
   
   A  $c_{t,4}^{-1}$-neighbourhood $U$ of the point $(\re z(t_5), \im z(t_5),  \eta(t_5), \kappa(t_5))$ belongs to the domain $D$. Denote by $U_1$ the  $0.5c_{t,4}^{-1}$-neighbourhood of the point $(\re z(t_5), \im z(t_5),  \eta(t_5))$ in $x$-space.
   Consider the time interval $[t_5, t_5+\eps^{-1/2}]$ and its subinterval $[t_5, t_*]$ on which the point $(\re z(t), \im z(t),  \eta(t))$  remains inside   $U_1$. The changes in  $\kappa$ and  $\dK$ on this  interval are $O(\eps^{1/2})$. Thus, equations for $z, \eta$ in  system (\ref {e_for_departure_1}) take the form
   \begin{equation}
        \begin{aligned}
        \label{e_for_departure_2}
        &\dot z =\lambda_1(\dK(\eps t))z +O(|z^2| +|z\eta| +|\eta|^2+ \eps^{1/2}|z| +\eps),\\
        &\dot \eta=B(\dK(\eps t))\eta +O(|z^2| +|z\eta| +|\eta|^2+\eps^{1/2}|\eta|+\eps).
   \end{aligned}
   \end{equation}
   On the considered time interval, we have $\re \lambda_1(\dK(\eps t))>(3/4)c_{t,2}^{-1},\   \re\lambda_{j}(\dK(\eps t))<-(3/4)c_{t,3}^{-1}, \ j=3,\ldots , n $. 
 Thus,
 \begin{equation}
        \begin{aligned}
   &\frac{d}{dt}|z|^2=2(\re \lambda_1(\dK(\eps t))|z|^2+ O(|z^3| +|z|^2|\eta| +|z||\eta|^2+ \eps^{1/2}|z|^2+\eps|z| )\\
   &\ge (3/2)c_{t,2}^{-1}|z|^2+ O(|z^3| +|z|^2|\eta| +|z||\eta|^2+ \eps^{1/2}|z|^2+\eps|z|).
       \end{aligned}
   \end{equation}
   According to \cite{bellman}, Sect 13, the linear homogeneous system with the matrix $B(\dK(\tau))$ for frozen $\tau$ admits a
quadratic Lyapunov function $W(\tau, \eta)$, whose derivative with respect to time  is  equal to $-(\eta\cdot\eta)$. Thus,
  
   \begin{equation}
        \begin{aligned}
   &\frac{d}{dt}W=-(\eta\cdot\eta)+ O(\eps|\eta|^2+ |\eta|^3| +|\eta|^2|z| +|z|^2|\eta|+ \eps^{1/2}|\eta|^2+\eps|\eta|).
    \end{aligned}
   \end{equation}
   Following the construction in  the proof  of the Chetaev instability theorem \cite{chetaev}, we consider the function
   $$
   V(z, \eta, \tau)=|z|^2-W(\tau, \eta).
   $$
   Then
   \begin{equation}
        \begin{aligned}
         &\frac{d}{dt}V\ge (3/2)c_{t,2}^{-1}|z|^2 + (\eta\cdot\eta) \\
         &+ O(|z^3| +|z^2\eta| +|z||\eta|^2+ \eps^{1/2}|z|^2 +\eps |\eta|^2 +|\eta^3| +|\eta|^2|z| +|z|^2|\eta|+ \eps^{1/2}|\eta|^2)\\
         &+O(\eps |z|+\eps|\eta|)\\
         &\ge c_{t,6}^{-1}(|z|^2 + (\eta\cdot\eta)) +O(\eps |z|+\eps|\eta|)
        \end{aligned}
   \end{equation}
   provided that  $c_{t,4}^{-1}$ is sufficiently small.
   
   We have $|z(t_5)|> c_{t,1}^{-1}\eps^{1/3} $. We chose $t_*$ such that $|z(t)|> 0.5c_{t,1}^{-1}\eps^{1/3} $ for
   $t_5\le t\le t_*$.  Thus, on this time interval we have $ |z|^2\gg \eps |z|$.
   For the values of $\eta$ on the considered time interval we distinguish  two cases a) $|\eta|\le\eps^{2/3}$, and b) $|\eta|>\eps^{2/3}$.
   In  case a), $|\eps\eta|\ll |z|^2$. In  case b), $|\eps \eta|\ll |\eta|^2$. Thus, in both cases
   $$
   \frac{d}{dt}V\ge c_{t,7}^{-1}(|z|^2 + |\eta|^2)\ge  c_{t,7}^{-1}V.
   $$
   Then
   $$
   V(z(t), \eta(t), \eps t)\ge  c_{t,8}^{-1}\eps^{2/3}e^{c_{t,7}^{-1}t},
   $$
   and thus
   \begin{equation}
   \label{e_grow_z}
   |z(t)|\ge c_{t,8}^{-1/2}\eps^{1/3}e^{0.5c_{t,7}^{-1}t}
   \end{equation}
   As $V$ remains positive, we have $W(\eta(t), \eps t)< |z|^2$, and hence $|\eta(t)| <c_{t,9}|z(t)|$. 
   This estimate together with (\ref{e_grow_z}) implies that there exists
       $t_d = t_5+O(|\ln \eps|)$ and a constant $c_{m,36}$ such that $z(t_d)= c_{m,36}^{-1}$, and for $t_5\le t\le t_d $ the point $(\re z(t), \im z(t),  \eta(t))$ does not leave  $U_1$ (i.e. $t_d<t_*$).
   We have $\eps t_5 =\tau_{+,C_{a,0}, \eps}$, and  $|\tau_{+,C_{a,0}, \eps}-\tau_{*}^{+}|= O(\eps|\ln \eps|)$. Therefore, $|\eps t_d-\tau_{*}^{+}|= O(\eps|\ln \eps|)$. 
   Thus, $\tau_d=\eps t_d$ is the time  moment  claimed in Lemma \ref {lem_escape}.

    \hskip 12cm $\square$

%% file: delay_proofs_sect_11.tex
\section{Proofs of Lemmas from Section 11}
\label{proofs_11}

{\bf Proof of lemma \ref {est_of_shorten}.} 
\medskip

Estimate $z(t)$. Each point in the time domain $ S(T)$   with $\im t \ge 0$ can be reached from $t_4$ by first moving along the real axis and then along the curve $\re \Psi_{\eps}={\rm const}$. Along the real axis, we have $z(t)=O(\eps^3)$. Let $\sigma$ denote the arc length along the curve $\re \Psi_{\eps}={\rm const}$. On this curve, $t=t(\sigma)$.
For the motion along  the curve  $\re \Psi_{\eps}={\rm const}$ in the domain $\re \tau \ge -c_{l,1}^{-1}$, the equation for $z$ in (\ref{a_shorten_+}) takes the form
\begin{equation}
\label{eq_z_Gamma_q}
\frac{d z}{d\sigma} = i\omega(\eps\sigma)  z +\eps^3\frac{dt}{d \sigma} O (\hat {d}_{+}^{-7/2}),
\end{equation}
where  $\omega=\Lambda_1(\dK_{\eps})({dt}/{d \sigma}) i^{-1}$  is a real-valued  function bounded away  from 0 by  $c_{a,1}^{-1}{\hat d_{+} }^{\, 1/2}$.

Estimates obtained from equation (\ref {eq_z_Gamma_q}) imply that 
$$
|z(t)|<O(\eps^3) +  \eps^2 O( \hat{d}_{+}^{-5/2}(\tau))<
c_{r,5}\eps^2  \hat {d}_{+}^{-5/2}(\tau).
$$
The last estimate can be improved.
  Each point in the time domain $ S(T)$   can be reached from $t_4$ by first moving  along the curve $\Gamma_{q,\eps}$ and then  vertically downward.  For the downward motion along a  vertical line $\re t = {\rm const}$,   equation (\ref{a_shorten_+}) takes the form
\begin{equation*}
\label{eq_z_vert_+}
 \frac{d z}{ds}
 =-i\Lambda_1(\dK_{\eps}(\eps t)) z +  \eps^3 O (\hat {d}_{+}^{-7/2}), \
    s= -\im t .
\end{equation*}
According to condition 6) in Section \ref{form_conditions}, each vertical line 
 crosses the curves $\re\Psi_{\eps}={\rm const}$ transversally. Thus, $\re\Psi_{\eps}$ decreases downward along the vertical line, and \\  $\re( -i\Lambda_1(\dK_{\eps}(\eps t))<-c_{a,2}^{-1}\hat d_{+}^{1/2}$. This implies that $|z(t)|$ decreases for motion downward along any line $\re t = {\rm const}$ while $| z(t)|> c_{a,3}^{-1}\eps^3 \hat {d}_{+}^{-4}$.  
 This, in particular, implies that at $0 \le \im \tau\le 2c_{l,1}^{-1}$ we have    $|z(t)|=O(\eps^3) $. 
 \medskip
 In a similar way, by considering motion vertically downward from the real axis, we obtain that   $|z(t)|=O(\eps^3) $ for $-2c_{l,1}^{-1} \le  \im \tau\le 0 $, and $|z(t)|<c_{r,8 }\eps^3 \hat d_{-}^{-3} $ for $\im \tau \le c_{l,1}^{-1}$.  This implies that  $|w(t)|<c_{r,8 }\eps^3 \hat d_{+}^{-3} $ for $\im \tau \ge -c_{l,1}^{-1}$
 \medskip

 \medskip
 Estimate $\eta(t)$.  For definiteness, consider the half-plane   $\im t\ge 0$. Equation for $\eta$ in (\ref{a_shorten_+}) has the form
 \begin{equation}
\begin{aligned}
\dot \eta=B(\dK_{\eps})\eta + \eps^3 \tilde O_2( d_+^{-3}), \  |\tilde O_2( d_+^{-3})|< 2c_{r,2}d_{+}^{-3}.
  \end{aligned}
\end{equation}
 Along the real axis, we have $  |\eta(t)|< c_{a,4}\eps^3$.  Each point in the half-plane  $\im t\ge 0$ can be reached from the real axis
 by first moving along the real axis and then along the curve
   $\re \Psi_{\eps}(t)={\rm const}$. We introduce the arc length $\sigma$  along this curve as a new time parameter. Thus $t=t(\sigma), \ |dt/d\sigma|=1$.
 
 The equation for $\eta$  on the curve $\re \Psi_{\eps}={\rm const}$ takes the form
\begin{equation}
\label{e_with_sigma_00}
\frac{d \eta}{d\sigma}=\frac{dt}{d\sigma}B(\dK_{\eps})\eta + \alpha, \ \alpha= \eps^3 \frac{dt}{d\sigma}\tilde O_2( d_+^{-3}), \
 |\alpha|< 2 c_{r,2}\eps^3d_{+}^{-3}.
\end{equation}
According to condition  5) in Section \ref {form_conditions}, all eigenvalues of the matrix $({dt}/{d\sigma})B(\dK_{\eps})$ have negative real parts.

   According to \cite{bellman}, Sect 13, the corresponding to  (\ref {e_with_sigma_00}) linear homogeneous system 
    (i.e. with $\alpha =0$) 
   for frozen $\tau$ has a
quadratic Lyapunov function $W(\tau, \eta)$ whose $\sigma$-derivative  for frozen $\tau$  is equal to  $ -( \eta\cdot \bar \eta)$.

  We have $c_{a,5}^{-1} |\eta|^2 \le W(\eta)\le c_{a,5} |\eta|^2$.   The derivative of $W$ in the original system is 
$$
\frac{dW}{d\sigma}=-(\eta\cdot\bar \eta)+O(\eps\hat {d}_{+}^{-1/2}) W_1(\eta) +W_2(\alpha, \eta) 
 $$
 with a hermitian quadratic form $W_1$ and hermitian bilinear form $W_2$.  
 
   We have $|W_2(\alpha, \eta)|\le c_{a,6}|\alpha||\eta|\le c_{a,7}\eps^3\hat {d}_{+}^{-3}\sqrt{W}$. This  implies
 $$
 \frac{dW}{d\sigma}\le - 0.5c_{a,5}^{-1} W+  c_{a,7}\eps^3\hat {d}_{+}^{-3}\sqrt{W}.
 $$
 For the starting point on the real time axis, we consider two cases: (a) $W^{1/2} < 4  c_{a,5}  c_{a,7}\eps^3\hat {d}_{+}^{-3}$, and (b) $W^{1/2} \ge 4 c_{a,5} c_{a,7}\eps^3\hat {d}_{+}^{-3}$. For the case (a), inequality  $W^{1/2} <4 c_{a,5} c_{a,7}\eps^3\hat {d}_{+}^{-3}$ will be satisfied up to the line $\re \tau =\tau_c$.  Indeed, starting from the moment, when this inequality is not satisfied, we should have ${dW}/{d\sigma}<0$, i.e. $W$ decays, while $\hat {d}_{+}^{-3}$ grows, which leads to a contradiction. For the case (b), starting from the real axis, $W$ decays, while $\hat {d}_{+}^{-3}$ grows. If $W^{1/2} \ge 4 c_{a,5} c_{a,7}\eps^3\hat {d}_{+}^{-3}$ for all time up to   $\re \tau =\tau_c$, then, because, initially,  $W^{1/2}=O(\eps^3)$, we would have  $W^{1/2}=O(\eps^3)$ for all this time interval. Suppose now that there is a moment of time when $W^{1/2} = 4 c_{a,5} c_{a,7}\eps^3\hat {d}_{+}^{-3}$ for the first time.  From this time onward,  we have $W^{1/2} \le 4 c_{a,5} c_{a,7}\eps^3\hat {d}_{+}^{-3}$, as in case (a). Thus, in both cases (a) and (b) we obtain $W^{1/2}=O(\eps^3\hat {d}_{+}^{-3})$. This implies that $|\eta|=O(\eps^3\hat {d}_{+}^{-3})$.
 
 \medskip
 For $\dot \kappa$, we have
 \begin{equation*}
 \begin{aligned}
  & \dot {\kappa}=\eps F(\kappa) +\eps^2 O(|z|^2d_+^{-2})+\eps^2 O(|\xi|_*^2d_+^{-3/2}) \\
 &+\eps O\left(|z|^4+|z|^5d_+^{-2}+ |\xi|_*^3d_+^{-1}\right)
 +\eps^3 O(|z|d_+^{-3} + |\xi|_*d_+^{-3/2})+\eps^4O(d_+^{-7/2}|\xi|),\\
      &\hskip 1 cm F= g(X(\kappa), \kappa,0)+\eps O_3( d_{+}^{-1}). 
\end{aligned}
\end{equation*}
For terms in $\dot \kappa$, at $\im \tau\ge -c_{l,1}^{-1}$,  we have
\begin{equation}
\begin{aligned}
&\eps^2 O(|z|^2d_+^{-2})+\eps^2 O(|\xi|_*^2d_+^{-3/2})=O(\eps^2(\eps^2d_+^{-5/2})^2d_+^{-2})
=O(\eps^6d_+^{-7}),\\
&\eps O\left(|z|^4\right)=O(\eps(\eps^2d_+^{-5/2})^4)
=O(\eps^9 d_+^{-10}),\\
&\eps O(|z|^5d_+^{-2})=O(\eps (\eps^2d_+^{-5/2})^5d_+^{-2})=O(\eps^{11}d_+^{-29/2}),\\
& \eps O(|\xi|_*^3d_+^{-1})=O(\eps (\eps^2d_+^{-5/2})^2(\eps^3d_+^{-3})d_+^{-1}=O(\eps^8 d_+^{-9}),\\
&\eps^3 O(|z|d_+^{-3} + |\xi|_*d_+^{-3/2})=O(\eps^3(\eps^2d_+^{-5/2})d_+^{-3})=
O(\eps^5 d_+^{-11/2}),\\
&\eps^4O(d_+^{-7/2}|\xi|)=O(\eps^4d_+^{-7/2}(\eps^2d_+^{-5/2})=O(\eps^6d_+^{-6}).
\end{aligned}
\end{equation}
We have
$$
O(\eps^6d_+^{-7})+O(\eps^9 d_+^{-10})+O(\eps^{11}d_+^{-29/2})+O(\eps^8 d_+^{-9})+O(\eps^5 d_+^{-11/2})=O(\eps^5 d_+^{-11/2}).
$$
Thus,
$$
\dot\kappa= \eps F(\kappa)+O(\eps^5  { d}_{+}^{-11/2})= \eps F(\kappa)+O(\eps^5  \hat { d}_{+}^{-11/2}).
$$
Similarly for $\im \tau\le c_{l,1}^{-1}$.

Thus, in the considered domain,
\begin{equation*}
|\kappa(t)-\dK_{\eps}(\eps t)|<c_{r,7}\eps^4  \hat { d}_{\pm}^{-9/2}.
\end{equation*}
(This also absorbs  the effect of  the difference in the initial conditions $|\kappa(t_4)-\dK_{\eps}(\eps t_4)|=O(\eps^6\ln \eps)$).

 \hskip 12cm $\square$

 \medskip
 
 {\bf Proof of lemma \ref {margin_D_l}.} 
\medskip

 For $\im \tau\ge -c_{l,1}^{-1}$, we have
  \begin{equation}
\begin{aligned}
  &\beta_1= \eps O(|z|^2 d_+^{-3/2})+
     \eps O(|\xi|_*^2 d_+^{-1/2})  +O(|\eta|(|\eta|+|w|))\\
       &+ O(|\xi|_*^3d_+^{-1/2})+O(|z|^5d_+^{-3/2})+\eps^3O(d_{+}^{-3}|\xi|), \\
  &\beta_2=O(|\eta|^2) + O(|\eta |(|z|+|w|)) +O(|zw|) \\
   & \skip 0.7cm+     \eps O((|z|^2  +|w|^2)d_+^{-1/2})  
     +O(|\xi|_*^3d_+^{-1/2})+ O(|z|^4)+\eps^3O(d_{+}^{-3}|\xi|),\\
       & \dot {\kappa}=\eps F(\kappa) +\eps^2 O(|z|^2d_+^{-2})+\eps^2 O(|\xi|_*^2d_+^{-3/2}) \\
 &
 +\eps O\left(|z|^4+|z|^5d_+^{-2}+ |\xi|_*^3d_+^{-1}\right)
 +\eps^3 O(|z|d_+^{-3} + |\xi|_*d_+^{-3/2})+\eps^4O(d_+^{-7/2}|\xi|),\\
  &\hskip 1cm F= g(X(\kappa), \kappa,0)+O(\eps  d_{+}^{-1}).
   \end{aligned}
\end{equation}
Estimate terms in these relations using (\ref{ind_1+}).

For terms in $\beta_1$,  provided that $ d_+>c_{*} \eps^{2/3}$ with sufficiently large constant $c_{*}$, we have
\begin{equation}
\begin{aligned}
& \eps O(|z|^2  d_+^{-3/2})=O(\eps  (\eps^2  { d}_{+}^{-5/2})^2 d_+^{-3/2})=O(\eps^5  d_+^{-13/2})<0.05c_{r,1}\eps^3  { d}_{+}^{-7/2},\\
&\eps O(|\xi|_*^2 d_+^{-1/2}) = O( \eps (\eps^2  { d}_{+}^{-5/2} \eps^3  { d}_{+}^{-3}  )  d_+^{-1/2})= O(\eps^6 d_+^{-6})<0.05c_{r,1}\eps^3 { d}_{+}^{-7/2},\\
&O(|\eta|(|\eta|+|w|))=O((\eps^3  { d}_{+}^{-3})^2) =O(\eps^6 { d}_{+}^{-6})<0.05c_{r,1}\eps^3  { d}_{+}^{-7/2},\\
&O(|\xi|_*^3  d_+^{-1/2})=O(( (\eps^2  { d}_{+}^{-5/2})^2 \eps^3  { d}_{+}^{-3})  d_+^{-1/2})
=O(\eps^7  { d}_{+}^{-17/2})<0.05c_{r,1}\eps^3 { d}_{+}^{-7/2},\\
&O(|z|^5d_+^{-3/2})=O((\eps^2  {d}_{+}^{-5/2})^5 { d}_{+}^{-3/2})=O(\eps^{10} {d}_{+}^{-14})<0.05c_{r,1}\eps^3  { d}_{+}^{-7/2},\\
&\eps^3O(d_{+}^{-3}|\xi|)=O(\eps^3d_{+}^{-3}(\eps^2  {d}_{+}^{-5/2}))=O(\eps^5d_{+}^{-11/2})<0.05c_{r,1}\eps^3  { d}_{+}^{-7/2}.
\end{aligned}
\end{equation}

For terms in $\beta_2$, provided that $ d_+>c_{*} \eps^{2/3}$ with sufficiently large constant $c_{*}$,  we have
\begin{equation}
\begin{aligned}
&O(|\eta|^2)=O((\eps^3  { d}_{+}^{-3}) ^2)=O(\eps^6 { d}_{+}^{-6})<0.05c_{r,2}\eps^3 { d}_{+}^{-3}\\
&O(|\eta z|)=O(\eps^{3}  { d}_{+}^{-3} \eps^2 { d}_{+}^{-5/2})=O(\eps^{5}  { d}_{+}^{-11/2})<
0.05c_{r,2}\eps^3  { d}_{+}^{-3},\\
&O(|\eta w|)=O(\eps^{3}  { d}_{+}^{-3}\eps^3{d}_{+}^{-3})=O(\eps^{6}  { d}_{+}^{-6})<0.05c_{r,2}\eps^3  { d}_{+}^{-3},\\
&O(|z w|)=O(\eps^2  { d}_{+}^{-5/2}\eps^{3}  { d}_{+}^{-3})=
O(\eps^5  { d}_{+}^{-11/2})<0.05 c_{r,2}\eps^3  { d}_{+}^{-3},\\
&\eps O(|z|^2){ d}_{+}^{-1/2}=O(\eps (\eps^2  { d}_{+}^{-5/2})^2 { d}_{+}^{-1/2})=O(\eps^5  { d}_{+}^{-11/2})<0.05c_{r,2}\eps^3  { d}_{+}^{-3},\\
&\eps O(|w|^2) { d}_{+}^{-1/2}=O(\eps (\eps^3  { d}_{+}^{-3})^2  d_{+}^{-1/2})=O(\eps^7  d_{+}^{-13/2})<0.05 c_{r,2}\eps^3  { d}_{+}^{-3},\\
&O(|\xi|^3_* { d}_{+}^{-1/2})=O((\eps^2  { d}_{+}^{-5/2})^2 (\eps^3{d}_{+}^{-3}) { d}_{+}^{-1/2})=O(\eps^7 { d}_{+}^{-17/2})<0.05c_{r,2}\eps^3{d}_{+}^{-3}\\
&O(|z|^4)=O((\eps^2  { d}_{+}^{-5/2})^4)=O(\eps^8  { d}_{+}^{-10})<0.05c_{r,2}\eps^3  { d}_{+}^{-3},\\
&\eps^3O(d_{+}^{-3}|\xi|)=O(\eps^3d_{+}^{-3}(\eps^2  {d}_{+}^{-5/2}))=O(\eps^5d_{+}^{-11/2})<0.05c_{r,2}\eps^3  { d}_{+}^{-3}.
\end{aligned}
\end{equation}

For $\im \tau\le c_{l,1}^{-1}$, we have
\begin{equation}
\label{d_equation_1}
 \beta_4=  O(|\eta|(|\eta|+|w|))+ \eps O(|z||w| d_-^{-3/2})   + \eps O(|\xi|^2d_{-}^{-1/2})
     +O(|\xi|_{**}^3d_{-}^{-1/2})+ O(|w|^3)+\eps^3O(d_{-}^{-3}|\xi|). 
     \end{equation}
     
     Estimate terms in this relation using (\ref{ind_1-}).
     
For terms in $\beta_4$, provided that $ d_->c_{*} \eps^{2/3}$ with sufficiently large constant $c_{*}$,  we have
\begin{equation}
\begin{aligned}
&O(|\eta|(|\eta|+|w|))=O(  (\eps^3  { d}_{-}^{-3})(\eps^2  { d}_{-}^{-5/2} ))=O(\eps^5  { d}_{-}^{-11/2})<0.05c_{r,4}\eps^3  { d}_{-}^{-3},\\
&\eps O(|z||w| d_-^{-3/2}) =O(\eps (\eps^3  { d}_{-}^{-3})(\eps^2  { d}_{-}^{-5/2} )d_{-}^{-3/2})
=O(\eps^6 d_{-}^{-7})<0.05c_{r,4}\eps^3  { d}_{-}^{-3},\\
&\eps O(|\xi|^2 d_-^{-1/2})=O(\eps (\eps^2  { d}_{-}^{-5/2})^2 { d}_{-}^{-1/2})=O(\eps^5 \ { d}_{-}^{-11/2})<0.05c_{r,4}\eps^3  { d}_{-}^{-3},\\
& O(|\xi|_{**}^3 d_-^{-1/2})=O((\eps^2  { d}_{-}^{-5/2})^2(\eps^3  { d}_{-}^{-3})  d_-^{-1/2})=O(\eps^7  { d}_{-}^{-17/2})<0.05c_{r,4}\eps^3  { d}_{-}^{-3},\\
&O(|w|^3)=O((\eps^2  { d}_{-}^{-5/2})^3)=O(\eps^6{ d}_{-}^{-15/2})<0.05c_{r,4}\eps^3  { d}_{-}^{-3},\\
&\eps^3O(d_{-}^{-3}|\xi|)=O(\eps^3d_{-}^{-3}(\eps^2  {d}_{-}^{-5/2}))=O(\eps^5d_{-}^{-11/2})<0.05c_{r,4}\eps^3  { d}_{-}^{-3}.
\end{aligned}
\end{equation}
Thus,
$$
|\beta_1|< 0.4 \eps^{3} c_{r,1} {d}_{+}^{-7/2}, \ |\beta_2| < 0.4  \eps^{3}c_{r,2} {d}_{\pm}^{-3}, \  |\beta_4| < 0.4  \eps^{3} c_{r,4} {d}_{-}^{-3}.
$$

\medskip

For $\kappa$,  we have
\begin{equation*}
|\kappa(t)-\dK_{\eps}(\eps t)|<c_{r,7}\eps^4  \hat { d}_{\pm}^{-9/2}<\mu_1\eps^4  \hat { d}_{\pm}^{-9/2}
\end{equation*}
(This also absorbs the effect of  the difference in the initial conditions $|\kappa(t_4)-\dK_{\eps}(\eps t_4)|=O(\eps^6\ln \eps)$).

Then
$$
|{ d}_{\pm}(\kappa(t))-\hat { d}_{\pm}(\eps t)|< c_{a,1}\eps^4   { d}_{\pm}^{-9/2}(\kappa(t))\ll { d}_{\pm}(\kappa(t)).
$$
This implies
$$
0.6 \hat d_{\pm}(\eps t)   \le d_{\pm}(\kappa(t)) \le 1.5 \hat d_{\pm}(\eps t).
$$
If constant $c_{e,5}$ is chosen sufficiently large, then $ d_{\pm}>c_{*} \eps^{2/3}$ with a sufficiently large constant $c_{*}$. Then the above estimate imply that,
 for any $\eps T< \re \tau_c$,  the assumptions   (\ref{induct_D_l}) are satisfied with a margin.
 
 \hskip 12cm $\square$
 
\medskip

 {\bf Proof of Lemma \ref{lem_0.5_2}.}
\medskip 

Denote by $\tau_u$ the upper endpoint of the considered segment. 
Denote\\ $\tilde d_{+}(\tau)= b\cdot(\dK_{\eps}(\tau)-\kappa_c ), \ \tilde d_u=\tilde d_{+}(\tau_u)= b\cdot(\dK_{\eps}(\tau_u)-\kappa_c )  $. Thus  $ \hat d_{+}(\tau) =|\tilde d_{+}(\tau)|,  \ \hat d_u =|\tilde d_u|.$
 
  Then
$$
|\tilde d_u-\tilde d_{+}(\tau)| =|b\cdot (\dK_{\eps}(\tau_u)-\dK_{\eps}(\tau)|\le c_{a,1}|\tau_u-\tau|.
$$
Taking  $|\tau_u-\tau|<0.5c_{a,1}^{-1}\hat  d_u$, we obtain result of the Lemma with
$c_{r,5}=2c_{a,1}$.

 \hskip 12 cm $\square$
 
   \medskip
 {\bf Proof of Lemma \ref{lem_about_width}.}
\medskip 
We should prove that
 $$c_{e,12} \eps \hat d_u^{-1/2}|\ln( c_{e,11}^{-1}\eps \hat d_u^{-3/2}C_q^{15/16})|< c_{r,5}^{-1}\hat d_u ,$$
 which can be rewritten as 
 $$
  c_{r,5} c_{e,12} \eps \hat d_u^{-3/2}|\ln( c_{e,11}^{-1}\eps \hat d_u^{-3/2}C_q^{15/16})|< 1.
 $$
 For $\eps^{2/3}(C_{q,1}+o(1)) \le\hat d_u\le c_{a,1}^{-1}$ the left hand side of this inequality is a monotonic function of $\eps \hat d_u^{-3/2}$. Its value is 
 maximal at
  $\hat d_u=\eps^{2/3}(C_{q,1}+o(1))$. The inequality takes the form
  $$
  c_{r,5} c_{e,12} (C_{q,1}+o(1))^{-3/2}|\ln( c_{e,11}^{-1}C_{q,1}+o(1))^{-3/2}C_q^{15/16})|<1.
  $$
 Since $C_{q,1}$ grows linearly in the principal approximation  with growth of $C_{q}$ for large  $C_{q}$, this inequality is satisfied for  sufficiently large $C_q$.

 \hskip 12 cm $\square$
 
    \medskip
 {\bf Proof of Lemma \ref{lem_closeness}.}
\medskip

Demonstrate closeness of tangent directions of curves $\tilde\Gamma'=\tilde\Gamma'_{q,r,d} $ and $\re \Psi_{\eps}={\rm const}$ (Fig.  \ref{closeness_1}).  
\begin{figure}[H]
\begin{center}
            \includegraphics[scale=0.5, angle=0.0]{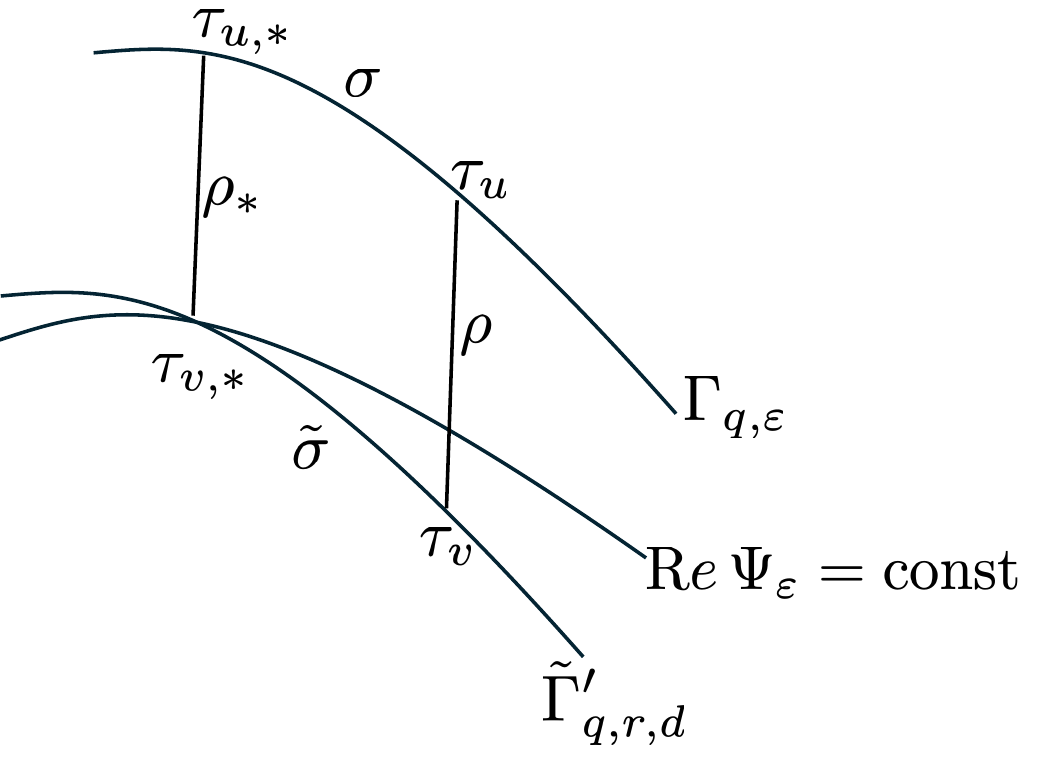}
            \end{center}
           \caption{Curves $\Gamma_{q,\eps}$,  $\tilde\Gamma'_{q,r,d}$ and  $\re \Psi_{\eps}={\rm const}$}
            \label{closeness_1}
\end{figure}
For a point $\tau_u\in  \Gamma_{q, \eps}$,   we denote $\hat d_u=\hat d(\tau_u), 
  \Lambda_{1,u}=\Lambda_{1}(\dK_{\eps}(\tau_u)),\\  \rho=-c_{e,12,1} \eps \hat d_u^{-1/2}\ln( c_{e,11,1}^{-1}\eps \hat d_u^{-3/2}C_q^{15/16})$. Denote  $\tau_v=\tau_u-i\rho\in \tilde \Gamma'$. 
   Fix some \\ $\tau_{u}=\tau_{u,*}\in \Gamma_{q, \eps}$ and the corresponding values 
   $\hat d_u=\hat d_{u,*},   \Lambda_{1,u}=  \Lambda_{1,u,*}, \rho=\rho_*$. Consider the level curve   $\re \Psi_{\eps}(\tau)={\rm const}$ passing through
   the point $\tau_{v,*}=\tau_{u,*}-i\rho_*$. Denote $\Lambda_{1,v,*}=\Lambda_{1}(\dK_{\eps}(\tau_{v,*}))$.

   The tangent direction to the curve $\re \Psi_{\eps}={\rm const}$  at the point $\tau_{v,*}$ is given by the complex number 
   $i\bar \Lambda_{1,v,*}/|\Lambda_{1,v,*}|$.
   
   For points on $\tilde\Gamma' $,  we have $\tau_{v}=\tau_{u}-i\rho$. Let $\sigma$ and  $\tilde \sigma$ denote the arc lengths along    $\Gamma_{q, \eps}$ and  $\tilde\Gamma' $, respectively. We have
   \begin{equation*}
  \left( \frac{d\tilde \sigma}{d \sigma}\right)_{\tau_{u}=\tau_{u,*}}= |\frac{d\tau_v}{d \tau_u}|_{\tau_{u}=\tau_{u,*}}=1+O\left(\eps \hat d_{u,*}^{-3/2}\ln(\eps \hat d_{u,*}^{-3/2}C_q^{15/16})\right).
    \end{equation*}

    The tangent direction to  $\tilde\Gamma' $ at the point $\tau_{v,*}$ is given by the complex number 
   \begin{equation*}
   \begin{aligned}
   &\left(\frac{d\tau_{v}}{d\tilde \sigma}\right)_{\tau_{u}=\tau_{u,*}} =\left(\frac{d\tau_{v}}{d \sigma}\frac{d\sigma}{d\tilde \sigma}\right)_{\tau_{u}=\tau_{u,*}} 
   =\left(\frac {d \tau_{u}}{d \sigma}-i
   \frac {d \rho}{d \hat d_u }  \frac {d \hat d_u}{d \sigma }
   \right)_{\tau_{u}=\tau_{u,*}}\left(1+O\left(\eps \hat d_{u,*}^{-3/2}\ln(\eps \hat d_{u,*}^{-3/2}C_q^{15/16})\right)\right)\\
   & =i\frac{\bar \Lambda_{1,u,*}}{|\Lambda_{1,u,*}|}+
   O\left(\eps \hat d_{u,*}^{-3/2}\ln(\eps \hat d_{u,*}^{-3/2}C_q^{15/16})\right)
   =i\frac{\bar\Lambda_{1,v,*}}{|\Lambda_{1,v,*}|}+O( \hat d_{u,*}^{-1}\rho_*)+
   O\left(\eps \hat d_{u,*}^{-3/2}\ln(\eps \hat d_{u,*}^{-3/2}C_q^{15/16})\right)\\
   &=i\frac{\bar \Lambda_{1,v,*}}{|\Lambda_{1,v,*}|}+O\left(\eps \hat d_{u,*}^{-3/2}\ln(\eps \hat d_{u,*}^{-3/2}C_q^{15/16})\right).
   \end{aligned}
   \end{equation*}
   Since $\eps \hat d_{u,*}^{-3/2}<c_{b,1}C_q^{-3/2}$, the angle between the tangent directions to the curves $\tilde\Gamma' $ and $\re \Psi_{\eps}={\rm const}$ can be made arbitrary small by choosing  $C_q$ sufficiently large.
   
   \medskip
   The same holds for the curve $\Gamma'_{q,r,d}$.

\medskip

\hskip 12cm $\square$

\medskip

{\bf Proof of Lemma \ref{est_of_shorten_r}.}

Estimate $z(t)$ on $\Gamma_{q,\eps}$.   Let $\sigma$ denote the arc length along the curve  $\Gamma_{q,\eps}$. On this curve, $t=t(\sigma)$.

\begin{lem}
\label{lem_on_curve_q_r}
For  motion along $\Gamma_{q,\eps}$, the equation for $z$ in (\ref{a_shorten_+r}) takes the form
\begin{equation}
\label{eq_z_Gamma_q_r}
\frac{d z}{d\sigma} = i\omega(\eps\sigma)  z +\frac{dt}{d \sigma}\alpha_1,
\end {equation}
where
 $\omega=\Lambda_1(\dK_{\eps})({dt}/{d \sigma}) i^{-1}$  is a real-valued  function bounded away from 0 by  $c_{a,1}^{-1}{\hat d_{+} }^{\, 1/2}$,
and
\begin{equation}
\begin{aligned}
|\alpha_1| &<c_{a,2}( |\kappa(t)-\dK(\eps t)| d_+^{-1/2}|z| + \eps |z|^2 d_+^{-3/2}+
 \eps |z|(\eps^{17/6} d_{+,q}^{-3})d_+ ^{-1/2}+\eps (\eps^{17/6} d_{+,q}^{-3})^2 d_+ ^{-1/2}\\
&+(\eps^{17/6} d_{+,q}^{-3})^2+|z|^2 (\eps^{17/6} d_{+,q}^{-3})d_+^{-1/2}+   |z| (\eps^{17/6} d_{+,q}^{-3})^2d_+^{-1/2}+  (\eps^{17/6} d_{+,q}^{-3})^3 d_+^{-1/2}  
\\
&   
 +|z|^5d_+^{-3/2} +\eps^3d_{+}^{-7/2}).
     \end{aligned}
     \end{equation}
\end{lem}
For the proof,  one only needs to substitute the estimates from (\ref{induct_D_r_d}) into the first equation in (\ref{d_equation}).
This lemma implies the following estimate.
\begin{lem}
\label{on_upper_curve_r}
If $\tau\in S(T)\cap\Gamma_{q,\eps}$,    then 
$|z(t)|< c_{a,3}  \eps^2 \hat d_{+,q}^{-5/2},\quad  |\kappa(t)- \dK_{\eps}(t)|<c_{a,4}\eps^4 \hat d_{+,q}^{-9/2}$.
\end {lem}

 

 \begin{lem}
 \label{lem_in_Dqr}
  If $\tau\in S(T) \cap D_{q,r,d} $,    then

 $|z(t)|< c_{r,6}  \eps^2 \hat d_{+,q}^{-5/2},\ |\eta (t)|<   c_{r,7} \eps^3 \hat d_{+,q}^{-3}$;
 
 \medskip
 additionally,  if
$\tau \in  S(T) \cap (D_{q,r,d}\setminus \tilde D_{q,r,d})$,  then

 $|z(t)|< c_{r,9} (\eps^{11/6}C_q^{-6}\hat {d}_{+}^{-1/2}+\eps^3 \hat {d}_{+}^{-4}),\
  |\eta(t)|<  c_{r,10}\eps^{3}\hat d_{+}^{-3};$ \
      \medskip
      
 additionally, if   $\tau \in  S(T) \cap \Gamma_{q,r,d}'$,  then  
 
 $|z(t)|< c_{r,12} \eps^{3}\hat d_{+}^{-4}$.
 
 \end{lem}
 
\medskip
To estimate $z(t)$ in  $S(T)\setminus (D_{q,r,d}\cup \bar D_{q,r,d})$, we use equation (\ref{a_shorten_+r}) with initial data on $\Gamma_{q,\eps,d}'$, and then equation (\ref{a_shorten_-r}) for motion vertically down along the lines $\re \tau={\rm const}$. 

\begin{lem}
\label{l_down_z}
 In  $S(T)\setminus (D_{q,r,d}\cup \bar D_{q,r,d})$,  we have
\begin{equation}
\begin{aligned}
& |z(t)|< c_{r,15} \eps^{3}\hat{ d}_{+}^{-4} \ \mbox{\rm if} \ \im \tau\ge -c_{l,1}^{-1},\\
 &|z(t)|< c_{r,16}\eps^{3} \hat{ d}_{-}^{-3} \  \mbox{\rm if} \ \im \tau\le c_{l,1}^{-1}.
  \end{aligned}     
\end{equation} 
\end{lem}

\begin{cor}
\label{c_down_z} 
The last relation implies
\begin{equation}
\begin{aligned}
\label{for_margin_w}
|w(t)|< c_{r,16}\eps^{3} \hat{ d}_{+}^{-3} \  \mbox{\rm if} \ \im \tau\ge -c_{l,1}^{-1}, \tau\notin  D_{q,r,d}.
\end{aligned}     
\end{equation}
\end{cor}

\medskip
To estimate $\eta (t)$ in  $S(T)\setminus (D_{q,r,d}\cup \bar D_{q,r,d})$ we use equation (\ref{a_shorten_+r}) for  $\im \tau\ge -c_{l,1}^{-1}$, and the corresponding equation with $d_{-}$ instead of $d_{+}$ for $\im \tau \le -c_{l,1}^{-1}$, with initial data given by Lemma \ref{lem_cont_D_q_l}. 
\begin{lem}
\label{l_right_eta}
In  $S(T)\setminus (D_{q,r,d}\cup \bar D_{q,r,d})$:
\begin{equation}
\begin{aligned}
|\eta(t)|< c_{r,16}\eps^{3} \hat{ d}_{+}^{-3} \  \mbox{\rm if} \ \im \tau\ge -c_{l,1}^{-1}\\
|\eta(t)|< c_{r,16}\eps^{3} \hat{ d}_{-}^{-3} \  \mbox{\rm if} \ \im \tau\le  c_{l,1}^{-1}.
\end{aligned}     
\end{equation}
\end{lem}
Now we obtain estimates for $\kappa(t)$.

\begin{lem}
\label{l_for_kappa_rup}
In part of $S(T)$ covered by the curves $\re \Psi_{\eps}={\rm const}$ passing through $D_{q,r,d}$,  we have
\begin{equation}
\label{e_for_kappa_rup}
 |\kappa(t)-\dK_{\eps}(\eps t)|<  c_{r,8}\eps^4 \hat{ d}_{+,q}^{-9/2}. 
\end{equation}
\end{lem}

According to Lemma \ref{lem_cont_D_q_l}, on the real axis for $\re \tau \le\re \tau_c$, we have \\  $|\kappa(t)-\dK_{\eps}(\eps t)|=O(\eps^4)$.  Using all above estimates, equations (\ref{with_beta_+r})  for $\im \tau\ge -c_{l,1}^{-1}$ and analogous equations for $\im \tau\le c_{l,1}^{-1}$,  in $S(T)\setminus (D_{q,r,d}\cup \bar D_{q,r,d})$ we obtain
\begin{equation}
\begin{aligned}
 |\kappa(t)-\dK_{\eps}(\eps t)|<  c_{r,17}\eps^4 \hat{ d}_{\pm}^{-9/2}   
   \end{aligned}
 \end{equation}
 (``+'' for $\im \tau\ge -c_{l,1}^{-1}$, ``-'' for $\im \tau\le c_{l,1}^{-1}$).
 \medskip
 
\begin{lem}
\label{l_down_down_z}
For $\tau \in (\bar D_{q,r,d}\setminus D_{q,r,d}) \cap S(T)$, we have
\begin{equation}
|z(t)|< c_{r,16}\eps^{3} \hat{ d}_{-,q}^{-3}.
\end{equation}

For $\tau \in (\bar D_{q,r,d}\setminus \overline {\tilde  D}_{q,r,d}) \cap S(T)$,  we have
\begin{equation}
|z(t)|< c_{r,16}\eps^{3} \hat{ d}_{-}^{-3}.
\end{equation}
\end{lem}

\begin{cor}
\label {c_down_down_z}
This lemma implies that $|w(t)|< c_{r,16}\eps^{3} \hat{ d}_{+}^{-3}$ for $\tau \in ( D_{q,r,d}\setminus {\tilde  D}_{q,r,d}) \cap S(T)$.
\end{cor}

It remains to obtain estimate of $   |\kappa(t)-\dK_{\eps}(\eps t)| $ in $D'_{q,r,d}\setminus  \tilde D_{q,r,d}$.

\begin{lem}
\label{lem_improved_kappa}
If $\tau\in  D'_{q,r,d}\setminus  \tilde D_{q,r,d}$,  
 then
   \begin{equation}
   \label{eq_improved_kappa}
    |\kappa(t)-\dK_{\eps}(\eps t)|<c_{r,11} \eps^4 \hat{ d}_{+}^{-9/2}.
    \end{equation}
\end{lem}

 This completes the proof of Lemma \ref{est_of_shorten_r}.
 
 \hskip 12 cm $\square$
 
 \medskip
 {\bf Proof of Lemma \ref{margin_D_r}.}
 
 The required properties for the domain $S(T)\cap  D_{q,r,d}$ are  satisfied due to Lemma  \ref{est_of_shorten_r}.
 
 The proof for the remaining part of $S(T)$  follows from estimates in  Lemma  (\ref{est_of_shorten_r})  completely analogously to the proof of  Lemma \ref{margin_D_l}, and we omit it. On should take into account that $\eps^3 d_{+}^{-4}<\eps^2 d_{+}^{-5/2}$,  if $C_q$ is sufficiently large. 
  
  \hskip 12 cm $\square$

%% file: delay_proofs_sect_12.tex
\section{Proofs of Lemmas from Section 12}
\label{proofs_sect_12}
{\bf Proof of Lemmas \ref{symmetry_triangle_1}, \ref{symmetry_triangle_2}.}

We have
\begin{equation}
 \begin{aligned}
& x=X(\kappa(t))+C(\kappa(t))\xi(t),  \quad x=X(\kappa_c)+C(\kappa_c) \xi_{sm}(t) ,\\
&z(t)=\xi_1(t)+i\xi_2(t),\ w(t)=\xi_1(t)-i\xi_2(t), \\
& z_{sm}(t)=\xi_{1,sm}(t)+i\xi_{2,sm}(t),\ w_{sm}(t)=\xi_{1,sm}(t)-i\xi_{2,sm}(t). \end{aligned}
\end{equation}
Here the value $t$ is taken on the lower boundary of the triangle $D_{\triangle}$, and we denote the first two  components of the vectors $\xi(t)$ and $\xi_{sm}(t)$ by
$\xi_1(t), \xi_2(t)$ and $\xi_{1,sm}(t), \xi_{2,sm}(t)$,   respectively.
 Thus,
 \begin{equation}
 \begin{aligned}
 \xi_{sm}(t)=C^{-1}(\kappa_c)(X(\kappa(t))-X(\kappa_c) )
 +C^{-1}(\kappa_c)C(\kappa(t))  \xi(t). 
 \end{aligned}
\end{equation}
We know that $\kappa(t)=\kappa_c+O(\eps^{2/3})$. This implies that $C(\kappa(t))=C(\kappa_c)+O(\eps^{1/3})$ and 
$C^{-1}(\kappa_c)C(\kappa(t))=I+O(\eps^{1/3})$,  where $I$ is the unit matrix. 

Moreover, $X(\kappa(t)) =X(\kappa_c)+O(\eps^{1/3})$,  and 
$$
C^{-1}(\kappa_c)(X(\kappa(t))-X(\kappa_c)) =\begin{pmatrix}(u+v)/2\\ -i(u-v)/2\\O(\eps^{2/3})   \end{pmatrix} , u=O(\eps^{1/3}),\ v=O(\eps^{2/3}).
$$
This is because the matrix $C^{-1}(\kappa_c)$ transforms $x-x_c$ into a coordinate system in which the matrix $A_c$ takes the block-diagonal form described in Lemma \ref{lem_transform_00}.

Thus, we obtain estimate (\ref{est1_triangle}) of the Lemma \ref {symmetry_triangle_1}. In a similar way, we obtain estimate (\ref{est2_triangle}) of this  Lemma.
 \medskip
 
 We also have 
\begin{equation}
 \begin{aligned}
& \xi_{sm}(t)=C^{-1}(\kappa_c)(X(\kappa(t))-X(\kappa_c)  )
 +C^{-1}(\kappa_c)C(\kappa(t)) \xi(t),\\
 & \xi_{ms}(\bar t)=C^{-1}(\bar \kappa_c)(X(\kappa(\bar t))-X(\bar \kappa_c)  )
 +C^{-1}(\bar \kappa_c)C(\kappa(\bar t)) \xi(\bar t).
 \end{aligned}
\end{equation}
Relations 
\begin{equation}
 \begin{aligned}
&X(\bar\kappa_c)=\bar X(\kappa_c),\ C(\bar \kappa_c)=\bar C(\kappa_c), \
\kappa(\bar t)=\bar \kappa( t), \\
&X(\kappa(\bar t))= \bar X(\kappa( t)),\ C(\kappa(\bar t))=\bar C(\kappa( t)),\
\xi((\bar t))=\bar \xi((t))
 \end{aligned}
\end{equation}
imply that $ \xi_{ms}(\bar t)=\bar \xi_{sm}( t)$. This, in turn, implies that
$$
z_{ms}(\bar t)=\bar w_{sm}( t),\ w_{ms}(\bar t)=\bar z_{sm}( t), \eta_{ms}(\bar t)=\bar \eta_{sm}( t).
$$
\hskip 12cm $\square$

%% file: delay_proofs_sect_13.tex
\section{Proofs of Lemmas from Section 13}
\label{proofs_sect_13}

  \medskip
 {\bf Proof of Lemma \ref{on_upper_curve_r}.}
 
 \medskip
Denote $t_q=\tau_q/\eps$.  According to Lemma \ref {lem_cont_D_q_l},  we have
 \begin{equation*}
| z(t_q)|<c_{e,7}\eps^2 \hat{ d}_{+,q}^{-5/2}, \quad  |\kappa(t_q)-\dK_{\eps}(\eps t_q)|<  c_{e,6}\eps^4 \hat{ d}_{+,q}^{-9/2}  .  
 \end{equation*}
 On the curve $\Gamma_{q, \eps}$,  we have
 \begin{equation}
\begin{aligned}
\label{eq_upper}
 &\dot z=  \Lambda_1(\kappa)z+  \eps O(|z|^2 d_+^{-3/2})+
     \eps O(|\xi|_*^2 d_+^{-1/2})  +O(|\eta|(|\eta|+|w|)\\
       &+ O(|\xi|_*^3d_+^{-1/2})+O(|z|^5d_+^{-3/2}) +\eps^3O(d_{+}^{-7/2}),\\
       & \dot {\kappa}=\eps F(\kappa) +\eps^2 O(|z|^2d_+^{-2})+\eps^2 O(|\xi|_*^2d_+^{-3/2}) \\
 &
 +\eps O\left(|z|^4+|z|^5d_+^{-2}+ |\xi|_*^3d_+^{-1}\right)
 +\eps^3 O(|z|d_+^{-3} + |\xi|_*d_+^{-3/2})+\eps^4O(d_+^{-7/2}|\xi|).
\end{aligned}
\end{equation}
Denote $z_d=\eps^2 \hat{ d}_{+,q}^{-5/2}<c_{b,1}\eps^{1/3}C_q^{-5/2}, \ \kappa_d=\eps^4 \hat{ d}_{+,q}^{-9/2}<c_{b,2}\eps C_q^{-9/2}$.
Denote by $z_m$ and  $\kappa_m$ the suprema  of the quantities    $|z(t)|$ and $|\kappa(t)-\dK_{\eps}(\eps t)|$ on $\Gamma_{q, \eps}$ over the  time interval from $t_q$ to some $t_m$ in $S(T)$.  Assume that on this time interval $z_m<\beta z_d$, where $\beta$ is a positive constant to be determined later. Note that this inequality with $\beta\ge c_{e,7}$  is certainly satisfied for  $t_m$  sufficiently close to $t_q$. 
Equations (\ref {eq_upper}) imply, for the time interval from $t_q$ to $t_m$,
\begin{equation}
\begin{aligned}
\label{forz_0}
|z(t)|&\le z_m<c_{e,7}\eps^2 \hat{ d}_{+,q}^{-5/2} +c_{b,3}\eps^2 \hat{ d}_{+,q}^{-5/2}+c_{b,4}\frac{1}{\eps}\kappa_m z_m\\
&+c_{b,5}\left(z_m^2\hat{ d}_{+,q}^{-1/2}+ z_m\eps^{17/6}\hat{ d}_{+,q}^{-3}+ \frac{1}{\eps}\eps^{17/3}\hat{ d}_{+,q}^{-6}\frac{z_m}{z_d}
+\frac{1}{\eps}z_m^2\eps^{17/6}\hat{ d}_{+,q}^{-3}+\frac{1}{\eps}z_m^5\hat{ d}_{+,q}^{-1/2}
\right).
\end{aligned}
\end{equation}
The first, second, and third terms  in the right hand side  of this inequality   appear, respectively,   due to initial condition for $z$ at $t=t_q$ , the last term in equation for $\dot z$ in (\ref{eq_upper}), and the term  $\Lambda_1(\kappa)z$ in (\ref{eq_upper}).
Equation (\ref{forz_0}) implies that
\begin{equation}
\begin{aligned}
z_m&\left(1-c_{b,6}   \left(z_m\hat{ d}_{+,q}^{-1/2}+ \eps^{17/6}\hat{ d}_{+,q}^{-3}+ \eps^{8/3}\hat{ d}_{+,q}^{-7/2}
+z_m\eps^{11/6}\hat{ d}_{+,q}^{-3}+\frac{1}{\eps}z_m^4\hat{ d}_{+,q}^{-1/2}
\right)\right)\\
&<c_{b,7}(z_d+\frac{1}{\eps}\kappa_m z_m).
\end{aligned}
\end{equation}
One can check that if $\beta < c_{b,8}^{-1}C_q^{21/8}$,  then on the time interval $[t_q, t_m)$ we have 
 \begin{equation}
 \label{induct _tm1}
 c_{b,6}   \left(z_m\hat{ d}_{+,q}^{-1/2}+ \eps^{17/6}\hat{ d}_{+,q}^{-3}+ \eps^{14/3}\hat{ d}_{+,q}^{-6}
+z_m\eps^{11/6}\hat{ d}_{+,q}^{-3}+\frac{1}{\eps}z_m^4\hat{ d}_{+,q}^{-1/2}
\right)<1/2.
 \end{equation}
 Indeed, consider the largest terms in (\ref{induct _tm1}):
 \begin{equation*}
 \begin{aligned}
 &z_m\hat{ d}_{+,q}^{-1/2}<\beta z_d \hat{ d}_{+,q}^{-1/2}< \beta c_{b,9}\eps^{1/3}C_q^{-5/2}(\eps^{2/3}C_q)^{-1/2}=
 \beta c_{b,9}C_q^{-3}<0.1(c_{b,6})^{-1},\\
 &\frac{1}{\eps}z_m^4\hat{ d}_{+,q}^{-1/2}<\frac{1}{\eps}(\beta z_d)^4\hat{ d}_{+,q}^{-1/2}
 <\beta^4c_{b,10}\frac{1}{\eps}(\eps^{1/3}C_q^{-5/2})^4(\eps^{2/3}C_q)^{-1/2}=\beta^4c_{
b,10}C_q^{-21/2}<0.1(c_{b,6})^{-1}, 
 \end{aligned}  
 \end{equation*}
 provided that $\beta < c_{b,8}^{-1}C_q^{21/8}$. Other terms in (\ref{induct _tm1}) are smaller and can be estimated using the bound $|\xi| <c_{t,5}\eps^{1/3}$.
 
  Thus, we have
 \begin{equation}
 \label{for_zm}
 z_m<2c_{b,7}(z_d+\frac{1}{\eps}\kappa_m z_m).
 \end{equation}  
 The second  equation (\ref {eq_upper}) implies that
 \begin{equation}
 \begin{aligned}
 &|\kappa(t) - \dK_{\eps}(\eps t)| \le \kappa_m <c_{b,11}\kappa_d+
 0.5 c_{b,12}z_m\left(\eps z_m \hat{ d}_{+,q}^{-1}+ \eps^{23/6}\hat{ d}_{+,q}^{-7/2}+z_m^3+ 
 z_m^4\hat d_{+,q}^{-1} \right. \\
 &
 \left.+z_m\eps^{17/6}\hat{ d}_{+,q}^{-3}(1+|\ln \hat{ d}_{+,q}|)+\eps^2 \hat d_{+q}^{-2} +\eps^{17/6} \hat d_{+,q}^{-1}+\eps^3 \hat d_{+,q}^{-5/2}\right).
 \end{aligned}
 \end{equation}
 One can check that if  $\beta < c_{b,13}^{-1}C_q^{3/2}$, then on the time interval $[t_q, t_m)$ we have 
 \begin{equation}
 \begin{aligned}
 \label{induct _tm2}
 &\left(\eps z_m \hat{ d}_{+,q}^{-1}+ \eps^{23/6}\hat{ d}_{+,q}^{-7/2}+z_m^3+ 
 z_m^4\hat d_{+,q}^{-1} \right. \\
 &
 \left.+z_m\eps^{17/6}\hat{ d}_{+,q}^{-3}(1+|\ln \hat{ d}_{+,q}|)+\eps^2 \hat d_{+q}^{-2} +\eps^{17/6} \hat d_{+,q}^{-1}+\eps^3 \hat d_{+,q}^{-5/2}\right)<2  \eps^{2}\hat d_{+,q}^{-2}.
 \end{aligned}
  \end{equation}
  Indeed, consider in (\ref{induct _tm2}) the largest terms, but the term $\eps^{2}d_{+,q}^{-2}$,  divided  by $\eps^{2}d_{+,q}^{-2}$:
 \begin{equation*}
 \begin{aligned} 
& \frac{\eps z_m \hat{ d}_{+,q}^{-1}}{\eps^{2}\hat d_{+,q}^{-2}}<\frac{\beta}{\eps} z_d \hat d_{+,q}
 <c_{b,14}\frac{\beta}{\eps}\eps^{1/3}C_q^{-5/2}\eps^{2/3}C_q=c_{b,14}\beta C_q^{-3/2}<0.1,\\
 &\frac{z_m^4 \hat{ d}_{+,q}^{-1}}{\eps^{2}\hat d_{+,q}^{-2}}<\frac{\beta^4}{\eps^2}z_d^4 \hat{ d}_{+,q}<c_{b,15}\frac{\beta^4}{\eps^2}\eps^{4/3}C_q^{-10}\eps^{2/3}C_q=c_{b,15}\beta^4C_q^{-9}<0.1,
 \end{aligned}
 \end{equation*}
  provided that $\beta < c_{b,13}^{-1}C_q^{3/2}$. Other terms in (\ref{induct _tm2}) are smaller and can be estimated using the bound $|\xi| <c_{t,5}\eps^{1/3}$.
  
   Thus, we have
  \begin{equation}
 \label{for_kappam}
  \kappa_m <c_{b,11}\kappa_d+c_{b,12}z_m \eps^{2}\hat d_{+,q}^{-2}.
  \end{equation} 
From (\ref{for_zm}) and (\ref{for_kappam}), we obtain
\begin{equation}
\begin{aligned}
\label{recurrent}
&|z(t)|\le z_m< c_{b,16}(z_d+\frac{1}{\eps}\kappa_m z_m),\\
&|\kappa(t) - \dK_{\eps}(\eps t)| \le \kappa_m < c_{b,17}(\kappa_d+\eps^{2}\hat{ d}_{+,q}^{-2}z_m).
\end{aligned}
\end{equation}
 Substitute  the right  inequality for $\kappa_m$   in (\ref{recurrent}) to the inequality for $z_m$. We obtain  
\begin{equation}
\label{zm_inequality}
z_m< c_{b,16}\left(z_d+\frac{1}{\eps}c_{b,17}\kappa_d z_m   + c_{b,17} \eps \hat{ d}_{+,q}^{-2}z_m^2\right),
\end{equation}
or
$$
c_{b,18}  \eps \hat{ d}_{+,q}^{-2}z_m^2 -(1-c_{b,18}\kappa_d/\eps) z_m +c_{b,16}z_d >0,
$$
where $c_{b,18}=c_{b,16}c_{b,17}$.

Consider the equality corresponding to this inequality:
$$
c_{b,18}  \eps \hat{ d}_{+,q}^{-2}\rho^2 -(1-c_{b,18}\kappa_d/\eps) \rho +c_{b,16}z_d=0.
$$
Denote by $\rho_1$ the smaller root of this equation. 
Calculate $\rho_1$:
$$
\rho_1=\frac{1}{2c_{b,18}  \eps \hat{ d}_{+,q}^{-2}}
\left( (1-c_{b,18}\kappa_d/\eps) - \sqrt {(1-c_{b,18}\kappa_d/\eps)^2-4c_{b,18}c_{b,16}\eps \hat{ d}_{+,q}^{-2} z_d }  \right).
$$
It is convenient to rewrite this as
$$
\rho_1=\frac{2c_{b,16}z_d }
{  \left( (1-c_{b,18}\kappa_d/\eps) + \sqrt {(1-c_{b,3}\kappa_d/\eps)^2-4c_{b,18}c_{b,16}\eps \hat{ d}_{+,q}^{-2} z_d }  \right) }.
$$
For sufficiently large $C_q$, we have $\rho_1<c_{b,19}z_d$. 

\medskip
We have $|z(t_q)|<\rho_1$. Assume that at some moment of time $t_*$ we have  $|z(t_*)|=\rho_1$ for the first time. Then for supremum of $|z(t)|$ on the time interval
from $t_q$ to $t_*$ we have $z_m=\rho_1$. Thus, 
$$
z_m= c_{b,16}\left(z_d+\frac{1}{\eps}c_{b,2}\kappa_d z_m   + c_{b,2} \eps\hat{ d}_{+,q}^{-2}z_m^2\right),
$$
which contradicts (\ref{zm_inequality}). Thus, $z_m<\rho_1<c_{b,19}z_d$.

Taking $\beta=c_{b,19}$, we obtain  that the assumption $z_m<\beta z_d$ is satisfied  on $\Gamma_{q, \eps}$ in $S(T)$, provided that $C_q$ is sufficiently large.

Then
$$
\kappa_m < c_{b,17}(\kappa_d+\eps^2\hat{ d}_{+,q}^{-2}\rho_1)\le
c_{b,20}(\kappa_d+\eps^2\hat{ d}_{+,q}^{-2}z_d)
\le c_{b,21} (\kappa_d+\eps^2\hat{ d}_{+,q}^{-2}\eps^2\hat{ d}_{+,q}^{-5/2})<c_{b,22} \kappa_d.
$$
Thus, on $\Gamma_{q, \eps}$ in $S(T)$,
$$|z(t)|< c_{a,3}  \eps^2 \hat d_{+,q}^{-5/2},\quad  |\kappa(t)- \dK_{\eps}(t)|<c_{a,4}\eps^4 \hat d_{+,q}^{-9/2}.$$

 \hskip 12 cm $\square$
  
 \medskip
 {\bf Proof of Lemma \ref{lem_in_Dqr}}
 
 The plan is to estimate $z(t)$, use this to estimate $\eta(t)$, and then use this to improve the estimate of $z(t)$.
        
 Each point in the time domain $D_{q,r,d}$   can be reached from $\Gamma_{q,\eps}$ by moving   vertically downward. 
 Consider downward motion  along a  vertical line $\re t = {\rm const}$,  while the relation
 \begin{equation}
 \label{induct_D_z}
 |z(t|< 2c_{a,3}\eps^2\hat {d}_{+,q}^{-5/2}
 \end{equation}
 is satisfied. Here $c_{a,3}$ is the constant from Lemma \ref {on_upper_curve_r}.
 For this motion, the  equation for $z$ in (\ref {d_equation}) takes the form
\begin{equation}
\begin{aligned}
\label{eq_z_vert_+r_d}
& \frac{d z}{ds}
 =\left(-i\Lambda_1(\dK_{\eps}(\eps t))+\alpha_2  \right) z +  \alpha_3+ \eps^3 (-i)O_1({d}_{+}^{-7/2}), \quad s= -\im t ,\\
 &\alpha_2=\mu_1 \eps^4O(\hat {d}_{+}^{-5}) +\eps O((\eps^2\hat {d}_{+,q}^{-5/2})  \hat {d}_{+}^{-3/2}) 
 +\eps O( (\eps^{17/6} \hat d_{+,q}^{-3}) \hat d_+^{-1/2})\\
 &+O((\eps^2\hat {d}_{+,q}^{-5/2})  (\eps^{17/6} \hat d_{+,q}^{-3}) \hat d_+^{-1/2})
  +O((\eps^{17/6} \hat d_{+,q}^{-3})^2 \hat d_+^{-1/2})
 + O((\eps^2\hat {d}_{+,q}^{-5/2})^4   \hat {d}_{+}^{-3/2} )+ \eps^3O(\hat {d}_{+}^{-3}), \\
 &\alpha_3=
     \eps O((  \eps^{17/6} \hat d_{+,q}^{-3})^2 \hat d_+^{-1/2})   +O((  \eps^{17/6} \hat d_{+.q}^{-3})^2 )
 + O((  \eps^{17/6} \hat d_{+,q}^{-3})^3 \hat d_+^{-1/2})+  \eps^3O((\eps^{17/6}\hat {d}_{+,q}^{-3})\hat {d}_{+}^{-3}). 
    \end{aligned}
\end{equation}
(We used (\ref{induct_D_r_d}) here.)
According to condition 6) in Section \ref{form_conditions}, each vertical line 
 crosses the curve $\Gamma_{q,\eps,d}$ and all curves  
 $\re\Psi_{\eps}={\rm const}$ transversally, if $c_{e,10}$ is sufficiently large. Thus,  $\re\Psi_{\eps}$ decreases downward along these lines, and  $\re( -i\Lambda_1(\dK_{\eps}(\eps t))<-c_{b,1}^{-1}\hat d_{+}^{1/2}$. 
 We also have
 $$
 |\alpha_2| < c_{b,2,1}C_q^{-9/2} \hat d_+^{1/2} , \ |\alpha_3|=O((  \eps^{17/6} d_{+.q}^{-3})^2 )< c_{b,2,2}\eps^{5/3}C_q^{-6}.
 $$
 This estimate for $ \alpha_2$ originates from the term $\eps O((\eps^2\hat {d}_{+,q}^{-5/2})  \hat {d}_{+}^{-3/2})$: 
 $$
 \frac{\eps^3\hat {d}_{+,q}^{-5/2}  \hat {d}_{+}^{-3/2})}{\hat d_+^{1/2}}=
 \eps^3\hat {d}_{+,q}^{-5/2}\hat {d}_{+}^{-2}=O(\eps^3\hat {d}_{+,q}^{-9/2})=O(C_q^{-9/2}).
 $$
 Other terms in $\alpha_2$ are smaller than this one.  
 
 Let us use notation $\hat d_u$ from Lemma \ref{lem_0.5_2}. Using this lemma, on the segment  of the considered line with length  less than  $c_{r,5}^{-1}\hat d_u$, and for sufficiently large $C_q$, we obtain the estimates 
  $$
 \re\left ( -i\Lambda_1(\dK_{\eps}(\eps t)  +\alpha_2   \right )<-c_{b,3}^{-1}\hat {d_{u}}^{1/2},\quad    |\alpha_3+\eps^3(-i) O_1({d}_{+}^{-7/2})|
 < c_{b,4} (\eps^{5/3}C_q^{-6}+\eps^3 \hat {d}_{u}^{-7/2}).
 $$
 
 Consider the   auxiliary equation (\ref {aux_r_1}), and apply Lemma \ref {l_aux_r_1} for the case 
\begin{equation}
\label{for_alpha_a}
  \nu= c_{b,3}^{-1}\hat {d_{u}}^{1/2}, \ \alpha_a=c_{b,4}(\eps^{5/3}C_q^{-6}+\eps^3 \hat {d}_{u}^{-7/2}), \
  u(s_0)=c_{a,3}\eps^2\hat d_{+,q}^{-5/2}.
\end{equation}
For sufficiently large $C_q$, we have $\alpha_a/\nu<u(s_0)$.  Indeed,
 \begin{equation*}
 \begin{aligned}
 &\frac{\alpha_a}{\nu u(s_0)}=O\left (\frac{\eps^{5/3}C_q^{-6}+\eps^3 \hat {d}_{u}^{-7/2}}
 {\hat {d_{u}}^{1/2}\eps^2\hat d_{+,q}^{-5/2}}\right)
 =O\left (  \frac{\eps^{5/3}C_q^{-6}}{  \eps^2\hat d_{+,q}^{-2}}+
  \frac{ \eps^3 \hat {d}_{u}^{-4}} {\eps^2\hat d_{+,q}^{-5/2}}              \right)
  =O\left (  \frac{\eps^{5/3}C_q^{-6}}{  \eps^{2/3}C_q^{-2}}+
  \frac{ \eps^3 \hat {d}_{+,q}^{-4}} {\eps^2\hat d_{+,q}^{-5/2}}              \right)\\
  &=O\left (\eps C_q^{-4} + \eps \hat d_{+,q}^{-3/2} \right)=O(C_q^{-3/2})<1.
 \end{aligned}
\end{equation*}

\medskip
Solution $u(t)$ of equation (\ref{aux_r_1}) can be used now to estimate $|z(t)|$ in (\ref{eq_z_vert_+r_d}) on the vertical line. We take $s_0$ as the value of $s$ at the intersection of the vertical line with the curve   $\Gamma_{q,\eps}$. According to Lemma
\ref {lem_compar_ODE_1}, for $s> s_0$, while condition  (\ref{induct_D_z}) is satisfied,  we have
$|z(t)|< u(s)$.
In particular, this  implies that condition (\ref{induct_D_z}), namely $|z(t|< 2c_{a,3}\eps^2\hat {d}_{+,q}^{-5/2}$, can not be violated in $D_{q,r,d}\cap S(T)$. Thus, in Lemma \ref {lem_in_Dqr} we can take $c_{r,6}=c_{a,3}$.

\medskip
Denote by $\tilde D_{q,r,d}$ the part of $D_{q,r,d}$ covered by vertical segments  of length less than  or equal to $(\eps /\nu)|\ln(\alpha_a/(\nu u(s_0))|$ 
down  from all points   $\tau_u\in \Gamma_{q,\eps}$ (later will use $c_{11}, c_{12}$ such that $\tilde D_{q,r,d}\subset  D_{q,r,d}$). Denote by $\tilde \Gamma_{q,r,d}$ the lower boundary of $\tilde D_{q,r,d}$. Denote by $\tilde \Gamma_{q,r,d}' $ the part of  $\tilde \Gamma_{q,r,d}$ consisting of points for which the lengths  of the considered vertical segments are  $(\eps/\nu)|\ln(\alpha_a/(\nu u(s_0))|$. According to Lemmas \ref {l_aux_r_1} and \ref{lem_compar_ODE_1}, in the whole $D_{q,r,d}$, we have $|z(t)|<c_{a,3}\eps^2\hat {d}_{+,q}^{-5/2}$, while on  $\tilde \Gamma_{q,r,d}' $ and in  $D_{q,r,d}\setminus \tilde D_{q,r,d}$ we have 
\begin{equation}
\label{for_z_improved_1}
|z(t)|< 2\alpha_a/ \nu<
c_{b,5}(\eps^{5/3}C_q^{-6}\hat {d}_{+}^{-1/2}+\eps^3 \hat {d}_{+}^{-4}).
\end{equation}
Obtain  estimate for $\eta(t)$ in $S(T)\cap D_{q,r,d}$. 
The  equation for $\eta $ in (\ref{with_beta_+r}) has the form
\begin{equation*}
\begin{aligned}
\dot \eta&=\left(B(\kappa) + O(|\eta|)  + O(|z|+|w| )\right)  \eta \\
&+  O(|zw|)+\eps O((|z|^2|+|w|^2|)d_+^{-1/2})\\
       &+O(|\xi|^3_*d_+^{-1/2})  + O(|z|^4)+\eps^3O_2(d_{+}^{-3}).
\end{aligned}     
\end{equation*} 
This implies that
    \begin{equation*}
\begin{aligned}
\dot \eta&=(B(\dK_{\eps}) + O(\eps^2\hat d_{+,q}^{-5/2}))  \eta \\
&+ O((\eps^2\hat d_{+,q}^{-5/2})(\eps^{17/6}  \hat  d_{+,q}^{-3}  )   
+        \eps O((\eps^2\hat d_{+,q}^{-5/2})^2 )\hat d_+^{-1/2})\\
& +O((\eps^2\hat d_{+,q}^{-5/2})^2)(\eps^{17/6}  \hat  d_{+,q}^{-3}  )d_+^{-1/2})
         + O((\eps^2\hat d_{+,q}^{-5/2})^4)+\eps^3O_2(\hat d_{+}^{-3}).
\end{aligned}     
\end{equation*} 
This in turn implies  that
\begin{equation}
\begin{aligned}
\label{for_eta_inside}
\dot \eta=(B(\dK_{\eps}) + O(\eps^{1/3}))  \eta +  \alpha_d,\  |\alpha_d|< c_{b,6}\eps^3 \hat  d_{+,q}^{-3}.
\end{aligned}     
\end{equation} 
Each point in the domain $S(T)\cap D_{q,r,d}$ can be reached from the axis $\re \tau = \re \tau_c$ by motion along a curve $\re \Psi_{\eps}={\rm const}$. Part of this curve is inside the domain  $S(T)\cap D_{q,r,d}$. For change in $\eta$ along this part, we can use equation 
(\ref {for_eta_inside}). It could be that a part of this curve is outside  $S(T)\cap D_{q,r,d}$.  At this part, according to  the equation (\ref {a_shorten_+}), we have 
  \begin{equation}
\begin{aligned}
\label{for_eta_outside}
\dot \eta=B(\dK_{\eps})  \eta +  \eps^3 \tilde O_2(d_{+}^{-3}), \   |\tilde O_2(d_{+}^{-3})|< 2c_{r,2}d_{+}^{-3}.
\end{aligned}     
\end{equation}
Thus, in both cases
   \begin{equation}
\begin{aligned}
\label{for_eta_both}
\dot \eta=(B(\dK_{\eps}) + O(\eps^{1/3}))  \eta +  \tilde \alpha_d,\  |\tilde\alpha_d|< 
c_{b,7}\eps^3 \hat  d_{+,q}^{-3}.
\end{aligned}     
\end{equation} 
\begin{lem}
\label{l_for_margin_eta}
In the part of $S(T)$ covered by the curves $\re \Psi_{\eps}={\rm const}$ passing through $D_{q,r,d}$, we have
\begin{equation}
\label{for_margin_eta}
|\eta(t)|<c_{r,7}\eps^3 \hat  d_{+,q}^{-3}.
\end{equation}
\end{lem}
We omit the proof of this lemma because it is analogous to the proof of Lemma \ref{l_right_eta} below.

\medskip
The obtained estimate for $\eta(t)$ allows for an improvement of the pervious estimates related to $z(t)$.
We now obtain  in (\ref {eq_z_vert_+r_d}), (\ref{for_alpha_a})
$$
|\alpha_3|=O((  \eps^{3} d_{+q}^{-3})(\eps^{17/6} d_{+q}^{-3})) ) < c_{b,2}\eps^{11/6}C_q^{-6}, 
\alpha_a=c_{b,4}(\eps^{11/6}C_q^{-6}+\eps^3 \hat {d}_{u}^{-7/2}).
$$
We redefine $\tilde D_{q,r,d}$, $\tilde \Gamma_{q,r,d}$ and $\tilde \Gamma_{q,r,d}'$ using this value of 
$\alpha_a$. On  $\tilde \Gamma_{q,r,d}' $ and in  $D_{q,r,d}\setminus \tilde D_{q,r,d}$, we have 
\begin{equation}
\label{for_z_improved_2}
|z(t)|< 2\alpha_a/\nu < c_{b,5}(\eps^{11/6}C_q^{-6}\hat {d}_{u}^{-1/2}+\eps^3 \hat {d}_{u}^{-4}).
\end{equation}
Now vertical distances between $\Gamma_{q,\eps}$ and $\tilde \Gamma_{q,r,d}'$ are
\begin{equation}
\begin{aligned}
\label{vert_distance_1}
&\eps \nu^{-1} |\ln \left(\frac{\alpha_a}{\nu u(s_0)}\right)|=\eps  c_{b,3}^{-1}\hat {d_{u}}^{-1/2}
|\ln\left(\frac{c_{b,3}c_{b,4}(\eps^{11/6}C_q^{-6}+\eps^3 \hat {d}_{u}^{-7/2})}
 {c_{a,3}\hat {d_{u}}^{1/2}\eps^2\hat d_{+,q}^{-5/2}} \right)|\\
 &<\eps  c_{b,3}^{-1}\hat {d_{u}}^{-1/2}
|\ln\left(\frac{c_{b,3}c_{b,4}\eps^3 \hat {d}_{u}^{-7/2}}
 {c_{a,3}\hat {d_{u}}^{1/2}\eps^2\hat d_{+,q}^{-5/2}} \right)|
 <c_{b,7}\eps \hat {d_{u}}^{-1/2}|\ln\left(c_{b,6}\eps^{8/3} \hat {d}_{u}^{-4}C_q^{5/2}\right)|\\
 &=c_{b,9}\eps \hat {d_{u}}^{-1/2}|\ln\left(c_{b,8}\eps \hat {d}_{u}^{-3/2}C_q^{15/16}\right)|.
\end{aligned}     
\end{equation}

Redefine $\tilde D_{q,r,d}$, $\tilde \Gamma_{q,r,d}$ and $\tilde \Gamma_{q,r,d}'$  using the value
$c_{e,12,2}\eps \hat {d_{u}}^{1/2}|\ln\left(c_{e,12,1}\eps \hat {d}_{u}^{-3/2}C_q^{15/16}\right)|$ 
to determine  the vertical distances between $\Gamma_{q,\eps}$ and $\tilde \Gamma_{q,r,d}'$, where $c_{e,12,1}=c_{b,8}$,  and\\$c_{e,12,2}>c_{b,9}$ is chosen so that this vertical distance on the line
$\re \tau =\re\tau_c$ is greater than the length of the segment $A_q$ (the notation $A_q$ is introduced after Lemma \ref{l_improved_I}).


Now we use this to improve the estimate for $\eta(t)$ in $D_{q,r,d}\setminus \tilde D_{q,r,d}$. Instead of estimate in  (\ref {for_eta_inside}), we now have  (see Appendix 1)
\begin{equation}
\begin{aligned}
\label {for_alpha_d_improved}
&|\alpha_d|= O(|zw|)+\eps O((|z|^2|+|w|^2|)d_+^{-1/2})+O((|w|^3+|w^2z|+|wz^2|)d_+^{-1/2}) \\
 &       + O(|z|^4)+\eps^3O_2(d_{+}^{-3})=\eps^3O(d_{+}^{-3}).
\end{aligned}     
\end{equation} 
This estimate allows an improvement of the estimate for $|\eta(t)|$ in $D_{q,r,d}\setminus \tilde D_{q,r,d}$.

\begin{lem}
\label{l_for_margin_eta_improved}
In $S(T)\cap (D_{q,r,d}\setminus \tilde D_{q,r,d})$, we have
\begin{equation}
\label{for_margin_eta_improved}
|\eta(t)|<c_{r,10}\eps^3d_{+}^{-3}.
\end{equation}
\end{lem}
This is one of estimates stated in Lemma \ref {lem_in_Dqr}.

\medskip
The  obtained estimate for $\eta(t)$ allows to improve the estimate for $z(t)$ on  the curve 
  $\Gamma_{q,r,d}'$. For equation (\ref{eq_z_vert_+r_d}) we now have 
 \begin{equation*}
\begin{aligned}
&|\alpha_3|=\eps O((  \eps^{17/6} d_{+}^{-3})^2 d_+^{-1/2})  
+O((  \eps^{3} d_{+}^{-3}) (  \eps^{17/6} d_{+}^{-3}))
+ O((  \eps^{17/6} d_{+}^{-3})^3  d_+^{-1/2})\\
&+O(\eps^3d_{+}^{-3}(  \eps^{17/6} d_{+}^{-3}))=o(\eps^3d_{+}^{-7/2}).
 \end{aligned}     
\end{equation*} 
Solutions of  equation  (\ref{eq_z_vert_+r_d}) should now be compared with solutions of equation (\ref{aux_r_1}) in which $\alpha_a=c_{b,11}\eps^3 \hat {d}_{u}^{-7/2}$. For this comparison, we take 
$s_0\in \tilde \Gamma_{q,r,d}'$ and initial condition  $u(s_0)=c_{b,12}(\eps^{11/6}C_q^{-6}\hat {d}_{u}^{-1/2}+\eps^3 \hat {d}_{u}^{-4})$.  We have
\begin{equation}
\label{sum_of_two}
\frac{\nu u(s_0)}{\alpha_a}<c_{b,3}^{-1}c_{b,11}^{-1}c_{b,12}\left(\frac{(\eps^{11/6}C_q^{-6}\hat {d}_{u}^{-1/2}+\eps^3 \hat {d}_{u}^{-4})}{\eps^3 \hat {d}_{u}^{-7/2}}{\hat {d}_{u}^{1/2}}\right ) = c_{b,3}^{-1}c_{b,11}^{-1}c_{b,12} \frac{C_q^{-6}}{\eps^{7/6}\hat {d}_{u}^{-7/2}}+ c_{b,3}^{-1}c_{b,11}^{-1}c_{b,12} .
\end{equation}
If the left-hand side of this relation is less than 1, then, according to  Lemma \ref{l_aux_r_1}, $u(s)<{\alpha_a}/\nu$ for any $s \ge s_0$. If the left-hand side of this relation is greater than or equal to 1, then, 
according to  Lemma \ref{l_aux_r_1}, $u(s)<2{\alpha_a}/\nu$ for
 $s \ge s_0+(1/\nu)|\ln(\alpha_a/(\nu u(s_0))|$. 

 If the left-hand side of (\ref{sum_of_two}) is  greater than or equal to 1, then the right-hand side should be greater than 1. In this case, of the two summands on the right-hand side, the larger one should be greater than 1/2.
 Suppose, first, that this is the first summand. 
In this case,  we have
 \begin{equation}
 \begin{aligned}
 \label{vert_distance_2}
&\ln\left( \frac{\nu u(s_0)}{\alpha_a}\right)\le \ln\left( 2c_{b,3}^{-1}c_{b,11}^{-1}c_{b,12} \frac{C_q^{-6}}{\eps^{7/6}\hat {d}_{u}^{-7/2}}\right)<c_{b,13}\ln\left(\frac{C_q^{-6}}{\eps^{7/6}\hat {d}_{u}^{-7/2}}\right)=\frac {7}{3}c_{b,13}\ln\left(\frac{C_q^{-6}}{\eps^{7/6}\hat {d}_{u}^{-7/2}}\right)^{3/7}\\
&=\frac {7}{3}c_{b,13}\ln\left(\frac{C_q^{-18/7}}{\eps^{1/2}\hat {d}_{u}^{-3/2}}\right) 
<c_{b,14}\ln\left(\frac{1}{\eps\hat {d}_{u}^{-3/2}C_q^{15/16}}\right).
 \end{aligned}
 \end{equation}
 If the second summand in (\ref{sum_of_two})  is the larger one, then we also have
 \begin{equation*}
 \begin{aligned}
\ln\left( \frac{\nu u(s_0)}{\alpha_a}\right)<c_{b,14}\ln\left(\frac{1}{\eps\hat {d}_{u}^{-3/2}C_q^{15/16}}\right).
\end{aligned}
 \end{equation*}
  According to Lemmas \ref{l_aux_r_1} and  \ref{lem_compar_ODE_1}, in all three  cases,  for
$$
s \ge s_0+c_{b,15}\hat {d}_{u}^{-1/2}|\ln(\eps\hat {d}_{u}^{-3/2}C_q^{15/16})|
$$ 
 we have 
 $$|z(s)|<2\alpha_a/ \nu=c_{b,16}\eps^3 \hat {d}_{u}^{-4}.
 $$ 
  
  At this point, we  introduce constants $c_{e,11}, c_{e,12}$. We take them such that the sum of the vertical distances given by estimates (\ref{vert_distance_1}) and (\ref{vert_distance_2}) is less than \\$c_{e, 12} \eps \hat d_u^{-1/2}|\ln( c_{e,11}^{-1}\eps \hat d_u^{-3/2}C_{q}^{15/16})|$. We then redefine the curve $\Gamma_{q,r,d}'$ so that its vertical distances from the curve  $\Gamma_{q,\eps}$ are equal to $c_{e, 12} \eps \hat d_u^{-1/2}|\ln( c_{e,11}^{-1}\eps \hat d_u^{-3/2}C_{q}^{15/16})|$.  On the curve 
  $\Gamma_{q,r,d}'$, we have $|z(t)|< c_{r,9}\eps^3d_{+}^{-4}$, as stated 
in Lemma \ref {lem_in_Dqr}.

\hskip 12cm $\square$

  \medskip
 {\bf Proof of Lemma \ref{l_down_z}.}
 
 \medskip
Each point in the domain   $S(T)\setminus (D_{q,r,d}\cup \bar D_{q,r,d})$ can be reached by moving vertically downward from the curve  $\Gamma_{q,r,d}'$. For such a motion, the equation for $z$ in (\ref {a_shorten_+}) for $\im \tau>-c_{l,1}^{-1}$ takes the form
\begin{equation}
\label{z(s)_equation}
\frac{dz}{ds} =-i  \Lambda_1(\dK_{\eps})z+ \alpha, \ |\alpha|< c_{b,1}\eps^3 \hat d_{+}^{-7/2}, \  s=-\im t.
 \end{equation}
 Let  $z(s)$ denote solution to this equation. For the  initial condition on $\Gamma_{q,r,d}'$, we have $|z(s_0)|< c_{r,12}\eps^3 \hat d_{+}^{-4}$. Condition 6) in Section \ref{form_conditions} implies that $\re( -i\Lambda_1(\dK_{\eps}(\eps t))<-c_{b,2}^{-1}\hat d_{+}^{1/2}$.
 
 The solution of equation (\ref {z(s)_equation}) can then be compared with the solution of equation
 \begin{equation}
\label{u(s)_equation}
\frac{du}{ds} =-\nu(\eps s)u+ \tilde \alpha (\eps s), \ 
\nu=-c_{b,2}^{-1}\hat d_{+}^{1/2}, \  \tilde \alpha=c_{b,1}\eps^3 \hat d_{+}^{-7/2}
 \end{equation}
 with the initial condition $u(s_0)= c_{r,12}\eps^3 \hat d_{+}^{-4}$ at $s=s_0$.
 
   On some time interval 
 $[s_0,s_*)$,  we have $|u(s)|<c_{b,3}\eps^3 \hat d_{+}^{-4}, \ c_{b,3}= \max (c_{r,12}, 3c_{b,2}c_{b,1})  $.  Let $s_*$  be the first value of $s$ for which this inequality is not satisfied.   At $s=s_*$,  we have
$$
 \frac{du}{ds}=-c_{b,2}^{-1}\hat d_{+}^{1/2} c_{b,3} \eps^3 \hat d_{+}^{-4}+c_{b,1}\eps^3 \hat d_{+}^{-7/2}< -c_{b,2}^{-1}\hat d_{+}^{1/2} 2c_{b,2}c_{b,1} \eps^3 \hat d_{+}^{-4}+c_{b,1}\eps^3 \hat d_{+}^{-7/2}=-c_{b,1}\eps^3 \hat d_{+}^{-7/2}.
$$
At this value of $s$,
$$
\left|\frac{d}{ds}\left(c_{b,3}\eps^3 \hat d_{+}^{-4} \right)\right  |<c_{b,4}\eps^4 \hat d_{+}^{-5}< c_{b,1}\eps^3 \hat d_{+}^{-7/2},
$$
if $c_{e,10}$ is sufficiently large. Thus, $u(s)-c_{b,3}\eps^3 \hat d_{+}^{-4}$ decays at $s=s_*$ and, therefore, is positive just before $s_*$. This contradicts the definition of $s_*$. Therefore, the inequality $|u(s)|<c_{b,3}\eps^3 \hat d_{+}^{-4}$ is satisfied at least while $\im \tau>-c_{l,1}^{-1}$. Comparing solutions of equations (\ref {z(s)_equation}) and (\ref {u(s)_equation}), according to Lemma  \ref {lem_compar_ODE_1}  we obtain that the inequality $|z(s)|<c_{b,3}\eps^3 \hat d_{+}^{-4}$ holds  at least while $\im \tau>-c_{l,1}^{-1}$. In a similar way, we obtain that the inequality
$|z(s)|<c_{b,5}\eps^3 \hat d_{-}^{-3}$ holds  for $\im \tau<c_{l,1}^{-1}$ up to reaching the boundary of the domain 
$\bar D_{q,r,d}$. 

\hskip 12cm $\square$

{\bf Proof of Lemma \ref{l_right_eta}.}

\medskip
Each point in the domain  $(S(T)\setminus (D_{q,r,d}\cup \bar D_{q,r,d}) \cap \{ \im \tau\ge -c_{l,1}^{-1}\}$ can be reached from the axis $\re \tau = \re \tau_c$  by moving along a curve $\re \Psi_{\eps}={\rm const}$. Some portions  of this curve could pass through the domain $D_{q,r,d}$. We replace these portions by segments of the curve  $\Gamma_{q,r,d}'$ having the same endpoints. By analyticity, there can be only a finite number of such segments. Thus,  each  point in the domain  $(S(T)\setminus (D_{q,r,d}\cup \bar D_{q,r,d}) \cap \{ \im \tau\ge -c_{l,1}^{-1}\}$
can be reached from the axis $\re \tau = \re \tau_c$ by moving along  a curve obtained by gluing together pieces of the curves
$\re \Psi_{\eps}={\rm const}$ and  $\Gamma_{q,r,d}'$.

Consider, in the domain  $(S(T)\setminus (D_{q,r,d}\cup \bar D_{q,r,d}) \cap \{ \im \tau\ge -c_{l,1}^{-1}\}$, a segment of a level curve $\re \Psi_{\eps}={\rm const}$. 
Let the left endpoint of this segment lies on the line $\re \tau=\re \tau_c$.  The right endpoint lies either on $\Gamma_{q,r,d}'$, on the line $ \im \tau=  -c_{l,1}^{-1}$,   or on the line $  \re t=T$. 
Introduce the arc-length parameter $\sigma $ along this curve
 as a new time  variable ($\sigma=0$ at   $\re \tau=\re \tau_c$). Thus, $t=t(\sigma), \ |dt/d\sigma|=1$. The equation for $\eta$  on the curve $\re \Psi_{\eps}={\rm const}$ takes the form
\begin{equation}
\label{e_with_sigma_1}
\frac{d \eta}{d\sigma}=\frac{dt}{d\sigma}B(\dK_{\eps})\eta + \alpha, \ \alpha= \eps^3 \frac{dt}{d\sigma}\tilde O_2( d_+^{-3}), \
 |\alpha|< 2 c_{r,2}\eps^3d_{+}^{-3}.
\end{equation}
According to condition  5) in Section \ref {form_conditions}, all eigenvalues of the matrix $({dt}/{d\sigma})B(\dK_{\eps})$ have negative real parts.

According to \cite{bellman}, Sect 13, the homogeneous system with frozen $\tau$  has a quadratic Lyapunov function $W(\eta)$ whose $\sigma$-derivative for frozen $\tau$  is $ -( \eta\cdot \bar \eta)$.
We have\\ $c_{b,1}^{-1} |\eta|^2 \le W(\eta)\le c_{b,1} |\eta|^2$.   The derivative of $W$ in the original system is 
$$
\frac{dW}{d\sigma}=-(\eta\cdot\bar \eta)+O(\eps\hat {d}_{+}^{-1/2}) W_1(\eta) +W_2(\alpha, \eta) 
 $$
 with a hermitian quadratic form $W_1$ and a hermitian bilinear form $W_2$.  
  
   We have $|W_2(\alpha, \eta)|\le c_{b,2}|\alpha||\eta|\le c_{b,3}\eps^3\hat {d}_{+}^{-3}\sqrt{W}$. This  implies that
 $$
 \frac{dW}{d\sigma} \le - c_{b,4}^{-1} W+  c_{b,5}\eps^3\hat {d}_{+}^{-3}\sqrt{W}.
 $$
 Denote $v=\sqrt{W}$. Then
 $$
 \frac{dv}{d\sigma} \le - c_{b,6}^{-1} v+  c_{b,7}\eps^3\hat {d}_{+}^{-3}.
 $$
 For the starting point on the axis $\re \tau= \re \tau_c$ we  have $v< c_{b,8}\eps^3\hat {d}_{+}^{-3}$. Take \\
 $c_{b,9}=\max\{2c_{b,8}, 3c_{b,6}c_{b,7} \}$.
 Consider the auxiliary  equation
 $$
 \frac{d\tilde v}{d\sigma} =- c_{b,6}^{-1} \tilde v+  c_{b,7}\eps^3\hat {d}_{+}^{-3}
 $$
 with the initial condition $ \tilde v=c_{b,8}\eps^3\hat {d}_{+}^{-3}$ at $\sigma=0$.  Then, on the whole considered interval of $\sigma$, we have $ \tilde v < c_{b,9}\eps^3\hat {d}_{+}^{-3}$. Indeed, assume that there exists $\sigma_*$ such that $ \tilde v = c_{b,9}\eps^3\hat {d}_{+}^{-3}$ for the first time. At $\sigma= \sigma_*$,  we have
 $$
 \frac{d}{d\sigma}(\tilde v-c_{b,9}\eps^3\hat {d}_{+}^{-3})\le- c_{b,6}^{-1} 2c_{b,6}c_{b,7}\eps^3\hat {d}_{+}^{-3}+  c_{b,7}\eps^3\hat {d}_{+}^{-3}+
 O(\eps^4\hat {d}_{+}^{-4} ) <0,
 $$
 which contradicts definition of $\sigma_*$. We also have $v<\tilde v$, which implies that\\
 $\eta(t)< c_{b,10}\eps^3\hat {d}_{+}^{-3}$.

 If the right endpoint of the considered segment of the curve $\re \Psi_{\eps}={\rm const}$ lies on $\Gamma_{q,r,d}'$, we introduce $\sigma$ as the arc length parameter along  $\Gamma_{q,r,d}'$. We again obtain an equation of the form (\ref{e_with_sigma_1}). If $C_q$ is sufficiently large, then all eigenvalues of matrix $({dt}/{d\sigma})B$ have negative real parts. Indeed, according to Lemma \ref{lem_closeness},  for   sufficiently large  $C_q$, 
 the tangent directions of the curves $\Gamma_{q,r,d}'$ and $\re \Psi_{\eps}={\rm const}$ are close to each other at each point of their intersection.
  Thus, we can treat a segment of the curve $\Gamma_{q,r,d}'$ in exactly the same way as a segment of the curve $\re \Psi_{\eps}={\rm const}$.  In this way, we obtain  the result of  the lemma for $\im \tau>-c_{l,1}^{-1}$.
 
 \medskip
 The property $\eta(\bar t) =\bar\eta(t)$ implies the conclusion  of the lemma for $\im \tau<c_{l,1}^{-1}$.
 
 \hskip 12cm $\square$ 
 
  \medskip
{\bf Proof of Lemma \ref{l_for_kappa_rup}.}

\medskip
For $\re\tau =\re \tau_c, \ \im\tau>\-c_{l,1}^{-1}$, according to Lemma \ref{lem_cont_D_q_l},   we have 
$$
|\kappa(t)-\dK_{\eps}(\eps t)|<  c_{e,6}\eps^4 \hat{ d}_{+,q}^{-9/2}. 
$$
According to (\ref{d_equation}),
\begin{equation*}
\begin{aligned}
& \dot {\kappa}=\eps F(\kappa) +\eps^2 O(|z|^2d_+^{-2})+\eps^2 O(|\xi|_*^2d_+^{-3/2}) \\
 &
 +\eps O\left(|z|^4+|z|^5d_+^{-2}+ |\xi|_*^3d_+^{-1}\right)
 +\eps^3 O(|z|d_+^{-3} + |\xi|_*d_+^{-3/2})+\eps^4O(d_+^{-7/2}|\xi|).
\end{aligned}
\end{equation*}
According to the previous lemmas, in the domain covered by curves $\re \Psi_{\eps}={\rm const}$ passing through $D_{q,r,d}$,  this reduces to
\begin{equation*}
\begin{aligned}
&\dot\kappa= \eps F(\kappa) +\eps^2 O((\eps^2 \hat d_{+,q}^{-5/2})^2d_+^{-2})
+\eps^2 O((\eps^{17/6}\hat d_{+,q}^{-3})(\eps^2 \hat d_{+,q}^{-5/2})d_+^{-3/2}) \\
 &
 +\eps O\left((\eps^2 \hat d_{+,q}^{-5/2})^4+(\eps^2 \hat d_{+,q}^{-5/2})^5 d_+^{-2}+ (\eps^2 \hat d_{+,q}^{-5/2})^2(\eps^{17/6}\hat d_{+,q}^{-3})d_+^{-1}\right)\\
 &+\eps^3 O((\eps^2 \hat d_{+,q}^{-5/2})d_+^{-3} + (\eps^{17/6}\hat d_{+,q}^{-3})d_+^{-3/2})+\eps^4O((\eps^2 \hat d_{+,q}^{-5/2})d_+^{-7/2})\\
&=\eps F(\kappa)
+O(\eps^6 \hat d_{+,q}^{-5} \hat d_{+}^{-2}
+\eps^{41/6} \hat d_{+,q}^{-11/2} \hat d_{+}^{-3/2}\\
&+\eps^9 \hat d_{+,q}^{-10} 
+\eps^{11} \hat d_{+,q}^{-25/2} \hat d_{+}^{-2}
+\eps^{47/6} \hat d_{+,q}^{-8} \hat d_{+}^{-1}
+\eps^5 \hat d_{+,q}^{-5/2} \hat d_{+}^{-3}
+\eps^{35/6} \hat d_{+,q}^{-3} \hat d_{+}^{-3/2}
+\eps^6 \hat d_{+,q}^{-5/2} \hat d_{+}^{-7/2}).
\end{aligned}
\end{equation*}
(Here we use the fact that we  know the estimate for $w$ outside $D_{q,r,d}$ due to Corollary \ref{c_down_z}.)
This gives
\begin{equation*}
\begin{aligned}
&|\kappa(t)-\dK_{\eps}(\eps t)|=O(\eps^4 \hat{ d}_{+,q}^{-9/2}+\eps^5 \hat{ d}_{+,q}^{-6} +\eps^{35/6}\hat{ d}_{+,q}^{-6}\\
&+\eps^8 \hat{ d}_{+,q}^{-10}
+\eps^{10} \hat{ d}_{+,q}^{-27/2}
+\eps^{41/6} \hat{ d}_{+,q}^{-8}(1+|\ln \hat{ d}_{+,q}|)
+\eps^{4} \hat{ d}_{+,q}^{-9/2}
+\eps^{29/6} \hat{ d}_{+,q}^{-7/2}
+\eps^{5} \hat{ d}_{+,q}^{-5})=O(\eps^4 \hat{ d}_{+,q}^{-9/2}).
\end{aligned}
\end{equation*}
 \hskip 12cm $\square$

{\bf Proof of Lemma \ref{l_down_down_z}.}

According to Lemmas \ref{l_down_z} and \ref{l_right_eta}, on the upper boundary of   domain    $\bar D_{q,r,d}\setminus  D_{q,r,d}$ we have 
$$
|z(t)< c_{r,16}\eps^{3}d_{-}^3, \  |\eta(t)|< c_{r,16}\eps^{3}d_{-}^3, \ |\kappa(t)-\dK_{\eps}(\eps t)|<  c_{r,17}\eps^4 \hat{ d}_{-}^{-9/2}   
$$ 

According (\ref{d_equation1}), in  $S(T) \cap \bar D_{q,r,d}\setminus  D_{q,r,d}$  we have
\begin{equation}
\begin{aligned}
\dot z&=  \Lambda_1(\kappa)z+  O(|\eta|(|\eta|+|w|)
+\eps O(|z||w| d_-^{-3/2})  + \eps O(|\xi|^2d_{-}^{-1/2})
     +O(|\xi|_{**}^3d_{-}^{-1/2})+ O(|w|^3) \\
     &+\eps^3O(d_{-}^{-3}|\xi|)
     +\eps^3O^*(d_{-}^{-3}).
\end{aligned}
\end{equation}
Here  $ |\xi|_{**}^3=|\xi|^3-|w|^3$.
This implies that
$$
\dot z= (\Lambda_1(\kappa)+\beta)z+\alpha,
$$
where
\begin{equation}
\begin{aligned}
&\beta=O(\eps (\eps^2 d_{-,q}^{-5/2})d_-^{-3/2})
+O(\eps(\eps^2 d_{-,q}^{-5/2})d_{-}^{-1/2})
+O(( \eps^{2}d_{-,q}^{-5/2})^2 d_{-}^{-1/2}) +O(\eps^3d_{-}^{-3}),\\
 &  \alpha=O(( \eps^{2}d_{-,q}^{-5/2})(\eps^{17/6}d_{-,q}^{-3}))+O(\eps (\eps^2 d_{-,q}^{-5/2})^2 d_{-}^{-1/2})
 +O((\eps^{2}d_{-,q}^{-5/2})^2   (\eps^{17/6}d_{-,q}^{-3})d_{-}^{-1/2})\\
 &+O( (\eps^{2}d_{-,q}^{-5/2})^3 +\eps^3O(d_{-}^{-3}).
\end{aligned}   
    \end{equation}
    This gives
    $$
    |\beta| < c_{b,1}\eps^{1/3}, \quad |\alpha|<c_{b,2}\eps^3d_{-,q}^{-3}.
    $$
    Lemma \ref{l_for_kappa_rup} implies that
    $$
    \Lambda_1(\kappa(t))=\Lambda_1(\dK_{\eps}(\eps t))+O(\eps^4 \hat{ d}_{-,q}^{-5})=\Lambda_1(\dK_{\eps}(\eps t))+O(\eps^{2/3}).
    $$
    in $\bar D_{q,r,d}$.
   
   Considering downward motion along the lines $\re t ={\rm const}$, similarly to proof of Lemma
   \ref{l_down_z}, we obtain $|z|< c_{r,7}\eps^3 d_{-,q}^{-3}$ in  the domain  $ \bar D_{q,r,d}\setminus  D_{q,r,d}$.
   \medskip
   
   In the domain  $\bar D_{q,r,d}\setminus \overline {\tilde D}_{q,r,d}$ one can obtain better estimate using the bounds
      $$
   |w(t)|<c_{r,9} (\eps^{11/6}C_q^{-6}\hat {d}_{-}^{-1/2}+\eps^3 \hat {d}_{-}^{-4}),\  |\eta(t)|<  c_{r,10}\eps^{3}\hat d_{+}^{-3}
   $$ 
   valid there. This gives (see Appendix 1)
   \begin{equation}
    \label{for_alpha_improved}
   |\alpha|=O(\eps^3\hat {d}_{-}^{-3}),
   \end{equation}
   which implies that
   $$
   |z(t)|< c_{r,13}\eps^3 d_{-}^{-3}.
   $$

   \hskip 12cm $\square$ 
   
   \medskip
   {\bf Proof of Lemma \ref{lem_improved_kappa}.}
   
   \medskip
   According to (\ref{d_equation}),  in the domain $D_{q,r,d}$ we have
   \begin{equation}
   \begin{aligned}
   \label{e_dot_kappa}
   & \dot {\kappa}=\eps F(\kappa) +\eps^2 O(|z|^2d_+^{-2})+\eps^2 O(|\xi|_*^2d_+^{-3/2}) \\
 &
 +\eps O\left(|z|^4+|z|^5d_+^{-2}+ |\xi|_*^3d_+^{-1}\right)
 +\eps^3 O(|z|d_+^{-3} + |\xi|_*d_+^{-3/2})+\eps^4O(d_+^{-7/2}|\xi|).\\
 %
 \end{aligned}
   \end{equation} 
   On  $\Gamma_{q,r,d}'$, we have
   \begin{equation}
   \label{e_kappa}
    |\kappa(t)-\dK_{\eps}(\eps t)|<  c_{r,17}\eps^4 \hat{ d}_{+}^{-9/2}. 
   \end{equation}
   We estimate $|\kappa(t)-\dK_{\eps}(\eps t)|$ in $D_{q,r,d}'\setminus \tilde D_{q,r,d}$ by moving vertically upward from $\Gamma_{q,r,d}'$. 
   In $D_{q,r,d}'\setminus \tilde D_{q,r,d}$, we have the estimates
   $$
   |z(t)|< c_{r,9} (\eps^{11/6}C_q^{-6}\hat {d}_{+}^{-1/2}+\eps^3 \hat {d}_{+}^{-4}),\
  |\eta(t)|<  \eps^{17/6}\hat d_{+}^{-3}, \
  |w(t)|<\eps^{17/6}\hat  d_{+}^{-3}.
   $$
 One can check that
   $$
   (\eps^{11/6}C_q^{-6}\hat {d}_{+}^{-1/2}+\eps^3 \hat {d}_{+}^{-4})>\eps^{17/6}\hat  d_{+}^{-3}
   $$
   if $\eps$ is sufficiently small. 
   
 The vertical width of  the domain $D_{q,r,d}'\setminus \tilde D_{q,r,d}$ is less than $c_{e, 12} \eps \hat d_u^{-1/2}|\ln( c_{e,11}^{-1}\eps \hat d_u^{-3/2}C_{q}^{15/16})|$. 
   We know that 
 $\partial F(\kappa)/\partial \kappa= O( \hat{ d}_{+}^{-1/2})$.
  According to Lemmas \ref{lem_0.5_2}, \ref{lem_about_width}, on the vertical line, $d_u$ can be replaced with  $\hat { d}_{+}$ in $O(\cdot)$-estimates.  
  Then (\ref{e_dot_kappa}), (\ref{e_kappa}) for $\tau\in D'_{q,r,d}\setminus\tilde  D_{q,r,d}$ imply 
    \begin{equation*}
    \begin{aligned}
    &|\kappa(t)-\dK_{\eps}(\eps t)|=O(\eps^4 \hat{ d}_{+}^{-9/2})
    + \left (\eps \hat d_{+}^{-1/2}|\ln( c_{e,11}^{-1}\eps \hat d_{+}^{-3/2}C_{q}^{15/16})|\right)\\
    &\cdot O\left( \eps(\eps^{11/6}C_q^{-6}\hat {d}_{+}^{-1/2}+\eps^3 \hat {d}_{+}^{-4})^2\hat d_+^{-2}
    +\eps(\eps^{11/6}C_q^{-6}\hat {d}_{+}^{-1/2}+\eps^3 \hat {d}_{+}^{-4})(\eps^{17/6}\hat{ d}_{+}^{-3})\hat{ d}_{+}^{-3/2}\right.\\
    &+(\eps^{11/6}C_q^{-6}\hat {d}_{+}^{-1/2}+\eps^3 \hat {d}_{+}^{-4})^4
    +(\eps^{11/6}C_q^{-6}\hat {d}_{+}^{-1/2}+\eps^3 \hat {d}_{+}^{-4})^5\ \hat{ d}_{+}^{-2}\\
    &+(\eps^{11/6}C_q^{-6}\hat {d}_{+}^{-1/2}+\eps^3 \hat {d}_{+}^{-4})^2(\eps^{17/6}\hat{ d}_{+}^{-3})\hat{ d}_{+}^{-1}\\
    &+ \eps^2(\eps^{11/6}C_q^{-6}\hat {d}_{+}^{-1/2}+\eps^3 \hat {d}_{+}^{-4}) \hat{ d}_{+}^{-3}
    +\eps^2(\eps^{17/6}\hat{ d}_{+}^{-3}) \hat d_+^{-3/2}
    +\left.\eps^3(\eps^{11/6}C_q^{-6}\hat {d}_{+}^{-1/2}+\eps^3 \hat {d}_{+}^{-4})\hat d_+^{-7/2}\right).
    \end{aligned}
    \end{equation*}
    The largest terms in this expression arise from  the terms $\eps^2 O(|z|^2d_+^{-2})$, $\eps O(|z|^5d_+^{-2})$,  and $\eps^3 O(|z|d_+^{-3})$ in (\ref{e_dot_kappa}).
    We estimate the ratios of the contributions of these terms to   $|\kappa(t)-\dK_{\eps}(\eps t)|$  relative  to  $\eps^4 \hat{ d}_{+}^{-9/2}$:
     \begin{equation*}
     \begin{aligned}
     &\frac{\left(\eps \hat d_{+}^{-1/2}|\ln( c_{e,11}^{-1}\eps \hat d_{+}^{-3/2}C_{q}^{15/16})|\right)\eps
     \left(\eps^{11/6}C_q^{-6}\hat {d}_{+}^{-1/2}+\eps^3 \hat {d}_{+}^{-4}\right)^2 \hat{ d}_{+}^{-2}}{\eps^4 \hat{ d}_{+}^{-9/2}}\\
     &=O\left( \frac{\left(\eps \hat d_{+}^{-1/2}|\ln(\eps \hat d_{+}^{-3/2}C_{q}^{15/16})|\right)\eps
     \left(\eps^{11/3}C_q^{-12}\hat {d}_{+}^{-1}+\eps^6 \hat {d}_{+}^{-8}\right) \hat{ d}_{+}^{-2}}{\eps^4 \hat{ d}_{+}^{-9/2}}\right)\\
     &=o(1)+O\left( \frac{\left(\eps \hat d_{+}^{-1/2}|\ln(\eps \hat d_{+}^{-3/2}C_{q}^{15/16})|\right)\eps
     \left(\eps^6 \hat {d}_{+}^{-8}\right) \hat{ d}_{+}^{-2}}{\eps^4 \hat{ d}_{+}^{-9/2}}\right)\\
     &=o(1)+O \left(\eps^4 \hat d_{+}^{-6}|\ln(\eps \hat d_{+}^{-3/2}C_{q}^{15/16})|\right)
     =o(1)+ O \left(\eps^4 \hat d_{+}^{-6}|\ln(\eps^4 \hat d_{+}^{-6}C_{q}^{15/4})|\right)\\
     &=O(C_{q}^{-6}|\ln(C_{q})|)=O(1).
      \end{aligned}
\end{equation*}
\medskip

\begin{equation*}
\begin{aligned}
&\frac{\left(\eps \hat d_{+}^{-1/2}|\ln( c_{e,11}^{-1}\eps \hat d_{+}^{-3/2}C_{q}^{15/16})|\right)
     \left(\eps^{11/6}C_q^{-6}\hat {d}_{+}^{-1/2}+\eps^3 \hat {d}_{+}^{-4}\right)^5 \hat{ d}_{+}^{-2}}{\eps^4 \hat{ d}_{+}^{-9/2}}\\
     &=O\left(\frac{\left(\eps \hat d_{+}^{-1/2}|\ln(\eps \hat d_{+}^{-3/2}C_{q}^{15/16})|\right)
     \left(\eps^{55/6}C_q^{-30}\hat {d}_{+}^{-5/2}+\eps^{15} \hat {d}_{+}^{-20}\right) \hat{ d}_{+}^{-2}}{\eps^4 \hat{ d}_{+}^{-9/2}}\right)\\
      &=o(1)+O\left(\frac{\left(\eps \hat d_{+}^{-1/2}|\ln(\eps \hat d_{+}^{-3/2}C_{q}^{15/16})|\right)
     \left(\eps^{15} \hat {d}_{+}^{-20}\right) \hat{ d}_{+}^{-2}}{\eps^4 \hat{ d}_{+}^{-9/2}}\right)\\
     &=o(1)+O\left(\eps^{12} \hat{ d}_{+}^{-18}|\ln(\eps \hat d_{+}^{-3/2}C_{q}^{15/16})|\right)=o(1)+O\left(\eps^{12} \hat{ d}_{+}^{-18}|\ln(\eps^{12} \hat d_{+}^{-18}C_{q}^{20})|\right)\\
     &=
     O(C_{q}^{-18}|\ln(C_{q})|)=O(1).
\end{aligned}
\end{equation*}

\begin{equation*}
     \begin{aligned}
     &\frac{\left(\eps \hat d_{+}^{-1/2}|\ln( c_{e,11}^{-1}\eps \hat d_{+}^{-3/2}C_{q}^{15/16})|\right)\eps^2
     \left(\eps^{11/6}C_q^{-6}\hat {d}_{+}^{-1/2}+\eps^3 \hat {d}_{+}^{-4}\right) \hat{ d}_{+}^{-3}}{\eps^4 \hat{ d}_{+}^{-9/2}}\\
     &=o(1)+O\left( \frac{\left(\eps \hat d_{+}^{-1/2}|\ln(\eps \hat d_{+}^{-3/2}C_{q}^{15/16})|\right)\eps^2
     \left(\eps^3 \hat {d}_{+}^{-4}\right) \hat{ d}_{+}^{-3}}{\eps^4 \hat{ d}_{+}^{-9/2}}\right)\\
     &=o(1)+O \left(\eps^2 \hat d_{+}^{-3}|\ln(\eps \hat d_{+}^{-3/2}C_{q}^{15/16})|\right)
     =o(1)+O \left(\eps^2 \hat d_{+}^{-3}|\ln(\eps^2 \hat d_{+}^{-3}C_{q}^{15/8})|\right)\\
     &=O(C_{q}^{-3}|\ln(C_{q})|)=O(1).
      \end{aligned}
\end{equation*}
Contributions of the remaining  terms are smaller; their ratios to  $\eps^4 \hat{ d}_{+}^{-9/2}$  tends to 0  as $\eps \to 0$. Thus,
    \begin{equation*}
    |\kappa(t)-\dK_{\eps}(\eps t)|<c_{r,11} \eps^4 \hat{ d}_{+}^{-9/2}.
       \end{equation*}
        \hskip 12cm $\square$ 
        

  \medskip

%% file: delay_proofs_sect_15.tex
\section{Proof for Lemma from Section 15}
\label{proof_sect_15}
\medskip

{\bf Proof of Lemma \ref{l_for_margin_eta_improved}.}

Each point in the domain $S(T)\cap (D_{q,r,d}\setminus \tilde D_{q,r,d} )$ can be reached from the axis $\re \tau = \re \tau_c$ by moving along a curve glued from pieces of the curves  $\re \Psi_{\eps}={\rm const}$,  $\tilde\Gamma'= \tilde\Gamma'_{q,r,d}$, and  $\Gamma' = \Gamma'_{q,r,d}$. According to condition 6) in Section \ref{form_conditions}, the linearised near the equilibrium fast system,  considered along any curve   $\re \Psi_{\eps}={\rm const}$, has $n-2$ eigenvalues with negative real parts, corresponding to variables $\eta$. The same is valid for this system considered along the curves $\tilde\Gamma'$ and $\Gamma'$, if $C_q$ is sufficiently large.
This follows because, according to Lemma \ref{lem_closeness}, by choosing $C_q$
sufficiently large, the tangent directions of the curve  $\tilde\Gamma' $ (respectively, $\Gamma' $) and of the curve $\re \Psi_{\eps}={\rm const}$  passing through the same point can be made arbitrarily close. On this basis, the result of Lemma \ref{l_for_margin_eta_improved} can be obtained using the Lyapunov function, in the same way as in the proof of Lemma \ref{l_right_eta}. We omit the details.

\hskip 12cm $\square$

%% file: appendix_1.tex
\section{Appendix 1. Auxiliary estimates.}
  \begin{lem}
 \label{l_aux_r_1}
 Consider the  linear ODE with constant coefficients  for a real variable $u$ with real independent variable $s$:
 \begin{equation}
 \label{aux_r_1}
 \frac{d u}{ds}=-\nu u +\alpha_a,\  \nu>0,\ \alpha_a>0.
 \end{equation}
 Consider the solution $u(s)$ of this equation with the initial condition $u(s_0)>0$. Then the following estimates hold.

(a) If $u(s_0)<\alpha_a/\nu$, then $u(s)<\alpha_a/\nu$ for  $s\ge s_0$.\\
(b)  If $u(s_0)\ge\alpha_a/\nu$, then $u(s)\le u(s_0)$ for $s\ge s_0$. 
  For $s \ge s_0+(1/\nu)|\ln(\alpha_a/(\nu u(s_0))|$, we have $u(s)<2\alpha_a/ \nu$.
\end{lem}
\begin{lem}
\label{lem_compar_ODE_1}
Consider  the following linear ODEs for a complex variable $z$ and a real variable $u$, with real independent variable $s \in [s_0, s_1]$:
 \begin{equation}
 \begin{aligned}
 \label{two_equations}
&  \frac{dz}{ds}= -\nu_1z +\alpha_{a,1},\\
&\frac{d u}{ds}=-\nu u +\alpha_a,\  \nu>0, \alpha_a>0 .
 \end{aligned}
 \end{equation}
 Here  $\nu_1$ and $\alpha_{a,1}$ are complex-valued functions, while  $\nu$ and $\alpha_a$ are real-valued functions such that  $\re \nu_1>\nu,\,  |\alpha_{a,1}|<\alpha_{a}$. Let $z(s), u(s)$ be solutions of these equations such that $|z(s_0)|\le u(s_0)$.  Then  $|z(s)|< u(s)$ for $s\in (s_0, s_1]$.
 \end{lem}

 {\bf Proof of Lemma \ref{l_aux_r_1}.}
  
  \medskip
  The solution to the equation  (\ref{aux_r_1}) is
  $$
  u(s)=\frac{\alpha_a}{\nu}+e^{-\nu(s-s_0)}\left(u(s_0)-\frac{\alpha_a}{\nu}\right).
  $$
  Thus,
  $$
  u(s)-u(s_0)=\left(\frac{\alpha_a}{\nu}-u(s_0)\right) \left(1- e^{-\nu(s-s_0)}\right).
  $$
  
  (a) For $u(s_0)<{\alpha_a}/{\nu}$, this implies that $u(s)<{\alpha_a}/{\nu}$.
  
 (b) For $u(s_0)\ge{\alpha_a}/{\nu}$, this implies that $u(s)\le u(s_0)$.
 Additionally, for  $s \ge s_0+(1/\nu)|\ln(\alpha_a/(\nu u(s_0))|$, we have
\begin{equation*}
u(s)< \frac{\alpha_a}{\nu}+\frac{\alpha_a}{\nu u(s_0)}u(s_0)= 2\frac{\alpha_a}{\nu}.
 \end {equation*}

 \hskip 10cm $\square$ 
 
 {\bf Proof of Lemma \ref{lem_compar_ODE_1}.}
 
 For $|z|\ne 0$, we have
 \begin{equation*}
  \frac{d|z|}{ds}= -\re(\nu_1)|z| +\frac{1}{2|z|}(\alpha_{a,1}\bar z+\bar \alpha_{a,1}z)\le
   -\re(\nu_1)|z| +|\alpha_{a,1}|.
\end{equation*}
The proof then follows the standard argument for the comparison lemma for solutions of one-dimensional real ODEs.  For $s>s_0$ sufficiently close to $s_0$, we have $|z(s)|<u(s)$. Let $s_*$ be the first  time such that $|z(s)|=u(s)$. At this  time we must have $d|z|/ds \ge du/ds$.  However, this can't occur, because  $|z(s_*)|=u(s_*)$ implies $d|z|/ds < du/ds$.

 \hskip 10cm $\square$ 
 
 {\bf Proof of estimate for $\alpha_d$ in (\ref{for_alpha_d_improved}).}
 
 We have
 \begin{equation*}
\begin{aligned}
&|\alpha_d|= O(|zw|)+\eps O((|z|^2|+|w|^2|)d_+^{-1/2})+O((|w|^3+|w^2z|+|wz^2|)d_+^{-1/2}) \\
 &       + O(|z|^4)+\eps^3O_2(d_{+}^{-3}),
\end{aligned}     
\end{equation*} 
and, according to (\ref {for_z_improved_2}) and (\ref{induct_D_r_d}),
$$
|z(t)|< c_{b,5}(\eps^{11/6}C_q^{-6}{d}_{+}^{-1/2}+\eps^3 \hat {d}_{+}^{-4}),\
 |w|<\eps^{17/6}{d}_{+}^{-3}.
$$
We should prove that $|\alpha_d|=\eps^3O(d_{+}^{-3})$.

\medskip
The estimates can be carried  out separately for $|z| \sim |z_a|=O(\eps^{11/6}C_q^{-6}{d}_{+}^{-1/2})$ and \\  $|z| \sim |z_b|=O(\eps^3 {d}_{+}^{-4})$.

\medskip
\underline {For $|z| \sim |z_a|=O(\eps^{11/6}C_q^{-6}{d}_{+}^{-1/2})$ .}
\medskip
\begin{equation*}
\begin{aligned}
&\frac{ |z_aw|  }{\eps^3d_{+}^{-3}}=O\left( \frac{ (\eps^{11/6}C_q^{-6}{d}_{+}^{-1/2}) ( \eps^{17/6}{d}_{+}^{-3}) }{\eps^3d_{+}^{-3}}\right)
=O\left(\eps^{5/3} {d}_{+}^{-1/2} C_q^{-6}  \right)=O\left(\eps^{4/3} C_q^{-13/2}  \right)\ll 1,\\
&\frac{ \eps |z_a^2| d_+^{-1/2} }{\eps^3d_{+}^{-3}}=O\left( \frac{\eps (\eps^{11/6}C_q^{-6}{d}_{+}^{-1/2})^2 {d}_{+}^{-1/2} }{\eps^3d_{+}^{-3}}\right)=O\left(\eps^{5/3}d_{+}^{3/2}C_q^{-12}\right)=O\left(\eps^{5/3}C_q^{-12}\right)\ll 1,\\
&\frac{ \eps |w^2| d_+^{-1/2} }{\eps^3d_{+}^{-3}}=O\left( \frac{\eps ( \eps^{17/6}{d}_{+}^{-3})^2 {d}_{+}^{-1/2} }{\eps^3d_{+}^{-3}}\right)
=O\left(\eps^{11/3} d_{+}^{-7/2}\right)=O\left(\eps^{4/3} C_q^{-7/2}\right)\ll1,\\
&\frac{ |w^3| d_+^{-1/2} }{\eps^3d_{+}^{-3}}=O\left( \frac{ ( \eps^{17/6}{d}_{+}^{-3})^3 {d}_{+}^{-1/2} }{\eps^3d_{+}^{-3}}\right)=O\left( \eps^{11/2}{d}_{+}^{-13/2}\right)=O\left( \eps^{7/6}C_q^{-13/2}\right)\ll1,\\
&\frac{ |w^2z_a| d_+^{-1/2} }{\eps^3d_{+}^{-3}}=O\left( \frac{ ( \eps^{17/6}{d}_{+}^{-3})^2  (\eps^{11/6}C_q^{-6}{d}_{+}^{-1/2}){d}_{+}^{-1/2} }{\eps^3d_{+}^{-3}}\right)=O\left(\eps^{27/6}{d}_{+}^{-4}C_q^{-6}\right)=O\left(\eps^{11/6}C_q^{-10}\right)\ll1,\\
&\frac{ |wz_a^2| d_+^{-1/2} }{\eps^3d_{+}^{-3}}=O\left( \frac{ ( \eps^{17/6}{d}_{+}^{-3})  (\eps^{11/6}C_q^{-6}{d}_{+}^{-1/2})^2{d}_{+}^{-1/2} }{\eps^3d_{+}^{-3}}\right)=O\left(\eps^{21/6}{d}_{+}^{-3/2}C_q^{-12}\right)=O\left(\eps^{5/2}C_q^{-27/2}\right)\ll1,\\
&\frac{ |z_a^4| }{\eps^3d_{+}^{-3}}=O\left( \frac{  (\eps^{11/6}C_q^{-6}{d}_{+}^{-1/2})^4 }{\eps^3d_{+}^{-3}}\right)
=O\left(\eps^{13/3}{d}_{+}C_q^{-24}\right)=O\left(\eps^{13/3}C_q^{-24}\right)\ll1.
\end{aligned}
 \end{equation*}
 
  \medskip
\underline {For $|z| \sim |z_b|=O(\eps^3{d}_{+}^{-4})$ .}
\medskip
\begin{equation*}
\begin{aligned}
&\frac{ |z_bw|  }{\eps^3d_{+}^{-3}}=O\left( \frac{ (\eps^3d_{+}^{-4}) ( \eps^{17/6}{d}_{+}^{-3}) }{\eps^3d_{+}^{-3}}\right)
=O\left(\eps^{17/6} {d}_{+}^{-4}  \right)=O\left(\eps^{1/6} C_q^{-4}  \right)\ll 1,\\
&\frac{ \eps |z_b^2| d_+^{-1/2} }{\eps^3d_{+}^{-3}}=O\left( \frac{\eps (\eps^3{d}_{+}^{-4})^2 {d}_{+}^{-1/2} }{\eps^3d_{+}^{-3}}\right)
=O\left(\eps^{4} {d}_{+}^{-11/2}  \right)=O\left(\eps^{1/3} C_q^{-11/2}\right) \ll 1,\\
&\frac{ |w^2z_b| d_+^{-1/2} }{\eps^3d_{+}^{-3}}=O\left( \frac{ ( \eps^{17/6}{d}_{+}^{-3})^2 (\eps^3 {d}_{+}^{-4}){d}_{+}^{-1/2} }{\eps^3d_{+}^{-3}}\right)=O\left(\eps^{17/3} {d}_{+}^{-15/2}  \right)=O\left(\eps^{2/3} C_q^{-15/2}  \right)\ll 1,\\
&\frac{ |wz_b^2| d_+^{-1/2} }{\eps^3d_{+}^{-3}}=O\left( \frac{ ( \eps^{17/6}{d}_{+}^{-3})  (\eps^3{d}_{+}^{-4})^2{d}_{+}^{-1/2} }{\eps^3d_{+}^{-3}}\right)=O\left(\eps^{35/6} {d}_{+}^{-17/2}  \right)=O\left(\eps^{1/6} C_q^{-17/2}  \right)\ll 1,\\
&\frac{ |z_b^4| }{\eps^3d_{+}^{-3}}=O\left( \frac{  (\eps^3{d}_{+}^{-4})^4 }{\eps^3d_{+}^{-3}}\right)
=O\left(\eps^{9}{d}_{+}^{-13}\right)=O\left(\eps^{1/3}C_q^{-13}\right)\ll 1.
\end{aligned}
 \end{equation*}
 \medskip
 
 Thus, $|\alpha_d|=\eps^3O(d_{+}^{-3})$.

\medskip
 
  {\bf Proof of estimate for $\alpha$ in (\ref{for_alpha_improved}). }
  
    We have
 \begin{equation*}
\begin{aligned}
&|\alpha|= O(|\eta|(|\eta|+|w|))  
          + \eps O(|w|^2d_{-}^{-1/2}) +\eps O(|w\eta|d_{-}^{-1/2})
     +O(|\eta| |w|^2 d_{-}^{-1/2})\\
     &+ O(|\eta|^2 |w| d_{-}^{-1/2})+O(|\eta|^3)+
     O(|w|^3) 
     +\eps^3O(d_{-}^{-3}|w|)
     +\eps^3O^*(d_{-}^{-3}).
\end{aligned}     
\end{equation*} 
and, according to Lemma \ref{lem_in_Dqr},
$$
|w(t)|< c_{r,9}(\eps^{11/6}C_q^{-6}{d}_{-}^{-1/2}+\eps^3 \hat {d}_{-}^{-4}),\
 |\eta(t)|<c_{r,10}\eps^{3}{d}_{-}^{-3}.
$$
We should prove that $|\alpha|=\eps^3O(d_{-}^{-3})$.

\medskip
The estimates can be carried out separately  for $|w| \sim |w_a|=O(\eps^{11/6}C_q^{-6}{d}_{-}^{-1/2})$ and  \\ $|w| \sim |w_b|=O(\eps^3 {d}_{-}^{-4})$.

\medskip
\newpage
\underline {For $|w| \sim |w_a|=O(\eps^{11/6}C_q^{-6}{d}_{-}^{-1/2})$ .}

\medskip
\begin{equation*}
\begin{aligned}
&\frac{ |\eta w_a|  }{\eps^3d_{-}^{-3}}
=O\left (\frac{(\eps^3d_{-}^{-3})(\eps^{11/6}C_q^{-6}{d}_{-}^{-1/2})}{\eps^3d_{-}^{-3}}\right)=O\left (\eps^{11/6}C_q^{-6}{d}_{-}^{-1/2}\right)=O\left ( \eps^{3/2}C_q^{-13/2}\right)\ll 1,\\
&\frac{\eps  | w_a|^2 d_{-}^{-1/2} }{\eps^3d_{-}^{-3}}
=O\left (\frac{\eps(\eps^{11/6}C_q^{-6}{d}_{-}^{-1/2})^2d_{-}^{-1/2} }{\eps^3d_{-}^{-3}}\right)=O\left (\eps^{5/3} d_{-}^{3/2}C_q^{-12}\right)=O\left (\eps^{5/3} C_q^{-12}\right)\ll 1,\\
&\frac{\eps  | w_a\eta|d_{-}^{-1/2}}{\eps^3d_{-}^{-3}}
=O\left (\frac{\eps(\eps^{11/6}C_q^{-6}{d}_{-}^{-1/2}\eps^3d_{-}^{-3} d_{-}^{-1/2} }{\eps^3d_{-}^{-3}}\right)=O\left (\eps^{17/6} d_{-}^{-1}C_q^{-6}\right)=O\left (\eps^{13/6} C_q^{-7}\right)\ll 1,\\
&\frac{|\eta| | w_a|^2 d_{-}^{-1/2} }{\eps^3d_{-}^{-3}}
=O\left (\frac{(\eps^3d_{-}^{-3})(\eps^{11/6}C_q^{-6}{d}_{-}^{-1/2})^2d_{-}^{-1/2} }{\eps^3d_{-}^{-3}}\right)=O\left(\eps^{11/3} d_{-}^{-3/2}C_q^{-12} \right)=O\left (\eps^{{8/3}}C_q^{-27/2}\right)\ll 1,\\
&\frac{|\eta|^2 | w_a| d_{-}^{-1/2} }{\eps^3d_{-}^{-3}}
=O\left (\frac{(\eps^3d_{-}^{-3})^2(\eps^{11/6}C_q^{-6}{d}_{-}^{-1/2})d_{-}^{-1/2} }{\eps^3d_{-}^{-3}}\right)=O\left(\eps^{29/6} d_{-}^{-4}C_q^{-6} \right)=O\left (\eps^{{5/6}}C_q^{-10}\right)\ll 1,\\
&\frac{| w_a|^3  }{\eps^3d_{-}^{-3}}
=O\left (\frac{(\eps^{11/6}C_q^{-6}{d}_{-}^{-1/2})^3 }{\eps^3d_{-}^{-3}}\right)=O\left (\eps^{5/2}d_{-}^{3/2}C_q^{-18}\right)=O\left (\eps^{5/2}C_q^{-18}\right)\ll 1.
\end{aligned}    
\end{equation*} 

\medskip
\underline {For $|w| \sim |w_b|=O(\eps^3 {d}_{-}^{-4})$. }
\medskip

\begin{equation*}
\begin{aligned}
&\frac{ |\eta w_b|  }{\eps^3d_{-}^{-3}}
=O\left (\frac{(\eps^3d_{-}^{-3})(\eps^3d_{-}^{-4})}{\eps^3d_{-}^{-3}}\right)=O\left(\eps^3d_{-}^{-4}\right)=O\left(\eps^{1/3}C_q^{-4}\right)\ll 1,\\
&\frac{\eps  | w_b|^2 d_{-}^{-1/2} }{\eps^3d_{-}^{-3}}
=O\left (\frac{\eps(\eps^3d_{-}^{-4})^2d_{-}^{-1/2} }{\eps^3d_{-}^{-3}}\right)=O\left (\eps^4 d_{-}^{-11/2}\right)=O\left (\eps^{1/3} C_q^{-11/2}\right)\ll 1,\\
&\frac{|\eta| | w_b|^2 d_{-}^{-1/2} }{\eps^3d_{-}^{-3}}
=O\left (\frac{(\eps^3d_{-}^{-3})(\eps^3 d_{-}^{-4})^2d_{-}^{-1/2}}{\eps^3d_{-}^{-3}}\right)=O\left (\eps^{6} d_{-}^{-17/2}\right)=O\left (\eps^{1/3}C_q^{-17/2}\right)\ll 1,\\
&\frac{| w_b|^3  }{\eps^3d_{-}^{-3}}
=O\left (\frac{(\eps^3d_{-}^{-4})^3 }{\eps^3d_{-}^{-3}}\right)=O\left (\eps^{6} d_{-}^{-9}\right)=O\left (C_q^{-9}\right)\ll 1.
\end{aligned}     
\end{equation*}
\medskip

Thus, $|\alpha|=\eps^3O(d_{-}^{-3})$.

\medskip

{\bf Lemma about crossing of curves $\re \Psi_{\eps}={\rm const}$ near $\Gamma_{q,r,d}$}

\medskip
The following lemma is not used in the proof of Theorem \ref{t1.th}. However, it is useful for obtaining a visual representation of the curves $\re \Psi_{\eps}={\rm const}$ near $\Gamma_{q,r,d}$ (cf. Fig. \ref{closeness_1}).  
Note that  $\Gamma_{q,r,d}'$ consists of points at distance
 $c_{e,12} \eps \hat d_u^{-1/2}|\ln( c_{e,11}^{-1}\eps \hat d_u^{-3/2})|$ from $\Gamma_{q, \eps}$, and $(\Gamma_{q,r,d}\setminus \Gamma_{q,r,d}')\subset \bar \Gamma_{q, \eps} $.
 \begin{lem}
 \label{crossing}
 The level curves $\re \Psi_{\eps}={\rm const}$ cross the part  of the curve $\Gamma_{q,r,d}'$ for which $\hat d_+< c_{e,22}/({\ln \eps)^2}$ in the upward direction as $\re \tau$ increases.
 \end{lem}
 
 \medskip
  
  {\bf Proof of Lemma \ref{crossing}.}
  
  \medskip
  For a point $\tau_u\in  \Gamma_{q, \eps}$,  we denote $\hat d_u=\hat d_+(\tau_u), 
  \Lambda_{1,u}=\Lambda_{1}(\dK_{\eps}(\tau_u)), \rho=-c_{e,12} \eps \hat d_u^{-1/2}\ln( c_{e,11}^{-1}\eps \hat d_u^{-3/2})$. Denote  $\tau_v=\tau_u-i\rho\in \Gamma_{q,r,d}'$. 
   Fix some $\tau_{u}=\tau_{u,*}\in \Gamma_{q, \eps}$ and the corresponding values 
   $\hat d_u=\hat d_{u,*},   \Lambda_{1,u}=  \Lambda_{1,*}, \rho=\rho_*$. Consider the level curve  $\re \Psi_{\eps}(\tau)={\rm const}$ passing through
   the point $\tau_{v,*}=\tau_{u,*}-i\rho_*$. Denote by $\tau_{\psi}$ the point on this curve lying below point $\tau_u$ (i.e. $\im \tau_{\psi}=\im \tau_u$). Denote $i\delta=\tau_u-\tau_{\psi}$. 
   
  We have
  \begin{equation}
  \begin{aligned}
  &\re \Psi_{\eps}(\tau_{\psi})=\re \Psi_{\eps}(\tau_{v,*})=\re \Psi_{\eps}(\tau_{u,*}-i\rho_*)
  =\re \Psi_{\eps}(\tau_{u,*})-\re(i\Lambda_{1,*})\rho_*+O(\rho_*^2 d_{u,*}^{-1/2})\\
 &=\re \Psi_{\eps}(\tau_{u,*})+\im(\Lambda_{1,*})\rho_*+O(\rho_*^2 d_{u,*}^{-1/2} ).
   \end{aligned}
   \end{equation}
   On the other hand,
    \begin{equation}
  \begin{aligned}
  &\re \Psi_{\eps}(\tau_{\psi})=\re \Psi_{\eps}(\tau_{u}-i\delta)
  =\re \Psi_{\eps}(\tau_{u})-\re(i\Lambda_{1})\delta+O(\delta^2 d_{u}^{-1/2})\\
 &=\re \Psi_{\eps}(\tau_{u})+\im(\Lambda_{1})\delta+O(\delta^2 d_{u}^{-1/2})=\re \Psi_{\eps}(\tau_{u,*})+\im(\Lambda_{1})\delta+O(\delta^2 d_{u}^{-1/2}).
   \end{aligned}
   \end{equation}
   Thus,
   \begin{equation}
  \begin{aligned}
 \re \Psi_{\eps}(\tau_{u,*})+ \im(\Lambda_{1,*})\rho_*+O(\rho_*^2 d_{u,*}^{-1/2} )=\re \Psi_{\eps}(\tau_{u,*})+\im(\Lambda_{1})\delta+O(\delta^2 d_{u}^{-1/2}).
   \end{aligned}
   \end{equation}
   In what follows, we  will proceed to the limit as  $\tau_{u}$ tends to $\tau_{u,*}$. Thus, we can replace $\rho_*$ with $\delta$, and $d_{u,*}$ with $d_{u}$ in $O(\cdot)$-symbols. We obtain
    \begin{equation}
  \begin{aligned}
\delta= \frac{ \im(\Lambda_{1,*})}{ \im(\Lambda_{1})}\rho_*+O(\rho_*^2 d_{u}^{-1} ).
  \end{aligned}
   \end{equation}
   Then
   \begin{equation}
  \begin{aligned}
 & \delta-\rho=\frac{ \im(\Lambda_{1,*})}{ \im(\Lambda_{1})}\rho_* -\rho+O(\rho_*^2 d_{u}^{-1} )
  =\left(\frac{ \im(\Lambda_{1,*})}{ \im(\Lambda_{1})} -\frac{\rho}{\rho_*}\right)\rho_*+O(\rho_*^2 d_{u}^{-1} )\\
  &=\frac{ \im(\Lambda_{1,*})}{ \im(\Lambda_{1})}
\left(1-\frac{ \im(\Lambda_{1})  \hat d_{u}^{-1/2}\ln( c_{e,11}\eps \hat d_{u}^{-3/2})}{ \im(\Lambda_{1,*})\hat d_{u,*}^{-1/2}\ln( c_{e,11}\eps \hat d_{u,*}^{-3/2})}
\right)\rho_*
+O(\rho_*^2 d_{u}^{-1}).
  \end{aligned}
   \end{equation}
   
   We have $  \im(\Lambda_{1})=\alpha(\hat d_u)\hat d_u^{1/2}+O(\hat d_u)$, where  $\alpha$ is a smooth function. Differentiate $\delta-\rho$ over $\hat d_u$ at $\hat  d_u=\hat  d_{u,*}$. Principal part of the derivative for small $\hat  d_u$ is given by differentiation of the logarithm. We have
   \begin{equation}
  \begin{aligned}
 \left( \frac{d(\delta-\rho)}{d(\hat d_u)}\right)_{\hat  d_u=\hat  d_{u,*}}
 =\frac{3}{2}\frac{\rho_*}{\hat  d_{u,*}\ln( c_{e,11}\eps \hat d_{u,*}^{-3/2})}+O(\hat  d_{u,*}^{-1/2}\rho_*)+
 O(\rho_*^2 d_{u,*}^{-2}).
   \end{aligned}
   \end{equation}
   The first term in this derivative is dominant  provided that $\hat  d_{u,*}<c_{e,22}/{(\ln \eps)^2}$ and  $C_q$ is sufficiently large. In this case, since value of the logarithm is negative, the value $(\delta-\rho)$ decreases from 0 (which is its value at  $\hat  d_u=\hat  d_{u,*}$). 
    This means that curves  $\re \Psi_{\eps}={\rm const}$ cross the curve $\Gamma_{q,r,d}'$ in the upward direction as $\re \tau$ increases.
  
   \hskip 12 cm $\square$